\documentclass[12pt]{amsart}

\makeatletter
\def\l@section{\@tocline{1}{10pt}{1pc}{}{}}
\def\l@subsection{\@tocline{2}{0pt}{1pc}{4.6em}{}}
\def\l@subsubsection{\@tocline{3}{0pt}{1pc}{7.6em}{}}
\renewcommand{\tocsection}[3]{%
  \indentlabel{\@ifnotempty{#2}{\makebox[2.3em][l]{%
    \ignorespaces#1 #2.\hfill}}}\textbf{#3}}
\renewcommand{\tocsubsection}[3]{%
  \indentlabel{\@ifnotempty{#2}{\hspace*{2.3em}\makebox[2.3em][l]{%
    \ignorespaces#1 #2.\hfill}}}#3}
\renewcommand{\tocsubsubsection}[3]{%
  \indentlabel{\@ifnotempty{#2}{\hspace*{4.6em}\makebox[3em][l]{%
    \ignorespaces#1 #2.\hfill}}}#3}
\makeatother

\usepackage{amsmath,amssymb,amscd,amsthm,graphicx,enumerate,tikz-cd}

\usepackage{hyperref}
\hypersetup{breaklinks=true}

\usepackage{enumitem} 
\newlist{condenum}{enumerate}{1} 
\setlist[condenum]{label=\bfseries Condition \arabic*.,  ref=\arabic*, wide}

\numberwithin{equation}{section}
\theoremstyle{plain}



\usepackage[utf8]{inputenc}
\usepackage{amssymb}
\usepackage{amsthm}
\usepackage{mathrsfs}
\usepackage{amsfonts}
\usepackage{amsmath}
\usepackage{mathtools}
\usepackage{makecell}
\usepackage{xcolor}
\usepackage[margin=1.25in]{geometry}

\usepackage{lipsum}
\makeatletter
\def\ps@pprintTitle{%
 \let\@oddhead\@empty
 \let\@evenhead\@empty
 \def\@oddfoot{}%
 \let\@evenfoot\@oddfoot}
\makeatother

\newcommand{\R}{\mathbb{R}}
\newcommand{\RR}{\mathbb{R}^2}

\newcommand{\ps}{\partial_s}
\newcommand{\pr}{\partial_r}

\newcommand{\LL}{\mathcal{L}}

\newcommand{\Rm}{\textnormal{Rm}}

\newcommand{\Ric}{\textnormal{Ric}}

\newcommand{\pt}{\partial_t}

\newcommand{\rii}{\rightarrow\infty}
\newcommand{\ri}{\rightarrow}

\newcommand{\cigar}{\textnormal{Cigar}}

\newcommand{\ga}{\Gamma_{\ge A}}

\newcommand{\yi}[1]{\textcolor{black}{#1}}

\numberwithin{equation}{section}

\newtheorem{theorem}{Theorem}[section]
\newtheorem{lem}[theorem]{Lemma}
\newtheorem{remark}[theorem]{Remark}
\newtheorem{prop}[theorem]{Proposition}
\newtheorem{claim}[theorem]{Claim}
\newtheorem{cor}[theorem]{Corollary} 

\theoremstyle{definition}
\newtheorem{defn}[theorem]{Definition}
\newtheorem*{theorem*}{Theorem}

\usepackage{xpatch}
\makeatletter
\xpatchcmd{\tableofcontents}{\contentsname \@mkboth}{\small\contentsname \@mkboth}{}{}
\xpatchcmd{\listoffigures}{\chapter *{\listfigurename }}{\chapter *{\small\listfigurename }}{}{}
\makeatother

\makeatletter
\def\blfootnote{\xdef\@thefnmark{}\@footnotetext}
\makeatother

\begin{document}

\begin{abstract}
We prove that all 3D steady gradient Ricci solitons are O(2)-symmetric. The O(2)-symmetry is the most universal symmetry in Ricci flows with any type of symmetries. Our theorem is also the first instance of symmetry theorem for Ricci flows that are not rotationally symmetric.
We also show that the Bryant soliton is the unique 3D steady gradient Ricci soliton with positive curvature that is asymptotic to a ray. 
\end{abstract}

\blfootnote{The author was supported by the NSF grant DMS-2203310.}

\title[O(2)-symmetry of 3D steady gradient Ricci solitons]{O(2)-symmetry of 3D steady gradient Ricci solitons}

\author[Yi Lai]{Yi Lai}
\email{yilai@stanford.edu}
\address[]{Department of Mathematics, Stanford University, CA 94305, USA}

\maketitle

\tableofcontents

\begin{section}{Introduction}\label{s: intro}
\subsection{Statement of the main results}
The concept of Ricci solitons was introduced by Hamilton \cite{cigar}. Ricci solitons generate self-similar solutions of Hamilton's Ricci flow \cite{Hamilton_ric}, and often arise as singularity models of Ricci flows \cite{supR,distanceH,Cao-Kahler,Chen-Zhu-Pinch}. 
They can be viewed as the fixed points under the Ricci flow in the space of Riemannian metrics modulo rescalings and diffeomorphisms. 
Ricci solitons are also natural generalizations of the Einstein metrics and constant curvature metrics. 

A complete Riemannian manifold $(M,g)$ is called a Ricci soliton, if there exist a vector field $X$ and a constant $\lambda\in\R$ such that
\begin{equation*}
    \Ric=\frac{1}{2}\LL_Xg+\lambda\,g.
\end{equation*}
The soliton is called \textit{shrinking} if $\lambda>0$, \textit{expanding} if $\lambda<0$, and \textit{steady} if $\lambda=0$. Moreover, if the vector field $X$ is the gradient of some smooth function $f$, then we say it is a \textit{gradient Ricci soliton}, and $f$ is the potential function.
In particular, a steady gradient Ricci soliton satisfies the equation
\begin{equation*}
    \Ric=\nabla^2 f.
\end{equation*}
The goal of this paper is to study steady gradient Ricci solitons in dimension 3.  We assume they are non-flat.

In dimension 2, the only steady gradient Ricci soliton is Hamilton's cigar soliton, which is rotationally symmetric \cite{cigar}. In dimension $n\ge3$, Bryant constructed a steady gradient Ricci soliton which is rotationally symmetric \cite{bryant}. See \cite{CaoHD,FIK,Lai2020_flying_wing} for more examples of Ricci solitons in dimension $n\ge 4$.

In dimension 3, we know that all steady gradient Ricci solitons are non-negatively curved \cite{ChenBL}, and they
are asymptotic to sectors of angle  $\alpha\in[0,\pi]$.
In particular, the Bryant soliton is asymptotic to a ray ($\alpha=0$), and the soliton $\R\times\cigar$ is asymptotic to a half-plane ($\alpha=\pi$).
If the soliton has positive curvature, it must be diffeomorphic to $\R^3$ \cite{petersen}, and asymptotic to a sector of angle in $[0,\pi)$ \cite{Lai2020_flying_wing}.
If the curvature is not strictly positive, then it is a metric quotient of $\R\times\cigar$ \cite{MT}.

Hamilton conjectured that there exists a 3D steady gradient Ricci soliton that is asymptotic to a sector with angle in $(0,\pi)$, which is called a \textit{3D flying wing} \cite{CaoHD,infinitesimal,Catino,Chow2007a,DZ,HaRF}.
The author confirmed this conjecture by constructing a family of $\mathbb{Z}_2\times O(2)$-symmetric 3D flying wings \cite{Lai2020_flying_wing}. More recently, the author showed that the asymptotic cone angles of these flying wings can take arbitrary values in $(0,\pi)$. 
It is then interesting to see whether a 3D steady gradient Ricci soliton with positive curvature must be either a flying wing or the Bryant soliton.
This is equivalent to ask whether the Bryant soliton is the unique 3D steady gradient Ricci soliton with positive curvature that is asymptotic to a ray. Our first main theorem gives an affirmative answer to this.

\begin{theorem}[Uniqueness theorem]\label{t: must look like a flying wing}
Let $(M,g)$ be a 3D steady gradient Ricci soliton with positive curvature. If $(M,g)$ is asymptotic to a ray, then it must be isometric to the Bryant soliton up to a scaling.
\end{theorem}

We mention that there are many other uniqueness results for the 3D Bryant soliton under various additional assumptions.
First, Bryant showed in his construction that the Bryant soliton is the unique rotationally symmetric steady gradient Ricci solitons \cite{bryant}.
More recently, a well-known theorem by Brendle proved that the Bryant soliton is the unique steady gradient Ricci soliton that is non-collapsed in dimension 3 \cite{brendlesteady3d}. See also \cite{DZ,Cao2009OnLC,Cao2011BachflatGS,Chen2011OnFA,Munteanu2019PoissonEO,Catino} for more uniqueness theorems for the Bryant soliton and Cigar soliton. 

Our Theorem \ref{t: must look like a flying wing} is the Ricci flow analogue of X.J. Wang's well-known theorem in mean curvature flow, which proves that the bowl soliton is the unique entire convex graphical translator in $\R^3$ \cite{Wangxujia}. 
Note that the analogue of 3D steady Ricci solitons in mean curvature flow are convex translators in $\R^3$, where the rotationally symmetric solutions are called bowl solitons.
Moreover, a 3D steady Ricci soliton asymptotic to a ray can be compared to a convex graphical translator whose definition domain is the entire $\R^2$.

There have been many exciting symmetry theorems in geometric flows \cite{Huisken2015ConvexAS,Brendle2011AncientST,angenent2022unique,Bamler2021OnTR,Brendle_jdg_high,zhu2022so,Zhu2021RotationalSO,Bourni_jdg,Bourni_convex,Brendle2019UniquenessOC,Brendle2021OC,brendle2023rotational,BrendleNaffDasSesum,du2021hearing}.
If one views the rotational symmetry as the `strongest' symmetry, then the $O(2)$-symmetry is naturally the `weakest', and the most universal symmetry in all ancient Ricci flow solutions.
For example, in dimension 2, the non-flat ancient Ricci flows are the shrinking sphere, the cigar soliton, and the sausage solution \cite{2dancientcompact,Sesum}, and they are all rotationally symmetric (i.e. $O(2)$-symmetric).
In dimension 3, the author's flying wing examples, the Fateev's examples \cite{Fateev} (see also \cite{Bakas2009AncientSO}) are all $O(2)$-symmetric but not rotationally symmetric (i.e. $O(3)$-symmetric).

It was conjectured by Hamilton and Cao that the 3D flying wings are $O(2)$-symmetric.
Our second main theorem confirms this conjecture.
In particular, this is the first instance of a symmetry theorem for Ricci flows that are not rotationally symmetric.

\begin{theorem}\label{t: symmetry of flying wing}
Let $(M,g)$ be a 3D flying wing, then $(M,g)$ is $O(2)$-symmetric. 
\end{theorem}

Here we say a complete 3D manifold is $O(2)$-symmetric if it admits an effective isometric $O(2)$-action, and the action fixes a complete geodesic $\Gamma$, such that the metric is a warped product metric on $M\setminus\Gamma$ with $S^1$-orbits. 
It is easy to see the Bryant soliton and $\R\times\cigar$ are also $O(2)$-symmetric.
Therefore, combining Theorem \ref{t: must look like a flying wing} and \ref{t: symmetry of flying wing}, we see that all 3D steady gradient Ricci solitons are $O(2)$-symmetric. 

\begin{theorem}\label{t: symmetry of all solitons}
Let $(M,g)$ be a 3D steady gradient Ricci soliton, then $(M,g)$ is $O(2)$-symmetric. 
\end{theorem}

In mean curvature flow, the `weakest' symmetry is the $\mathbb{Z}_2$-symmetry, which are usually obtained using the standard maximum principle method. 
More precisely, if we compare 3D steady gradient Ricci solitons with convex translators in $\R^3$,
then the $O(2)$-symmetry is compared with the $\mathbb{Z}_2$-symmetry (reflectional symmetry) there.
The convex translators in $\R^3$ have been classified to be the tilted Grim Reapers, the flying wings, and the bowl soliton, all of which are $\mathbb{Z}_2$-symmetry \cite{white}.
However, as its analogy in Ricci flow, the $O(2)$-symmetry is not `discrete' at all, and no maximum principle is available.


We also obtain some geometric properties for the 3D flying wings.
First, we show that the scalar curvature $R$ always attains its maximum at some point, which is also the critical point of $f$.
The analogue of this statement in mean curvature flow is that the graph of the convex translator has a maximum point, which relies on the well-known convexity theorem by Spruck-Xiao \cite{Spruck2020CompleteTS}.

\begin{theorem}\label{t': max}
Let $(M,g,f)$ be a 3D steady gradient Ricci soliton with positive curvature, then there exists $p\in M$ which is a critical point of the potential function $f$, and the scalar curvature $R$ achieves the maximum at $p$.
\end{theorem}

We study the asymptotic geometry of 3D flying wings. First, we show that the soliton is $\mathbb{Z}_2$-symmetric at infinity, in the sense that the limits of $R$ at the two ends of $\Gamma$ are equal to a same positive number. Here $\Gamma$ is a complete geodesic fixed by the $O(2)$-isometry.
After a rescaling we may assume this positive number is $4$, then we show that there are two asymptotic limits, one is 
$\R\times\cigar$ with $R(x_{tip})=4$, and the other is $\RR\times S^1$ with the diameter of the $S^1$-factor equal to $\pi$. Note that in a cigar soliton where $R=4$ at the tip, the diameter of the $S^1$-fibers in the warped-product metric converges to $\pi$ at infinity.
See \cite{distanceH,Lu_Wang_Ricci_shrinker,De15,chow2022four,Lai2020_flying_wing} for more discussions on the asymptotic geometry of Ricci solitons.

\begin{theorem}[$\mathbb{Z}_2$-symmetry at infinity]\label{t: R_1=R_2}
Let $(M,g,f)$ be a 3D flying wing. Then after a rescaling we have
\begin{equation*}
    \lim_{s\rii}R(\Gamma(s))=\lim_{s\ri-\infty}R(\Gamma(s))=4.
\end{equation*}
For any sequence of points $p_i\rii$, the pointed manifolds $(M,g,p_i)$ smoothly converge to either $\R\times\cigar$ with $R(x_{tip})=4$ or $\RR\times S^1$ with the diameter of the $S^1$-factor equal to $\pi$. Moreover, if $p_i\in\Gamma$, then the limit is $(\R\times\cigar,x_{tip})$, and if $d_g(\Gamma,p_i)\rii$, then the limit is $\R^2\times S^1$.
\end{theorem}

We also obtain a quantitative relation between the limit of $R$ along $\Gamma$, the asymptotic cone angle, and $R(p)$ where $p$ is the critical point of $f$.
This is also true in the Bryant soliton, and thus is true for all 3D steady gradient Ricci solitons with positive curvature.

\begin{theorem}\label{t': quantitative relation}
Let $(M,g,f,p)$ be a 3D steady gradient Ricci soliton with positive curvature.
Assume $(M,g)$ is asymptotic to a sector with angle $\alpha$, then we have 
\begin{equation*}
    \lim_{s\rii}R(\Gamma(s))=\lim_{s\ri-\infty}R(\Gamma(s))=R(p)\sin^2\frac{\alpha}{2}.
\end{equation*}
\end{theorem}

It has been conjectured whether there is a dichotomy of the curvature decay rate of steady gradient solitons, that is, the curvature decays either exactly linearly or exponentially \cite{Munteanu2019PoissonEO,DZ,chan2022dichotomy,Chan2019CurvatureEF}.
In dimension 3, the curvature of Bryant soliton decays linearly in the distance to the tip, and the curvature in $\R\times\cigar$ decays exponentially in distance to the line of cigar tips. 
In this paper, we prove that in a 3D flying wing, the curvature decays faster than any polynomial function in $r$, and slower than an exponential function in $r$, where $r$ is the distance function to $\Gamma$.

\begin{theorem}\label{t': curvature estimate}
Let $(M,g,f,p)$ be a 3D flying wing. Suppose $\lim_{s\rii}R(\Gamma(s))=4$. Then for any $\epsilon_0>0$ there exists $C(\epsilon_0)>0$, and for any $k\in\mathbb{N}$ there exists $C_k>0$ such that the following holds for all $x\in M$,
\begin{equation*}
    C^{-1}e^{-2(1+\epsilon_0)\,d_g(x,\Gamma)}\le R(x)\le C_k\,d_g^{-k}(x,\Gamma).
\end{equation*}
\end{theorem}

\subsection{Outline of difficulties and proofs} 
In Ricci flow, the `strongest' symmetry, i.e. the rotational symmetry was first studied by Brendle in dimension 3 with many novel ideas that are successfully generalize to prove rotational symmetry in higher dimensions \cite{brendlesteady3d,Brendle_jdg_high,brendle2023rotational,BrendleNaffDasSesum}. In contrast to the rotational symmetry,
one of the major difficulties in studying any weaker symmetries is the non-uniqueness of asymptotic limits. 
For the $O(2)$-symmetry, this requires us to study the two different asymptotic limits $\mathbb{R}\times\cigar$ and $\R^2\times S^1$ separately and glue up these estimates in a delicate way.
Our $O(2)$-symmetry theorem is the first instance of tackling this issue in Ricci flow. 
Our method may be generalized to study the $O(n-k)$-symmetries, $k=1,...,n-2$, which are weaker than the rotational symmetry (i.e. $O(n)$-symmetry). For example, the author constructed $n$-dimensional steady solitons that are non-collapsed, $O(n-1)$-symmetric but not $O(n)$-symmetric \cite{Lai2020_flying_wing}.

Therefore, we need to develop new tools and methods to prove the $O(2)$-symmetry.
Some of our methods are independent of the soliton structure and the dimension, such as the distance distortion estimates, the curvature estimates, and the symmetry improvement theorem, which should have more applications.
In particular, as a consequence of the non-uniqueness issue of limits, Brendle's construction of the killing field is not applicable in our setting.
So we introduce a new stability method to construct the killing field (see more in Section \ref{s: lie} and \ref{s: killing field}).
Our stability method generalizes Brendle's method (see \cite[Lemma 4.1]{brendlesteady3d}) from steady solitons to flows, in the sense that it is much less restricted by the steady soliton structure, and hence may be applied in more general settings. 
For example, our method was recently applied in \cite{ZhaoZhu} to study the rigidity of the non-collapsed Bryant soliton in any dimension $n\ge 4$. 

We now outline the structure of the paper.
In the following we assume $(M,g,f)$ is a 3D steady gradient Ricci soliton that is not the Bryant soliton.
Let $\{\phi_t\}_{t\in\R}$ be the diffeomorphisms generated by $\nabla f$, $\phi_0=id$, and let $g(t)=\phi_{-t}^*g$. Then $(M,g(t))$, $t\in(-\infty,\infty)$, is the Ricci flow of the soliton.

\textbf{Section \ref{s: pre}} We give most of the definitions and standard Ricci flow results that will be used in the following proofs.

\textbf{Section \ref{s: asymptotic geometry}} We study the asymptotic geometry of the soliton in this section. First, by the splitting theorem of 3D Ricci flow we can show that 
for any sequence of points $x_i\in M$ going to infinity as $i\rii$, the rescaled manifolds $(M,r^{-2}(x_i)g,x_i)$ converge to a smooth limit which splits off a line, where $r(x_i)>0$ is some non-collapsing scale at $x_i$.
We show that such an asymptotic limit is either isometric to $\R\times\cigar$ or $\R^2\times S^1$. 
Moreover, we find two integral curves $\Gamma_1,\Gamma_2$ of $\nabla f$ tending to infinity at one end, such that the asymptotic limits are isometric to $\R\times\cigar$ along them, and are $\RR\times S^1$ away from them. 
The two integral curves also correspond to the two edge rays in the sector which is the blow-down limit of the soliton.

Second, we prove Theorem \ref{t': max} of the existence of the maximum point of $R$.
This is also the unique critical point of $f$.
We do this by a contradicting argument. Suppose there does not exist a maximum of $R$. Then we can find an integral curve $\gamma$ of $\nabla f$ which goes to infinity at both ends. We show that $R$ is non-increasing along the curve and has a positive limit at one end.
Using that the asymptotic limits along both ends of $\gamma$ are isometric to $\R\times\cigar$, we can compare the geometry at the two ends, and by a convexity argument we can show that $R$ is actually constant along $\gamma$, so that the soliton is isometric to $\R\times\cigar$. This gives a contradiction to our positive curvature assumption.
So we have a closed subset $\Gamma$, which is the union of the critical point and two integral curves of $\nabla f$, such that $\Gamma$ is invariant under the diffeomorphisms generated by $\nabla f$, and the soliton converges to $(\R\times\cigar,x_{tip})$ under rescalings along the two ends of $\Gamma$.

Next, we prove a quadratic curvature decay away from the edge $\Gamma$.
This corresponds to the case when $k=2$ in Theorem \ref{t': curvature estimate}.  
The proof uses Perelman's curvature estimate, which gives the upper bound $R(x,0)\le\frac{C}{r^2}$ on scalar curvature in a non-negatively curved Ricci flow $(M,g(t))$, $t\in[-r^2,0]$, assuming the flow is non-collapsed at $x$ on scale $r$.
In our situation,  we will show by methods of metric comparison geometry that the soliton is not non-collapsed at $x$ on scale $d_g(x,\Gamma)$ in a local universal covering. So we can apply Perelman's estimates on the local universal covering and obtain the desired quadratic decay $R(x)\le\frac{C}{d_g^2(x,\Gamma)}$.

Lastly, we prove Theorem \ref{t: must look like a flying wing} and \ref{t: R_1=R_2}.
In proving the two theorems, we will work in the backwards Ricci flow $(M,g(-\tau))$, $\tau\ge0$, and reduce the change of various geometric quantities to the distortion of distances and lengths under the flow.
More specifically,
for a fixed point $x\in M$ at which $(M,g(-\tau))$ is close to $\RR\times S^1$ on scale $h(\tau)$. Let $H(\tau)$ be the 
$g(-\tau)$-distance from $x$ to $\Gamma$. Then we can show 
\begin{equation*}
    \begin{cases}
    H'(\tau)\ge C^{-1}\cdot h^{-1}(\tau)\\
    h'(\tau)\le C\cdot H^{-2}(\tau)\cdot h(\tau),
    \end{cases}
\end{equation*}
which guarantees that $h(\tau)$ stays bounded, and $H(\tau)$ grows at least linearly as $\tau\rii$.
Using this we can show that $R$ has two positive limits $R_1,R_2$ at the two ends of $\Gamma$, and the asymtotic cone angle is non-zero.
To show $R_1=R_2$, first we can find two points $x_1,x_2$ at which the soliton $(M,g)$ is $\epsilon(\tau)$-close to $\RR\times S^1$, on the scales $2R_1^{-1/2}$ and $2R_2^{-1/2}$. Here $\epsilon(\tau)\ri0$ as $\tau\rii$.
Then we can show that $x_1,x_2$ stay in a bounded distance to each other as we move backwards long the Ricci flow, and hence $(M,g(-\tau))$ is $\epsilon(\tau)$-close to $\RR\times S^1$ at $x_1,x_2$ on a uniform scale.
So $R_1=R_2$ follows by controlling the scale change at $x_1,x_2$ in the flow.
Theorem \ref{t: R_1=R_2} is a key ingredient in proving the $O(2)$-symmetry theorem.

\textbf{Section \ref{s: curvature etimates}} We prove Theorem \ref{t': curvature estimate} of the curvature estimates in this section.
It is needed in the proof of the $O(2)$-symmetry.
First, we derive the exponential curvature lower bound of $R$, which needs an improved Harnack inequality for non-negatively curved Ricci flows.
For a Ricci flow solution with non-negative curvature operator, the following conventional integrated Harnack inequality can be obtained by integrating Hamilton's differential Harnack inequality and use the inequality $\Ric(v,v)\le|v|^2R$,
\begin{equation*}
    \frac{R(x_2,t_2)}{R(x_1,t_1)}\ge \exp\left(-\frac{1}{2}\frac{d^2_{g(t_1)}(x_1,x_2)}{t_2-t_1}\right),
\end{equation*}
see for example \cite[Theorem 4.40]{MT}.
We observe that the inequality $\Ric(v,v)\le|v|^2R$ can be improved to $\Ric(v,v)\le\frac{1}{2}|v|^2R$, using which we can prove the following improved Harnack inequality
\begin{equation*}
    \frac{R(x_2,t_2)}{R(x_1,t_1)}\ge \exp\left(-\frac{1}{4}\frac{d^2_{g(t_1)}(x_1,x_2)}{t_2-t_1}\right).
\end{equation*}

Using this improved Harnack inequality and some distance distortion estimates we can prove the exponential curvature lower bound. 
Note that
this exponential lower bound $C^{-1}(\epsilon_0)e^{-(2+\epsilon_0)\,d_g(\cdot,\Gamma)})$ is
sharp, because $\epsilon_0$ can be arbitrarily small, and in $\R\times\cigar$ where $R=4$ at the cigar tip, we have $\Gamma=\R\times\{x_{tip}\}$ and $R$ decays like $e^{-2\,d_g(\cdot,\Gamma)}$.

Next, we derive the polynomial upper bound of $R$, which states that $R$ decays faster than $d^{-k}_g(\cdot,\Gamma)$, for any $k\in\mathbb{N}$. 
We prove this by induction.
First, the case of $k=2$ is proved in Section \ref{s: asymptotic geometry}.
Now assume by induction that $R\le C_k\, d^{-k}_g(\cdot,\Gamma)$ for some $k\ge 2$.
Since $R$ evolves by $\partial_t R=\Delta R+2|\Ric|^2(x,t)$ under the Ricci flow, we have the following reproduction formula for all $s<t$,
\begin{equation*}\begin{split}
    R(x,t)&=\int_M G(x,t;y,s)R(y,s)\,d_sy+2\int_s^t\int_M G(x,t;z,\tau)|\Ric|^2(z,\tau)\,d_{\tau}z\,d\tau,
\end{split}\end{equation*}
where $G$ is the heat kernel of the heat equation $\pt u=\Delta u$. 
Using a heat kernel estimate on $G$ and the inductive assumption, we can show the first term goes to zero as we choose $s\ri-\infty$,
and the second term is bounded by $C\cdot d_{g(t)}^{-(2k-1)}(x,\Gamma)$.
Note that $k\ge 2$ implies $2k-1>k$, this completes the induction process.

\textbf{Section \ref{s: semi-local}} 
In this section we prove a local stability theorem, which is another key ingredient of the $O(2)$-symmetry theorem.
It states that the degree of $SO(2)$-symmetry improves as we move forward in time along the Ricci flow of the soliton. Here $SO(2)$-symmetric means that the manifold admits an isometric $SO(2)$-action whose principal orbits are circles.

First, we prove the symmetry improvement theorem in the linear case.
For a symmetric 2-tensor $h$ on a $SO(2)$-symmetric manifold, it has a decomposition as a sum of a rotationally invariant mode and an oscillatory mode. 
We show that if $h$ which satisfies the linearized Ricci-Deturck flow $\partial_t h=\Delta_L h$ on the cylindrical plane $\RR\times S^1$, then the oscillatory mode of $h$ decays exponentially in time. By a limiting argument we generalize this theorem to the non-linear case for the Ricci-Deturck flow perturbation, whose background is a $SO(2)$-symmetric Ricci flow that is sufficiently close to $\R^2\times S^1$. 

Moreover, the symmetry improvement theorem also describes the decay of $|h|$, in the case that it is bounded by an exponential function instead of a constant. More precisely, for $x_0\in M$, if $|h|(\cdot,0)\le e^{\alpha\,d_g(x_0,\cdot)}$ for any $\alpha\in[0,2.02]$, then $|h|(x_0,T)\le e^{-\delta_0 \,T}\cdot e^{2\alpha\,T}$ holds for some $\delta_0>0$. Applying the theorem to a 3D flying wing in which $R$ limits to $4$ along the edges, the increasing factor $e^{2\alpha\,T}$ will be compensated by the cigar tip contracting along the edges under the Ricci flow.
It is crucial that $\alpha$ can be slightly greater than $2$, using which we can construct a $SO(2)$-symmetric approximating metric in Section \ref{s: Approximating $SO(2)$-symmetric metrics}, so that the error decays like $e^{-(2+\delta)\,d_g(\cdot,\Gamma)}$ for some small but positive $\delta$. So the error decays faster than that of $R$ by the exponential lower bound in Theorem \ref{t': curvature estimate}.
We will use this fact to construct a killing field in Section \ref{s: lie} and \ref{s: killing field}.

\textbf{Section \ref{s: Approximating $SO(2)$-symmetric metrics}} In this section we construct an approximating $SO(2)$-symmetric metric $\overline{g}$ satisfying suitable error estimates. First, we construct a $SO(2)$-symmetric metric $\overline{g}_1$ away from $\Gamma$, which satisfies
\begin{equation}\label{e:  g_2}
    |\overline{g}_1-g|_{C^{100}}\le  e^{-(2+\epsilon_0)\,d_g(\cdot,\Gamma)},
\end{equation}
for some $\epsilon_0>0$.
To show this, we impose the following inductive assumption.

\textbf{Inductive assumption one:}
There are a constant $\delta\in(0,0.01)$ and an increasing arithmetic sequence $\alpha_n>0$, with $\delta\le\alpha_{n+1}-\alpha_n\le0.01$, and if $\alpha_n\le2.02$, then
there is a $SO(2)$-symmetric metric $\widehat{g}_{n}$ such that
\begin{equation}\label{e: intoinduction on n}
    |\widehat{g}_{n}-g|_{C^{100}}\le  e^{-\alpha_nd_g(\cdot,\Gamma)}.
\end{equation}
If this is true for all $n\in\mathbb{N}$, then $\widehat{g}_N$ will satisfy \eqref{e:  g_2} for a large enough $N\in\mathbb{N}$.

Now assume inductive assumption one holds for $n$, to show it also holds for $n+1$, 
we want to apply the symmetry improvement theorem to the Ricci flow of the soliton.
After applying the symmetry improvement theorem $i$ times, the error to a symmetric metric will decay by $C^{-i}$ for some $C>0$. 
So for points at larger distance to $\Gamma$, we need to apply the symmetry improvement theorem more times to achieve the error estimate $e^{-\delta\,d_g(\cdot,\Gamma)}$.
Therefore, we need a second induction to apply the symmetry improvement theorem infinitely many times so that eventually the error estimate $e^{-\delta\,d_g(\cdot,\Gamma)}$ holds everywhere. 

\textbf{Inductive assumption two:} There is a sequence of $SO(2)$-symmetric metrics $\{\widehat{g}_{n,k}\}_{k=1}^{\infty}$ such that $\widehat{g}_{n,k}$
satisfies \eqref{e: intoinduction on n}, and for some $C>0$ we have
\begin{equation}\begin{split}\label{e: intoinduction on k}
|\widehat{g}_{n,k}-g|_{C^{100}}\le   e^{-\alpha_{n} d_g(\cdot,\Gamma)}\cdot C^{-i}\quad \textit{on}\quad\Gamma_{\ge iD},\quad i=0,...,k,
\end{split}\end{equation}
where $\Gamma_{\ge iD}=\{x\in M:d_g(x,\Gamma)\ge iD\}$.
If inductive assumption two is true for all $k\in\mathbb{N}$, we take $\widehat{g}_{n+1}$ to be a subsequential limit of $\widehat{g}_{n,k}$ as $k\rii$, then $\widehat{g}_{n+1}$ satisfies \eqref{e: intoinduction on n} for $n+1$.

Inductive assumption two clearly holds for $k=0$ by taking $\widehat{g}_{n,0}=\widehat{g}_n$ and using inductive assumption one for $n$.
Now assume it holds for $k\ge0$, we verify it for $k+1$ by applying the symmetry improvement theorem.
More precisely, we consider the harmonic map heat flow from $(M,g(t))$ to the Ricci flow $\widehat{g}_{n,k}(t)$ starting from $\widehat{g}_{n,k}$ on $[0,T]$.
Then the error between $g(t)$ and $\widehat{g}_{n,k}(t)$ is then described by the Ricci Deturck flow perturbation. 
Let $\widehat{g}_{n,k+1}$ be the final-time metric $\widehat{g}_{n,k}(T)$ modulo the rotationally invariant part of the error and a diffeomorphism.
Since oscillatory part of the error decays exponentially in time by the symmetry improvement theorem, we can show that $\widehat{g}_{n,k+1}$ satisfies \eqref{e: intoinduction on k}.
This completes the two inductions and hence we obtain a $SO(2)$-symmetric metric $\overline{g}_1$ satisfying \eqref{e:  g_2}.

Lastly, we modify the metric $\overline{g}_1$ to obtain the desired approximating $SO(2)$-symmetric metric $\overline{g}$, which satisfies both \eqref{e: g_2} and 
\begin{equation}\label{e: g3}
    |\overline{g}-g|_{C^{100}}(x)\ri0\quad\textit{as}\quad x\rii.
\end{equation}
Note that $\overline{g}_1$ already satisfies \eqref{e: g3} as we move away from $\Gamma$, we just need to extend this estimate near $\Gamma$. 
Since the soliton converges along the two ends of $\Gamma$ to $\R\times\cigar$ which is $SO(2)$-symmetric, we can obtain $\overline{g}$ by gluing up $\overline{g}_1$ with $\R\times\cigar$ in suitable neighborhoods of the two ends of $\Gamma$.

\textbf{Section \ref{s: lie} } 
The goal of Section \ref{s: lie} and \ref{s: killing field} is to construct a non-trivial killing field of the soliton. 
We do this by a global stability argument using a heat kernel method, which is consistent with our curvature estimates and estimates of approximating metrics.

In this section, we study the solution to the following initial value problem of the linearized Ricci-Deturck flow equation,
\begin{equation}\begin{cases}\label{e: special heat equation}
    \pt h(t)=\Delta_{L,g(t)}h(t),\\
    h(0)=\LL_Xg,
\end{cases}\end{equation}
where $X$ is the killing field of the approximating $SO(2)$-symmetric metric obtained in Section \ref{s: Approximating $SO(2)$-symmetric metrics}.
By the conditions \eqref{e: g_2} and \eqref{e: g3}, and the exponential lower bound from Theorem \ref{t': curvature estimate}, we can deduce that $|\frac{\LL_Xg}{R}|(x)\ri0$ as $x\rii$.
We show that $|h(t)|\ri0$ as $t\rii$.

To prove this, we first observe by the Anderson-Chow's curvature pinching \cite{AC} that $|h|$ satisfies the following inequality 
\begin{equation}\label{e: think}
    |h(x,t)|\le\int_M G(x,t;y,0)|h|(y,0)\,d_0y,
\end{equation}
where $G(x,t;y,s)$ is the heat kernel to the following heat type equation,
\begin{equation}\label{e: ok}
    \partial_t u=\Delta u+\frac{2|\Ric|^2(x,t)}{R(x,t)}u.
\end{equation}
Our key estimate is to show a vanishing theorem of the heat kernel $G(x,t;y,s)$ for any fixed pair $(y,s)$ at $t=\infty$.
Using this vanishing theorem we can show that the integral in \eqref{e: think} in any compact subset is arbitrarily small when $t\rii$.
For the integral outside the compact subset, by the initial condition it is an arbitrary small multiple of $R$ integrated against the heat kernel $G$, which is bounded by the maximum of $R$ seeing that $R$ is also a solution to \eqref{e: ok}.

\textbf{Section \ref{s: killing field} } 
In this section we construct a killing field of the soliton metric.
Let $X$ be the killing field of the approximating $SO(2)$-symmetric metric obtained in Section \ref{s: Approximating $SO(2)$-symmetric metrics}.
Let $(M,g(t))$ be the Ricci flow of the soliton. Let $Q(t)=\pt \phi_{t*}(X)-\Delta_{g(t)} \phi_{t*}(X)-\Ric_{g(t)}(\phi_{t*}(X))$ and $Y(t)$ be a time-dependent vector field which solves the equation
\begin{equation*}
\begin{cases}
    \pt Y(t)-\Delta Y(t)-\Ric(Y(t))=Q(t),\\
    Y(0)=0.
\end{cases}
\end{equation*}
Moreover, let $X(t):=\phi_{t*}(X)-Y$. Then $X(t)$ solves the following initial value problem,
\begin{equation*}
\begin{cases}
    \partial_t X(t)-\Delta X(t)-\Ric(X(t))=0,\\
    X(0)=X,
\end{cases}
\end{equation*}
and the symmetric 2-tensor field $\LL_{X(t)}g(t)$ satisfies the equation \eqref{e: special heat equation}.
Therefore, by the result from Section \ref{s: lie} we see that $\LL_{X(t)}g(t)$ tends to zero as $t\rii$. So the limit of $X(t)$ as $t\rii$ is a killing field of $(M,g)$.

To show that the killing field is non-zero, we first show that $Q(t)$ satisfies a polynomial decay away from $\Gamma$ as a consequence of \eqref{e: g_2} and the polynomial curvature upper bound from Theorem \ref{t': curvature estimate}.
Then by some heat kernel estimates on $|Y(t)|$ we show that it also satisfies the polynomial decay away from $\Gamma$, which guarantees the non-vanishing of the limit of $X(t)$ as $t\rii$.

\textbf{Section \ref{s: O(2)-symmetry} } We prove Theorem \ref{t: symmetry of flying wing} of the $O(2)$-symmetry in this section. 
First, let $X$ be the killing field obtained in Section \ref{s: killing field}, and $\chi_{\theta}$, $\theta\in\R$, be the isometries generated by $X$. We show that $\chi_{\theta}$ are commutative with the diffeomorphisms $\phi_t$ generated by $\nabla f$. Then we show that $\chi_{\theta}$ is a $SO(2)$-isometry. This uses the existence of a maximum point of $R$, which must be fixed by the isometries $\chi_{\theta}$.
Since the maximum point of $R$ is also a critical point of $f$, it follows that $f$ is invariant under the isometries.
Using this we can show that $\chi_{\theta}$ is a $SO(2)$-isometry and fixes the edge $\Gamma$.

Lastly, in order to show that the soliton is also $O(2)$-symmetric, it remains to show that the curvature form of the $SO(2)$-isometry vanishes everywhere. By using the soliton equation and the curvature formula under the $SO(2)$-isometry, we can reduce this to the vanishing of a scaling invariant quantity at a point on $\Gamma$. By a limiting argument and the scaling invariance, this can be further reduced to the Euclidean space $\R^3$ where the $SO(2)$-isometry is the rotation around the $z$-axis, and hence the vanishing assertion clearly holds.

\subsection{Acknowledgments}
I would like to thank Richard Bamler for his constant help and encouragement, and Otis Chodosh, Huai-Dong Cao, Bennett Chow, Xiaohua Zhu, Brian White for helpful discussions and comments.

\end{section}

\section{Preliminaries}\label{s: pre}
In the following we present most of the definitions and concepts that are needed in the statement and proofs of the main results of this paper.

\subsection{Steady gradient Ricci solitons}
\begin{defn}[Steady Ricci soliton]
We say a smooth complete Riemannian manifold $(M,g)$ is a steady Ricci soliton if it satisfies
\begin{equation}\label{e: soliton}
    \Ric=\LL_X g,
\end{equation}
for some smooth vector field $X$. If moreover the vector field is the gradient of some smooth function $f$, then we say it is a steady gradient Ricci soliton, and $f$ is the potential function. In this case, the soliton satisfies the equation
\begin{equation*}
    \Ric=\nabla^2 f.
\end{equation*}
By a direct computation using \eqref{e: soliton}, the family of metrics $g(t)=\phi_t^*(g)$, $t\in(-\infty,\infty)$, satisfies the Ricci flow equation, where $\{\phi_t\}_{t\in(-\infty,\infty)}$ is the one-parameter group of diffeomorphisms generated by $-\nabla f$ with $\phi_0$ the identity. We say $g(t)$ is the Ricci flow of the soliton.
\end{defn}

Throughout the paper, we use the triple $(M,g,f)$ to denote a steady gradient soliton $(M,g)$ and a potential function $f$, and use the quadruple $(M,g,f,p)$ to denote the soliton when $p\in M$ is a critical point of $f$.

For 3D steady gradient Ricci solitons, by the maximum principle they must have non-negative sectional curvature \cite{ChenBL}. Moreover, by the strong maximum principle, see e.g. \cite[Lemma 4.13, Corollary 4.19]{MT}, we see that a 3D steady gradient Ricci soliton must be isometric to quotients of $\R\times\cigar$ if the curvature is not strictly positive everywhere. 
Therefore, throughout the paper we will assume our soliton has positive curvature. So by the soul theorem, the manifold is diffeomorphic to $\R^3$, see e.g. \cite{petersen}.

There are several important identities for the steady gradient Ricci solitons due to Hamilton, see e.g. \cite{HaRF}. In particular, we will use frequently 
\begin{equation}\label{e: two}
    \begin{split}
        \langle\nabla R,\nabla f\rangle&=-2\Ric(\nabla f,\nabla f),\\
        R+|\nabla f|^2&=\textnormal{const}.
    \end{split}
\end{equation}
By the second equation, a critical point of $f$ must be the maximum point of $R$.
For a 3D steady gradient Ricci soliton, since the Ricci curvature is positive, the first equation implies that a maximum point of $R$ is also a critical point of $f$. We will show in Section \ref{s: asymptotic geometry} that the critical point exists in all 3D steady gradient Ricci solitons.

Hamilton's cigar soliton is the first example of Ricci solitons \cite{cigar}. It is rotationally symmetric and has positive curvature. The cigar soliton is an important notion in this paper. In the following we review the definition of the cigar soliton and some properties we will use, including a precise description of the curvature decay and the tip contracting rate.

\begin{defn}[Cigar soliton, c.f. \cite{cigar}]\label{d: cigar}
Hamilton's cigar soliton is a complete Riemannian surface $(\R^2,g_c,f)$, where
\begin{equation*}
    g_c:=\frac{dx^2+dy^2}{1+x^2+y^2},\quad\textit{and}\quad f=\log(1+x^2+y^2).
\end{equation*}
As a solution of Ricci flow, its time-dependent version is
\begin{equation*}
    g_c(t):=\frac{dx^2+dy^2}{e^{4t}+x^2+y^2}.
\end{equation*}
Let $s$ denote the distance to the cigar tip $(0,0)$, then we may rewrite $g_c$ as
\begin{equation}\label{e: warped cigar}
    g_c=ds^2+\tanh^2s\,d\theta^2,
\end{equation}
and the scalar curvature of $g_c$ is
\begin{equation}\label{e: cigar R}
    R_{\Sigma}=4\,\textnormal{sech}^2s=\frac{4}{(e^s+e^{-s})^2}.
\end{equation}
In particular, $R(x_{\textnormal{tip}})=4$ and $K(x_{\textnormal{tip}})=2$.
For a fixed $\theta_0\in[0,2\pi)$, the curve $\gamma(s):=(\theta_0,s)$ is a unit speed ray starting from the tip, and we can also compute that  
\begin{equation}\label{e: integrate Ricci}
    \int_{0}^{r}\Ric_{\Sigma}(\gamma'(s),\gamma'(s))\,ds=\int_0^{r}2\,\textnormal{sech}^2s\,ds=2\tanh s|^{r}_0=2\left(1-\frac{2}{e^{2r}+1}\right),
\end{equation}
which converges to $2$ as $r\rii$.
Note that this integral is the speed of a point st distance $r$ drifts away from the tip in the backward Ricci flow $g_c(t)$. 
\end{defn}

Throughout the paper, by abuse of notation, we will use $g_c$ to denote the metric on the cigar, and the product metric on $\R\times\cigar$ such that $R(x_{tip})=4$; and use $g_{stan}$ to denote the product metrics on $\R\times S^1$ and $\RR\times S^1$ such that the length of the $S^1$-fiber is equal to $2\pi$.

With this convention, it is easy to see from \eqref{e: warped cigar} that for any sequence of points $x_i\rii$, the pointed manifolds $(\cigar,g_c,x_i)$ smoothly converges to $(\R\times S^1,g_{stan})$ in the Cheeger-Gromov sense.

Next, we introduce the concept of collapsing and non-collapsing.

\begin{defn}[Collapsing and non-collapsing]
Let $(M^n,g)$ be an n-dimensional Riemannian manifold. We say the it is non-collapsed (resp. collapsed) if there exists (resp. does not exist) a constant $\kappa>0$ such that the following holds: For all $x\in M$, if  $|\Rm|(x)<r^{-2}$ on $B_g(x,r)$ for some $r>0$, then
\begin{equation*}
    vol_g(B_g(x,r))\ge\kappa r^n.
\end{equation*}
\end{defn}

It is easy to see that an n-dimensional Riemannian manifold is collapsed if there is an asymptotic limit isometric to $\R^{n-2}\times\cigar$.

\begin{lem}\label{l: cigar implies collapsing}
Let $(M^n,g)$ be an n-dimensional Riemannian manifold. Suppose there exists a sequence of points $x_i\in M$ and constants $r_i>0$ such that the pointed manifolds $(M,r_i^{-2}g,x_i)$ smoothly converge to $\R^{n-2}\times\cigar$. Then $(M,g)$ is collapsed.
\end{lem}

\begin{proof}
By the assumption we may choose a sequence of points $y_i\in M$ such that $(M,r_i^{-2}g,y_i)$ converge to $\R^{n-1}\times S^1$. So there is a sequence of constants $A_i\rii$ such that $|\Rm|_{r_i^{-2}g}\le A_i^{-2}$ on $B_{r_i^{-2}g}(y_i,A_i)$, and $\lim_{i\rii}\frac{vol_{r_i^{-2}g}B_{r_i^{-2}g}(y_i,A_i)}{A_i^n}=0$. After rescaling, this implies $|\Rm|_{g}\le (A_ir_i)^{-2}$ on $B_{g}(y_i,A_ir_i)$ and $\lim_{i\rii}\frac{vol_{g}B_{g}(y_i,A_ir_i)}{(A_ir_i)^n}=0$. So $(M,g)$ is collapsed.
\end{proof}

In Section \ref{s: asymptotic geometry} we will see that the opposite of Lemma \ref{l: cigar implies collapsing} is also true for all 3D steady gradient Ricci solitons, i.e. any collapsed solitons have asymptotic limits isometric to $\R\times\cigar$. 
Note all steady gradient Ricci solitons except the Bryant soliton are collapsed \cite{brendlesteady3d}.

\subsection{Local geometry models}
We will show in Section \ref{s: asymptotic geometry} that $\R\times\cigar$ and $\RR\times S^1$ are asymptotic limits in 3D steady gradient Ricci solitons that are not Bryant solitons.
In this subsection we define $\epsilon$-necks, $\epsilon$-cylindrical planes and $\epsilon$-tip points, which are local geometry models corresponding to these asymptotic limits.
Moreover, to obtain the asymptotic limits, we need to rescale the soliton by factors that are comparable to volume scale at the points.

\begin{defn}[Volume scale]
Let $(M,g)$ be a 3D Riemannian manifold. We define the volume scale $r(\cdot)$ to be
\begin{equation*}
    r(x)=\sup\{s>0: vol_g(B_g(x,s))\ge\omega_0s^3\},
\end{equation*}
where $\omega_0>0$ is chosen such that $r(x)=1$ for all $x\in \RR\times S^1$. It is clear that $\omega_0$ is less than the volume of the radius one ball in the Euclidean space $\R^3$.
\end{defn}

We measure the closeness of two pointed Riemannian manifolds by using the following notion of $\epsilon$-isometry.

\begin{defn}[$\epsilon$-isometry between manifolds]\label{d: epsilon isometry}
Let $\epsilon>0$ and $m\in\mathbb{N}$.
Let $(M^n_i,g_i)$, $i=1,2$, be an n-dimensional Riemannian manifolds, $x_i\in M_i$. We say a smooth map $\phi:B_{g_1}(x_1,\epsilon^{-1})\ri M_2$, $\phi(x_1)=x_2$, is an $\epsilon$-isometry in the $C^m$-norm if it is a diffeomorphism onto the image, and
\begin{equation}\label{e: parti}
    |\nabla^{k}(\phi^*g_2-g_1)|\le \epsilon\quad \textit{on}\quad B_{g_1}(x_1,\epsilon^{-1}),\quad k=0,1,...,m,
\end{equation}
where the covariant derivatives and norms are taken with respect to $g_1$.
We also say $(M_2,g_2,x_2)$ is $\epsilon$-close to $(M_1,g_1,x_1)$ in the $C^m$-norm. 
In particular,  if $m=[\epsilon^{-1}]$, then we simply say $(M_2,g_2,x_2)$ is $\epsilon$-close to $(M_1,g_1,x_1)$ and $\phi$ is an $\epsilon$-isometry. 
\end{defn}

\begin{defn}[$\epsilon$-isometry between Ricci flows]\label{d: RF-epsilon isometry}
Let $\epsilon>0$.
Let $(M^n_i,g_i(t))$, $i=1,2$, $t\in[-\epsilon^{-1},\epsilon^{-1}]$, be an n-dimensional Ricci flow, $x_i\in M_i$. We say a smooth map $\phi:B_{g_1(0)}(x_1,\epsilon^{-1})\ri M_2$, $\phi(x_1)=x_2$, is an $\epsilon$-isometry between the two Ricci flows if it is a diffeomorphism onto the image, and
\begin{equation*}
    |\nabla^{k}(\phi^*g_2(t)-g_1(t))|\le \epsilon\quad \textit{on}\quad B_{g_1(0)}(x_1,\epsilon^{-1})\times[-\epsilon^{-1},\epsilon^{-1}],\quad k=0,1,...,[\epsilon^{-1}],
\end{equation*}
where the covariant derivatives and norms are taken with respect to $g_1(0)$.
We also say $(M_1,g_1(t),x_1)$ is $\epsilon$-close to $(M_2,g_2(t),x_2)$.
\end{defn}

In the following, we will choose the target manifolds to be the cylinder and cylindrical plane, and call the manifolds that are close to them the $\epsilon$-neck and $\epsilon$-cylindrical plane.

\begin{defn}[$\epsilon$-neck]\label{d: neck}
We say a 2D Riemannian manifold $(N,g)$ is an $\epsilon$-neck at $x_0\in M$ for some $\epsilon>0$, if there exists a constant $r>0$ such that $(N,r^{-2}g,x_0)$ is $\epsilon$-close to the cylinder $\R\times S^1$.

We say $r$ is the scale of the $\epsilon$-neck and $x_0$ is a center of an $\epsilon$-neck. 
On $\R\times S^1$, let the function $x:\R\times S^1$ be the projection onto the $\R$-factor, and $\partial_{x}$ be the vector field. Then by an abuse of notation, we will denote the corresponding function and vector field on $N$ modulo the $\epsilon$-isometry by $x$ and $\partial_{x}$.
\end{defn}

\begin{defn}[$\epsilon$-cap]
Let $(M,g)$ be a complete 2D Riemannian manifold.
We say a compact subset $\mathcal{C}\subset M$ is an $\epsilon$-cap if $\mathcal{C}$ is diffeomorphic to a 2-ball and the boundary $\partial\mathcal{C}$ is the central circle of an $\epsilon$-neck $N$ in $(M,g)$.
We say that the points in $\mathcal{C}\setminus N$ are centers of the $\epsilon$-cap.
\end{defn}

\begin{defn}[$\epsilon$-cylindrical plane]\label{d: cylindrical plane}
We say a 3D Riemannian manifold $(M,g)$ is an $\epsilon$-cylindrical plane at $x_0\in M$ for some $\epsilon>0$, if there exists a constant $r>0$ such that $(M,r^{-2}g,x_0)$ is $\epsilon$-close to the cylindrical plane $\RR\times S^1$.

We say $r$ is the scale of the $\epsilon$-cylindrical plane, and $x_0$ is the center of the $\epsilon$-cylindrical plane.
Let $x,y:\R^2\times S^1\ri\R$ and $\theta:\R^2\times S^1\ri S^1$ be the projection to the three product factors.
By abuse of notation, we use $\partial_{x},\partial_{y},\partial_{\theta}$ to denote the corresponding vector fields on $M$ modulo the $\epsilon$-isometry.
In particular, we call $\partial_{\theta}$ to be the $SO(2)$-killing field of the $\epsilon$-cylindrical plane.
\end{defn}

At the center of an $\epsilon$-cylindrical plane, we introduce another scale in the following, which is comparable to the volume scale. Since this scale is measured by the length of curves, it is more useful than the volume scale in the Ricci flow of the soliton when combining with suitable distance distortion estimates.

\begin{defn}[Scale at an $\epsilon$-cylindrical plane]\label{d: h}
Let $(M,g)$ be a 3D Riemannian manifold. Suppose $x\in M$ is the center of an $\epsilon$-cylindrical plane. We denote by $h(x)$ the infimum length of all closed smooth curves at $x$ that are homotopic to the $S^1$-factor of the $\epsilon$-cylindrical plane in $B(x,1000 r(x))$, where $r(x)>0$ is the volume scale at $x$.
It is clear that $h(x)$ is achieved by a geodesic loop at $x$. 

Moreover, by the definition of volume scale we have $r(x)=1$ and $h(x)=2\pi$ for all $x\in\RR\times S^1$. So when $\epsilon$ is sufficiently small we have 
\begin{equation*}
    1.9\pi\,r(x)\le h(x)\le 2.1\pi\,r(x).
\end{equation*}
\end{defn}

In a manifold that is $\epsilon$-close to the cigar soliton, we call the points that are $\epsilon$-close to the tip of cigar under the $\epsilon$-isometry to be the $\epsilon$-tip points.

\begin{defn}[$\epsilon$-tip point]\label{d: tip points}
Let $\epsilon>0$.
Let $(M,g)$ be a 2D Riemannian manifold, $x\in M$. 
If there is an $\epsilon$-isometry from $(M,r^{-2}(x)g,x)$ to $(\cigar,r^{-2}(x_0)g_c,x_0)$, for some $x_0\in \cigar$ such that $d_{g_c}(x_0,x_{tip})\le\epsilon$, then we say $x$ is an $\epsilon$-tip point.

Similarly, if $(M,g)$ is a 3D Riemannian manifold, $x\in M$. 
Suppose there is an $\epsilon$-isometry from $(M,r^{-2}(x)g,x)$ to $(\R\times\cigar,r^{-2}(x_0)g_c,x_0)$, for some $x_0\in \R\times\cigar$ such that $d_{g_c}(x_0,x_{tip})\le\epsilon$ where $x_{tip}$ is the tip of the cigar with the same $\R$-coordinate as $x_0$. Then we say that $x$ is an $\epsilon$-tip point.
\end{defn}

\subsection{Distance distortion estimates and curvature estimates}
In this subsection, we review some standard distance distortion estimates and curvature estimates, which are originally due to Hamilton and Perelman \cite{Hamilton_singularity_formation,Pel1}.  

The following lemma gives an upper bound on the speed of distance shrinking between two points, using only local curvature bounds near the two points. The proof uses the second variation formula, see e.g. \cite[Theorem 18.7]{RFTandA3}.

\begin{lem}\label{l: distance laplacian}
Let $(M,g(t))_{t\in[0,T]}$ be a Ricci flow of dimension $n$. Let $K,r_0>0$.
\begin{enumerate}
    \item Let $x_0\in M$ and $t_0\in(0,T)$. Suppose that $\Ric\le(n-1)K$ on $B_{t_0}(x_0,r_0)$. Then the distance function $d(x,t)=d_t(x,x_0)$ satisfies the following inequality in the outside of $B_{t_0}(x_0,r_0)$:
\begin{equation*}
    (\pt-\Delta)|_{t=t_0}d\ge -(n-1)\left(\frac{2}{3}Kr_0+r_0^{-1}\right).
\end{equation*}  
   \item Let $t_0\in[0,T)$ and $x_0,x_1\in M$. Suppose 
\begin{equation*}
    \Ric(x,t_0)\le(n-1)K,
\end{equation*}
for all $x\in B_{t_0}(x_0,r_0)\cup B_{t_0}(x_1,r_0)$.
Then 
\begin{equation*}
    \pt|_{t=t_0}d_t(x_0,x_1)\ge -2(n-1)(Kr_0+r_0^{-1}).
\end{equation*}
\end{enumerate}
\end{lem}

The following lemmas control how fast a metric ball shrinks and expands along Ricci flow, using the nearby curvature assumptions. They can be proved by using the Ricci flow equation, see e.g. \cite[Lemma 2.1, 2.2]{simon2021local}.

\begin{lem}\label{l: expanding lemma}
Let $(M,g(t))_{t\in[0,T]}$ be a Ricci flow of dimension $n$, $x_0\in M$.
Let $K,A>0$.
\begin{enumerate}
    \item Suppose $\Ric_{g(t)}\ge -K$ on $ B_t(x_0,A)\subset\subset M$ for all $t\in[0,T]$. Then the following holds for all $t\in[0,T]$:
\begin{equation*}
    B_0(x_0,A\,e^{-KT})\subset B_t(x_0,A\,e^{-K(T-t)}).
\end{equation*}
    \item Suppose $\Ric_{g(t)}\le K$ on $ B_t(x_0,R)\subset\subset M$ for all $t\in[0,T]$. Then the following holds for all $t\in[0,T]$:
\begin{equation*}
    B_t(x_0,A
\,e^{-Kt})\subset B_0(x_0,A).
\end{equation*}
\end{enumerate}

\end{lem}

The following curvature estimate is also due to Perelman \cite[Corollary 11.6]{Pel1}. It provides a curvature upper bound at points in a Ricci flow, if the local volume has a positive lower bound.
For a more general version of this estimate see \cite[Proposition 3.2]{BamA}.

\begin{lem}[Perelman's curvature estimate]\label{l: Perelman}
For any $\kappa>0$ and $n\in \mathbb{N}$, there exists $C>0$ such that the following folds:
Let $(M^n,g)\times [-T,0]$ be an $n$-dimensional Ricci flow (not necessarily complete).
Let $x\in M$ be a point with $B_g(x,r)\times[-r^2,0]\subset\subset M\times [-T,0]$ for some $r>0$. Assume also $vol(B_g(x,r))\ge \kappa r^n$. Then $R(x,0)\le \frac{C}{r^2}$.
\end{lem}

\subsection{Metric comparisons}
We need the following notions and facts from metric comparison geometry, see \cite{BGP}. Let $(M,g)$ be a complete n-dimensional Riemannian manifold with non-negative sectional curvature.

\begin{lem}[Monotonicity of angles]
For any triple of points $o,p,q\in M$, the comparison angle $\widetilde{\measuredangle}poq$ is the corresponding angle formed by minimizing geodesics with lengths equal to $d_g(o,p),d_g(o,q),d_g(p,q)$ in Euclidean space.

Let $op,oq$ be two minimizing geodesics in $M$ between $o,p$ and $o,q$, and $\measuredangle poq$ be the angle between them at $o$, then $\measuredangle poq\ge\widetilde{\measuredangle}poq$.
Moreover, for any $p'\in op$ and $q'\in oq$, we have $\widetilde{\measuredangle}p'oq'\ge \widetilde{\measuredangle}poq$.  
\end{lem}

In a non-negatively curved complete non-compact Riemannian manifold, we can equip a length metric on the space of geodesic rays. Moreover, a blow-down sequence of this manifold converges to the metric cone over the space of rays in the Gromov-Hausdorff sense, see e.g. \cite[Prop 5.31]{MT}.

Let $\gamma_1,\gamma_2$ be two rays with unit speed starting from a point $p\in M$, the limit $\lim_{r\rii}\widetilde{\measuredangle}\gamma_1(r)p\gamma_2(r)$ exists by the monotonicity of angles and we say it is the angle at infinity between $\gamma_1$ and $\gamma_2$, and denote it as $\measuredangle(\gamma_1,\gamma_2)$.

\begin{lem}[Space of rays]
Let $p\in M$ and $S_{\infty}(M,p)$ be the space of equivalent classes of rays starting from $p$, where two rays are equivalent if and only if the angle at infinity between them is zero, and the distance between two rays is the limit of the angle at infinity between them. Then $S_{\infty}(M,p)$ is a compact length space.
\end{lem}

\begin{lem}[Asymptotic cone]\label{l: converge to asymp}
Let $p\in M$ and $\mathcal{T}(M,p)$ be the metric cone over $S_{\infty}(M,p)$.
Then for any $\lambda_i\rii$, the  sequence of pointed manifolds $(M,\lambda_i^{-1}g,p)$ converge to $\mathcal{T}(M,p)$ with $p$ converging to the cone point in the pointed Gromov-Hausdorff sense.
Moreover, for $p,q\in M$, we have that $\mathcal{T}(M,p)$ is isometric to $\mathcal{T}(M,q)$.
We say $(M,g)$ is asymptotic to $\mathcal{T}(M,p)$, and $\mathcal{T}(M,p)$ is the asymptotic cone of $M$.
\end{lem}

It is clear that the asymptotic cone $\mathcal{T}(M,p)$ is in fact independent
of the choice of p. 
It is easy to see that the asymptotic cones of the Bryant soliton and $\R\times\cigar$ are a ray and a half-plane.
In \cite{Lai2020_flying_wing}, the author constructed a family of 3D steady gradient Ricci solitons that are asymptotic to metric cones over an interval $[0,a]$, $a\in(0,\pi)$. 
We will show in Section \ref{s: asymptotic geometry} that this is true for all
3D steady gradient Ricci soliton with positive curvature that is not a Bryant soliton.

In the rest of this subsection we introduce a very useful technical notion called strainer \cite{BGP}.
It is similar to the notion of an orthogonal frame in the Euclidean space $\R^n$, that provides a local coordinate system in the metric space.
\begin{defn}[$(m,\delta)$-strainer]
Let $\delta>0$ and $1\le m\in\mathbb{N}$. A $2m$-tuple $(a_1,b_1,...,a_m,b_m)$ of points  in a metric space $(X,d)$ is called an $(m,\delta)$-strainer around a point $x\in X$ if
\begin{equation*}
\begin{cases}
    \widetilde{\measuredangle}a_ixb_j>\frac{\pi}{2}-\delta \quad\textit{for all }i\neq j,\quad i,j=1,...,m,\\
     \widetilde{\measuredangle}a_ixa_j>\frac{\pi}{2}-\delta \quad\textit{for all }i\neq j,\quad i,j=1,...,m,\\
      \widetilde{\measuredangle}b_ixb_j>\frac{\pi}{2}-\delta \quad\textit{for all }i\neq j,\quad i,j=1,...,m,\\
      \widetilde{\measuredangle}a_ixb_i>\pi-\delta\quad\textit{for all }i=1,...,m.
\end{cases}
\end{equation*}
The strainer is said to have size $r$ if $d(x,a_i)=d(x,b_i)=r$ for all $i=1,...,m$. It is said to have size at least $r$ if $d(x,a_i)\ge r$ and $d(x,b_i)\ge r$ for all $i=1,...,m$.
\end{defn}

We also introduce the notion of $(m+\frac{1}{2},\delta)$-strainers. Similarly, this notion provides a local coordinate system in the metric space that look at a half-plane $\R^n_+=\R^n\times\R_+$.

\begin{defn}[$(m+\frac{1}{2},\delta)$-strainer]
Let $\delta>0$ and $1\le m\in\mathbb{N}$. A $2m+1$-tuple $(a_1,b_1,...,a_m,b_m,a_{m+1})$ of points in a metric space $(X,d)$ is called an $(m+\frac{1}{2},\delta)$-strainer around a point $x\in X$ if
\begin{equation*}
\begin{cases}
       \widetilde{\measuredangle}a_ixb_j>\frac{\pi}{2}-\delta , &\text{for all } i\neq j,\quad i=1,...,m+1,\quad j=1,...,m,\\
        \widetilde{\measuredangle}a_ixa_j>\frac{\pi}{2}-\delta , &\text{for all } i\neq j,\quad i,j=1,...,m+1,\\
        \widetilde{\measuredangle}b_ixb_j>\frac{\pi}{2}-\delta , &\text{for all } i\neq j,\quad i,j=1,...,m,\\
        \widetilde{\measuredangle}a_ixb_i>\pi-\delta , &\text{for all } i=1,...,m.
\end{cases}
\end{equation*} 
The strainer is said to have size $r$ if $d(x,a_i)=d(x,b_j)=r$ for all $i=1,...,m+1$ and $j=1,...,m$. It is said to have size at least $r$ if $d(x,a_i)\ge r$ for all $i=1,...,m+1$ and $d(x,b_j)\ge r$ for all $j=1,...,m$.
\end{defn}

\subsection{Heat kernel estimates}
We prove a few lemmas using the standard heat kernel estimates of the heat equations under the Ricci flows.
Let $G(x,t;y,s)$, $x,y\in M$, $s<t$, be the heat kernel of the heat equation $\pt u=\Delta u$ under $g(t)$, that is,
\begin{equation}\label{e: standard heat kernel}
\begin{split}
    \pt G(x,t;y,s)&=\Delta_{x,t} G(x,t;y,s),\\
    \lim_{t\searrow s}G(\cdot,t;y,s)&=\delta_{y}.
\end{split}
\end{equation}

It is easy to see that $G(x,t;y,s)$ is also the heat kernel of the conjugate heat equation, that is,
\begin{equation*}
\begin{split}
    -\partial_s G(x,t;y,s)&=\Delta_{y,s} G(x,t;y,s)-R(y,s)\,G(x,t;y,s),\\
    \lim_{s\nearrow t}G(x,t;\cdot,s)&=\delta_{x}.
\end{split}
\end{equation*}

We have the following Gaussian upper bound for $G$.

\begin{lem}[Upper bound of the heat kernel for an evolving metric](c.f. \cite[Theorem 26.25]{RFTandA3})\label{l: heat kernel lower bound implies upper bound}
Let $(M^n,g(t))$ be a complete Ricci flow on $[0,T]$ with $|\Rm|\le K$. There exists a constant $C<\infty$ depending only on $n,T,K$ such that the conjugate heat kernel satisfies 
\begin{equation}\label{e: theorem 2625}
    G(x,t;y,s)\le \frac{C}{vol^{1/2}(B_s(x,\sqrt{\frac{t-s}{2}}))\cdot vol^{1/2}(B_s(y,\sqrt{\frac{t-s}{2}}))}\, \textnormal{exp}\left(-\frac{d^2_s(x,y)}{C(t-s)}\right)
\end{equation}
for any $x,y\in M$ and $0\le s<t\le T$.
\end{lem}

Using this we can prove the following lemma, which gives a time and distance dependent upper bound on subsolutions to the heat equation, which satisfy certain initial upper bounds.

\begin{lem}\label{l: heat equation}
For any $C_0>0$ and $\alpha>0$, there are a constant $\delta>0$ and a non-decreasing function $C:[0,\infty)\ri [0,\infty)$ such that the following holds:

Let $(M,g(t),y_0)$ be a complete Ricci flow on $[0,T]$. 
Assume the following holds:
\begin{enumerate}
    \item $|\Rm|(x,t)\le C_0$ for all $x\in M$ and $t\in[0,T]$.
    \item $u(x,t)$ is a function with $\partial_t u\le\Delta u$, satisfying
    \begin{equation*}
        |u|(x,0)\le e^{\alpha\,d_0(x,y_0)}.
    \end{equation*}
\end{enumerate}
Then $|u|(x,t)\le C(\alpha,T,C_0)\,e^{(\alpha+1)D}$ for any $(x,t)\in B_0(y_0,D)\times[0,T]$.
\end{lem}

We need the following heat kernel estimates, see e.g. \cite[Theorem 26.25]{RFTandA3}.

\begin{proof}
First, by the curvature assumption $|\Rm|\le K$, it follows immediately from the Ricci flow equation that the metrics at different times are comparable to each other, i.e.
\begin{equation*}
    e^{-CK}g(s)\le g(t)\le e^{CK} g(s),
\end{equation*}
where $C$ depends only on $T$. 
From now on we will use $C$ to denote all constants that depend only on $T,K$, which may vary from line to line.

By the reproduction formula we have
\begin{equation}\label{e: I+II}
    u(x,t)=\int_M G(x,t;y,0)\cdot u(y,0)\,d_0y.
\end{equation}
Now assume $x\in B_0(x_0,D)$, and split the integral into two parts
\begin{equation*}
    u(x,t)=\int_{B_0(x,1)} G(x,t;y,0)\cdot u(y,0)\,d_0y+\int_{M\setminus B_0(x,1)} G(x,t;y,0)\cdot u(y,0)\,d_0y.
\end{equation*}
For the first part, note that for any $y\in B_0(x,1)$, we have
\begin{equation*}
    d_0(y,x_0)\le d_0(y,x)+d_0(x,x_0)\le 1 +D.
\end{equation*}
So by the assumption on $u$ we have $u(y,0)\le e^{\alpha(1+D)}$, and hence
\begin{equation}\label{e: I}
    \int_{B_0(x,1)} G(x,t;y,0)\cdot u(y,0)\,d_0y\le e^{\alpha(1+D)}\int_{B_0(x,1)} G(x,t;y,0)\,d_0y\le e^{\alpha(1+D)}.
\end{equation}

To estimate the second part in \eqref{e: I+II}, we first claim that for any $y\in M\setminus B_0(x,1)$, the following holds,
\begin{equation}\label{claim: heat kernel}
    G(x,t;y,0)\le C \cdot\textnormal{exp}\left(-\frac{d^2_0(x,y)}{Ct}\right).
\end{equation}
To this end, we note that if $t\ge 1$, the volumes of the two balls $B_0(x,\sqrt{\frac{t}{2}})$ and $B_0(y,\sqrt{\frac{t}{2}})$ are bounded below by $C^{-1}$.
So the claim follows immediately from \eqref{e: theorem 2625}.
If $t<1$, then by the assumption on the injectivity radius and the curvature, 
we see that the volumes of the two balls $B_0(x,\sqrt{\frac{t}{2}})$ and $B_0(y,\sqrt{\frac{t}{2}})$ are bounded below by $C^{-1}\left(\frac{t}{2}\right)^{n/2}$.
Note also $\left(\frac{t}{2}\right)^{n/2}\le C\cdot\textnormal{exp}\left(-\frac{1}{Ct}\right)$, we obtain
\begin{equation*}
    G(x,t;y,0)\le C\left(\frac{t}{2}\right)^{-n/2}\textnormal{exp}\left(-\frac{d^2_0(x,y)}{Ct}\right)\le C \cdot\textnormal{exp}\left(-\frac{d^2_0(x,y)}{Ct}\right),
\end{equation*}
which proves the claim.

Since for any $y\in M\setminus B_0(x,1)$, we have 
\begin{equation*}
    u(y,0)\le e^{\alpha d_0(x_0,y)}\le e^{\alpha(d_0(x_0,x)+d_0(x,y))}\le e^{\alpha D}\cdot e^{\alpha d_0(x,y)}.
\end{equation*}
Combining this with \eqref{claim: heat kernel}, we see that the second part in \eqref{e: I+II} satisfies
\begin{equation*}\begin{split}\label{e: II final}
    \int_{M\setminus B_0(x,1)} G(x,t;y,0)\cdot u(y,0)\,d_0y&\le C\,e^{\alpha D}\int_{M} \textnormal{exp}\left(-\frac{d^2_0(x,y)}{Ct}\right)\cdot  e^{\alpha d_0(x,y)}\,d_0y\\
    &\le C(\alpha,T)\,e^{\alpha D}\int_{M} \textnormal{exp}\left(-\frac{d^2_0(x,y)}{Ct}\right)\,d_0y\\
    &\le C(\alpha,T,C_0)\,e^{\alpha D}.
\end{split}\end{equation*}
where in the last inequality we used the curvature bound $|\Rm|\le C_0$ and $t\le T$, which allows us to apply a volume comparison to estimate the last integral term, see also \cite[Lemma 2.8]{Lai2019}. 
Combining the two inequalities \eqref{e: I} and \eqref{e: II final} we get $|u|(x,t)\le C(\alpha,T,C_0)\,e^{(\alpha+1)D}$, which proves the lemma.
\end{proof}

\section{Asymptotic geometry at infinity}\label{s: asymptotic geometry}
In this section,
we study the asymptotic geometry of a 3D steady gradient soliton that is not a Bryant soliton. We show that it dimension reduces to the cigar soliton along two integral curves of $\nabla f$. Moreover, we show that the asymptotic cone of the soliton is isometric to a metric cone over an interval $[0,a]$, $a\in(0,\pi)$.
We also prove a few other geometric properties of the soliton, one of which is that the scalar curvature attains its maximum at some point.

\subsection{Classification of asymptotic limits}
The main result in this subsection is Lemma \ref{l: DR to cigar}, which states that there are two asymptotic limits in the soliton, which are $\RR\times S^1$ and $\R\times\cigar$.

We will use the following lemma to show that $\R\times T^2$ can not be an asymptotic limit.

\begin{lem}\label{l: T2}
Let $(M,g)$ be a 3D complete Riemannian manifold with positive curvature, and $\R\times T^2=\R\times S^1\times S^1$ is equipped with some product metric $g_0$ (in which the lengths of the two $S^1$-fibers are not necessarily equal), $x_0\in\R\times T^2$. Then there exists an $\epsilon_0>0$ such that for any $y\in M$, the pointed manifold $(M,r^{-2}(y)g,y)$ is not $\epsilon_0$-close to $(\R\times T^2,r^{-2}(x_0)g_0,x_0)$.
\end{lem}

\begin{proof}
Suppose $\epsilon_0$ does not exist. Then there exists a sequence of points $y_k\in M$ such that $(M,r^{-2}(y_k)g,y_k)$ smoothly converges to $(\R\times T^2,r^{-2}(x_0)g_0,x_0)$. It is clear that $y_k\rii$, because otherwise the manifold is isometric to $\R\times T^2$, a contradiction to the positive curvature.
So there exists $\epsilon_k\ri0$, an open neighborhood $U_k$ of $y_k$ and a diffeomorphism $\phi_k$ such that $\phi_k:[-\epsilon_k^{-1},\epsilon_k]\times T^2\ri U_k\subset M$ which maps $x_0\in\{0\}\times T^2$ to $y_k$, and $\phi_k$ is an $\epsilon_k$-isometry.
We say $T^2_k:=\phi_k(\{0\}\times T^2)$ is the central torus, which is homeomorphic to $T^2$.

In the rest of the proof we show that a connected component of the level set of the function $d_{y_0}(\cdot):=d_g(y_0,\cdot)$ at $y_k$ is homeomorphic to a 2-torus. Suppose this claim is true, then for $k$ sufficiently large, $d_g(y_0,y_k)$ is sufficiently large, which contradicts the fact that a level set of a distance function to a fixed point in a positively-curved 3D Riemannian manifold must be homeomorphic to a 2-sphere at all large distances, see e.g. \cite[Corollary 2.11]{MT}. Then this will prove the lemma.

To show the claim, without loss of generality, we may assume after a rescaling that $r(y_k)=1$.  
Let $s:\R\times T^2\ri \R$ be the coordinate function in the $\R$-direction of $\R\times T^2$, and let $X=\phi_k(\partial_{x})$ be a vector field on $U_k\subset M$.
For any small $\epsilon>0$, it is easy to see that for all large $k$, the angle formed by any minimizing geodesic from $y_0$ to a point $p\in U_k$ and $X(p)$ is less than $\epsilon$.
Let $\chi_{\mu}$, $\mu\in(-100,100)$, be the flow generated by $X$ on $\phi_k((-100,100)\times T^2)\subset M$, then the distance function $d_{y_0}(\cdot)$ increases along $\chi_{\mu}$ at a rate bounded below by $1-C_0\epsilon$,
where $C_0>0$ is a universal constant.
In particular, an integral curve of $X$ intersects a level set of $d_{y_0}(\cdot)$ in a single point.

Therefore, there is a continuous function  $\tilde{s}:T^2_k\ri\R$
such that for any $x\in T^2_k$, we have $ d_{y_0}(\chi_{\tilde{s}(x)}(x))=d_{y_0}(y_k)$. Let $F:T^2_k\ri d_{y_0}^{-1}(d_{y_0}(y_k))$ be defined by $F(x)=\chi_{\tilde{s}(x)}(x)$, then $F$ is continuous. We show that
$F$ is an injection: suppose $F(x_1)=F(x_2)=y$, then $x_i=\chi_{-\tilde{s}(x_i)}(y)$, $i=1,2$.
Since $x_1,x_2\in T^2_k$, it follows that $(\phi_k^{-1})^*s(x_1)=(\phi_k^{-1})^*s(x_2)=0$ and
\begin{equation*}
    0=(\phi_k^{-1})^*s(x_1)-(\phi_k^{-1})^*s(x_2)
    =\tilde{s}(x_2)-\tilde{s}(x_1).
\end{equation*}
So $\tilde{s}(x_2)=\tilde{s}(x_1)$, and hence $x_1=x_2$.
Since $T^2_k$ is compact, $F$ is a homeomorphism from the 2-torus onto the image which is a connected component of $d_{y_0}^{-1}(d_{y_0}(y_k))$.
This proves the claim.

\end{proof}

The following lemma will be used to show that all asymptotic limit splits off a line.
\begin{lem}\label{l: splitting}
Let $(M,g)$ be a complete Riemannian manifold with non-negative sectional curvature and $\{y_k\}_{k=0}^{\infty}\subset M$ is a sequence of points with $d_g(y_0,y_k)\rii$. 
Then after passing to a subsequence of $\{y_k\}_{k=0}^{\infty}$, there exists a ray $\sigma:[0,\infty)\ri M$ with $\sigma(0)= y_0$ and a sequence of numbers $s_k\rii$ such that for $z_k=\sigma(s_k)$, we have $d_g(z_k,y_k)=d_g(y_0,y_k)$ and 
\begin{equation*}
    \widetilde{\measuredangle}y_0y_kz_k\ri\pi,
\end{equation*}
as $k\rii$.
\end{lem}

\begin{proof}
A standard metric comparison argument. See for example \cite[Lemma 5.1.5]{Bamler_thesis}.
\end{proof}

\begin{lem}\label{l: DR to cigar}
Let $(M,g)$ be a 3D steady gradient soliton with positive curvature that is not a Bryant soliton. Let $p\in M$ be a fixed point.  Then for any $\epsilon>0$, there is $D(\epsilon)>0$ such that for all $x\in M\setminus B_{g}(p,D)$, the pointed manifold $(M,r^{-2}(x)g,x)$ is $\epsilon$-close to exactly one of the following,
\begin{enumerate}
    \item $(\mathbb{R}\times\cigar,r^{-2}(\widetilde{x})g_c,\widetilde{x})$;
    \item $(\RR\times S^1, g_{stan},\widetilde{x})$.
\end{enumerate}
We call these two limits the asymptotic limits of the soliton.
\end{lem}

\begin{proof}
First, we show that there exists $D>0$ such that $(M,r^{-2}(x)g,x)$ is $\epsilon$-close to some product space with an $\R$-factor, for all $x\in M\setminus B_{g}(p,D)$. 
This suffices to show that for any sequence of points $\{y_k\}_{k=0}^{\infty}$, the rescaled Ricci flows $(M,r^{-2}(y_k)g(r^2(y_k)t),(y_k,0))$ converge to an ancient Ricci flow that split-off a line.

First, we claim $\frac{r(y_k)}{d_g(y_0,y_k)}\ri0$ as $k\rii$.
If this is not true, then $(M,g)$ has Euclidean volume growth, and hence is flat by Perelman's curvature estimate Lemma \ref{l: Perelman}, contradiction.
So by passing to a subsequence we may assume $(M,r^{-2}(y_k)g(r^2(y_k)t),(y_k,0))$ converges to an ancient 3D Ricci flow, see \cite[Lemma 3.3]{Lai2020_flying_wing}.

So by Lemma \ref{l: splitting} and the strong maximum principle \cite[Lemma 4.13, Corollary 4.19]{MT}, the limit flow splits off an $\R$-factor.
Therefore, by the classification of ancient 2D Ricci flows \cite{Chu,Sesum}, the limit flow must be isometric to one of the following Ricci flows up to a rescaling,
\begin{enumerate}
    \item $\R\times\cigar$;
    \item $\RR\times S^1$;
    \item $\R\times T^2$;
    \item $\R\times S^2$;
    \item $\R\times \textnormal{sausage solution}$.
\end{enumerate}

\yi{First, item (3) is impossible by Lemma \ref{l: T2}. 
Second, we can argue in the same way as
 \cite[Theorem 3.7]{Lai2020_flying_wing} to exclude item (5): 
 Note that in \cite[Theorem 3.7]{Lai2020_flying_wing}, we argue under the $O(2)$-symmetry assumption which is not available here, use the curve $\Gamma(s)$ fixed under the $O(2)$-symmetry to define a diameter function $F(s)$, and show $\lim_{s\rightarrow\infty}F(s)=\infty$. In our setting, we first find a curve $\tilde{\Gamma}(s):[0,\infty)\rightarrow M$ going to infinity, such that for all $s\ge0$, $(M,R(\tilde{\Gamma}(s))g,\tilde{\Gamma}(s))$ is $\epsilon$-close to either $\R\times\textnormal{Cigar}$ or a time-slice of $\R\times \textnormal{Sausage solution}$ at the tip. Then we can define the diameter function $F(s)$ as in \cite[Theorem 3.7]{Lai2020_flying_wing} using $\tilde{\Gamma}(s)$, and show $\lim_{s\rightarrow\infty}F(s)=\infty$ in the same way.
This implies that $\R\times \textnormal{Sausage solution}$ can not appear in the blow-up limit.
}
 
\yi{Finally, we exclude item (4) as follows: Suppose $\R\times S^2$ is an asymptotic limit, we claim that all asymptotic limits are $\R\times S^2$. If so, then by Brendle's result \cite{brendlesteady3d} we know that a 3D steady Ricci solitons must be Bryant soliton if it dimension reduces to $S^2$ along any sequence of points tending to infinity. This contradicts with our assumption that the soliton is not the Bryant soliton.
To show the claim, suppose by contradiction that there is another type of limit, which by the above argument has to be item (1) or (2). First, we observe the following fact: Let $(M_i,g_i)$, $i=1,2,4$ be one of the limits (1), (2), (4), and $p\in M_i$ be an arbitrary point. For any $\epsilon>0$, there exists $C(\epsilon)>0$ so that for any $C> C(\epsilon)$, the metric space $(M_i,d_{C^{-1}g_i},p)$ is $\epsilon$-close in the pointed Gromov-Hausdorff sense to an $\epsilon^{-1}$-ball in $\R\times\R_+$ (for $i=1$), $\R^2$ (for $i=2$), or $\R$ (for $i=4$). 
In particular, there exists $\epsilon_0>0$ such that for all $C\ge C_0>0$, $(M_4,d_{C^{-1}g_4},p)$ are not $\epsilon_0$-close in the pointed Gromov-Hausdorff sense to neither $(M_1,d_{C^{-1}g_1},p)$ nor $(M_2,d_{C^{-1}g_2},p)$.}

\yi{Since $\R\times S^2$ is not the unique limit, let $C>C(\frac{\epsilon}{3})$, then by an open-closed argument we can find a sequence of points $z_k\in M$ going to infinity as $k\rii$, such that $(M,C^{-1}r^{-2}(z_k)g,z_k)$ is not $\frac{\epsilon_0}{3}$-close to neither of the limits (1)(2)(4) in the pointed Gromov-Hausdorff sense, where $r(z_k)$ is the volume scale at $z_k$. This contradicts the above fact and hence proves the claim.
}
\end{proof}

The following lemma shows that in the Ricci flow of the soliton, the closeness of a time-slice to the asymptotic limit leads to the closeness in a parabolic region of a certain size.

\begin{cor}\label{c: RF closeness}
Let $(M,g)$ be a 3D steady gradient soliton with positive curvature that is not a Bryant soliton. Let $p\in M$ be a fixed point.  Then for any $\epsilon>0$, there is $D(\epsilon)>0$ such that for all $x\in M\setminus B_{g}(p,D)$, the pointed rescaled Ricci flow $(M,r^{-2}(x)g(r^2(x)t),x)$ is $\epsilon$-close to exactly one of the following two Ricci flows
\begin{enumerate}
    \item $(\mathbb{R}\times\cigar,r^{-2}(\widetilde{x})g_c(r^2(\widetilde{x})t),\widetilde{x})$;
    \item $(\RR\times S^1, g_{stan},\widetilde{x})$.
\end{enumerate}
\end{cor}

\begin{proof}
Since the Ricci flows $g_{\mathbb{R}\times\cigar}(t)$ and $g_{\mathbb{R}^2\times S^1}(t)$ on $\R\times\cigar$ and $\RR\times S^1$ are both eternal Ricci flows which have bounded curvature. The assertion now follows from Lemma \ref{l: DR to cigar} by a standard limiting argument as in \cite[Lemma 3.4]{Lai2021Producing3R}.
\end{proof}

In the remaining of this section we show that $\R\times\cigar$ is a stable asymptotic limit when we move forward in the Ricci flow of the soliton, in the sense that a region close to the $\R\times\cigar$ stays close to it until it is not close to any asymptotic limits.

This is based on the observation that the Ricci flow of the cigar soliton contracts all points to the tip when we move forward in time along the flow.
By using this and the closeness to the Ricci flow of cigar, we show in the next lemma that an $\epsilon$-tip point $x$ (see Definition \ref{d: tip points}) stays an $2\epsilon$-tip point outside a compact subset, when we move along an integral curve of $\phi_{-t}(x)$ $-\nabla f$.
Note that this amounts to moving forward along the Ricci flow of the soliton, since  $g(t)=\phi_{-t}^*g$ satisfies the Ricci flow equation, where $\{\phi_t\}_{t\in\R}$ are the diffeomorphisms generated by $\nabla f$.

\begin{lem}\label{l: tip contracting}
Fix some $p\in M$. For any $\epsilon>0$, there exists $D(\epsilon)>0$ such that the following holds.

For any point $x\in M\setminus B_g(p,D)$, $\phi_{-t}(x)$ is the integral curve of $-\nabla f$, $t\in[0,\infty)$.
Suppose $x$ is an $\epsilon$-tip point. Then $\phi_{-t}(x)$ is a $2\epsilon$-tip point for all $t\in[0,t(x))$, where
$t(x)\in(0,\infty]$ is the supremum of $t$ such that $\phi_{-t}(x)\in M-B_g(p,D)$. 
\end{lem}

\begin{proof}

For the fixed $\epsilon$, we choose $T(\epsilon)>0$ to be the constant such that in the Ricci flow of the cigar soliton, the metric ball of radius $2\epsilon$ at time $0$ centered at the tip contracts to a metric ball of radius $\epsilon$ at time $T(\epsilon)$.
Let $0<\epsilon_1<<\epsilon$ be a constant whose value will be chosen later, then choose $D(\epsilon_1)>0$ to be the constant from Corollary \ref{c: RF closeness}.

If $x_0\notin B_g(p,D)$ is an $\epsilon$-tip point, in the following we will show that $\phi_{-r^2(x_0)\,t}(x_0)$ is a $2\epsilon$-point for all $t\in[0,T(\epsilon)]$ such that $\phi_{-r^2(x_0)\,t}(x_0)\notin B_g(p,D)$, and  $\phi_{-r^2(x_0)\,T(\epsilon)}(x_0)$ is again an $\epsilon$-tip point.
Suppose this is true, then the lemma follows by induction immediately.

By Corollary \ref{c: RF closeness}, there is an $\epsilon_1$-isometry $\psi$ between the two pointed Ricci flows $(M,r^{-2}(x_0)g(r^2(x_0)t),x_0)$ and $(\R\times\cigar,r^{-2}(\psi(x_0))g_0(r^{2}(\psi(x_0))t),\psi(x_0))$. 
Note that $\psi$ is also a $100\epsilon_1$-isometry between time-slices, and hence an $\epsilon$-isometry for $\epsilon_1<\frac{\epsilon}{100}$.

Let $x_{tip}$ be the tip of the cigar in $\R\times\cigar$ which has the same $\R$-coordinate as $\psi(x_0)$.
Then by
taking $\epsilon_0$ sufficiently small depending on $\epsilon$, and using the distance shrinking of the cigar, it is easy to see that $d_t(\psi(x_0),x_{tip})<2\epsilon$ in the Ricci flow $r^{-2}(\psi(x_0))g_0(r^{2}(\psi(x_0))t)$, for all $t\ge0$.
This implies that the first half of the claim that $\phi_{-r^2(x_0)\,t}(x_0)$ is a $2\epsilon$-point.
The second half of the claim follows by the choice of $T(\epsilon)$ and taking $\epsilon_1$ sufficiently small such that $\epsilon_1^{-1}>T(\epsilon)$.

\end{proof}

\subsection{The geometry near the edges}
We study the local and global geometry at points that look like $\R\times\cigar$.
First, we show in Lemma \ref{l: two chains} that there are two chains of infinitely many topological 3-balls that cover all $\epsilon$-tip points. 
Using this we show in Lemma \ref{l: two rays} that the asymptotic cone of the soliton is a metric cone over an interval $[0,a]$, $a\in[0,\pi)$, and the points in these two chains correspond to the boundary points of the cone.
Next, in Lemma \ref{l: Gamma} we construct two smooth curves going to infinity inside the two chains, such that they are two integral curves of $\nabla f$ or $-\nabla f$, and along them the soliton converges to $(\mathbb{R}\times\cigar,x_{tip})$.

Fix a point $p\in M$, in the following technical lemma we show that the velocity vector of a minimizing geodesic from $p$ to an $\epsilon$-tip point is almost parallel to the $\R$-direction in $\R\times\cigar$. 
The idea is to study the geometry near an $\epsilon$-tip point $q$ in three different scales: In the largest scale $d(p,q)$, the soliton looks like its asymptotic cone; In the smallest scale by the volume scale $r(q)$, it looks like $\R\times\cigar$; In some intermediate scale between $r(q)$ and $d(p,q)$, it looks like a 2-dimensional upper half-plane. Note when there is no confusion, we will omit the subscript $g$ and write $d_g(\cdot,\cdot)$ as $d(\cdot,\cdot)$ and $B_g(p,\cdot)$ as $B(p,\cdot)$.

\begin{lem}\label{l: parallel}
Let $(M,g)$ be a 3D steady gradient soliton with positive curvature that is not a Bryant soliton. Let $p\in M$ be a fixed point.  
For any $\delta>0$, there exists $\epsilon>0$ such that the following holds:

Let $q\in M$ be an $\epsilon$-tip point, and $\phi$ be an $\epsilon$-isometry.  Let $\partial_{r}$ be the unit vector field in the $\R$-direction in $\R\times\cigar$.
Let $\gamma:[0,1]\ri M$ be a minimizing geodesic from $p$ to $q$.
Then $||\cos\measuredangle(\gamma'(1),\phi_*\left(\partial_{r}\right))|-1|\le\delta$.
\end{lem}

\begin{proof}
Suppose not, then there are $\delta>0$,  $\epsilon_i\ri0$ and $\epsilon_i$-tip points $q_i\rii$ such that 
\begin{equation}\label{e: bigger than delta}
   \measuredangle\left(\gamma_i'(1),\phi_{i*}\left(\partial_{r}\right)\right)\in(\delta,\pi-\delta),
\end{equation}
where $\gamma_i:[0,1]\ri M$ is a minimizing geodesic from $p$ to $q_i$, and $\phi_i$ is the inverse of an $\epsilon_i$-isometry at $q_i$.
For convenience, we will use $\epsilon_i$ to denote any sequences $C\epsilon_i$ where $C>0$ is a constant independent of $i$.

Since $d_g(p,q_i)\rii$ and the curvature is positive, after passing to a subsequence we may assume that the rescaled manifold $(M,d^{-2}(q_i,p)g,p)$ converges to the asymptotic cone in the Gromov-Hausdorff sense, see Lemma \ref{l: converge to asymp}. So we can find a point $z_i\in M$ such that the pair of points $(p,z_i)$ is a $(1,\epsilon_i)$-strainer at $q_i$ of size $d_g(q_i,p)$.
Let $p_i$ be a point on $\gamma_i$ and $w_i$ be a point on the minimizing geodesic connecting $q_i,z_i$ such that $d(w_i,q_i)=d(p_i,q_i)=\frac{r_i}{\delta_i}$, where $\delta_i>\frac{r(q_i)}{d(q_i,p)}$ is a sequence converging to zero, which we may adjust later. Since $\frac{r(q_i)}{\delta_i}<d(q_i,p)$, by the monotonicity of angles, $(p_i,w_i)$ is a $(1,\epsilon_i)$-strainer at $q_i$ of size $\frac{r(q_i)}{\delta_i}$. 
So
\begin{equation}\label{e: large angle}
    \widetilde{\measuredangle} p_iq_iw_i\ge\pi-\epsilon_i.
\end{equation}

Next, consider the rescaled pointed manifold $(M,r^{-2}(q_i)g,q_i)$, which is $\epsilon_i$-close to $(\R\times\cigar,x_{tip})$. 
Then there is a sequence of points $o_i\in M$ with $d(o_i,q_i)=\frac{r(q_i)}{\delta_i}$ such that the minimizing geodesic $\sigma_i:[0,1]\ri M$ from $q_i$ to $o_i$ satisfies $\measuredangle\left(\sigma_i'(0),\phi_{i*}\left(\partial_{r}\right)\right)\ri0$ as $i\rii$.
Combining this fact with \eqref{e: bigger than delta} we get 
\begin{equation*}
    \measuredangle p_iq_io_i\in\left(\delta/2,\pi-\delta/2\right).
\end{equation*}
Since $\epsilon_i\ri0$, by choosing $\delta_i\ri0$ properly we have $|\widetilde{\measuredangle}p_iq_io_i-\measuredangle p_iq_io_i|<\frac{\delta}{10}$, and hence
\begin{equation}\label{e: small angle}
    \widetilde{\measuredangle} p_iq_io_i\in\left(\delta/4,\pi-\delta/4\right)
\end{equation}
for all sufficiently large $i$.

Now consider the rescaled pointed manifold $(M,\delta^2_ir^{-2}(q_i)g,q_i)$. Since $\delta_i\ri0$, after passing to a subsequence we may assume that it converges to the upper half-plane $\R\times\R_+$ in the Gromov-Hausdorff sense, with $q_i\ri(0,0),o_i\ri(1,0)\in \R\times\R_+$ modulo the approximation maps.
Assume $p_i$ converges to $(x_0,y_0)\in \R\times\R_+$, then by \eqref{e: small angle} we have $y_0>\frac{\delta}{100}$. 
On the other hand, \eqref{e: large angle} implies that $(x_0,y_0)$ is one point in a $(1,0)$-strainer at $(0,0)$ of size $1$. So it is clear that $|x_0|=1$, $y_0=0$, a contradiction. This proves the lemma.

\end{proof}

In the next few definitions, we introduce the concept of $\epsilon$-solid cylinders. These are topological 3-balls that look like a large neighborhood of the tip in $\R\times\cigar$. A chain of $\epsilon$-solid cylinder is a sequence of these cylinders meeting nicely. In this subsection, we will show in Lemma \ref{l: two chains} that all $\epsilon$-tip points are covered by exactly two such chains.

\begin{defn}[$\epsilon$-solid cylinder]
Let $x\in M$ be an $\epsilon$-tip point, and $\phi_x$ be the inverse of the corresponding $\epsilon$-isometry. We say that the neighborhood $\nu_{L,D}:=\phi_x([-L,L]\times B_{g_c}(x_{tip},D))$, is an $\epsilon$-solid cylinder centered at $x$, where $L,D>0$ are constants. 
\end{defn}

In order to make sure that the union of two intersecting $\epsilon$-solid cylinders is still a topological 3-ball, we want them to meet nicely. So we introduce the concept of good intersection between two $\epsilon$-solid cylinders, see e.g. \cite[Section 5.6]{MT2}.

\begin{defn}[Good intersection]\label{d: good intersection}
Let $y_1,y_2$ be two $\epsilon$-tip points, and $\phi_{y_i}$ be the inverses of corresponding $\epsilon$-isometries.
Let $\widetilde{\gamma}_i:[-L_i,L_i]\ri M$ be a curve passing through $y_i$ such that $\phi_{y_i}^{-1}(\widetilde{\gamma}_i(s))=(s,x_{tip})$. 
We say $\nu(i)=\phi_{y_i}([-L_i,L_i]\times B_{g_c}(x_{tip},D_i))$, $i=1,2$, have good intersection if after possibly reversing the directions either or both of the $\R$-factors, the following hold:
\begin{enumerate}
    \item The projection $r_1$ in the direction of $\R$ is an increasing function along $\widetilde{\gamma}_2$ at any point of $\widetilde{\gamma}_2\cap \nu(1)$.
    \item There is a point in the negative end of $\nu(2)$ that is contained in $\nu(1)$, and the positive end of $\nu(2)$ is disjoint from $\nu(1)$.
    \item Either $(1.1)D_1r(y_1)\le D_2r(y_2)$ or $D_2r(y_2)\le (0.9)D_1r(y_1)$.
\end{enumerate}
\end{defn}

With the notion above, if two $\epsilon$-solid cylinders have good intersection, then the intersection is homeomorphic to a 3-ball, see \cite[Lemma 5.19]{MT2}.

\begin{defn}[Chain]
Suppose that we have a sequence of $\epsilon$-solid cylinders $\{\nu(1),\cdots,\nu(k)\}$, $k\in\mathbb{N}\cup\{\infty\}$, with the curves $\widetilde{\gamma}_i$ from Definition \ref{d: good intersection}. We say that they form a chain of $\epsilon$-solid cylinders if the following hold:
\begin{enumerate}
    \item For each $1\le i<k$ the open sets $\nu(i)$ and $\nu(i+1)$ have a good intersection with the given orientations.
    \item If $\nu(i)\cap\nu(j)\neq\emptyset$ for some $i\neq j$, then $|i-j|=1$.
\end{enumerate}
\end{defn}

\begin{lem}(c.f. \cite[Lemma 5.22]{MT2})
Suppose that $\{\nu(1),\cdots,\nu(k)\}$ is a chain of $\epsilon$-solid cylinders. Then $\nu(1)\cup\cdots\cup\nu(k)$ is homeomorphic to a 3-ball and its boundary is the union of the negative end of $\nu(1)$, the positive end of $\nu(k)$, and an annulus $A$.
\end{lem}

For an $\epsilon$-tip point $x$, let $\nu$ be an $\epsilon$-solid cylinder centered at it. In Lemma \ref{l: one chain}, we construct a chain of $\epsilon$-solid cylinders starting from $\nu$, which extends to infinity on one end. Moreover, the $\epsilon$-solid cylinder on the other end meets a 2-sphere metric sphere $\partial B(p,D_0)$ in a spanning disk, where $D_0>0$ is independent of $x$.

\begin{lem}[Extend an $\epsilon$-solid cylinder to a chain]\label{l: one chain}
Let $(M,g)$ be a 3D steady gradient soliton with positive curvature that is not a Bryant soliton. Let $p\in M$ be a fixed point. There exists $\overline{D}>0$ such that the following is true.

For any $\epsilon$-tip point $x\in M\setminus B(p,\overline{D})$,
there exist an integer $N\in\mathbb{N}$, a sequence of $\epsilon$-tip points $\{x_i\}_{i=-N}^{\infty}$ going to infinity, $x_{-N}\in \partial B(p,\overline{D})$, and an infinite chain of $\epsilon$-solid cylinders
$\{\nu(i)\}_{i=-N}^{\infty}$ centered at $x_i$, where $\nu(i)=\phi_{x_i}([-900,900]\times B_{g_c}(x_{tip}, L_i))$, $L_i\in[500,1000]$,
such that if $\nu(i)$ intersects with $\partial B(p,D)$, $D\ge \overline{D}$, then $\nu(i)$ meets $\partial B(p,D)$, $D\ge \overline{D}$, in a spanning disk.

\end{lem}

\begin{proof}
By Lemma \ref{l: DR to cigar} we can choose $\overline{D}>0$ sufficiently large such that the rescaled soliton at any point in $M\setminus B(p,\overline{D}/2)$ is $\epsilon$-close to the asymptotic limits.
For an $\epsilon$-tip point $x_0\in M\setminus B(p,\overline{D})$, assume $d(p,x_0)=D_0$.
Let $\phi_{x_0}$ be the inverse of the $\epsilon$-isometry. Then
$\nu(0):=\phi_{x_0}([-900,900]\times B_{g_c}(x_{tip}, 1000))$ is an $\epsilon$-solid cylinder. 
Let $r$ be the projection onto the the $\R$-direction in $\R\times\cigar$. Then by Lemma \ref{l: parallel}, after possibly replacing $r$ by $-r$, we may assume $\measuredangle(\nabla d(p,\cdot),\phi_{x_0*}(\partial_{r}))<\epsilon$. So the two points $y_{\pm}:=\phi_{x_0}((\pm1000,x_{tip}))\in M$ satisfy
\begin{equation*}
    d(y_+,p)\ge(1-\epsilon)1000\,r(x_0)+D_0\quad\textit{and}\quad d(y_-,p)\le-(1-\epsilon)1000\,r(x_0)+D_0.
\end{equation*}

By the choice of $\overline{D}$ we can find
two $\epsilon$-tip points $x_{\pm1}\in B(y_{\pm},r(x_0))$. In particular, we have $x_{\pm1}\notin \phi_{x_0}([-900,900]\times B_{g_c}(x_{tip}, 1000))$, and
\begin{equation*}
    d(x_1,p)>900\,r(x_0)+D_0\quad \textit{and}\quad d(x_{-1},p)<-900\,r(x_0)+D_0.
\end{equation*}
Similarly, let $\phi_{x_{\pm1}}$ be the inverse of the $\epsilon$-isometry at $x_{\pm1}$, then $\nu(\pm1):=\phi_{x_{\pm1}}([-900,900]\times B_{g_c}(x_{tip}, 500))$ is an $\epsilon$-solid cylinder at $x_{\pm1}$.
It is clear that $\nu(0)$ and $\nu(\pm1)$ have good intersections.

Repeating this, we can obtain a sequence of $\epsilon$-tip points $\{x_i\}_{i=-N}^{\infty}$, $x_{-N}\in B(p,\overline{D})$,
and a sequence of $\epsilon$-solid cylinders $\nu(i):=\phi_{x_i}([-900,900]\times B_{g_c}(x_{tip}, D_i))$ centered at $x_i$, where $D_i=500$ when $i$ is odd and $D_i=1000$ when $i$ is even, such that $x_{i+1}\in B(\phi_{x_i}((1000,x_{tip})),10 r(x_i))$.
Therefore, by triangle inequalities it is easy to see that $\nu(i)$ only intersects with $\nu(i-1)$ and $\nu(i+1)$, and have good intersections with them. In particular, for all $i\ge0$ we have
\begin{equation*}
    d(x_i,p)>900 \sum_{k=0}^{i-1}r(x_k)+D_0\rii.
\end{equation*}
So $\{\nu(i)\}_{i=1}^{\infty}$ is an infinite chain of $\epsilon$-solid cylinders.
Moreover, we see that $\nu(i)$ meets all metric spheres $\partial B(p,D)$, $D\ge \overline{D}$, in a spanning disk since by Lemma \ref{l: parallel} we have $\measuredangle(\nabla d(p,\cdot),\phi_{x_i*}(\partial_{r}))<\epsilon$ for all $i$.
\end{proof}

In the following, we use Lemma \ref{l: one chain} to show that all $\epsilon$-tip points outside of a compact subset are contained in the union of finitely many chains of $\epsilon$-solid cylinders.

\begin{lem}[Disjoint chains containing all $\epsilon$-tip points]\label{l: finitely many chains}
Under the same assumption as in Lemma \ref{l: one chain}, and let $\overline{D}$ be from Lemma \ref{l: one chain}.
There exist $k$ chains of $\epsilon$-solid cylinders  $\mathcal{C}_1,\cdots,\mathcal{C}_k$, $k\in\mathbb{N}$, each of which satisfies the conclusions in Lemma \ref{l: one chain}. Moreover, all $\epsilon$-tip points in $M\setminus B(p,\overline{D})$ are contained in the union of $\mathcal{C}_1,\cdots,\mathcal{C}_k$.
\end{lem}

\begin{proof}
Assume $x_1$ is an $\epsilon$-tip point, $x_1\in M\setminus B(p,\overline{D})$. 
Then let $\mathcal{C}_1$ be a chain of $\epsilon$-solid cylinders produced by Lemma \ref{l: one chain} whose union contains $x_1$.
If there exists an $\epsilon$-tip point $x_2\in M\setminus(\mathcal{C}_1\cup B(p,\overline{D}))$, then by Lemma \ref{l: one chain} we can construct a new chain $\mathcal{C}_2$ of $\epsilon$-solid cylinders containing $x_2$.
We claim that $\mathcal{C}_1\cap\mathcal{C}_2=\emptyset$, i.e. the union of all $\epsilon$-solid cylinders in $\mathcal{C}_1$ is disjoint from that in $\mathcal{C}_2$.

First, for two $\epsilon$-tip points $x,y$ with $d(x,y)\le1000\,r(x)$, it is easy to see that $d(x,y)\le 0.1r(x)$ when $\epsilon$ is sufficiently small.
Using this fact we see that $x_2$ is at least $900r(x_2)$-away from 
$\mathcal{C}_1$, and hence $\mathcal{C}_1\cap\mathcal{C}_2\cap \partial B(p,D_1)=\emptyset$, where $D_1=d(x_2,p)\ge \overline{D}$. Let $a,b>0$ be the infimum and supremum of $r$ such that $\mathcal{C}_1\cap\mathcal{C}_2\cap \partial B(p,r)=\emptyset$.
Then $a<D_1<b$.
Suppose by contradiction that $a\ge \overline{D}$.
Then $\mathcal{C}_1\cap\mathcal{C}_2\cap \partial B(p,a)\neq\emptyset$, and we can find an $\epsilon$-tip $y\in\mathcal{C}_1\cap\mathcal{C}_2\cap \partial B(p,a)$.
However, by Lemma \ref{l: parallel} this implies that $\nu$ the $\epsilon$-solid cylinder centered at $y$, is contained in $\mathcal{C}_1\cap\mathcal{C}_2$ and the positive end of $\nu$ is at distance $a+1000r(y)$ to $p$, which is greater than $a$.
This contradicts the infimum assumption of $a$.
By a similar argument we can show $b=\infty$.
This proves the claim.

Repeating this procedure, we obtain a sequence of chains of $\epsilon$-solid cylinders whose unions are disjoint.
This must stop in finite steps, because these chains intersect with $\partial B(p,\overline{D})$ in spanning disks whose areas are uniformly bounded below.
So we may assume these chains are $\mathcal{C}_1,\cdots,\mathcal{C}_k$ and they contain all $\epsilon$-tip points.

\end{proof}

Now we show that the number of these chains are exactly equal to two.
To do this, we need the following lemma, \cite[Proposition 4.4]{MT2}, which enables us to glue up $\epsilon$-cylindrical planes that intersect, and produce a global $S^1$-fibration on their union.

\begin{lem}[global $S^1$-fibration]\label{l: MT-glue} (\cite[Proposition 4.4]{MT2})
Let $M$ be a 3D Riemannian manifold. Given $\epsilon'>0$, the following holds for all $\epsilon>0$ less than a positive constant $\epsilon_1(\epsilon')$.
Suppose $K\subset M$ is a compact subset and each $x\in K$ is the center of an $\epsilon$-cylindrical plane.
Then there is an open subset $V$ containing $K$ and a smooth $S^1$-fibration structure on $V$.

Furthermore, if $U$ is an $\epsilon$-cylindrical plane that contains a fiber $F$ of the fibration on $V$, then $F$ is $\epsilon'$-close to a vertical $S^1$-factor in $U$ and $F$ generates the fundamental group of $U$. 
In particular, the diameter of $F$ is at most twice the length of any circle in the $\epsilon$-cylindrical plane centered at any point of $F$.
\end{lem}

\begin{lem}[Two chains containing all $\epsilon$-tip points]\label{l: two chains}
Let $(M,g)$ be a 3D steady gradient soliton with positive curvature that is not a Bryant soliton. Let $p\in M$ be a fixed point. There exists $\overline{D}>0$ such that the following is true.

There are exactly $2$ chains of $\epsilon$-solid cylinders  $\mathcal{C}_1,\mathcal{C}_2$, each of which satisfies the conclusions in Lemma \ref{l: one chain}. Moreover, all $\epsilon$-tip points in $M\setminus B(p,\overline{D})$ are contained in the union of $\mathcal{C}_1$ and $\mathcal{C}_2$.

\end{lem}

\begin{proof}

First, we show $k\ge 1$, which is equivalent to that the $\R\times\cigar$ is indeed an asymptotic limit of the soliton $M$.
Suppose this is not the case.
On the one hand, since all $\epsilon$-tip points are contained in $\mathcal{C}_1,\cdots,\mathcal{C}_k$,
the complement of them in $M\setminus B(p,\overline{D})$ is covered by $\epsilon$-cylindrical planes.
Then by Lemma \ref{l: MT-glue} we can find a connected open subset $V_1$ containing $B(p,2\overline{D})-B(p,\overline{D})$ which carries a smooth $S^1$-fibration.
Consider the homotopy exact sequence
\begin{equation*}
    \cdots\ri\pi_2(V_{1,0})\ri\pi_1(S^1)\ri\pi_1(V_1)\ri\pi_1(V_{1,0})\ri\cdots
\end{equation*}
Since the base space $V_{1,0}$ of this fibration is connected and non-compact, we have $\pi_2(V_{1,0})=0$, and 
it follows that the $S^1$-fiber is incompressible in $V_1$, i.e. the map $\pi_1(S^1)\ri\pi_1(V_1)$ is an injection, see e.g. \cite{Hatcher}.

On the other hand, let $U$ be an $\epsilon$-cylindrical plane contained in $B(p,2\overline{D})-B(p,\overline{D})$, $x\in U$ be the center of the $\epsilon$-cylindrical plane, and $F$ be the $S^1$-fiber of the fibration on $V$ that passes through $x$. 
Then by Lemma \ref{l: MT-glue}, the fiber $F$ is $\epsilon$-close to the vertical $S^1$-factor in $U$, and hence $F$ is also contained in $U$, and hence contained in $B(p,2\overline{D})-B(p,\overline{D})$.
Note that the metric spheres $\partial B(p,\overline{D})$ and $\partial B(p,2\overline{D})$ are diffeomorphic to 2-spheres, and $B(p,2\overline{D})-B(p,\overline{D})$ is difffeomorphic to $\R\times S^2$.
So if $k=0$, we see that $B(p,2\overline{D})-B(p,\overline{D})-\mathcal{C}_1-\cdots-\mathcal{C}_k$ is diffeomorphic to $\R\times S^2$; if $k=1$, it is diffeomorphic to $\R^3$, both of which have trivial fundamental groups. Therefore, $F$ bounds a disk in $B(p,2\overline{D})-B(p,\overline{D})$.
However, this contradicts with the above fact that the fiber $F$ is incompressible.
So $k\ge 2$.

Next, we show $k=2$.
Suppose not, then $k\ge3$.
Since the subset $K:=B(p,2\overline{D})-B(p,\overline{D})-\mathcal{C}_1-\mathcal{C}_2-\cdots-\mathcal{C}_k$ is covered by $\epsilon$-cylindrical planes, by Lemma \ref{l: MT-glue} we can obtain a smooth $S^1$-fibration on a subset $V$ containing $K$. Let $C_1,C_2$ be two circles of the fibration on $V$ that are contained in $\mathcal{C}_1,\mathcal{C}_2$, respectively.
Then by Lemma \ref{l: MT-glue}, $C_i$
bounds a spanning disk $D_i$ in $\mathcal{C}_i$, $i=1,2$.
Let $A\subset B(p,2\overline{D})-B(p,\overline{D})$ be an annulus bounded by $C_1,C_2$, which is saturated by $S^1$-fibers. Then the union $S:=D_1\cup D_2\cup A$ is a 2-sphere which is isotopic to the two metric spheres $\partial B(p,\overline{D})$ and $\partial B(p,2\overline{D})$. In particular, $S$ separates $\partial B(p,\overline{D})$ from $\partial B(p,2\overline{D})$.

Since $k\ge3$, it follows that $\mathcal{C}_3$ intersects with $S$ at some $\epsilon$-tip point $y$.
First, $y$ cannot be in the annulus $A$ because $A$ is covered by $\epsilon_0$-cylindrical planes for a very small $\epsilon_0>0$. Second, $y$ cannot be in either of the two spanning disks $D_1$ and $D_2$, because  $\mathcal{C}_3$ is disjoint from $\mathcal{C}_i$, $i=1,2$.
So this contradiction proves $k=2$.

\end{proof}

It is clear that at the volume scale, an $\epsilon$-tip point is different from a point at which the manifold is an $\epsilon$-cylindrical plane.
We show in the following technical lemma that they can be also distinguished from each another at an even larger scale. Roughly speaking, the former point looks like a boundary point in the half-plane, while the latter looks like a point in the plane. 

\begin{lem}\label{l: tip point has to be edge}
Let $(M,g)$ be a 3D steady gradient soliton with positive curvature that is not a Bryant soliton. Let $p\in M$ be a fixed point.
For any $\delta>0$, there exists $\epsilon>0$ such that for any $\epsilon$-tip point $x$, there exists a constant $0<r_x<\delta d(p,x)$ such that there is no 2-strainer at $x$ of size larger than $r_x$.
\end{lem}

\begin{proof}
First, note that for $x$ sufficiently far away from $p$, we have by volume comparison that $\frac{r(x)}{d(p,x)}\le \delta^2$.
Take $r_x=r(x)/\delta$, then $r_x<\delta d(p,x)$ and the rescaling $(M,r_x^{-2}g, x)$ is the scaling down of $(M,r^{-2}(x)g, x)$ by $\delta^{-1}$.
The scaling down of $(\R\times\cigar,x_{tip})$ by $\delta^{-1}$ is $C\delta$-close to the 2-dimensional upper half-plane $(\R\times\R_+,(0,0))$ in the pointed Gromov-Hausdorff sense.
Take $\epsilon<\delta$, by the definition of $\epsilon$-tip point, we see that $(M,r^{-2}(x)g, x)$ is $\epsilon$-close to $(\R\times\cigar,x_{tip})$.
Therefore, $(M,r_x^{-2}g, x)$ is $C\epsilon\delta$-close in the Gromov-Hausdorff sense to the metric ball $B((0,0),1)$ in the 2-dimensional upper half-plane $(\R\times\R_+,(0,0))$.

It is easy to see that there is not a $(2,\epsilon_0)$-strainer at $(0,0)$ in $\R\times\R_+$ of size larger than $1$, which is the same in $(M,r_x^{-2}g, x)$.
Scaling back, it follows that there is not a $(2,\epsilon_0)$-strainer at $x$ in $(M,g)$ of size larger than $r_x$.
\end{proof}

Recall that $S_{\infty}(M,p)$ is the metric space of equivalent classes of rays starting from $p$.
We show in the following lemma that $S_{\infty}(M,p)$ is a closed interval $[0,\theta]$ for some $\theta\in[0,\pi)$.
Moreover, the two chains of $\epsilon$-solid cylinders $\mathcal{C}_1,\mathcal{C}_2$ are two ``edges" of the soliton $(M,g)$, in the sense that the two rays $\gamma_1,\gamma_2$ correspond to the two end points in $[0,\theta]$.
Note we will show $\theta>0$ in Corollary \ref{c: Asymptotic to a sector} so that the soliton is indeed a flying wing.

\begin{lem}\label{l: two rays}
Let $(M,g)$ be a 3D steady gradient soliton with positive curvature that is not a Bryant soliton. Let $p\in M$ be a fixed point. Then $S_{\infty}(M,p)$ is isometric to an interval $[0,\theta]$ for some $\theta\in[0,\pi)$.

Moreover, for any sequence of points $p^{(i)}_k\in\mathcal{C}_i$, $i=1,2$, $p^{(i)}_k\rii$ as $k\rii$, the minimizing geodesics $pp^{(i)}_k$ subsequentially converge to two rays $\gamma_1,\gamma_2$, such that $[\gamma_1]=0$, $[\gamma_2]=\theta$, after possibly switching $\gamma_1,\gamma_2$.
\end{lem}

\begin{proof}
Fix some $p\in M$. We shall say that two quantities $d_1,d_2>0$ are comparable if $C^{-1}d_1<d_2<Cd_1$ for some universal constant $C>0$. 
By the non-euclidean volume growth of $(M,g)$, the asymptotic cone is a 2-dimensional metric cone over $S_{\infty}(M,p)$, where $S_{\infty}(M,p)$ is a 1-dimensional Alexandrov space, and hence is an interval or a circle.
In the latter case, we can find a $(2,\epsilon)$-strainer of size comparable to $d(p,x)$ at any point $x\in M$. However, this is impossible at an $\epsilon$-tip point by Lemma \ref{l: tip point has to be edge}. 
So $S_{\infty}(M,p)$ is isometric to an interval $[0,\theta]$, $\theta\in[0,\pi]$.
Moreover, we have $\theta<\pi$, because otherwise the manifold splits off a line and is isometric to $\R\times\cigar$, which does not have strictly positive curvature.

First, we show that for any points going to infinity in $\mathcal{C}_1$, the minimizing geodesics from $p$ to them subsequentially converge to rays that are in the same equivalent class in $S_{\infty}(M,p)$.
Suppose by contradiction that this is not true. Then we can find two sequences of points $p_{1k},p_{2k}\in\mathcal{C}_1$ going to infinity, such that the minimizing geodesics $pp_{1k},pp_{2k}$ converge to two rays $\sigma_1,\sigma_2$ with $\widetilde{\measuredangle}(\sigma_1,\sigma_2)>0$.

Let $\beta_k:[0,1]\ri\mathcal{C}_1$ be a smooth curve joining $p_{1k},p_{2k}$, which consists of $\epsilon$-tip points. 
By Lemma \ref{l: parallel} we may assume that $d(p,\beta_k(s))$ is monotone in $s$.
Let $\theta_{ik}(s)=\widetilde{\measuredangle}\sigma_i(d(p,\beta_k(s)))\,p\,\beta_k(s)$, $i=1,2$.
Then it is clear that $\theta_{1k}(0),\theta_{2k}(1)$ converge to $0$, and $\theta_{1k}(1),\theta_{2k}(0)$ converge to $\widetilde{\measuredangle}(\sigma_1,\sigma_2)>0$, as $k\rii$.
So for sufficiently large $k$, we have  $\frac{\theta_{1k}(0)}{\theta_{2k}(0)}<\epsilon$ and $\frac{\theta_{1k}(1)}{\theta_{2k}(1)}>\epsilon^{-1}$.

By continuity, there exists a point $q_k=\beta_k(s_k)$, $s_k\in(0,1)$, such that $\theta_{1k}(s_k)=\theta_{2k}(s_k)$, and hence
\begin{equation*}
    \widetilde{\measuredangle}\sigma_1(d(p,q_k))\,p\,q_k=\widetilde{\measuredangle}\sigma_2(d(p,q_k))\,p\,q_k.
\end{equation*}
Since $d(p,p_{1k}),d(p,p_{2k})\rii$, we have $d(p,q_k)\rii$. 
Therefore, after passing to a subsequence $pq_k$ converges to a ray $\sigma_3$, which satisfies $\widetilde{\measuredangle}(\sigma_3,\sigma_1)=\widetilde{\measuredangle}(\sigma_3,\sigma_2)$.
Since $S_{\infty}(M,p)$ is an interval, this implies
\begin{equation*}
    \widetilde{\measuredangle}(\sigma_3,\sigma_1)=\widetilde{\measuredangle}(\sigma_3,\sigma_2)= \widetilde{\measuredangle}(\sigma_1,\sigma_2)/2.
\end{equation*}
So for sufficiently large $k$, we can find a $(2,\epsilon)$-strainer at $q_k$ of size comparable to $d(q_k,p)$, which is impossible by Lemma \ref{l: tip point has to be edge} because $q_k$ is an $\epsilon$-tip point.

Now it remains to show that the ray $\sigma_1$ corresponds to one of the two end points in $S_{\infty}(M,p)=[0,\theta]$. This can be shown by a similar argument:
Suppose $\theta>0$ and $[\sigma_1]\in (0,\theta)$. 
Then we can find $(2,\epsilon_0)$-strainers at a sequence of $\epsilon$-tip points $x_k\rii$ at scales comparable to $d(p,x_k)$, which contradicts with Lemma \ref{l: tip point has to be edge}.

\end{proof}

Now we prove our main result in this subsection.
We construct two smooth curves $\Gamma_i:[0,\infty)\ri\mathcal{C}_i$ tending to infinity, $i=1,2$, which are integral curves of either $\nabla f$ or $-\nabla f$, such that the rescaled manifold converges to $\R\times\cigar$ pointed at the tip along $\Gamma_i$.

\begin{lem}\label{l: Gamma}
Let $(M,g)$ be a 3D steady gradient soliton with positive curvature that is not a Bryant soliton. 
There exists a smooth curve $\Gamma_i:[0,\infty)\ri\mathcal{C}_i$ which is an integral curve of $\nabla f$ or $-\nabla f$, such that $\lim_{s\rii}\Gamma_i(s)=\infty$, and for any sequence of points $x_i\rii$ along $\Gamma_i$, the pointed manifolds $(M,r^{-2}(x_i)g,x_i)$ smoothly converge to $(\R\times\cigar,r^{-2}(x_{tip})g_c,x_{tip})$.
\end{lem}

\begin{proof}
Fix some point $p\in M$, and let $D_0>0$ be a large constant such that if $f$ has a critical point (which must be unique), then $B(p,D_0)$ contains the critical point.
We will use $\epsilon(D)>0$ denote all positive constants depending on $D$ such that $\epsilon\ri0$ as $D\rii$.

Take a sequence of $\epsilon_k$-tip points $p_k\in \mathcal{C}_1$ going to infinity, where $\epsilon_k\ri0$.
Assume $p_k\notin B(p,D_0)$ for all $k$.
Denote by $\gamma_{p_k}:[0,s_k]\ri M$ the integral curve of $-\nabla f$ starting from $p_k$, where $s_k\in(0,\infty]$ is the smallest value of $s$ such that $\gamma_{p_k}(s)\in \partial B(p,D_0)$.
It is easy to see that $s_k\rii$ as $k\rii$.
By Lemma \ref{l: tip contracting} we see that $\gamma_{p_k}([0,s_k])\subset\mathcal{C}_1$ and if $\gamma_{p_k}(s)\notin B(p,D)$ for some $D>0$, then $\gamma_{p_k}(s)$ is an $\epsilon(D)$-tip point, where $\epsilon(D)\ri0$ as $D\rii$ is independent of $k$.

First, assume $s_k$ is finite for all $k$. 
Let $q_k=\gamma_{p_k}\in\mathcal{C}_1\cap \partial B(p,D_0)$, then after passing to a subsequence $q_k$ converges to a point $q_{\infty}\in\mathcal{C}_1\cap \partial B(p,D_0)$.
Denote by $\widetilde{\gamma}_{q_k}$ the integral curve of $\nabla f$ starting from $q_k$. Then $\widetilde{\gamma}_{q_k}$ converges to the integral curve $\Gamma_1:[0,\infty)\ri\mathcal{C}_1$ of $\nabla f$ starting from $q_{\infty}$, which satisfies all the assertions.

Now assume $s_k=\infty$ for some $k=k_0$. Since the critical point of $f$ (if exists) is contained in $B(p,D_0)$, this implies that $\gamma_{p_{k_0}}(s)\rii$ as $s\rii$.
Therefore, we may take $\gamma_{p_{k_0}}$ to be $\Gamma_1$, which is an integral curve of $-\nabla f$ and satisfies all assertions as a consequence of Lemma \ref{l: tip contracting}.
By the same argument we can find $\Gamma_2\subset\mathcal{C}_2$ satisfying the assertions.

\end{proof}

\begin{remark}
When $(M,g)$ is isometric to $\R\times\cigar$, and the potential function $f$ could be a non-constant linear function in the $\R$-direction, then one of $\Gamma_1,\Gamma_2$ is the integral curve of $\nabla f$ and the other is of $-\nabla f$.
However, if $(M,g)$ has strictly positive curvature, we will show in Theorem \ref{t: Rmax critical point} that $f$ has a unique critical point, and $\Gamma_1,\Gamma_2$ are both integral curves of $\nabla f$.
\end{remark}

Let $\Gamma=\Gamma_1([0,\infty))\cup\Gamma_2([0,\infty))$, the following lemma shows that the distance to the subset $\Gamma$ must be achieved at interior points, so that the minimizing geodesics connecting them to $\Gamma$ are orthogonal to $\Gamma$ by the first variation formula.

\begin{lem}\label{l: inf achieved on the non-compact portion}
Under the same assumptions as in Lemma \ref{l: Gamma}. Suppose there are a sequence of points $p_i\rii$ and a constant $s_0>0$, such that the following holds:
\begin{equation*}
d(p_i,\Gamma)=d(p_i,\Gamma_1([0,s_0])\cup\Gamma_2([0,s_0]))
\end{equation*}
Then $M$ is isometric to $\R\times\cigar$.
\end{lem}

\begin{proof}
We will show that after passing to a subsequence, the minimizing geodesics from $p$ to $p_i$ converge to a geodesic ray $\gamma$ such that $\measuredangle(\gamma,\gamma_i)\ge\pi/2$, where $\gamma_i$ are the two rays corresponding to the two end points in $S_{\infty}(M,p)$, see Lemma \ref{l: two rays}. Then by Lemma \ref{l: two rays} this implies $S_{\infty}(M,p)=[0,\pi]$ and the assertion of the lemma follows immediately. 
In the proof we use $\epsilon$ to denote all positive constants that goes to zero as $i\rii$.

Let $q_i\in\Gamma$ such that $d(p_i,\Gamma([-s_0,s_0]))=d(p_i,q_i)$. Choose a sequence of points $o_i\in \Gamma_1$ with $d(q_i,o_i)=\alpha_i d(p_i,q_i)$, such that $\alpha_i\ri0$ and $d(q_i,o_i)\rii$.
Since $\alpha_i\ri0$, we have $\widetilde{\measuredangle}o_ip_iq_i\ri0$, and hence
\begin{equation*}
    \widetilde{\measuredangle}o_iq_ip_i+\widetilde{\measuredangle}q_io_ip_i\ge\pi-\epsilon.
\end{equation*}
Since $d(p_i,o_i)\ge d(p_i,\Gamma)=d(p_i,q_i)$, the segment $o_iq_i$ is the longest in the comparison triangle $o_ip_iq_i$, so it must be opposite to the largest comparison angle, i.e.
\begin{equation*}
    \widetilde{\measuredangle}o_iq_ip_i\ge \widetilde{\measuredangle}q_io_ip_i.
\end{equation*}
So the last two inequalities imply
\begin{equation}\label{e: soon}
    \widetilde{\measuredangle}o_iq_ip_i\ge\pi/2-\epsilon. 
\end{equation}
By Lemma \ref{l: two rays} the minimizing geodesics $po_i$ converge to the ray $\gamma_1$.
After passing to a subsequence we may assume $pp_i$ converge to a ray $\gamma$.
Then by the boundeeness of $\{q_i\}$ and \eqref{e: soon}, it is easy to see that
$\measuredangle (\gamma_1,\gamma)\ge\pi/2$.
By a symmetric argument we also have $\measuredangle (\gamma_2,\gamma)\ge\pi/2$. So $\measuredangle (\gamma_1,\gamma_2)\ge\pi$ by Lemma \ref{l: Gamma}, which implies a splitting of $(M,g)$, and hence it is isometric to $\R\times\cigar$. 
\end{proof}

\subsection{Quadratic curvature decay}
The main result in this subsection is the following theorem of the quadratic curvature decay, which corresponds the upper bound in Theorem \ref{t': curvature estimate}. We show that there is a uniform $C>0$ such that the scalar curvature has the upper bound $R\le\frac{C}{d^2(\cdot,\Gamma)}$, where $\Gamma=\Gamma_1([0,\infty))\cup\Gamma_2([0,\infty))$ and $\Gamma_1,\Gamma_2$ are from Lemma \ref{l: Gamma}.

\begin{theorem}\label{l: curvature upper bound initial}
Let $(M,g)$ be a 3D steady gradient soliton with positive curvature that is not a Bryant soliton. 
There exists $C>0$ such that for any $x\in M\setminus\Gamma$, 
\begin{equation*}
 R(x)\le\frac{C}{d^2(x,\Gamma)}.
\end{equation*}
\end{theorem}

In the following we will introduce some constants and we may further adjust their values. The dependence of these constants is subject to the following order,
\begin{equation*}
    \delta,\alpha,D_2,\epsilon,D_1,\omega,C,
\end{equation*}
such that each constant is chosen depending only on the preceding ones.
More precisely, we will choose $D_1,D_2$ sufficiently large such that $\delta$ can be arbitrarily close to zero.

Let $\epsilon>0$ be some very small number, and fix a point $p_0\in M$.
Then by Lemma \ref{l: DR to cigar}, \ref{l: two chains} and \ref{l: Gamma} we see that there are $D_1>D_2>0$ such that the following holds:
First, the soliton is an $\epsilon$-cylindrical plane at all points $x$ which satisfies $d(x,p_0)\ge \frac{1}{2}D_1$ and $d(x,\Gamma)\ge \frac{1}{2}D_2 \,r(x)$.
Second, the soliton is $\epsilon$-close to $(\R\times\cigar,r^{-2}(x_0)g_c,x_0)$ for some $x_0$, $d_{g_c}(x_{tip},x_0)<10 D_2$.

By compactness we may choose $C>0$ large enough, so that Theorem \ref{l: curvature upper bound initial} holds in $B(p_0,D_1)$.
Moreover, for points $x\in M\setminus B(p_0,D_1)$ which satisfies $d(x,p_0)\ge D_1$ and $d(x,\Gamma)\le D_2 \,r(x)$, by the definition of the volume scale  and a volume comparison argument, it is clear that there exists $\omega>0$ such that
\begin{equation*}
     vol(B(\widetilde{x},d(x,\Gamma)))\ge \omega\,d(x,\Gamma)^3.
\end{equation*}
So Theorem \ref{l: curvature upper bound initial} holds at $x$ by Perelman's curvature estimate, see Lemma \ref{l: Perelman}.

Therefore, from now on we assume that 
\begin{equation}\label{e: captain}
    d(x,p_0)\ge D_1\quad\textit{and}\quad d(x,\Gamma)\ge D_2 \,r(x).
\end{equation}
So the metric ball $B(x,d(x,\Gamma)/2)$ is covered by $\epsilon$-cylindrical planes.
By the $S^1$-fibration Lemma \ref{l: MT-glue}, we see that there is an open subset $U$ containing $B(x,d(x,\Gamma)/2)$ which has a global $S^1$-fibration, and the $S^1$-fibers are incompressible.

The following lemma gives a uniform lower bound for the volume ratio on the universal coverings of $U$. So by applying Perelman's curvature estimate in the universal covering we obtain the inequality in Theorem \ref{l: curvature upper bound initial} at a lift of $x$, which implies the same inequality at $x$. So Theorem \ref{l: curvature upper bound initial} reduces to the following lemma.

\begin{lem}\label{l: non-collapsing in universal cover}
There exists $\omega>0$ such that in the universal cover $\widetilde{U}$ of $U$, we have $B(\widetilde{x},d(x,\Gamma)/2)\subset\subset \widetilde{U}$, and
\begin{equation*}
 vol(B(\widetilde{x},d(x,\Gamma)/2))\ge \omega\,d^3(x,\Gamma),
\end{equation*}
where $\widetilde{x}\in\widetilde{U}$ is a lift of $x$.
\end{lem}

\begin{proof}[Proof of Lemma \ref{l: non-collapsing in universal cover}]
To prove the lemma, we will follow the idea in \cite[Lemma 2.2]{BamD} to construct a $(3,\delta)$-strainer near $\widetilde{x}$ of size comparable to $d(x,\Gamma)$ for some small $\delta>0$, then use this to obtain a lower bound on the volume ratio as in \cite[Lemma 2.2]{BamD} and \cite[Theorem 10.8.18]{BuragoBuragoIvanov}.
We first construct a $(2,\delta)$-strainer at $\widetilde{x}$ using Claim \ref{c: alpha triangle} and \ref{l: alp}.

\begin{claim}\label{c: alpha triangle}
There exist $p,q\in M$ such that the triple of points $\{p,q,x\}$ forms a $\frac{\pi}{6}$-triangle and $d(x,p)=d(x,\Gamma)$. Here by a $\mu$-triangle we mean that $\mu>0$ and 
\begin{equation*}
    \widetilde{\measuredangle}xpq,\widetilde{\measuredangle}xqp,\widetilde{\measuredangle}pxq>\mu.
\end{equation*}
\end{claim}

\begin{proof}
Let $\epsilon\in(0,\frac{1}{100})$ be some fixed small number.
Let $p\in\Gamma$ to be the closest point to $x$ on $\Gamma$, then by Lemma \ref{l: inf achieved on the non-compact portion} $p$ is an interior point in $\Gamma$, and $p$ is an $\epsilon$-tip point by Lemma \ref{l: Gamma}. Therefore, it is easy to see that the minimizing geodesic $px$ and $\psi_{p*}(\partial_{r})$ is almost orthogonal at $p$, where $\psi_p$ is the $\epsilon$-isometry to $(\R\times\cigar,r^{-2}(x_{tip})g_c,x_{tip})$ at $p$, i.e.
\begin{equation}\label{e: plane}
    \left|\measuredangle \left(px,\psi_{p*}\left(\partial_{r}\right)\right)-\frac{\pi}{2}\right|\le\epsilon.
\end{equation}

Note that $d(x,\Gamma)>D_2\,r(x)$ and $d(x,p)=d(x,\Gamma)$, we claim that
\begin{equation}\label{e: hudson}
    d(x,\Gamma)>\frac{1}{10}\,D_2\,r(p).
\end{equation}
Because otherwise we may choose $D_1$ sufficiently large (depending on $D_2$) such that $d(x,p)\le \frac{1}{10}\,D_2\,r(p)$ must imply $r(p)\le 10r(x)$, which implies $d(x,p)\le \,D_2\,r(x)$, a contradiction. So \eqref{e: hudson} holds.

By Lemma \ref{l: Gamma} we can take $D_1$ to be large so that $\epsilon$ is sufficiently small, where $\epsilon$ is determined by $D_2$ such that the following holds:
By \eqref{e: plane} and the $\epsilon$-closeness to $\R\times\cigar$ in the region containing the $B(p,D_2\,r(p))$, we have for all points $y\in B(p,D_2\,r(p))$ that
\begin{equation}\label{e: Sully}
    \widetilde{\measuredangle}ypx'\le\frac{\pi}{2}+\frac{1}{100}\le\frac{2\pi}{3},
\end{equation}
where $x'$ is a point on the minimizing geodesic $px$ such that 
\begin{equation*}
    d(p,x')=\frac{1}{10}\,D_2\,r(p)<d(x,\Gamma).
\end{equation*}

Now let $q\in\Gamma$ be a point such that $d(q,p)=d(x,\Gamma)$. 
By \eqref{e: hudson} we can choose a point $q'$ on the minimizing geodesic $pq$ such that
\begin{equation*}
    d(p,q')=\frac{1}{10}\,D_2\,r(p)<d(x,\Gamma).
\end{equation*}
Then using \eqref{e: Sully} we obtain
\begin{equation*}
    \widetilde{\measuredangle}q'px'\le\frac{\pi}{2}+\frac{1}{100}\le\frac{2\pi}{3}.
\end{equation*}
By the monotonicity of angles, the last three inequalities imply
\begin{equation*}
    \widetilde{\measuredangle}qpx\le\frac{\pi}{2}+\frac{1}{100}\le\frac{2\pi}{3}.
\end{equation*}

Note $d(x,q)\ge d(x,\Gamma)=d(x,p)=d(p,q)$. The segment $|\widetilde{x}\widetilde{q}|$ is the longest in the comparison triangle $\widetilde{\triangle}\widetilde{x}\widetilde{q}\widetilde{p}$, and thus it must be opposite to the largest comparison angle, i.e.
\begin{equation*}
    \widetilde{\measuredangle}qpx\ge \widetilde{\measuredangle}qxp=\widetilde{\measuredangle}pqx.
\end{equation*}
This combined with $\widetilde{\measuredangle}qpx+ \widetilde{\measuredangle}qxp+\widetilde{\measuredangle}pqx=\pi$ implies
\begin{equation*}
    \widetilde{\measuredangle}qpx\ge \widetilde{\measuredangle}qxp=\widetilde{\measuredangle}pqx\ge\frac{\pi}{6}.
\end{equation*}
So $\{p,q,x\}$ forms a $\frac{\pi}{6}$-triangle.
\end{proof}

\begin{claim}\label{l: alp}
There exists $\alpha>0$ such that when $D_1,D_2$ are sufficiently large, the following holds for all $x\in M$ satisfying \eqref{e: captain}:
There is a point $x'\in B(x,\frac{1}{10} \,d(x,\Gamma))$ such that there is a $(2,\delta)$-strainer at $x'$ of size $\alpha\,d(x,\Gamma)$. 
\end{claim}

\begin{proof}
Suppose this is not true. Then there is a sequence of points $x_i\in M$ going to infinity and $d(x_i,\Gamma)\rii$ such that the claim fails.
Let $p_i,q_i$ be points from Claim \ref{c: alpha triangle} which form an $\frac{\pi}{6}$-triangle together with $x_i$, and $d(p_i,x_i)=d(x_i,\Gamma)$.

After passing to a subsequence, the pointed manifolds $(M,d^{-2}(x_i,\Gamma)g,x_i)$ converge to a complete non-compact length space $(X,d_X)$, which is an Alexandrov space with non-negative curvature \cite{BGP}.
The triples $(p_i,q_i,x_i)$ converge to a triple $(p_{\infty},q_{\infty},x_{\infty})$ in the limit space, which forms a $\frac{\pi}{6}$-triangle.
So the Alexandrov dimension of the limit space is at least two.
Because otherwise $X$ must be isometric to a ray or a line, which does not contain any $\frac{\pi}{6}$-triangles.

Since the set of $(2,\delta)$-strainers is open and dense in an Alexandrov space of dimension 2, we can find a point $x'\in B_{d_X}(x,\frac{1}{10})$ such that there is a $(2,\delta)$-strainer at $x'$ of size $\alpha$ for some $\alpha>0$. 
This induces a contradiction for sufficiently large $i$.
\end{proof}

Assume we can show that the volume of $B(x',1)$ is bounded below, then by a volume comparison using that $x'\in B(x,\frac{1}{10} \,d(x,\Gamma))$, this would imply a lower bound on the volume of $B(x,1)$. 
Therefore, we may assume without loss of generality that there is a $(2,\delta)$-strainer $(a_1,b_1,a_2,b_2)$ at $x$ of size $\alpha\,d(x,\Gamma)$.

From now on we will work with the rescaled metric $h=d^{-2}(x,\Gamma)g$ and bound the 1-ball around a lift of $x$ in the universal cover from below by a universal constant. 

Let $\pi:\widetilde{U}\ri U$ be a universal covering, and $\widetilde{x},\widetilde{a}_i,\widetilde{b}_i$ be lifts of $x,a_i,b_i$ in the universal cover $\widetilde{U}$ such that 
\begin{equation}\label{e: equal length}
    d(\widetilde{x},\widetilde{a}_i)=d(x,a_i)=d(\widetilde{x},\widetilde{b}_i)=d(x,b_i)=\alpha.
\end{equation}
Then since the covering map $\pi$ is $1$-Lipschitz, we have
\begin{equation}\label{e: bigger length}
    d(\widetilde{a}_i,\widetilde{b}_j)\ge d(a_i,b_j),\quad 
    d(\widetilde{a}_i,\widetilde{a_j})\ge d(a_i,a_j),\quad
    d(\widetilde{b}_i,\widetilde{b}_j)\ge d(b_i,b_j).
\end{equation}
So the comparison angles between $\widetilde{x},\widetilde{a}_i,\widetilde{b}_j$ at $\widetilde{x}$ are at least as large as those between $x,a_i,b_j$ at $x$, i.e.
\begin{equation*}
    \widetilde{\measuredangle}\widetilde{a}_i\widetilde{x}\widetilde{b}_j\ge \widetilde{\measuredangle} a_ixb_j,\quad  \widetilde{\measuredangle}\widetilde{a}_i\widetilde{x}\widetilde{b}_j\ge \widetilde{\measuredangle} a_1xa_2,\quad  \widetilde{\measuredangle}\widetilde{a}_i\widetilde{x}\widetilde{b}_j\ge \widetilde{\measuredangle} b_1xb_2.
\end{equation*}
So $(\widetilde{a}_1,\widetilde{a}_2,\widetilde{b}_1,\widetilde{b}_2)$ is a $(2,\delta)$-strainer at $\widetilde{x}$.

Next, we will extend the $(2,\delta)$-strainer to a $(2+\frac{1}{2},\delta)$-strainer at $\widetilde{x}$.
Since the $S^1$-fiber in $U$ is incompressible, we can
find a sequence $\widetilde{x}_i$ of lifts of $x$ that is unbounded. We may assume that the the consecutive distances of $\{\widetilde{x}_i\}$ are at most $D_2^{-1/2}$.
Because otherwise, there would be two points $\widetilde{x}_i,\widetilde{x}_j$ such that $d(\widetilde{x}_i,\widetilde{x}_j)\ge D_2^{-1/2}$. Rescaling back to the metric $g$ and by Lemma \ref{l: MT-glue}, we see that the length of the circle in the $\epsilon$-circle plane at $x$ is at least $D_2^{-1/2}\,d(x,\Gamma)$. This implies $r(x)\ge D_2^{-1/2}\,d(x,\Gamma)$, which contradicts our assumption \eqref{e: captain}.
So we can find an $i\in\mathbb{N}$ such that with $\widetilde{y}=\widetilde{x}_i$ we have
\begin{equation*}
    |d(\widetilde{y},\widetilde{x})-D_2^{-\frac{1}{4}}|<D_2^{-\frac{1}{2}}.
\end{equation*}

\begin{claim}
The tuple $(\widetilde{a}_1,\widetilde{a}_2,\widetilde{b}_1,\widetilde{b}_2,\widetilde{y})$ is a $(2+\frac{1}{2},\delta)$-tuple at $\widetilde{x}$ of size at least $D_2^{-\frac{1}{4}}-D_2^{-\frac{1}{2}}$.
\end{claim}

\begin{proof}
Note that in the triangle $\triangle \widetilde{y}\widetilde{x}\widetilde{a}_i$, the segment $|\widetilde{y}\widetilde{a}_i|$ has the longest length, and thus must be opposite to the largest comparison angle, i.e.
\begin{equation*}
    \widetilde{\measuredangle}\widetilde{a}_i\widetilde{x}\widetilde{y}\ge \widetilde{\measuredangle}\widetilde{x}\widetilde{y}\widetilde{a}_i.
\end{equation*}
Since $d(\widetilde{x},\widetilde{y})\ri0$ as $D_2\rii$, we find
\begin{equation}\label{e: smaller than delta}
    \widetilde{\measuredangle}\widetilde{y}\widetilde{a}_i\widetilde{x}<\delta.
\end{equation}
We also have 
\begin{equation}\label{e: euclid equal to pi}
    \widetilde{\measuredangle}\widetilde{y}\widetilde{a}_i\widetilde{x}+\widetilde{\measuredangle}\widetilde{a}_i\widetilde{x}\widetilde{y}+\widetilde{\measuredangle}\widetilde{x}\widetilde{y}\widetilde{a}_i=\pi.
\end{equation}
So the last three inequalities imply $\widetilde{\measuredangle}\widetilde{a}_i\widetilde{x}\widetilde{y}\ge\frac{\pi}{2}-\delta$. The same is true with $\widetilde{a}_i$ replaced by $\widetilde{b}_i$. So
\begin{equation}\label{e: hero}
    \widetilde{\measuredangle}\widetilde{a}_i\widetilde{x}\widetilde{y}\ge\frac{\pi}{2}-\delta\quad \textit{and}\quad \widetilde{\measuredangle}\widetilde{b}_i\widetilde{x}\widetilde{y}\ge\frac{\pi}{2}-\delta.
\end{equation}
and hence the claim holds.
\end{proof}

\begin{claim}\label{c: 2half strainer at y}
The tuple $(\widetilde{a}_1,\widetilde{a}_2,\widetilde{b}_1,\widetilde{b}_2,\widetilde{x})$ is a $(2+\frac{1}{2},\delta)$-tuple at $\widetilde{y}$ of size at least $D_2^{-\frac{1}{4}}-D_2^{-\frac{1}{2}}$.
\end{claim}

\begin{proof}
First, since $|d(\widetilde{y},\widetilde{a}_i)-d(\widetilde{x},\widetilde{a}_i)|<D_2^{-\frac{1}{4}}$ and $|d(\widetilde{y},\widetilde{b}_i)-d(\widetilde{x},\widetilde{b}_i)|<D_2^{-\frac{1}{4}}$, we see that $(\widetilde{a}_1,\widetilde{a}_2,\widetilde{b}_1,\widetilde{b}_2)$ is a $(2,\delta)$-strainer at $\widetilde{y}$ of size at least $\alpha-D_2^{-\frac{1}{4}}-2D_2^{-\frac{1}{2}}$.
Note we may assume $D_2$ sufficiently large such that this is at least $\frac{\alpha}{2}>D_2^{-\frac{1}{4}}-D_2^{-\frac{1}{2}}$.

By metric comparison we have
\begin{equation*}
    \widetilde{\measuredangle}\widetilde{a}_i\widetilde{x}\widetilde{y}+
     \widetilde{\measuredangle}\widetilde{b}_i\widetilde{x}\widetilde{y}+
      \widetilde{\measuredangle}\widetilde{a}_i\widetilde{x}\widetilde{b}_i\le2\pi,
\end{equation*}
which by \eqref{e: hero} and $\widetilde{\measuredangle}\widetilde{a}_i\widetilde{x}\widetilde{b}_i\ge\pi-\delta$ implies
\begin{equation*}
    \widetilde{\measuredangle}\widetilde{a}_i\widetilde{x}\widetilde{y}\le\frac{\pi}{2}+\delta\quad \textit{and}\quad \widetilde{\measuredangle}\widetilde{b}_i\widetilde{x}\widetilde{y}\le\frac{\pi}{2}+\delta.
\end{equation*}
Combining this with \eqref{e: euclid equal to pi} and \eqref{e: smaller than delta} we obtain 
\begin{equation*}
    \widetilde{\measuredangle}\widetilde{x}\widetilde{y}\widetilde{a}_i\ge\frac{\pi}{2}-\delta\quad\textit{and}\quad \widetilde{\measuredangle}\widetilde{x}\widetilde{y}\widetilde{b}_i\ge\frac{\pi}{2}-\delta.
\end{equation*}
So the claim holds.
\end{proof}

Now take $\widetilde{m}$ be the midpoint of a minimizing geodesic between $\widetilde{y}$ and $\widetilde{x}$.
\begin{claim}\label{c: 3-strainer}
The tuple $(\widetilde{a}_1,\widetilde{a}_2,\widetilde{b}_1,\widetilde{b}_2,\widetilde{y},\widetilde{x})$ is a $(3,\delta)$-strainer at $\widetilde{m}$ of size at least $\frac{1}{2}D_2^{-\frac{1}{4}}-D_2^{-\frac{1}{2}}$.
\end{claim}

\begin{proof}
First, by the monotonicity of comparison angles we have
\begin{equation*}
    \widetilde{\measuredangle}\widetilde{m}\widetilde{x}\widetilde{a}_i\ge \widetilde{\measuredangle}\widetilde{y}\widetilde{x}\widetilde{a}_i\ge\frac{\pi}{2}-\delta\quad\textit{and}\quad \widetilde{\measuredangle}\widetilde{m}\widetilde{x}\widetilde{b}_i\ge \widetilde{\measuredangle}\widetilde{y}\widetilde{x}\widetilde{b}_i\ge\frac{\pi}{2}-\delta.
\end{equation*}
Then repeating the same argument as in Claim \ref{c: 2half strainer at y} replacing $\widetilde{y}$ by $\widetilde{m}$, we wee that
\begin{equation*}
    \widetilde{\measuredangle}\widetilde{a}_i\widetilde{m}\widetilde{x},\;
    \widetilde{\measuredangle}\widetilde{b}_i\widetilde{m}\widetilde{x}>\frac{\pi}{2}-\delta.
\end{equation*}
Replacing $\widetilde{x}$ by $\widetilde{y}$, then similarly we can obtain
\begin{equation*}
    \widetilde{\measuredangle}\widetilde{a}_i\widetilde{m}\widetilde{y},\;
    \widetilde{\measuredangle}\widetilde{b}_i\widetilde{m}\widetilde{y}>\frac{\pi}{2}-\delta.
\end{equation*}
Moreover, similarly as before we can see that $(\widetilde{a}_1,\widetilde{a}_2,\widetilde{b}_1,\widetilde{b}_2)$ is a $(2,\delta)$-strainer at $\widetilde{m}$.
Finally, $\widetilde{\measuredangle}\widetilde{y}\widetilde{m}\widetilde{x}=\pi$ is trivially true. So the claim holds.
\end{proof}

Now using the $(3,\delta)$-strainer in Claim \ref{c: 3-strainer}, one can construct a $100$-bilipschitz map $f:B(\widetilde{m},\lambda \,D_2^{-\frac{1}{4}})\ri\R^3$ for some sufficiently small $\lambda$ as in \cite[Lemma 2.2(i)]{BamD} and \cite[Theorem 10.8.18]{BuragoBuragoIvanov}, which implies
\begin{equation*}
    vol(B(\widetilde{m},\lambda \,D_2^{-\frac{1}{4}}))>c(\lambda \,D_2^{-\frac{1}{4}})^3,
\end{equation*}
for some universal $c>0$. This implies Lemma \ref{l: non-collapsing in universal cover} by volume comparison.
\end{proof}

Now Theorem \ref{l: curvature upper bound initial} follows from Lemma \ref{l: non-collapsing in universal cover} and Perelman's curvature estimate.

\subsection{Existence of a critical point}

The main result in this subsection is Theorem \ref{t: Rmax critical point}, which proves the existence of the maximum point of the scalar curvature, which is also the critical point of the potential function.

In Lemma \ref{l: geometry of level set near tip} and \ref{l: 3D neck is 2D neck}, we study the geometry of level sets of the potential function $f$.
To start, we first note that the second fundamental form of a level set of $f$ satisfies
\begin{equation}\label{e: second fundamental form}
    \mathrm{I\!I}=-\left.\frac{\nabla ^2 f}{|\nabla f|}\right|_{f^{-1}(a)}=-\left.\frac{\Ric}{|\nabla f|}\right|_{f^{-1}(a)}\le 0.
\end{equation}
Recall the Gauss equation that for a manifold $N$ embedded in a Riemannian manifold $(M,g)$, the curvature tensor $\Rm_N$ of $N$ with induced metric can be expressed using the second fundamental form and $\Rm_M$, the curvature tensor of $M$:
\begin{equation*}
   \langle \Rm_N(u,v)w,z \rangle=\langle \Rm_M(u,v)w,z \rangle+\langle\mathrm{I\!I}(u,z),\mathrm{I\!I}(v,w)\rangle-\langle\mathrm{I\!I}(u,w),\mathrm{I\!I}(v,z)\rangle
\end{equation*}
So \eqref{e: second fundamental form} implies
the level sets of $f$ with induced metric have positive curvature.

More precisely, for a level set of $f$ which passes an $\epsilon$-tip point, Lemma \ref{l: geometry of level set near tip} shows that at a point that is sufficiently far away from this $\epsilon$-tip point in the level set, the soliton is close to $\RR\times S^1$ under a suitable rescaling. We show this by a limiting argument.

\begin{lem}\label{l: geometry of level set near tip}
Let $(M,g,f)$ be a 3D steady gradient soliton with positive curvature that is not a Bryant soliton. 
For any $\delta>0$ and $C_0>0$, there exists $\overline{D}>0$ such that the following holds:

For all $D\ge\overline{D}$, there exists $\epsilon>0$ such that the following holds:
Suppose $q\in M$ is an $\epsilon$-tip point with $|\nabla f|(q)\ge C_0^{-1}$, and let $\Sigma=f^{-1}(f(q))$ be the level set passing through $q$.
Suppose also that $z\in\Sigma$ is a point with $d_{\Sigma}(q,z)=D\,R^{-1/2}(q)$, where $d_{\Sigma}$ is the length metric of the induced metric on $\Sigma$. Then the soliton $(M,g)$ is a $\delta$-cylindrical plane at $q$.
\end{lem}

\begin{proof}
Suppose this is false, then there is some $\delta>0$, such that for any large $\overline{D}$, there exist a constant $D>\overline{D}$, a sequence of constants $\epsilon_k\ri0$, and a sequence of $\epsilon_k$-tip points $q_k\rii$, and a sequence of points $z_k\in \Sigma_k:=f^{-1}(f(q_k))$ with $d_{\Sigma_k}(q_k,z_k)= D\,R^{-1/2}(q_k)$ such that the soliton $(M,g)$ is not a $\delta$-cylindrical plane at $z_k$. 

In the following we will show that when $\overline{D}$ is large enough, the level sets $\Sigma_k$ under suitable rescalings converge to a level set of a smooth function on $\R\times\cigar$, and $z_k$ converge to a point at which $\R\times\cigar$ is a $\delta/2$-cylindrical plane, which will imply a contradiction for sufficiently large $k$.
We now divide the discussion into three cases depending on the limit of $R(q_k)$. To start, note that $R+|\nabla f|^2=C_1^2$ for some $C_1>0$. So $R\le C_1^2$ and $|\nabla f|\le C_1$.

\textbf{Case 1:} $\limsup_{k\rii} R(q_k)>C^{-2}>0$ for some $C>0$ which may depend on the sequence.
Then we can deduce from the soliton equation $\nabla^2 f=\Ric$ that 
derivatives of order at least two of the functions $\widetilde{f}_k:=f-f(q_k)$ are uniformly bounded.
Moreover, at $q_k$ we have $\widetilde{f}_k(q_k)=0$ and
\begin{equation*}
    C_0^{-1}\le|\nabla \widetilde{f}_k|(q_k)=|\nabla f|(q_k)\le C_1.
\end{equation*}
Let $\widetilde{\nabla}$ denote the covariant derivatives of $\widetilde{g}_k=\frac{1}{16}R(q_k)g$, then we have that $|\widetilde{\nabla} \widetilde{f}_k|_{\widetilde{g}_k}= 4R^{-1/2}(q_k)|\nabla f|$ is uniformly bounded above, and in particular at $q_k$ we have
\begin{equation*}
    |\widetilde{\nabla} \widetilde{f}_k|(q_k)\ge 4\,C_0^{-1}R^{-1/2}(q_k)\ge 4\,C_0^{-1}C_1^{-1}.
\end{equation*}
So after passing to a subsequence we may assume that the functions $\widetilde{f}_k$ converge to a smooth function $f_{\infty}$ on $\R\times\cigar$, which satisfies $f_{\infty}(x_{tip})=0$, $\nabla^2 f_{\infty}=\Ric$ and 
\begin{equation*}
    |\nabla f_{\infty}|(x_{tip})\ge 4\,C_0^{-1}C_1^{-1}.
\end{equation*}
Note that $C_0,C_1$ are constants that only depend only on the soliton but not the sequence.

Since $\nabla^2 f_{\infty}=\Ric$, by the uniqueness of the potential function on the cigar soliton we see that $f_{\infty}$ is the sum of the potential function on $\cigar$ which vanishes at the tip and a linear function along the $\R$-factor whose derivative has absolute value at least $4\,C_0^{-1}C_1^{-1}$ and vanishes at $x_{tip}$. In particular, $0$ is a regular value of $f_{\infty}$ and the level set $\Sigma_{\infty}:=f^{-1}_{\infty}(0)$ is a non-compact complete rotationally symmetric 2D manifold.  

Therefore, after passing to a subsequence, as the manifolds $(M,\frac{1}{16}R(q_k)g,q_k)$ smoothly converge to $(\R\times\cigar,g_c,x_{tip})$,  the level sets $(\Sigma_k,\frac{1}{16}R(q_k)g_{\Sigma_k},q_k)$ of $\widetilde{f}_k$ with the induced metrics smoothly converge to the level set $(\Sigma_{\infty},g_{\Sigma_{\infty}},x_{tip})$ of $f_{\infty}$, and $z_k\in\Sigma_k$ converge to a point
$z_{\infty}\in\Sigma_{\infty}$ with $d_{\Sigma_{\infty}}(x_{tip},z_{\infty})=\frac{1}{4}D$.
Since $\R\times\cigar$ is a $\delta/2$-cylindrical plane at $z_{\infty}$ when $\overline{D}$ is sufficiently large depending on $\delta$ and $4\,C_0^{-1}C_1^{-1}$, we obtain a contradiction for all sufficiently large $k$.

\textbf{Case 2:} $\lim_{k\rii} R(q_k)=0$. Consider the rescaled metrics $\widetilde{g}_k=\frac{1}{16}R(q_k)g$ and the rescaled functions $\widetilde{f}_k:=\frac{f-f(q_k)}{4R^{-1/2}(q_k)}$, then $\widetilde{f}_k$ satisfies $\widetilde{f}_k(q_k)=0$ and
\begin{equation}\label{e: higher derivatives of f}
    \widetilde{\nabla}^2\widetilde{f}_k=\nabla^2\widetilde{f}_k=\frac{\nabla^2 f}{4R^{-1/2}(q_k)}=\frac{\Ric}{4R^{-1/2}(q_k)}=\frac{\widetilde{\Ric}}{4R^{-1/2}(q_k)},
\end{equation}
and also
\begin{equation*}
    |\widetilde{\nabla}\widetilde{f}_k|_{\widetilde{g}_k}=
    4R^{-1/2}(q_k)|\nabla\widetilde{f}_k|=
    |\nabla f|\le C_1.
\end{equation*}
In particular, at $q_k$ we have
\begin{equation}\label{e: the norm of gradient}
    |\widetilde{\nabla}\widetilde{f}_k|_{\widetilde{g}_k}(q_k)=|\nabla f|(q_k)\in[C_0^{-1},C_1].
\end{equation}
By \eqref{e: higher derivatives of f} and \eqref{e: the norm of gradient}, the derivatives of $\widetilde{f}_k$ are uniformly bounded. Thus there is a subsequence of $\widetilde{f}_k$ converging to a smooth function $f_{\infty}$ on $\R\times\cigar$ with $f_{\infty}(x_{tip})=0$.

By \eqref{e: higher derivatives of f} and \eqref{e: the norm of gradient} we have $\nabla^2f_{\infty}=0$ and $|\nabla f_{\infty}|(x_{tip})>0$. 
So $f_{\infty}$ is a non-constant linear function in the $\R$-direction.
In particular, $0$ is a regular value of $f_{\infty}$, and the level set $\Sigma_{\infty}:=f_{\infty}^{-1}(0)$ is equal to $\{a\}\times\cigar\subset\R\times\cigar$, for some $a\in\R$.

Therefore, after passing to a subsequence, as the manifolds $(M,\frac{1}{16}\,R(q_k)g,q_k)$ smoothly converge to $(\R\times\cigar,g_c,x_{tip})$,  the level sets $(\Sigma_k,\frac{1}{16}\,R(q_k)g_{\Sigma_k},q_k)$ of $\widetilde{f}_k$ with induced metrics smoothly converge to the level set $(\Sigma_{\infty},g_{\Sigma_{\infty}},x_{tip})$ of $f_{\infty}$, and the points $z_k\in\Sigma_k$ converge to a point
$z_{\infty}\in\Sigma_{\infty}$ with $d_{\Sigma_{\infty}}(x_{tip},z_{\infty})=\frac{1}{4}D$.
So it follows when $\overline{D}$ is sufficiently large that $\R\times\cigar$ is a $\delta/2$-cylindrical plane at $z_{\infty}$. This is a contradiction for large $k$.

\end{proof}

The next lemma shows that for a point at which the soliton looks sufficiently like the cylindrical plane $\RR\times S^1$, the level set of $f$ passing through it looks like the cylinder $\R\times S^1$. We prove this lemma by a limiting argument. 

\begin{lem}\label{l: 3D neck is 2D neck}
Let $(M,g,f)$ be a 3D steady gradient soliton with positive curvature that is not a Bryant soliton. 
For any $\delta>0$ and $C_0>0$, there exists $\epsilon>0$ such that if $M$ is an $\epsilon$-cylindrical plane at $x\in M$ and $r(x)\ge C_0^{-1}$, then the level set $f^{-1}(f(x))$ of $f$ passing through $x$ is a $\delta$-neck at $x$ at scale $r(x)$. 
\end{lem}

\begin{proof}
Suppose the conclusion is not true, then for some $\delta>0$ and $C_0>0$, there is a sequence of points $x_i\in M$ at which $(M,g)$ is an $\epsilon_i$-cylindrical plane,  $\epsilon_i\ri0$, such that $\Sigma_i:=f^{-1}(f(x_i))$ is not a $\delta$-neck at $x_i$ at scale $r(x_i)$.

Consider the rescalings of the metrics $\widetilde{g}_i:=r^{-2}(x_i)g$, and the rescalings of the functions $\widetilde{f}_i:=\frac{f_i-f_i(x_i)}{r(x_i)}$. We have $\widetilde{f}_i(x_i)=0$ and
\begin{equation*}
\begin{split}
    \widetilde{\nabla}^{k+2}\widetilde{f}_i&=r^{-1}(x_i)\widetilde{\nabla}^{k+2} f_i
    =r^{-1}(x_i)\widetilde{\nabla}^{k} (\widetilde{\nabla}^{2}f_i)
    =r^{-1}(x_i)\widetilde{\nabla}^{k}(\widetilde{\Ric}),\quad k\ge0\\
\end{split}
\end{equation*}
and also
\begin{equation*}
    \widetilde{\nabla}\widetilde{f}_i=r^2(x_i)\nabla\widetilde{f}_i=r(x_i)\nabla f_i.
\end{equation*}
Therefore, using $r(x_i)\ge C_0^{-1}$ we obtain
\begin{equation*}
\begin{split}
|\widetilde{\nabla}^{k+2}\widetilde{f}_i|_{\widetilde{g}_i}&=r^{-1}(x_i)|\widetilde{\nabla}^{k}(\widetilde{\Ric})|_{\widetilde{g}_i}\le C_0 |\widetilde{\nabla}^{k}(\widetilde{\Ric})|_{\widetilde{g}_i}\\
\end{split}
\end{equation*}
which goes to zero for each $k\ge 0$ since $\epsilon_i\ri0$. Note $R+|\nabla f|^2=C_1^2$ for some $C_1>0$, we also have $|\nabla f|\le C_1$ and 
\begin{equation*}
    |\widetilde{\nabla} \widetilde{f}_i|_{\widetilde{g}_i}=|\nabla f|\le C_1.
\end{equation*}
In particular, since $r(x_i)\ge C_0^{-1}$ and $\epsilon_i\ri0$, it follows that $R(x_i)\ri0$ and $|\nabla f|(x_i)\ge\frac{1}{2}C_1$ for all large $i$. So at $x_i$ we have
\begin{equation*}
    |\widetilde{\nabla} \widetilde{f}_i|_{\widetilde{g}_i}(x_i)=|\nabla f|(x_i)\in\left[C_1/2,C_1\right].
\end{equation*}

So after passing to a subsequence we may assume that
the manifolds $(M,\widetilde{g}_i,x_i)$ smoothly converge to $(\RR\times S^1,g_{stan},x_{\infty})$, and
the functions $\widetilde{f}_i$ converge to a smooth function $f_{\infty}$ on $\RR\times S^1$, which satisfies
$f_{\infty}(x_{\infty})=0$ and
\begin{equation}\label{e: happyday}
    |\nabla^{k+2} f_{\infty}|=0,\quad{k\ge0},\quad {and}\quad |\nabla f_{\infty}|(x_{\infty})\in\left[C_1/2,C_1\right].
\end{equation}

By \eqref{e: happyday} it is easy to see that $f_{\infty}$ is a constant in each $S^1$-factor in $\RR\times S^1$, and a non-constant linear function on the $\RR$-factor. After a possible rotation on $\RR$, we may assume the level set $\Sigma_{\infty}:=f^{-1}_{\infty}(0)=\{(x,0,\theta):x\in\R,\theta\in[0,2\pi)\}$. In particular, $0$ is a regular value of $f_{\infty}$, and $\Sigma_{\infty}$ is isometric to $(\R\times S^1,g_{stan})$. 
Therefore, the level sets $(\Sigma_i,r^{-2}(x_i)g_{\Sigma_i},x_i)$ of $\widetilde{f}_i$ smoothly converge to the level set $(\Sigma_{\infty},g_{stan},x_{\infty})$ of $f_{\infty}$.
This implies that $\Sigma_i$ is a $\delta$-neck when $i$ is sufficiently large, a contradiction.
\end{proof}

The following lemma compares the value of $f$ at two points $x$ and $y$, when the minimizing geodesic from $x$ to $y$ is orthogonal to $\nabla f$ at $y$.

\begin{lem}\label{l: compare f}
Let $x,y$ be two points in $M$.
Suppose that a minimizing geodesic $\sigma$ from $y$ to $x$ is orthogonal to $\nabla f$ at $y$. Then $f(x)\ge f(y)$.
\end{lem}

\begin{proof}
Since $\langle\nabla f(\sigma(0)),\sigma'(0)\rangle=0$, computing by calculus variation we have
\begin{equation*}
\begin{split}
    f(x)-f(y)&=f(\sigma(1))-f(\sigma(0))
    =\int_{0}^{1}\langle\nabla f(\sigma(r)),\sigma'(r)\rangle\,dr\\
    &=\int_{0}^1\int_0^r\nabla^2f(\sigma'(s),\sigma(s))\,ds\,dr\\
    &=\int_{0}^1\int_0^r\Ric(\sigma'(s),\sigma(s))\,ds\,dr
    \ge 0.
\end{split}
\end{equation*}
\end{proof}

The following lemma compares the scales of two $\epsilon$-necks in a positively-curved non-compact complete 2D manifold. It says that the scale of the inner $\epsilon$-neck is almost not larger than that of the outer $\epsilon$-neck. 

\begin{lem}\label{l: 2D surface}
For any $\delta>0$, there exists $\epsilon>0$ such that the following holds:

Let $(M,g)$ be a 2D complete non-compact Riemannian manifold with positive curvature and let $p$ be a soul for it. Then for any $\epsilon$-neck $N$ disjoint from $p$ the central circle of $N$ separates the soul from the end of the manifold.
In particular, if two $\epsilon$-necks $N_1$ and $N_2$ in $M$ are disjoint from each other and from $p$, then the central circles of $N_1$ and $N_2$ are the boundary components of a region in $M$ diffeomorphic to $S^1\times I$.

Moreover, assume $N_1$ is contained in the 2-ball bounded by the central circle of $N_2$, then the scales $r_1,r_2$ of $N_1,N_2$ satisfy
\begin{equation*}
    r_1<(1+\delta)r_2.
\end{equation*}
\end{lem}

\begin{proof}
The proof is a slight modification of the proof of \cite[Lemma 2.20]{MT} using Busemann functions.
\end{proof}

Now we prove the critical point theorem.

\begin{theorem}\label{t: Rmax critical point}
Let $(M,g)$ be a 3D steady gradient soliton with positive curvature that is not a Bryant soliton. Then there exists $p\in M$ such that $R$ attains its maximum at $p$ and $\nabla f(p)=0$.
\end{theorem}

\begin{proof}
Let $\epsilon>0$ be some constant we can take arbitrarily small, and we will use $\delta>0$ to denote all constants that converge to zero as $\epsilon\ri0$.
We suppose by contradiction that $R_{\max}$ does not exists.

First, the level sets of $f$ are non-compact: Suppose not, then for some $a\in\R$, the level set $f^{-1}(a)$ is compact and hence is diffeomorphic to $S^2$.
So $f^{-1}(a)$ separates the manifold into a compact and a non-compact connected component. Since $f$ is convex and non-constant, by the maximum principle, it attains the minimum in the compact region.
This contradicts our assumption.

Let $\Gamma_1,\Gamma_2:[0,\infty)\in  M$ be the two integral curves of $\nabla f$ or $-\nabla f$ from Lemma \ref{l: Gamma}, which extend to infinity on the open ends.
First, we claim that $\Gamma_2$ and $\Gamma_1$ can not be integral curves of $-\nabla f$ at the same time.
Because otherwise, on the one hand, we have $d(\Gamma_1(s),\Gamma_2(s))\rii$ as $s\rii$ by Lemma \ref{l: two rays}.
On the other hand, since $\Gamma_1,\Gamma_2$ are integral curves of $-\nabla f$, it follows by the positive curvature and distance expanding along the backwards Ricci flow of the soliton that $d(\Gamma_1(s),\Gamma_2(s))\le d(\Gamma_1(s_0),\Gamma_2(s_0))$ for any $s
\ge s_0$. This contradiction shows the claim.

So we may assume $\Gamma_1$ is an integral curve of $\nabla f$.
In Claim \ref{c: gamma goes to infinity}, \ref{c: DR}, and \ref{c: in C_2}, we will construct a complete integral curve $\Gamma:(-\infty,+\infty)\in M$ of $\nabla f$ such that for some $s_0>0$, $\Gamma([s_0,\infty))\subset \mathcal{C}_1,\Gamma((-\infty,-s_0])\subset \mathcal{C}_2$ and moreover the manifolds $(M,r^{-2}(\Gamma(s))g,\Gamma(s))$ converge to $(\R\times\cigar,r^{-2}(x_{tip})g_c,x_{tip})$ as $s\ri\pm\infty$. Note that $\Gamma((-\infty,\infty))$ is invariant under the diffeomorphisms $\phi_t$ generated by $\nabla f$.

\begin{claim}\label{c: gamma goes to infinity}
Take $\gamma_1:[0,\infty)\ri M$ to be the integral curve of $-\nabla f$ starting from $\Gamma_1(0)$.
Then $\gamma_1(s)$ goes to infinity as $s\rii$.
\end{claim}
\begin{proof}
Suppose otherwise, assume for some $s_i\rii$ and a compact subset $V$ there is $\gamma_1(s_i)\in V$. By the compactness of $V$, there is $c>0$ such that $\Ric\ge c g$ and $c\le|\nabla f|\le c^{-1}$ in $V$. So by the first identity in \eqref{e: two} we have
\begin{equation*}
    \frac{d}{ds}|_{s=s_i}R(\gamma_1(s))\ge c.
\end{equation*}
Moreover, by the increasing of $R(\gamma_1(s))$ we get $\frac{d}{ds}|_{s=s_i}R(\gamma_1(s))\ge 0$ for all $s$.
It is clear that there is a uniform $C_0>0$ such that $\left|\frac{d^2}{ds^2}R(\gamma_1(s))\right|\le C_0$ for all $s$.
We may choose the sequence $s_i$ such that $s_{i+1}>s_i+1$. Then 
\begin{equation*}
    R(\gamma_1(s_{i+1}))\ge R(\gamma(s_1)) + \sum_{k=1}^{i}(R(\gamma_1(s_{k+1}))-R(\gamma_1(s_{k})))
    \ge R(\gamma(s_1)) + i\,c^2C_0^{-1}/2\rii,
\end{equation*}
which is impossible.
\end{proof}

\begin{claim}\label{c: DR}
The manifolds $(M,r^{-2}(\gamma_1(s))g,\gamma_1(s))$ converge smoothly to the manifold $(\R\times\cigar,r^{-2}(x_{tip})g_{\Sigma},x_{tip})$ as $s\rii$. In particular,  $\gamma_1(s)$ is an $\epsilon$-tip point for all sufficiently large $s$. 
\end{claim}

\begin{proof}[Proof of Claim \ref{c: DR}]
It follows from Theorem \ref{l: DR to cigar} that $(M,r^{-2}(\gamma_1(s))g,\gamma_1(s))$ converge smoothly to either $(\R\times\cigar,r^{-2}(x_0)g_{\Sigma},x_0)$, or $(\RR\times S^1,g_{stan},x_0)$.
We show that it must be the first case: Since $\liminf_{s\rii}R(\gamma_1(s))>0$, by the quadratic curvature decay Theorem \ref{l: curvature upper bound initial}, it follows that $\gamma_1(s)$ is within uniformly bounded distance to the $\epsilon$-tip points on $\Gamma_1\cup\Gamma_2$. So the limit must be $(\R\times\cigar,r^{-2}(x_0)g_{\Sigma},x_0)$.

Moreover, since $\gamma_1(s)$ is an integral curve of $-\nabla f$, by the distance shrinking in the cigar soliton it is easy to see that $x_0$ must be a tip point in $\R\times\cigar$. 
\end{proof}

\begin{claim}\label{c: in C_2}
$\gamma_1(s)\subset\mathcal{C}_2$ for all $s$ sufficiently large.
\end{claim}

\begin{proof}[Proof of Claim \ref{c: in C_2}]
Since the two chains $\mathcal{C}_1,\mathcal{C}_2$ contains all $\epsilon$-tip points by Lemma \ref{l: two chains}, we have either $\gamma_1(s)\in \mathcal{C}_2$ or $\gamma_1(s)\in\mathcal{C}_1$ for all sufficiently large $s$.

Suppose $\gamma_1(s)\in\mathcal{C}_1$ for all large $s$.
On the one hand, by Claim \ref{c: DR}, we have $d(\gamma_1(s),\Gamma_1)\ri0$ as $s\rii$. Since $\Gamma_1(s)$ is the integral curve of $\nabla f$, we see that $|\nabla f|(\Gamma_1(s))$ increases in $s$, and hence $\liminf_{s\rii}|\nabla f|(\gamma_1(s))=\liminf_{s\rii}|\nabla f|(\Gamma_1(s))>0$. So by Claim \ref{c: DR} we have that 
$\measuredangle(\nabla f,\phi_{*}(\partial_{r}))<\epsilon$ at all $\epsilon$-tip points in $\mathcal{C}_1$ after possibly replacing $r$ by $-r$.

On the other hand, for a fixed point $p_0$, by Lemma \ref{l: parallel}, $\measuredangle(\nabla d(p_0,\cdot),\phi_{*}(\partial_{r}))<\epsilon$ holds at all $\epsilon$-tip points in $\mathcal{C}_1$ after possibly replacing $r$ by $-r$.
So either $\measuredangle(\nabla f,\nabla d(p_0,\cdot))<\epsilon$ or $\measuredangle(-\nabla f,\nabla d(p_0,\cdot))<\epsilon$ has to hold at all $\epsilon$-tip points in $\mathcal{C}_1$.
Note that Claim \ref{c: gamma goes to infinity} implies that $d(p_0,\gamma_1(s))\rii$ as $s\rii$, and hence (1) must hold. But the fact $d(p_0,\Gamma_1(s))\rii$ as $s\rii$ implies (2) must hold, a contradiction.

\end{proof}

Therefore, letting $\Gamma(s)=\Gamma_1(s)$ for $s\ge s_0$, and $\Gamma(s)=\gamma_1(s_0-s)$ for $s\le s_0$ we get the desired complete integral curve $\Gamma:(-\infty,+\infty)$ of $\nabla f$.
So we may assume $\Gamma_2(s)=\gamma_1(s-s_0)$ for $s\ge s_0$, then $\Gamma_1,\Gamma_2$ still satisfy the conclusions in Lemma \ref{l: Gamma}, and moreover satisfy the additional properties that $\lim_{s\rii}R(\Gamma_2(s))>0$ and $\Gamma_1,\Gamma_2$ are both parts of a complete integral curve $\Gamma$.

After a rescaling we may assume $\lim_{s\rii}R(\Gamma_2(s))=4$. 
Then for some $s_1>0$ whose value will be determined later, we can find a point $p$ which is the center of an $\epsilon$-cylindrical plane, such that $|h(p)-2\pi|\le\epsilon$ and $d(p,\Gamma)=d(p,\Gamma_2)=s_1$, see Definition \ref{d: h} for $h(\cdot)$.
Let $\gamma_p(t)$ be the integral curve of $\nabla f$ starting from $p$.
Then $d(\gamma_p(t),\Gamma)$ increases in $t$.
In particular, $d(\gamma_p(t),\Gamma(0))\ge d(\gamma_p(t),\Gamma)\ge s_1$. 
So by Lemma \ref{l: inf achieved on the non-compact portion} we see that when $s_1$ is sufficiently large, the distance $d(\gamma_p(t),\Gamma)$ for any fixed $t$ is always attained in $\Gamma((-\infty,-s_0)\cup(s_0,\infty))$ where $\Gamma(s)$ are $\epsilon$-tip points.

In particular, the minimizing geodesic connecting $\gamma_p(t)$ to some point $y_t\in\Gamma((-\infty,-s_0)\cup(s_0,\infty))$ such that $d(\gamma_p(t),y_t)=d(\gamma_p(t),\Gamma)$ is orthogonal to $\Gamma$ at the $\epsilon$-tip point $y_t$, and $y_t\rii$ as $t\rii$.
On the one hand, by distance distortion estimate we have
\begin{equation}\label{e: fast}
    \frac{d}{dt}d(\gamma_p(t),\Gamma)\ge\sup_{\gamma\in \mathcal{Z}(t)}\int_{\gamma}\Ric(\gamma'(s),\gamma'(s))\,ds
\end{equation}
in the backward difference quotient sense (see \cite[Lemma 18.1]{RFTandA3}), where $\mathcal{Z}(t)$ is the space of minimizing geodesics $\gamma$ that realize the distance of $d(\gamma_p(t),\Gamma)$.
In particular, if $y_t\in\Gamma_2$ and $\gamma$ is a minimizing geodesic connecting $\gamma_p(t)$ to $y_t$, we have
\begin{equation}\label{e: cigarricci}
    \frac{d}{dt}d(\gamma_p(t),\Gamma)\ge\int_{\gamma}\Ric(\gamma'(s),\gamma'(s))\,ds\ge 2-\epsilon \quad \textit{for all}\quad t\in[0,T],
\end{equation}
where in the second inequality we used \eqref{e: integrate Ricci} that in a cigar soliton with $R(x_{tip})=4$, the integral of Ricci curvature along a geodesic ray starting from the tip is equal to $2$.
On the other hand, we have
\begin{equation}\label{e: slow}
    \frac{d}{dt}d(\gamma_p(t),\Gamma_1(s_0+t))\le\inf_{\gamma\in \mathcal{W}(t)}\int_{\gamma}\Ric(\gamma'(s),\gamma'(s))\,ds,
\end{equation}
in the forward difference quotient sense, where $\mathcal{W}(t)$ is the space of all minimizing geodesics between $\gamma_p(t)$ and $\Gamma_1(s_0+t)$.
Since $R(\Gamma(s))$ strictly decreases in $s$, because otherwise $(M,g)$ is isometric to $\R\times\cigar$, we may assume that for some $c_1>0$ we have $R(\Gamma_1(s_0,\infty))\le 4-c_1$.
So by \eqref{e: slow} there exists some $c_2>0$ such that
\begin{equation}\label{e: largecigarricci}
    \frac{d}{dt}d(\gamma_p(t),\Gamma_1(s_0+t))\le2-c_2.
\end{equation}
Therefore, it follows from \eqref{e: cigarricci} and \eqref{e: largecigarricci} that $d(\gamma_p(t),\Gamma)=d(\gamma_p(t),\Gamma_1)$ for sufficiently large $t$.

Therefore, we may let
\begin{equation*}
    T=\sup\{t: d(\gamma_p(t),\Gamma)=d(\gamma_p(t),\Gamma_2)\}<\infty,
\end{equation*}
then $d(\gamma_p(T),\Gamma_1)=d(\gamma_p(T),\Gamma_2)$ and $d(\gamma_p(t),\Gamma)=d(\gamma_p(t),\Gamma_2)$ for all $t\le T$. 
Integrating \eqref{e: cigarricci} from $0$ to $T$ we obtain
\begin{equation*}
    d(\gamma_p(t),\Gamma)\ge d(p,\Gamma)+(2-\epsilon)t\ge s_1+(2-\epsilon)t.
\end{equation*}
\yi{Since $\gamma_p(t)$ is the integral curve of $\nabla f$ starting from $p$, it follows by the definition of $h$ (see Definition \ref{d: h}) that $h(\gamma_p(t))$ is equal to the length of a minimizing 
geodesic loop at $p$ with respect to $g(t)$. So by the Ricci flow equation, $\Rm\ge0$, and $|\Ric|\le C R$, we see that $h(\gamma_p(t))$ is non-decreasing in $t$ and the following evolution inequality holds
\begin{equation*}
    \frac{d}{dt} h(\gamma_p(t))\le C\cdot R(\gamma_p(t))\cdot h(\gamma_p(t)).
\end{equation*}
}
Combining this with the following curvature upper bound from Theorem \ref{l: curvature upper bound initial}, 
\begin{equation*}
    R(\gamma_p(t))\le\frac{C}{d^2(\gamma_p(t),\Gamma)}\le \frac{C}{(s_1+(2-\epsilon)t)^2},
\end{equation*}
we obtain
\begin{equation*}
    \frac{d}{dt} h(\gamma_p(t))\le \frac{C}{(s_1+(2-\epsilon)t)^2}\cdot h(\gamma_p(t)).
\end{equation*}
Assuming $s_1$ is sufficiently large and integrating this we obtain
\begin{equation}\label{e: ell less than}
    h(\gamma_p(t))\le h(p)(1+\epsilon)\le 2\pi(1+\epsilon).
\end{equation}

Let $q\in \Gamma_1$ be a point such that 
\begin{equation*}
    d(\gamma_p(T),q)=d(\gamma_p(T),\Gamma).
\end{equation*}
So by Lemma \ref{l: compare f} we have
\begin{equation*}
    f(q)< f(\gamma_p(T)).
\end{equation*}
Since $R(\Gamma_1(s))$ decreases and $f(\Gamma_1(s))$ increases in $s$, there is $q_2\in\Gamma_1$ such that 
\begin{equation*}
 f(q_2)=f(\gamma_p(T))>f(q) \quad \textit{and} \quad   R(q)> R(q_2).
\end{equation*}

In the rest of proof we will show $R(q_2)\ge 4-\delta$.
First, if $d(\gamma_p(T),q_2)<\delta^{-1}R^{-1/2}(q_2)$, then since $q_2$ is an $\epsilon$-tip point, we obtain
\begin{equation*}
    |h(\gamma_p(T))-4\pi\cdot R^{-1/2}(q_2)|\le\delta.
\end{equation*}
This fact combined with \eqref{e: ell less than} gives
\begin{equation}\label{e: upper bound of h}
    R^{-1/2}(q_2)\le\frac{1}{4\pi}\,h(\gamma_p(T))+\delta
    \le \frac{1}{2}+\delta,
\end{equation}
and hence $R(q_2)\ge 4-\delta$.

So we may assume from now on that
$d(\gamma_p(T),q_2)\ge\delta^{-1}R^{-1/2}(q_2)$.

\begin{claim}\label{claim: neck}
There exists an $\delta$-cylindrical plane at some $z\in f^{-1}(f(q_2))$ at scale $r=R^{-1/2}(q_2)$.
\end{claim}

\begin{proof}[Proof of Claim \ref{claim: neck}]
Let $\phi:(\mathbb{R}\times\cigar,x_{tip}):\ri (M,q_2)$ be the inverse of an $\epsilon$-isometry.
For the interval $[-1/\delta,1/\delta]\subset\R$ and the ball $B_{g_c}(x_{tip},1/\sqrt{\delta})$ in $(\cigar,g_c)$,
we consider image of their product $U:=\phi([-1/\delta,1/\delta]\times B_{g_c}(x_{tip},1/\sqrt{\delta}))$.
Let $\sigma$ be a smooth curve connecting $q_2$ to $\gamma_p(T)$ in the level set $f^{-1}(f(q_2))$. Since $\gamma_p(T)\notin U$, by continuity $\sigma$ must exit $U$ at some $z\in\partial U$. 

Denote $\partial U_{\pm}:=\phi(\{\pm1/\delta\}\times B_{g_c}(x_{tip},1/\sqrt{\delta}))$. 
We will show that $z\in \phi([-1/\delta,1/\delta]\times \partial B_{g_c}(x_{tip},1/\sqrt{\delta}))=\partial U-\partial U_--\partial U_+$.
Replacing $+,-$ if necessary we may assume $f(\phi(1/\delta))>f(q_2)>f(\phi(-1/\delta))$.

On the one hand, for any $y\in \partial U_-$, let $\sigma:[0,\ell]\ri M$ be a unit speed geodesic from $y$ to some point $q_-\in \phi([-1/\delta,1/\delta]\times x_{tip})$, which achieves the distance from $y$ to it. Then we have
\begin{equation*}
\begin{split}
    f(y)-f(q_-)&=\int_0^{\ell}\langle\nabla f,\sigma'(r)\rangle dr
    =\int_0^{\ell}\int_0^r\nabla^2 f(\sigma'(s),\sigma'(s))\,ds dr\\
    &\le \int_0^{\ell}\int_0^{\ell}\Ric(\sigma'(s),\sigma'(s))\,ds dr
    \le C/\sqrt{\delta},
\end{split}
\end{equation*}
where we used that the length of $\sigma$ satisfies $\ell\in[\delta^{-1/2}R^{-1/2}(q_2),2\delta^{-1/2}R^{-1/2}(q_2)]$.
This then implies
\begin{equation*}
    f(y)\le f(q_-)+\frac{C}{\sqrt{\delta}}\le f(q_2)-\frac{1}{C\delta}+\frac{C}{\sqrt{\delta}}<f(q_2),
\end{equation*}
and hence $\partial U_-$ is disjoint from $f^{-1}(f(q_2))$.

On the other hand, for any $y\in \partial U_+$, let $q_+$ be a point in $\phi([-1/\delta,1/\delta]\times x_{tip})$ that is closest to $y$. Then
\begin{equation*}
    f(y)\ge f(q_+)\ge f(q_2)+\frac{1}{C\delta}>f(q_2),
\end{equation*}
and hence $\partial U_+$ is also disjoint from $f^{-1}(f(q_2))$.
So the claim holds.
\end{proof}

By Lemma \ref{l: 3D neck is 2D neck}, the two $\delta$-cylindrical plane at $\gamma_p(T)$ and $z$ produce the two $\delta$-necks in the level set surface $f^{-1}(f(q_2))$: \yi{One $\delta$-neck denoted by $N_1$ is centered at $\gamma_p(T)$ which has scale $2(1+\delta)$, because $h(\gamma_p(T))\le 2\pi(1+\delta)$ by \eqref{e: upper bound of h}, and $h(\gamma_p(T))\ge h(p)\ge 2\pi-\epsilon$ by the choice of $p$ and the monotonicity of $h$ along integral curves;} and the other $\delta$-neck denoted by $N_2$ is centered at $z$ with scale $R^{-1/2}(q_2)$.
By the choice of $N_2$ and $d(\gamma_p(T),q_2)\ge\delta^{-1}R^{-1/2}(q_2)$, it is clear that $N_2$ is in the 2-ball bounded by the central circle of $N_1$.
Since $f^{-1}(f(q_2))$ is positively-curved, we can apply Lemma \ref{l: 2D surface} and deduce that
\begin{equation*}
    R^{-1/2}(q_2)\le 2(1+\delta), \quad\textit{hence}\quad R(q_2)\ge 4(1-\delta).
\end{equation*}
Now letting $\epsilon$ go to zero, by the monotonicity of $R$ along $\Gamma$, this implies that $R$ is a constant along $\Gamma$. So $R\equiv2$ on $\Gamma$.
So by the soliton identity,
\begin{equation*}
    \Ric(\nabla f,\nabla f)=-\langle\nabla R,\nabla f\rangle=0.
\end{equation*}
The Ricci curvature vanishes along $\Gamma$ in the direction of $\nabla f$.
So the soliton splits off a line and it is isometric to $\R\times\cigar$, contradiction.
This proves the existence of a critical point of $f$.
\end{proof}

The following corollary follows immediately from the Theorem, Lemma \ref{l: Gamma} and \ref{l: two rays}.

\begin{cor}\label{l: new Gamma}
There are two integral curves $\Gamma_i:(-\infty,\infty)\ri M$ of $\nabla f$, $i=1,2$, such that the followings hold: 
\begin{enumerate}
    \item Let $p$ be the critical point of $f$. Then $\lim_{s\ri-\infty}\Gamma_i(s)=p$;
    \item The pointed manifolds $(M,r^{-2}(\Gamma_i(s))g,\Gamma_i(s))$ smoothly converge to the manifold $(\R\times\cigar,r^{-2}(x_{tip})g_c,x_{tip})$ as $s\rii$;
    \item For any $p^{(i)}_k\in\Gamma_i$, $p^{(i)}_k\rii$ as $k\rii$, the minimizing geodesics $pp^{(i)}_k$ subsequentially converge to two rays $\gamma_1,\gamma_2$, such that $[\gamma_1]=0,\,[\gamma_2]=\theta\in S_{\infty}(M,p)=[0,\theta]$, $\theta\in[0,\pi)$. 
\end{enumerate}
\end{cor}

\subsection{An ODE lemma for distance distortion estimates}

We will use the following ODE lemma of two time-dependent scalar functions to estimate certain distance distortion in Theorem \ref{l: wing-like}. 
This method generalizes the bootstrap argument
in \cite[Theorem 1.3]{Lai2020_flying_wing}, which relies on the $O(2)$-symmetric structure of the soliton.

\begin{lem}(An ODE Lemma)\label{l: ODE}
Let $H,d:[0,T]\rightarrow(0,\infty)$ be two differentiable functions satisfying the following
\begin{equation}\label{e: ODE derivative assump}
    \begin{cases}
     H'(t)\ge C_1\cdot h^{-1}(t)\\
     h'(t)\le C_2\cdot H^{-2}(t)\cdot h(t),
    \end{cases}
\end{equation}
for some constants $C_1,C_2>0$. Suppose 
\begin{equation}\label{e: C5}
\frac{H(0)}{h(0)}> \frac{C_2}{C_1}.
\end{equation}
Let $C_3:=C_1 h^{-1}(0)-C_2H^{-1}(0)>0$. Then we have
\begin{equation*}
    \begin{cases}
    H(t)\ge C_3t+H(0)\\
    h(t)\le h(0)e^{\frac{C_2}{C_3H(0)}},
    \end{cases}
\end{equation*}
for all $t\in[0,T]$.
\end{lem}

In Section \ref{s: asymptotic geometry} we will show that the soliton outside of a compact subset is covered by two regions: The edge region consists of two solid cylindrical chains where the local geometry is close to $\R\times\cigar$, and the almost flat region carries a $S^1$-fibration and the local geometry looks like $\RR\times S^1$. 
Fix a point $x$ in the almost flat region, the two functions $h(t)$ and $H(t)$ are essentially the length of the $S^1$-fiber at $\phi_t(x)$, and the distance from $\phi_t(x)$ to $\Gamma$, where $\{\phi_t\}_{t\in\R}$ are the diffeomorphisms generated by $\nabla f$.

\begin{proof}
Dividing both sides by $h(t)$ in the second inequality in \eqref{e: ODE derivative assump} we get
\begin{equation*}
    \pt(\ln h(t))\le C_2H^{-2}(t).
\end{equation*}
Integrating this from $0$ to $t$ we get
\begin{equation*}
    \ln h(t)\le C_2\int_{0}^{t}H^{-2}(s)ds + \ln h(0),
\end{equation*}
and hence
\begin{equation*}
    h(t)\le h(0)\,e^{C_2\int_{0}^{t}H^{-2}(s)\,ds},
\end{equation*}
plugging which into the first inequality in \eqref{e: ODE derivative assump} we get
\begin{equation*}
    H'(t)\ge C_1\,h(0)\,e^{-C_2\int_0^{t}H^{-2}(s)\,ds}.
\end{equation*}
Let $H_0(t)$ be a solution to the following problem:
\begin{equation}\label{e: H_0}
    \begin{cases}
    H_0(0)=H(0)\\
    H'_0(t)=C_1\,h^{-1}(0)\,e^{-C_2\int_0^{t}H_0^{-2}(s)\,ds}.
    \end{cases}
\end{equation}
Then it is easy to see that 
\begin{equation}\label{e: H bigger than H0}
    H(t)\ge H_0(t)>0,
\end{equation}
for all $t\ge0$.

The second equation in \eqref{e: H_0} implies
\begin{equation*}
    \ln  (H'_0(t))=\ln (C_1\, h^{-1}(0))-C_2\int_0^t H_0^{-2}(s)ds,
\end{equation*}
differentiating which at both sides we obtain
\begin{equation*}
\begin{split}
    \pt( H'_0(t)&-C_2H_0^{-1}(t))=0.
\end{split}
\end{equation*}
Integrating this and using \eqref{e: H_0},\eqref{e: C5} we obtain
\begin{equation*}
    H'_0(t)-C_2H_0^{-1}(t)= H'_0(0)-C_2H_0^{-1}(0)=C_1 h^{-1}(0)-C_2H^{-1}(0)=C_3>0.
\end{equation*}
So by \eqref{e: H bigger than H0} we obtain
\begin{equation*}
    H(t)\ge H_0(t)\ge C_3t+H(0).
\end{equation*}
Substituting this into the second inequality in \eqref{e: ODE derivative assump} we get
\begin{equation*}
    \pt(\ln h(t))\le \frac{C_2}{(C_3t+H(0))^2},
\end{equation*}
integrating which we obtain
\begin{equation*}
    h(t)\le h(0)e^{\frac{C_2}{C_3H(0)}-\frac{C_2}{C_3(C_3t+H(0))}}\le h(0)e^{\frac{C_2}{C_3H(0)}}.
\end{equation*}
\end{proof}

\subsection{Asymptotic cone is not a ray}
In this subsection, we show that the scalar curvature has positive limits along the two integral curves $\Gamma_1,\Gamma_2$.
As a consequence of this, the asymptotic cone is isometric to a sector with non-zero angle.

A key step in the proof is to choose two suitable functions that evolves by the conditions in Lemma \ref{l: ODE}. Roughly speaking, for a fixed point $x$ at which the soliton is an $\epsilon$-cylindrical plane, the two functions $h(t)$ and $H(t)$ are essentially the time-$t$-length of the $S^1$-fiber at $x$, and the time-$t$-distance to the edges $\Gamma$, where $t$ is the backwards time variable in the Ricci flow of the soliton.

These two functions satisfy the inequalities in Lemma \ref{l: ODE}:
On the one hand, by Perelman's curvature estimate we will see that the curvature in the almost flat region is bounded above by $
H^{-2}(t)$. So by the Ricci flow equation, $h(t)$ evolves by the second inequality in \eqref{e: ODE derivative assump}.
On the other hand, the increase of $H(t)$ is contributed by the regions that look like $\R\times\cigar$, whose volume scale is roughly $h(t)$.
So $H(t)$ evolves under the first inequality in \eqref{e: ODE derivative assump}.
Then applying the lemma we will see that $H(t)$ increases at least linearly, and $h(t)$ stays bounded as $t\rii$.

We first prove a technical lemma using metric comparison.

\begin{lem}\label{l: exactly two caps in 2D}
There exists $\epsilon>0$ such that the following holds:
Let $\Sigma$ be a 2D complete Riemannian manifold with non-negative curvature. Then there can not be more than two disjoint $\epsilon$-caps.

Moreover,
suppose there are two disjoint $\epsilon$-caps centered at $p_1,p_2$, and $\Sigma$ is an $\epsilon$-neck at a point $p\in\Sigma$ such that $p$ is not in the two $\epsilon$-caps.
Then the central circle at $p$ separates $p_1$ and $p_2$.
\end{lem}

\begin{proof}
For the first claim, suppose by contradiction that there are three disjoint $\epsilon$-caps $\mathcal{C}_1,\mathcal{C}_2,\mathcal{C}_3$ centered at $p_1,p_2,p_3$.
We shall use $\delta(\epsilon)$ to denote all constants that go to zero as $\epsilon$ goes to zero.

Assume the minimizing geodesics $p_1p_2,p_1p_3$ intersect the boundary of $\mathcal{C}_1$ at $x_2$ and $x_3$ respectively, which are centers of two $\epsilon$-necks. So we have
\begin{equation*}
    d(x_2,x_3)<\delta(\epsilon)d(x_2,p_1),
\end{equation*}
which by the monotonicity of angles implies
\begin{equation*}
    \widetilde{\measuredangle}p_2p_1p_3\le\widetilde{\measuredangle}x_2p_1x_3\le \delta(\epsilon).
\end{equation*}
In the same way we obtain that $\widetilde{\measuredangle}p_1p_2p_3,\widetilde{\measuredangle}p_1p_3p_2\le\delta(\epsilon)$.
But then we have
\begin{equation*}
    \widetilde{\measuredangle}p_1p_2p_3+\widetilde{\measuredangle}p_2p_1p_3+\widetilde{\measuredangle}p_1p_3p_2\le3\delta(\epsilon)<\pi,
\end{equation*}
which is impossible.

For the second claim,
suppose the central circle at $p$ does not separates $p_1$ and $p_2$. 
Let $\psi:(-\epsilon^{-1},\epsilon^{-1})\times S^1\ri \Sigma$ be the inverse of the $\epsilon$-isometry of the $\epsilon$-neck at $p$.
Let $\gamma_{\pm}=\psi(\{\pm\epsilon^{-1}\}\times S^1)$. Then after possibly replacing $+$ with $-$, we claim that the minimizing geodesics $pp_1,pp_2,p_1p_2$ are all contained in the the component of $\Sigma$ separated by $\gamma_{-}$ which contains $\gamma_{+}$:
First, since $p_1,p_2$ are in the same component of $\Sigma$ separated by $\psi(\{0\}\times S^1)$,
suppose $pp_1$ intersects $\gamma_{+}$, then it follows that $pp_2$ also intersects $\Sigma_{+}$, and the claim follows by the minimality of these geodesics.

By the claim, we can use a similar argument as before to deduce 
\begin{equation*}
    \widetilde{\measuredangle}p_1p_2p+\widetilde{\measuredangle}p_2p_1p+\widetilde{\measuredangle}p_1pp_2<\pi,
\end{equation*}
which is a contradiction.
\end{proof}

Now we prove the main theorem in this section. 
First, it states that the soliton is $\mathbb{Z}_2$-symmetric at infinity, in the sense that $R$ has equal positive limits along the two ends of $\Gamma$.
Moreover, assume this positive limit is equal to $4$ after a proper rescaling, then any sequence of points going to infinity converges to either $\RR\times S^1$ or $\R\times\cigar$, without any rescalings.

We remark that this $\mathbb{Z}_2$-symmetry at infinity is also true in mean curvature flow: A mean curvature flow flying wing in $\R^3$ is a graph over a finite slab. Moreover, the slab width is equal to that of its asymptotic translators, which are two tilted Grim-Reaper hypersurfaces  \cite{Spruck2020CompleteTS,white}.

\begin{theorem}\label{l: wing-like}
Let $(M,g)$ be a 3D steady gradient soliton with positive curvature that is not a Bryant soliton. Let $\Gamma_1,\Gamma_2$ be the two integral curves of $\nabla f$ from Corollary \ref{l: new Gamma}. Then after a rescaling of $(M,g)$, we have
\begin{equation*}
    \lim_{s\rii}R(\Gamma_1(s))=\lim_{s\rii}R(\Gamma_2(s))=4.
\end{equation*}
Moreover, for any sequence of points $q_k\rii$, the sequence of pointed manifolds $(M,g,q_k)$ converge to either $(\RR\times S^1,g_{stan})$, or  $(\R\times\cigar,g_c)$. In particular, if $\{q_k\}\subset\Gamma_1\cup\Gamma_2$, then $(M,g,q_k)$ converges to $\R\times\cigar$.
\end{theorem}

\begin{proof}
We will first prove $\lim_{s\rii}R(\Gamma_i(s))>0$, $i=1,2$.
By Theorem \ref{t: Rmax critical point} we know that $f$ has a unique critical point $x_0$. Assume $f(x_0)=0$. Then it is easy to see that all level sets $f^{-1}(s)$ for all $s>0$ are diffeomorphic to 2-spheres, and the induced metrics have positive curvature.
Suppose by contradiction that $\lim_{s\rii}R(\Gamma_i(s))>0$, $i=1,2$, does not hold.
Let $\Gamma_1,\Gamma_2$ be from Corollary \ref{l: new Gamma}, and $\Gamma=\Gamma_1(-\infty,\infty)\cup\Gamma_2(-\infty,\infty)\cup\{x_0\}$. Then the subset $\Gamma$ is invariant under the diffeomorphism $\phi_t$.

Let $C$ denote all positive universal constants, $\epsilon$ denote all positive constants that we may take arbitrarily small, and $\delta$ denote all positive constants that converge to zero as $\epsilon\ri0$. The value of $\delta$ may change from line to line.
Suppose by contradiction that the theorem does not hold. Then we may assume $\lim_{s\rii}R(\Gamma_1(s))=0$.

Choose a point $p\in M$ such that $(M,g)$ is an $\epsilon$-cylindrical plane at $p$, and $d(p,\Gamma)=d(p,\Gamma_1)$. Let $\gamma_p(t)$ be the integral curve of $\nabla f$ starting from $p$. Then by Lemma \ref{l: tip contracting} it follows that $(M,g)$ is always an $\epsilon$-cylindrical plane at $\gamma_p(t)$ for $t\ge0$.
So we can define $h(\gamma_p(t))$ as in Definition \ref{d: h}.
Denote $d(\gamma_p(t),\Gamma)$ by $H(\gamma_p(t))$, and abbreviate $\gamma_p(t)$ as $p_t$.
By some distortion estimates and Theorem \ref{l: curvature upper bound initial} it is easy to see that
\begin{equation}\label{e: H h initial}
    \begin{cases}
    \pt H(p_t)\ge C^{-1}\cdot r^{-1}(q_t)\\
    \pt h(p_t)\le C\cdot H^{-2}(p_t)\cdot h(p_t),
    \end{cases}
\end{equation}
In the following we will show that $r(q_t)<C\,h(p_t)$, and $H(p_t),h(p_t)$ satisfy the conditions in the ODE Lemma \ref{l: ODE}.

First, let $q_t\in\Gamma$ be a point such that $d(p_t,\Gamma)=d(p_t,q_t)$. 
We claim that $R(q_t)\ri0$ as $t\rii$. First, this is clear
if $\lim_{s\rii}R(\Gamma_2(s))=0$, because $q_t\rii$ as $t\rii$ by Lemma \ref{l: inf achieved on the non-compact portion}.
Second, if $\lim_{s\rii}R(\Gamma_2(s))>0$.
We may assume 
\begin{equation}\label{e: inf bigger than sup}
    \sup_{s\in[s_1,\infty)} R(\Gamma_1(s))< 100\inf_{s\in[s_1,\infty)} R(\Gamma_2(s)),
\end{equation}
for some $s_1>0$. We may also assume by Lemma \ref{l: inf achieved on the non-compact portion} that $q_t\in\Gamma_1([s_1,\infty))\cup\Gamma_2([s_1,\infty))$.
Since $d(p,\Gamma)=d(p,\Gamma_1)$, by \eqref{e: inf bigger than sup} 
and a distance distortion estimate we see that the closest point $q_t$ on $\Gamma$ is always on the segment $\Gamma_1([s_1,\infty))$
for all $t\ge 0$. So $R(q_t)\ri0$. So the claim holds.

We fix some sufficiently large $t$ so that $R(q_t)<\frac{1}{2000C^2}$,  where the value of $C>0$ will be clear later.
For simplicity, we will omit the subscript $t$ in $q_t$ and $p_t$.
Let $\Sigma:=f^{-1}(f(q))$.
Then $\Sigma$ is diffeomorphic to $S^2$, and it separates $M$ into a bounded component $f^{-1}([0,f(q)))$ diffeomorphic to a 3-ball, and an unbounded component $f^{-1}((f(q),\infty))$ diffeomorphic to $\R\times S^2$. So $\Gamma_2\cap\Sigma\neq\emptyset$. Let $q_2\in\Gamma_2\cap\Sigma$, and we may assume it is an $\epsilon$-tip point.

\begin{claim}\label{c: one half}
$d_{\Sigma}(p_1,q_2)\ge\frac{1}{2}\,d_{\Sigma}(q,p_1)$.
\end{claim}

\begin{proof}[Proof of the claim]
Let $\gamma:[0,1]\ri M$ be a minimizing geodesic from $q$ to $p$, then by Lemma \ref{l: compare f} we have $f(\gamma([0,1]))\ge f(q)$, so $p\in f^{-1}([f(q),\infty))$. 
Therefore, there is a smooth  non-negative function $T:[0,1]\ri\R$ such that $\overline{\gamma}:=\phi_{-T(r)}(\gamma(r))\in\Sigma$, $r\in[0,1]$. 
Let $p_1=\overline{\gamma}(1)=\phi_{-T(1)}(p)$. Then by the discussion in the beginning of the proof, we may assume $(M,g)$ is an $\epsilon$-cylindrical plane at $p_1$.
First, by the positive curvature and distance shrinking in the Ricci flow $g(t)=\phi_{-t}^*g$ of the soliton, we have
\begin{equation}\label{e: two d}
    h(p_1)\le h(p).
\end{equation}
Consider the smooth map $\chi:[0,1]\times\R\ri M$ defined by $\chi(r,t)=\phi_t(\overline{\gamma}(r))$.
Since $\phi_t$ is the flow of $\nabla f$, we have $f\circ\chi(r,t)=t+f(q)$ and hence $\langle\chi_*(\partial_t),\chi_*(\partial_r)\rangle=\langle\nabla f,\chi_*(\partial_r)\rangle=0$. So we can compute that
\begin{equation}\label{e: useful}
\begin{split}
    d(p,q)=L(\gamma)&=\int_0^1|\chi_{*(r,T(r))}(\partial_r)+T'(r)\cdot\chi_{*(r,T(r))}(\partial_t)|\,dr\\
    &\ge\int_0^1|\chi_{*(r,T(r))}(\partial_r)|\,dr
    =\int_0^1|\phi_{T(r)*}(\overline{\gamma}'(r))|\,dr\\
    &\ge\int_0^1|\overline{\gamma}'(r)|\,dr=L(\overline{\gamma})\ge d_{\Sigma}(q,p_1),
\end{split}
\end{equation}
where $d_{\Sigma}$ denotes the intrinsic metric on $\Sigma$. 

Since $|\nabla f|\ge C^{-1}>0$ on $M\setminus B(x_0,1)$, we have 
\begin{equation}\label{e: one half}
\begin{split}
    d(p,p_1)\le C(f(p)-f(p_1))&=C(f(p)-f(q))\\
    &=C\int_0^1\int_0^r\Ric(\gamma'(s),\gamma'(s))\,ds\,dr\\
    &\le C\,d(p,q)\int_0^1\Ric(\gamma'(r),\gamma'(r))\,dr
    \le \frac{1}{2} d(p,q),
\end{split}
\end{equation}
where in the last inequality we used $\int_0^1\Ric(\gamma'(r),\gamma'(r))\,dr\le \frac{1}{2C}$, which follows from the second variation formula and the assumption $R(q)<\frac{1}{2000C^2}$.

By the choice of $q$ we have $d(p,q_2)\ge d(p,\Gamma) 
=d(p,q)$. Together with inequalities \eqref{e: useful} and \eqref{e: one half} this implies
\begin{equation*}\begin{split}\label{e: distance compare for volume comparison}
    d_{\Sigma}(p_1,q_2)\ge d(p_1,q_2)\ge d(p,q_2)-d(p,p_1)
    \ge d(p,q)-d(p,p_1)\ge\frac{1}{2}d_{\Sigma}(q,p_1).
\end{split}\end{equation*}
This proves the claim.
\end{proof}

\begin{claim}\label{c: compare r}
$r(q)\le 1200\,r(p)$.
\end{claim}
\begin{proof}[Proof of the claim]
Since $q$ is an $\epsilon$-tip point, we may assume without loss of generality that $d_{\Sigma}(q,p_1)\ge \overline{D} R^{-1/2}(q)$ where $\overline{D}$ is from Lemma \ref{l: geometry of level set near tip}, because otherwise the claim clearly holds for sufficiently small $\epsilon$.
So we can find a point $q'_1$ on the $\Sigma$-minimizing geodesic between $q$ and $p_1$ such that $d_{\Sigma}(q,q'_1)= \overline{D} R^{-1/2}(q)$ and hence
\begin{equation}\label{e: point}
    r(q)\le 10\,r(q'_1).
\end{equation}
Moreover, by Lemma \ref{l: geometry of level set near tip} it follows that $(M,g)$ is a $\delta$-cylindrical plane at $q'_1$. So by Lemma \ref{l: 3D neck is 2D neck} this implies that the level set $\Sigma$ is a $\delta$-neck at $p_1$ and $q'_1$ at scale $r(p_1)$ and $r(q'_1)$.

We may also assume  $d_{\Sigma}(p_1,q'_1)>10\,r(p_1)$, because otherwise $r(q)\le 10 r(q'_1)\le 20\,r(p_1)$ and the claim holds. We will that the following  inequalities hold,  
\begin{equation}\label{e: VC}
d_{\Sigma}(q_2,p_1)\le d_{\Sigma}(q_2,q'_1)\le 3\,d_{\Sigma}(q_2,p_1).
\end{equation}
For the first inequality in \eqref{e: VC},
since $q_2$ and $q'_1$ are in two disjoint $\delta$-caps, it follows by Lemma \ref{l: exactly two caps in 2D} that the central circle at $p_1$ separates $q'_1$ and $q_2$ in $\Sigma$.
So a minimizing geodesic between $q'_1$ and $q_2$ intersects the central circle at $p_1$, and hence
\begin{equation*}
    d_{\Sigma}(q_2,q'_1)\ge d_{\Sigma}(p_1,q_2)-10 h(p_1)+d_{\Sigma}(p_1,q'_1)\ge d_{\Sigma}(p_1,q_2),
\end{equation*}
where in the last inequality we used
$d_{\Sigma}(p_1,q'_1)>10\,h(p_1)$.
For the second inequality in \eqref{e: VC}, 
by Claim \ref{c: one half} we have
\begin{equation*}
d_{\Sigma}(q_2,q'_1)\le d_{\Sigma}(q_2,p_1)+d_{\Sigma}(p_1,q'_1)
\le d_{\Sigma}(q_2,p_1)+d_{\Sigma}(p_1,q)\le 3\,d_{\Sigma}(q_2,p_1).
\end{equation*}

By the first inequality in \eqref{e: VC} and the positive curvature on $\Sigma$, we can deduce by the volume comparison that
\begin{equation}\label{e: oil}
    \frac{|\partial B_{\Sigma}(q_2,d_{\Sigma}(q_2,q'_1))|}{d_{\Sigma}(q_2,q'_1)}\le \frac{|\partial B_{\Sigma}(q_2,d_{\Sigma}(q_2,p_1))|}{d_{\Sigma}(q_2,p_1)}.
\end{equation}
Since $\Sigma$ is a $\delta$-neck at both points $p_1$ and $q'_1$, it is easy to see that
\begin{equation*}
    \frac{1}{2}\le\frac{|\partial B_{\Sigma}(q_2,d_{\Sigma}(q_2,q'_1))|}{r(q'_1)}\le2\quad\textit{and}\quad \frac{1}{2}\le\frac{|\partial B_{\Sigma}(q_2,d_{\Sigma}(q_2,p_1))|}{r(p_1)}\le2.
\end{equation*}
Therefore, by the second inequality in \eqref{e: VC} and \eqref{e: oil} we get
$r(q'_1)\le 12\,r(p_1)$.
Then by \eqref{e: point}, $r(p_1)\le 10\,h(p_1)$, and \eqref{e: two d} we get
\begin{equation*}
    r(q)\le 120\,r(p_1)\le 1200\,h(p_1)\le 1200\,h(p).
\end{equation*}

\end{proof}

Restoring the subscript $t$ in $p_t,q_t$, we proved $r(q_t)<1200\,h(p_t)$. So by the evolution inequalities \eqref{e: H h initial} we see  that $h(p_t)$ and $H(p_t)$ satisfy the following inequalities
\begin{equation*}
    \begin{cases}
    \pt H(p_t)\ge (1200C)^{-1}\cdot h^{-1}(q_t),\\
    \pt h(p_t)\le C\cdot H^{-2}(p_t)\cdot h(p_t).
    \end{cases}
\end{equation*}
Since $p$ is the center of an $\epsilon$-cylindrical plane, it follows that by assuming $\epsilon$ to be sufficiently small, we have $\frac{H(p)}{h(p)}>1200\,C^2$.
Therefore, the two functions $H(p_t),h(p_t)$ satisfy all assumptions in the ODE Lemma \ref{l: ODE}, applying which we can deduce
\begin{equation*}
    H(p_t)\ge C^{-1}t\quad\textit{and}\quad h(p_t)\le C,
\end{equation*}
for all sufficiently large $t$.
So by Claim \ref{c: compare r} and $r(p_t)\le 10\,h(p_t)$ we obtain
\begin{equation*}
    r(q_t)\le 1200\,r(p_t)\le 12000\,h(p_t)\le C.
\end{equation*}
Note $q_t\in \Gamma_1$ are $\epsilon$-tip points and $q_t\rii$ as $t\rii$. 
This implies $\lim_{s\rii}R(\Gamma_1(s))>0$, a contradiction.
Therefore, we proved $\lim_{s\rii}R(\Gamma_i(s))>0$.

Lastly, we prove $\lim_{s\rii}R(\Gamma_1(s))=\lim_{s\rii}R(\Gamma_2(s))$ and the rest assertions of the theorem.
Let $x_1,x_2\in M$ be any two $\epsilon$-cylindrical points. We will show $|h(x_1)-h(x_2)|\le \epsilon$. If this is true, then combining with the convergence to $\R\times\cigar$ after rescalings along $\Gamma_1$ and $\Gamma_2$, this implies the theorem.

On the one hand, similarly as before, we can show that 
\begin{equation}\label{enheng}
    |h(x_i)-h(\phi_t(x_i))|\le\epsilon,\quad i=1,2.
\end{equation}
On the other hand, we claim
\begin{claim}\label{c: lnt}
$\lim_{t\rii} d(\phi_t(x_1),\phi_t(x_2))<\infty$.
\end{claim}

\begin{proof}[Proof of the claim]
First, we have that $d(\phi_t(x_i),\Gamma)\ge d(x_i,\Gamma)+C^{-1}t$, $i=1,2$,
so by the distance distortion Lemma \ref{l: distance laplacian} and Theorem \ref{l: curvature upper bound initial} we have 
\begin{equation*}
    \frac{d}{dt}d(\phi_t(x_1),\phi_t(x_2))\le \max\left\{\frac{C}{d(x_1,\Gamma)+C^{-1}t},\frac{C}{d(x_2,\Gamma)+C^{-1}t}\right\},
\end{equation*}
integrating which we obtain
\begin{equation}\label{e: lnt}
    d(\phi_t(x_1),\phi_t(x_2))\le d(x_1,x_2)+C\ln t.
\end{equation}
Therefore, for any sufficiently large $t$, let $\gamma:[0,1]\ri M$ be a minimizing geodesic between $\phi_t(x_1),\phi_t(x_2)$, by triangle inequalities we have $d(\gamma([0,1],\Gamma)> C^{-1}t$.
So by Theorem \ref{l: curvature upper bound initial} we have $\sup_{s\in[0,1]}R(\gamma(s))\le \frac{C}{t^2}$, and hence \eqref{e: lnt} implies
\begin{equation*}
    \frac{d}{dt}d(\phi_t(x_1),\phi_t(x_2))\le\int_{\gamma}\Ric(\gamma'(s),\gamma'(s))\,ds\le \frac{C}{t^{\frac{3}{2}}},
\end{equation*}
integrating which we proved the claim.
\end{proof}

Note that the points $\phi_t(x_2)$ converge to a rescaling of $\RR\times S^1$ as $t\rii$, so by Claim \ref{c: lnt} we see that 
\begin{equation*}
    |h(\phi_t(x_1))-h(\phi_t(x_2))|\le\epsilon,
\end{equation*}
for all sufficiently large $t$.
Combining this with \eqref{enheng}, this implies 
\begin{equation*}
    |h(x_1)-h(x_2)|\le\epsilon,
\end{equation*}
which proves the theorem.

\end{proof}

In the following we show that the soliton is asymptotic to a sector. Therefore, 3D steady gradient solitons are all flying wings except the Bryant soliton.

\begin{cor}[Asymptotic to a sector]\label{c: Asymptotic to a sector}
Let $(M,g)$ be a 3D steady gradient soliton with positive curvature. If the asymptotic cone of $(M,g)$ is a ray, then $(M,g)$ is isometric to a Bryant soliton.
\end{cor}

\begin{proof}
Suppose that $(M,g)$ is not a Bryant soliton.
Let $C>0$ denotes all constants depending on the soliton $(M,g)$ and $\epsilon>0$ be some sufficiently small number.
Let $\Gamma_1,\Gamma_2$ be the integral curves from Corollary \ref{l: new Gamma}. By Theorem \ref{l: wing-like} we may assume  $\lim_{s\rii}R(\Gamma_i(s))=4$.
Let $\Gamma=\Gamma_1([0,\infty))\cup\Gamma_2([0,\infty))$.
Let $p\in M$ be the center of an $\epsilon$-cylindrical plane, then we have
\begin{equation}\label{e: scallop0}
    d(\phi_t(p),\Gamma)\ge 1.9\,t+d(p,\Gamma),
\end{equation}
for $t\ge0$.
Suppose $\max R=R(x_0)\le C$, then we have
\begin{equation*}
    d(x_0,\phi_t(p))\le 20\,C\,t,
\end{equation*}
which combining with \eqref{e: scallop0} implies
\begin{equation}\label{e: scallop}
    d(\phi_t(p),\Gamma)\ge C^{-1}d(x_0,\phi_t(p))
\end{equation}
for all large $t$.

Let $q_k\in\Gamma_1$, $\overline{q}_k\in\Gamma_2$ be sequences of points such that $d(x_0,q_k)=d(x_0,\overline{q}_k)$, and let $\sigma_k:[0,1]\ri M$ be a minimizing geodesic connecting $q_k,\overline{q}_k$.
Then it is easy to see $d(x_0,\sigma_k([0,1]))\rii$ as $k\rii$.
Since $d(x_0,\phi_t(p))\rii$ as $t\rii$, the integral curve $\phi_t(p)$ must pass through the $S^1$-factor of an $\epsilon$-cylindrical plane centered at some point on $\sigma_k(0,1)$. In particular, we can find $t_k>0$ and $s_k\in(0,1)$ such that $t_k\rii$ as $k\rii$, and
\begin{equation*}
    d(\phi_{t_k}(p),\sigma_k(s_k))<2\pi.
\end{equation*}
Since $d(q_k,\overline{q}_k)\ge d(\sigma_k(s_k),q_k)\ge d(\sigma_k(s_k),\Gamma)$, this implies by triangle inequality that
\begin{equation*}
    d(q_k,\overline{q}_k)
    \ge d(\phi_{t_k}(p),\Gamma)-d(\phi_{t_k}(p),\sigma_k(s_k))
    \ge d(\phi_{t_k}(p),\Gamma)-2\pi,
\end{equation*}
which together with \eqref{e: scallop} implies
\begin{equation}\label{berry}
    d(q_k,\overline{q}_k)\ge C^{-1}d(x_0,\phi_{t_k}(p))
\end{equation}
for all large $k$.
Since $(M,g)$ is not isometric to $\R\times\cigar$, we have
\begin{equation*}
    d(q_k,\overline{q}_k)<(2-C^{-1})d(x_0,q_k).
\end{equation*}
Combining it with the following triangle inequality  
\begin{equation*}
    d(x_0,\phi_{t_k}(p))+d(q_k,\overline{q}_k)\ge d(x_0,q_k)+d(x_0,\overline{q}_k)-2\pi=2d(x_0,q_k)-2\pi,
\end{equation*}
we obtain
\begin{equation*}
    d(x_0,\phi_{t_k}(p))\ge C^{-1}d(x_0,q_k).
\end{equation*}
So by \eqref{berry} this implies $d(q_k,\overline{q}_k)\ge C^{-1}d(x_0,q_k)$ and thus $\widetilde{\measuredangle}q_kx_0\overline{q}_k\ge C^{-1}$.
Lastly, by Lemma \ref{l: two rays}, the minimizing geodesics $x_0q_k,x_0\overline{q}_k$ converge to two rays $\sigma_1,\sigma_2$ with $\widetilde{\measuredangle}(\sigma_1,\sigma_2)\ge C^{-1}>0$. So the soliton is asymptotic to a sector.

\end{proof}

\section{Upper and lower curvature estimates}\label{s: curvature etimates}
In this subsection, we prove Theorem \ref{t': curvature estimate} of the two-sided curvature estimates. 
For the lower bound, Theorem \ref{l: R>e^{-2r}} shows that $R$ decays at most exponential fast away from $\Gamma$ by using the improved Harnack inequality in Corollary \ref{l: Harnack}.
For the upper bound, Theorem \ref{t: R upper bd} shows that $R$ decays at least polynomially fast away from $\Gamma$.
Theorem \ref{t: R upper bd} is proved using the quadratic curvature decay from Theorem \ref{l: curvature upper bound initial} and a heat kernel method.

\subsection{Improved integrated Harnack inequality}
In this subsection, we prove an improved integrated Harnack inequality for Ricci flows with non-negative curvature operators. 
This improved Harnack inequality will be used to deduce the exponential curvature lower bound in Theorem \ref{l: R>e^{-2r}}.

First, we state Hamilton's traced differential Harnack inequality and its integrated version.

\begin{theorem}
Let $(M,g(t)),t\in(0,T]$ be an n-dimensional Ricci flow with complete time slices and non-negative curvature operator. Assume furthermore that the curvature is bounded on compact time intervals. Then for any $(x,t)\in M\times(0,T]$ and $v\in T_xM$,
\begin{equation}\label{e: Harnack}
    \pt R(x,t)+\frac{R}{t}+2\langle v,\nabla R\rangle+2\Ric(v,v)\ge 0.
\end{equation}
Moreover, integrating this inequality appropriately yields:
For any $(x_1,t_1),(x_2,t_2)\in M\times (0,T]$ with $t_1<t_2$, we have
\begin{equation}\label{e: integrated version}
      \frac{R(x_2,t_2)}{R(x_1,t_1)}\ge\frac{t_1}{t_2} \exp\left(-\frac{1}{2}\frac{d^2_{g(t_1)}(x_1,x_2)}{t_2-t_1}\right).
\end{equation}
\end{theorem}

\begin{remark}
By the soliton identities it is not hard to see that the equality in the differential Harnack inequality \eqref{e: Harnack} is achieved
if $(M,g(t))$ is the Ricci flow of an expanding gradient Ricci soliton with non-negative curvature operator and $v=\nabla f_t$. Note we adopt the convention that the flow satisfies $\Ric(g(t))+\frac{1}{2t}g(t)=\nabla^2 f_t$, $t>0$, see e.g. \cite[Chapter 10.4]{HaRF}.
\end{remark}

\begin{remark}
In dimension 2, using $\Ric=\frac{1}{2}R\,g$ one can prove the following slightly better integrated Harnack inequality, 
\begin{equation}\label{e: integrated version better}
      \frac{R(x_2,t_2)}{R(x_1,t_1)}\ge\frac{t_1}{t_2} \exp\left(-\frac{1}{4}\frac{d^2_{g(t_1)}(x_1,x_2)}{t_2-t_1}\right).
\end{equation}

\end{remark}

The main result of this subsection shows that \eqref{e: integrated version better} actually holds in all dimensions. Our key observation is the following curvature inequality.

\begin{lem}\label{l: Ric compare to R}
Let $(M,g)$ be an n-dimensional Riemannian manifold with non-negative curvature operator. Then 
\begin{equation}\label{e: Ric}
   \Ric\le \frac{1}{2}\,R\,g.
\end{equation}
\end{lem}

\begin{proof}
To show this, let $p\in M$ and choose an orthonormal basis $\{e_i\}_{i=1}^n$ of $T_pM$ under which the Ricci curvature is diagonal: 
\begin{equation*}
    \Ric=(\lambda_1,...,\lambda_n),
\end{equation*}
where $\lambda_1\ge...\ge\lambda_n$ are the $n$ eigenvalues. 
Let $k_{ij}=\Rm(e_i,e_j,e_j,e_i)$. Then since $\Rm\ge0$, we have
\begin{equation*}
    k_{1i}\le\sum_{j\neq i}k_{ji}=\lambda_i
\end{equation*}
for all $i=2,...,n$. So
\begin{equation*}
    \lambda_1=k_{12}+k_{13}+\dots+k_{1n}\le\lambda_2+\lambda_3+\dots+\lambda_n,
\end{equation*}
hence we have
\begin{equation*}
    \Ric(v,v)\le\lambda_1|v|^2\le\left(\frac{\lambda_1+\dots+\lambda_n}{2}\right)|v|^2=\frac{1}{2}|v|^2R.
\end{equation*}
which proves the lemma.
\end{proof}

Now we prove the improved integrated Harnack inequality.

\begin{theorem}[Improved integrated Harnack inequality]
\label{l: Harnack_with_time}
Let $(M,g(t))$, $t\in[0,T]$, be a Ricci flow with non-negative curvature operator. Then for any $x_1,x_2\in M$ and $0<t_1<t_2\le T$, we have
\begin{equation*}
    \frac{R(x_2,t_2)}{R(x_1,t_1)}\ge\frac{t_1}{t_2} \exp\left(-\frac{1}{4}\frac{d^2_{g(t_1)}(x_1,x_2)}{t_2-t_1}\right).
\end{equation*}
\end{theorem}

\begin{proof}
In the Harnack inequality \eqref{e: Harnack}, let $v=-\nabla(\log R)(x,t)$ and $T_0=0$, we get
\begin{equation}\label{e: Harnack_0}
 R^{-1}\pt R+\frac{1}{t}-2|\nabla\log R|^2+\frac{2\Ric(\nabla\log R,\nabla\log R)}{R}\ge0.   
\end{equation}

Then by Lemma \ref{l: Ric compare to R} we obtain
\begin{equation}\label{e: equality holds on cigar}
    \frac{\partial}{\partial t}\log(tR)\ge|\nabla \log R|^2.
\end{equation}
Let $\mu:[t_1,t_2]\rightarrow M$ be a $g(t_1)$-minimizing geodesic from $x_1$ to $x_2$. Then
\begin{equation*}
    \begin{split}
        \log\left(\frac{t_2R(x_2,t_2)}{t_1R(x_1,t_1)}\right)&=\int_{t_1}^{t_2}\frac{d}{dt}\log(tR(\mu(t),t))\;dt\\
        &=\int_{t_1}^{t_2} \frac{\partial}{\partial t}\log(tR)+\left\langle\nabla\log (tR),\frac{d\mu}{dt}\right\rangle\;dt\\
        &\ge\int_{t_1}^{t_2} |\nabla \log R|^2+\left\langle\nabla\log R,\frac{d\mu}{dt}\right\rangle\;dt\\
        &\ge\int_{t_1}^{t_2}|\nabla\log R|^2-|\nabla\log R|\left|\frac{d\mu}{dt}\right|\;dt\\
        &\ge-\frac{1}{4}\int_{t_1}^{t_2}\left|\frac{d\mu}{dt}\right|^2\;d\mu
        \ge-\frac{1}{4}\frac{d^2_{g(t_1)}(x_1,x_2)}{t_2-t_1}.
    \end{split}
\end{equation*}
\end{proof}

Note that if moreover the Ricci flow $(M,g(t))$ is ancient, then \eqref{e: equality holds on cigar} becomes
\begin{equation}\label{e: equality holds on cigar2}
    \frac{\partial}{\partial t}\log R\ge|\nabla \log R|^2.
\end{equation}
In particular, the following integrated Harnack inequality is a direct consequence of Theorem \ref{l: Harnack_with_time}.

\begin{cor}[Improved Harnack inequality, ancient flow]\label{l: Harnack}
Let $(M,g(t))$, $t\in(-\infty,0]$, be a Ricci flow. Suppose $\Rm_{g(t)}\ge0$ for all $t\in(-\infty,0]$. Then for any $x_1,x_2\in M$ and $t_1<t_2$, we have
\begin{equation*}
    \frac{R(x_2,t_2)}{R(x_1,t_1)}\ge \exp\left(-\frac{1}{4}\frac{d^2_{g(t_1)}(x_1,x_2)}{t_2-t_1}\right).
\end{equation*}
\end{cor}

\begin{remark}
Note that the equality in Lemma \ref{l: Ric compare to R} holds on any 2-dimensional solutions. In the cigar soliton, it is easy to see that the equality in \eqref{e: equality holds on cigar2} is achieved, but the equality is lost in the integrated version.
Nevertheless, the factor $\frac{1}{4}$ in Corollary \ref{l: Harnack} is still sharp in the sense that using it we can obtain a curvature lower bound on cigar soliton which is arbitrarily close to the actual curvature decay in the cigar at infinity: 

Let $(\Sigma,g_c)$ be a cigar soliton and $R(x_{tip})=4$. Let $(\Sigma,g_c(t))$ be the Ricci flow of the soliton. For any $x\in \Sigma$, let  $t=-\frac{d_0(x,x_{tip})}{2}$, then by the distance distortion estimate \eqref{e: integrate Ricci} in $g_c(t)$, we have
\begin{equation*}
    d_t(x,x_{tip})\le d_0(x,x_{tip})+(2-\epsilon)(-t),
\end{equation*}
where $\epsilon>0$ denotes all constants depending on $d_0(x,x_{tip})$, such that $\epsilon\ri0$ as $d_0(x,x_{tip})\rii$.
So applying the improved Harnack inequality we get
\begin{equation*}
    R(x,0)\ge R(x_{tip},t)\,e^{-(2-\epsilon)\,d_0(x,x_{tip})}= 4\,e^{-(2-\epsilon)\,d_0(x,x_{tip})}.
\end{equation*}
This can be compared with the curvature formula of the cigar soliton \eqref{e: cigar R},
\begin{equation*}
    R(x,0)=\frac{16}{(e^{d_0(x,x_{tip})}+e^{-d_0(x,x_{tip})})^2}\le 16 \,e^{-2\,d_0(x,x_{tip})}.
\end{equation*}
\end{remark}

\subsection{Exponential lower bound of the curvature}
In this subsection we use the improved Harnack inequality to deduce the exponential curvature lower bound. 
\begin{theorem}[Scalar curvature exponential lower bound]\label{l: R>e^{-2r}}
Let $(M,g,f,p)$ be a 3D steady gradient soliton that is not a Bryant soliton. Assume $\lim_{s\rii}R(\Gamma_1(s))=\lim_{s\rii}R(\Gamma_2(s))=4$. Then for any $\epsilon_0>0$, there exists $C>0$ such that
\begin{equation}\label{e: conde}
    R(x)\ge C^{-1}e^{-2(1+\epsilon_0)d_g(x,\Gamma)}.
\end{equation}
\end{theorem}

\begin{proof}
For any $\epsilon_0>0$, let $\epsilon>0$ denote all small constants depending on $\epsilon_0$ whose values may change from line to line.
Let $(M,g(t))$, $t\in(-\infty,\infty)$, be the Ricci flow associated to the soliton $(M,g)$, $g(0)=g$. 
Consider the subset $U$ consisting of all points $x$ such that the distance $d_t(x,\Gamma)$ for all $t\le0$ must be achieved at $\epsilon$-tip points on $\Gamma$.
Then it is clear by Lemma \ref{l: inf achieved on the non-compact portion} that the complement of $U$ is compact.
So we can find a constant $C>0$ such that $R\ge C^{-1}$ on $M\setminus U$.
Therefore, it suffices to prove the curvature lower bound \eqref{e: conde} for points $(x,0)\in U\times\{0\}\subset M\times(-\infty,\infty)$.

Let $x\in U$ and $t\le0$.
By a distance distortion estimate we have
\begin{equation}\label{e: seminar}
    -\frac{d}{dt}d_t(x,\Gamma)\le\sup_{\gamma\in\mathcal{Z}(t)}\int_{\gamma}\Ric(\gamma'(s),\gamma'(s))\,ds, 
\end{equation}
where the derivative is the backward difference quotient, and $\mathcal{Z}(t)$ is the space of all minimizing geodesics $\gamma$ which realize the distance $d_t(x,\Gamma)$.
For any such $\gamma$ connecting $x$ to a point $y_t\in\Gamma$, $d_t(x,y_t)=d_t(x,\Gamma)$, we have that $y_t$ is an $\epsilon$-tip point since $x\in U$, and $\gamma$ is orthogonal at $y_t$ to $\Gamma$. Moreover, by the assumption $\lim_{s\rii}R(\Gamma_1(s))=\lim_{s\rii}R(\Gamma_2(s))=4$ we have
\begin{equation}\label{e: oo}
    R(y_t,t)\ge4-\epsilon.
\end{equation}
So by taking $\epsilon$ small, \eqref{e: seminar} implies
\begin{equation*}
    -\frac{d}{dt}d_t(x,\Gamma)\le2(1+\epsilon_0),
\end{equation*}
integrating which we get
\begin{equation}\label{e: tt}
    d_t(x,y_t)=d_t(x,\Gamma)\le d_0(x,\Gamma)+2(1+\epsilon_0)(-t).
\end{equation}

Now applying Corollary \ref{l: Harnack} (Improved Harnack inequality) and using \eqref{e: oo},\eqref{e: tt} we obtain
\begin{equation*}
    \begin{split}
        R(x,0)\ge R(y_t,t)e^{-\frac{d^2_{t}(x,y_t)}{4(-t)}}
        \ge (4-\epsilon) e^{-\frac{(d_0(x,\Gamma)+2(1+\epsilon_0) (-t))^2}{4(-t)}}.
    \end{split}
\end{equation*}
Letting $t=-\frac{d_0(x,\Gamma)}{2}$, this implies
\begin{equation*}
    R(x,0)\ge  (4-\epsilon)e^{-2(1+\epsilon_0)d_0(x,\Gamma)}.
\end{equation*}
\end{proof}

\begin{remark}
Note that this curvature estimate is
sharp: In the manifold $\R\times\cigar$, the curvature decays like $O(e^{-2\,d_g(\cdot,\Gamma)})$, so our lower bound estimate $O(e^{-(2+\epsilon_0)\,d_g(\cdot,\Gamma)})$ gets arbitrarily close to it as the distance $d_g(\cdot,\Gamma)$ goes to infinity.
\end{remark}

\subsection{Polynomial upper bound of the curvature}

In Theorem \ref{t: R upper bd} 
we show that the quadratic curvature decay from Theorem \ref{l: curvature upper bound initial} can be improved to polynomial decay at any rate.  
The proof relies on the following heat kernel estimate. This estimate shows that the heat kernel starting from $(x,t)$ behaves like a Gaussian, and it is centered at the $(x,s)$ for all $s< t-2$.
For $s\in[t-2,t)$, the Gaussian bound also holds by Lemma \ref{l: heat kernel lower bound implies upper bound}.

\begin{lem}\label{l: L-geodesic}
Let $(M,g,f,p)$ be a 3D steady gradient soliton that is not a Bryant soliton and $(M,g(t))$ be the Ricci flow of the soliton. Let $G(x,t;y,s)$, $x,y\in M$, $s< t-2$, be the heat kernel of the heat equation $\pt u=\Delta u$ under $g(t)$.
Then there exists $C>0$ such that
\begin{equation*}
    G(x,t;y,s)\le C\,(t-s)^\frac{3}{2}\, \textnormal{exp}\left(-\frac{d^2_s(x,y)}{4C(t-s)}\right).
\end{equation*}
\end{lem}

\begin{proof}
After a rescaling we assume $\lim_{s\rii}R(\Gamma_1(s))=\lim_{s\rii}R(\Gamma_2(s))=4$. We shall use $C$ to denote all constants depending only on the soliton $(M,g)$.
Without loss of generality, we may assume $s=0$. 
For any $s\in[0,1]$ and $z\in B_s(x,1)$, let $\gamma:[s,t]\ri M$ be a curve such that $\gamma|_{[s,2]}$ is a minimizing geodesic connecting $x$ and $z$ with respect to $g(0)$, and $\gamma|_{[2,t]}\equiv y$.

For any $\tau\in[0,t-s]$, by Theorem \ref{l: curvature upper bound initial} we have
$R(y,t-\tau)\le \frac{C}{r^2(y,t-\tau)}$.
Moreover, denote $d_t(x,\Gamma)$ by $r(x,t)$, then by Theorem \ref{l: wing-like} and distance distortion estimates, we have
\begin{equation}\label{e: distance change}
    r(x,t)+1.9(t-s)\le r(x,s)\le r(x,t)+2.1\,(t-s).
\end{equation}
So $R(y,t-\tau)\le \frac{C}{(r(x,t)+C^{-1}\tau)^2}$.
For $\tau\in[t-2,t-s]$,
we have $R(\gamma(t-\tau),t-\tau)\le C$ and $|\gamma'|(t-\tau)\le C$.
Putting these together we can estimate the $\LL$-length of $\gamma$,
\begin{equation*}
    \begin{split}
        \mathcal{L}(\gamma)&=\int_{0}^{t-2}\sqrt{\tau}R(y,\tau)\,d\tau+
        \int_{t-2}^{t-s}\sqrt{\tau}(R(\gamma(t-\tau),t-\tau)+|\gamma'|^2)\,d\tau\\
        &\le \int_{0}^{t-2}\sqrt{\tau}\frac{C}{(r(x,t)+C^{-1}\tau)^2}\,d\tau+\int_{t-2}^{t-s}C\sqrt{\tau}\,d\tau\le C\sqrt{t}.
    \end{split}
\end{equation*}
Let $\ell(z,s):=\ell_{(x,t)}(z,s)$ be the reduced length from $(x,t)$ to $(z,s)$, then 
\begin{equation*}
    \ell(z,s)=\frac{\mathcal{L}_{(x,t)}(z,s)}{2\sqrt{t-s}}\le \frac{\mathcal{L}(\gamma)}{2\sqrt{t-s}}\le C.
\end{equation*}
Recall the heat kernel lower bound by Perelman in
\cite[Corollary 9.5]{Pel1}
we get:
\begin{equation}\label{reduced length}
    G(x,t;z,s)\ge \frac{1}{4\pi(t-s)^{3/2}}e^{-\ell(z,s)}\ge\frac{C}{t^{\frac{3}{2}}},
\end{equation}
for all $s\in[0,1]$ and $z\in B_s(x,1)$, integrating which in $B_s(x,1)$ we get
\begin{equation}\label{low}
    \int_{B_s(x,1)}G(x,t;z,s)\,d_sz\ge \frac{C}{t^{3/2}},
\end{equation}
for all $s\in [0,1]$.

Let $y\in M$, then by the multiplication inequality for the heat kernel in \cite[Theorem 1.30]{HN} we have 
\begin{equation}\label{e: hein-naber}
     \left(\int_{B_s(x,1)}G(x,t;z,s)\,d_sz\right)\left(\int_{B_s(y,1)}G(x,t;z,s)\,d_sz\right)\le C\, \textnormal{exp}\left(-\frac{(d_s(x,y)-2)^2}{4C(t-s)}\right).
\end{equation}
So by substituting \eqref{low} into \eqref{e: hein-naber} and using the distance distortion estimate $d_s(x,y)-2\ge C^{-1}d_0(x,y)-2\ge (2C)^{-1}d_0(x,y)$,
we obtain
\begin{equation*}
    \left(\int_{B_s(y,1)}G(x,t;z,s)\,d_sz\right)\le C\,t^\frac{3}{2}\, \textnormal{exp}\left(-\frac{d_0(x,y)^2}{4C(t-s)}\right).
\end{equation*}
Integrating this for all $s\in[0,1]$, and then applying the parabolic mean value inequality (see e.g. \cite{RFTandA3}) to $G(x,t;\cdot,\cdot)$ at $(y,0)$, we obtain 
\begin{equation*}
    G(x,t;y,0)\le C\,t^\frac{3}{2}\, \textnormal{exp}\left(-\frac{d^2_0(x,y)}{4Ct}\right).
\end{equation*}

\end{proof}

\begin{theorem}[Scalar curvature polynomial upper bound]\label{t: R upper bd}
Let $(M,g,f)$ be a 3D steady gradient soliton that is not a Bryant soliton. Then for any integer $k\ge2$, there exists $C_k>0$ such that 
\begin{equation*}
    R\le \frac{C_k}{d_g^k(\cdot,\Gamma)}.
\end{equation*}
\end{theorem}

\begin{proof}
By Theorem \ref{l: curvature upper bound initial} this is true for $k=2$. Let $(M,g(t))$ be the Ricci flow of the soliton. We denote $d_t(x,\Gamma)$ by $r(x,t)$. 
After a rescaling we assume $\lim_{s\rii}R(\Gamma_1(s))=\lim_{s\rii}R(\Gamma_2(s))=4$, so \eqref{e: distance change} holds.

Suppose by induction that this is true for $k\ge2$, we will show that this is also true for $k+1$. In the following $C$ denotes all positive constants that depend on $k$, the maximum of $R$ and the limits of $R$ at the two ends of $\Gamma$.
Since $R$ satisfies the evolution equation
\begin{equation*}
    \partial_tR=\Delta R+2|\Ric|^2,
\end{equation*}
for a fixed pair $(x,t)\in M\times(-\infty,\infty)$ we have
\begin{equation*}\begin{split}
    R(x,t)&=\int_M G(x,t;y,s)R(y,s)\,d_sy+2\int_s^t\int_M G(x,t;z,s)|\Ric|^2(z,\tau)\,d_{\tau}z\,d\tau\\
    &:=I(s)+II(s).
\end{split}\end{equation*}

First, we claim that $\lim_{s\ri-\infty}I(s)=0$.
To show this, we split $I(s)$ into two integrals on $B_s(x,\frac{r(x,s)}{1000})$ and $M\setminus B_s(x,\frac{r(x,s)}{1000})$, and denote them respectively by $I_1(s)$ and $I_2(s)$.
Then for $I(s)$, using $\int_M G(x,t;y,s)\,d_sy=1$ we can estimate that
\begin{equation*}
    I_1(s)\le\left(\int_M G(x,t;y,s)\,d_sy\right)\cdot\left(\sup_{B_{s}(x,\frac{r(x,s)}{1000})}R(\cdot,s)\right)=\sup_{B_{s}(x,\frac{r(x,s)}{1000})}R(\cdot,s).
\end{equation*}
For any $y\in B_{s}(x,\frac{r(x,s)}{1000})$, we have $r(y,s)\ge r(x,s)-\frac{r(x,s)}{1000}\ge\frac{r(x,s)}{2}$. So $R(y,s)\le\frac{C}{r^k(x,s)}$ by the inductive assumption. 
So it follows by \eqref{e: distance change} that $I_1(s)\le\frac{C}{r^2(x,s)}$ which goes to zero as $s\ri-\infty$.

For $I_2(s)$, since by \eqref{e: distance change} we have
\begin{equation*}
    \frac{d_s^2(y,x)}{t-s}\ge C^{-1}r(x,s)\ge C^{-1}(t-s),
\end{equation*}
for all $y\in M\setminus B_s(x,\frac{r(x,s)}{1000})$, it follows by the heat kernel estimates Lemma \ref{l: L-geodesic} and Lemma \ref{l: heat kernel lower bound implies upper bound} that
\begin{equation*}
    G(x,t;y,s)\le\,C\,e^{-\frac{d_s^2(y,x)}{C(t-s)}},
\end{equation*}
which implies
\begin{equation*}
    I_2(s)\le\,C\,\int_{M\setminus B_s(x,\frac{r(x,s)}{1000})}e^{-\frac{d_s^2(y,x)}{C(t-s)}}\,d_sy\le Ce^{-\frac{r(x,s)}{C}}\ri0, \quad\textit{as }s\ri-\infty.
\end{equation*}

Next, we estimate $II(s)$. Let 
\begin{equation*}
    J(\tau):=\int_M G(x,t;z,\tau)|\Ric|^2(z,\tau)\,d_{\tau}z,
\end{equation*}
and split it into two integrals on $B_{\tau}(x,\frac{r(x,\tau)}{1000})$ and $M\setminus B_{\tau}(x,\frac{r(x,\tau)}{1000})$, and denote them respectively by $J_1(\tau)$ and $J_2(\tau)$.
Then by a similar argument as above and using the inductive assumption, we see that $J_1(\tau)\le\frac{C}{r^{2k}(x,\tau)}$ and $J_2(\tau)\le C\,e^{-\frac{r(x,\tau)}{C}}\le \frac{C}{r^{2k}(x,\tau)}$.

Therefore, integrating $J(s)$ we obtain
\begin{equation*}\begin{split}
    II(s)&=2\int_s^tJ(\tau)\,d\tau=2\int_s^t(J_1(\tau)+J_2(\tau))\,d\tau
    \le\int_s^t\frac{C}{r^{2k}(x,\tau)}\,d\tau\\
    &\le\int_s^t\frac{C}{(r(x,t)+1.9\,(t-\tau))^{2k}}\,d\tau\\
    &=C\left(\frac{1}{r^{2k-1}(x,t)}-\frac{1}{(r(x,t)+1.9\,(t-s))^{2k-1}}\right)
    \le\frac{C}{r^{2k-1}(x,t)}.
\end{split}\end{equation*}
Combining this with the estimate on $I(s)$, it follows that
\begin{equation*}
    R(x,t)\le\limsup_{s\ri-\infty}(I(s)+II(s))\le\frac{C}{r^{2k-1}(x,t)}.
\end{equation*}
Since $k\ge2$, we have $2k-1\ge k+1$, which proves the theorem by induction.

\end{proof}

\section{Symmetry improvement theorems}\label{s: semi-local}
In this section we will study the Ricci-Deturck perturbations $h$ whose background metric is a $SO(2)$-symmetric complete Ricci flow which is sufficiently close to the cylindrical plane $\RR\times S^1$.
Such symmetric 2-tensor can be decomposed as $h=h_++h_-$, where $h_+$ is the rotationally invariant mode and $h_-$ is the oscillatory mode.
We show that the oscillatory mode $h_-$ decays in time exponentially in a certain sense.
We will first prove the linear version of this symmetry improvement theorem, that is, the oscillatory mode of a linearized Ricci-Deturck flow on $\RR\times S^1$ decays exponentially in time.
Then we can obtain the theorem from its linear version by using a limiting argument.

More explicitly, $|h_-|$ decays exponentially in time in the following sense: First, if $|h_-|$ is initially bounded uniformly by a constant, then the theorem shows that it decays as $e^{-\delta_0t}$ for some $\delta_0>0$.
Moreover, if $|h_-|(\cdot,0)$ has an exponential growth in the space direction, then the theorem shows that $|h_-|(\cdot,t)$ still decays as $e^{-\delta_0t}$ modulo the same exponential growth rate in the space direction.

\subsection{SO(2)-decomposition of a symmetric 2-tensor}
For a 3D Riemannian manifold $(M,g)$,
we say it is $SO(2)$-symmetric if it admits 
an effective isometric $SO(2)$-action.
Equivalently, this means that there is a one parameter group of isometries $\psi_{\theta}$, $\theta\in\R$, such that $\psi_{\theta}=i.d.$ if and only if $\theta=2k\pi$, $k\in\mathbb{Z}$.
Throughout this section, we will moreover assume that $(M,g)$ is a 3D $SO(2)$-symmetric Riemannian manifold such that there exist a 2D Riemannian manifold $(N,g_0)$ and a Riemannian submersion $\pi:(M,g)\ri(N,g_0)$ which maps an orbit of the $SO(2)$-action to a point in $N$. This can be ensured when the $SO(2)$-action is free.

Let $U$ be a local coordinate chart on $N$ with coordinates $\rho:(x,y)\in U_0\subset\RR\ri\rho(x,y)\in U$.   
Take a section $s: U\ri\pi^{-1}(U)\subset M$.
Parametrize $SO(2)$ by $\theta:[0,2\pi)\ri SO(2)$.
Then we obtain a local coordinate on $\pi^{-1}(U)$ by $(x,y,\theta)\ri \theta\cdot s(\rho(x,y))$.

Let $h$ be a symmetric 2-tensor on $M$, and 
\begin{equation}\label{e: h_+}
    h_+(y):=\frac{1}{2\pi}\int_0^{2\pi}(\theta^*h)(y)\,d\theta,
\end{equation}
and $h_-:=h-h_+$. 
Then $h_+,h_-$ are two symmetric 2-tensors.
For any $\theta_0\in SO(2)$, we have
\begin{equation*}
    \theta_0^*h_+=\frac{1}{2\pi}\int_0^{2\pi}\theta_0^*(\theta^*h)\,d\theta=\frac{1}{2\pi}\int_0^{2\pi}(\theta_0+\theta)^*h\,d\theta=\frac{1}{2\pi}\int_{\theta_0}^{2\pi+\theta_0}\theta^*h\,d\theta=h_+.
\end{equation*}
So we say $h_+$ is the rotationally invariant part and $h_-$ is the oscillatory part of $h$, and $h=h_++h_-$ the $SO(2)$-decomposition of $h$. 
Similarly, We say $h$ is rotationally invariant if $h=h_+$, and oscillatory if $h_+=0$.

We now analyze the structure of the oscillatory mode more carefully.
Since the one forms $dx,dy,d\theta$ are invariant under the $SO(2)$-action, it follows that the basis  $\{dx^2,dy^2,dxdy,dxd\theta,dyd\theta,d\theta^2\}$ of the space of all symmetric 2-tensors are rotationally invariant. So the $SO(2)$-decomposition of $h$ reduces to that the decomposition of components under this basis: $h$ can be written as below under the local coordinates,
\begin{equation*}
    h=F_{1}dx^2+F_{2}dy^2+F_3dxdy+F_4dxd\theta+F_5dyd\theta+F_6d\theta^2,
\end{equation*}
where $F_i(x,y,\theta):U_0\times S^1\ri\R$ are functions.
Let $F_{i,\pm}$ be the i-th component in $h_{\pm}$.
Then 
\begin{equation*}
    F_{i,+}(x,y,\theta)=\frac{1}{2\pi}\int_0^{2\pi}F_{i}(x,y,\theta')\,d\theta',
\end{equation*}
which is independent of $\theta$, and
\begin{equation*}
    F_{i,-}(x,y,\theta)=\sum_{j=1}^{\infty}A_{i,j}(x,y)\cos(j\theta)+B_{i,j}(x,y)\sin(j\theta),
\end{equation*}
where 
\begin{equation*}\begin{split}
    A_{i,j}(x,y)&=\frac{1}{\pi}\int_0^{2\pi}F_i(x,y,\theta')\cos(j\theta')\,d\theta',\\ B_{i,j}(x,y)&=\frac{1}{\pi}\int_0^{2\pi}F_i(x,y,\theta')\sin(j\theta')\,d\theta'.
\end{split}
\end{equation*}

We have the following observations:
Suppose $\{M_i,g_i,x_i\}$ is a sequence of $SO(2)$-symmetric Riemannian manifold, which smoothly converges to a $SO(2)$-symmetric Riemannian manifold $(M_{\infty},g_{\infty},x_{\infty})$, and the convergence if $SO(2)$-equivariant.
Suppose also that $h_i$ is a sequence of symmetric 2-tensors on $M_i$ that smoothly converges to a symmetric 2-tensor on $M_{\infty}$.
Write $h_i=h_{i,+}+h_{i,-}$ and $h_{\infty}=h_{\infty,+}+h_{\infty,-}$ for the $SO(2)$-decomposition. Then $h_{i,+}$ smoothly converges to $h_{\infty,+}$, and $h_{i,-}$ smoothly converges to $h_{\infty,-}$.

\subsection{A symmetry improvement theorem in the linear case}

Note that straightforward computation shows that the decomposition $h=h_++h_-$ is compatible with the linearized Ricci Deturck flow $\partial_th=\Delta_Lh$ on the cylindrical plane $\RR\times S^1$.
In the following, we consider an oscillatory symmetric 2-tensor $h$ on $\R^2\times S^1$ which solves the linearized Ricci DeTurck equation. Assume $|h|(\cdot,0)$ satisfies an exponential growth bound,    then the following proposition shows that $|h|$ decays exponentially in time in a certain sense.

\begin{prop}(On $\RR\times S^1$, linear)\label{l: h on R2xS1}
Let $\delta_0=0.01$. There exists $T_0>0$ such that the following holds for all
\begin{equation*}
    \alpha\in[0,2.02],\quad T\ge T_0.
\end{equation*}
Let $h(\cdot,t)$, $t\in[0,T]$, be a continuous family of oscillatory tensors on $\RR\times S^1$,  which is smooth on $t\in(0,T]$ and satisfies the linearized Ricci deturck flow $\partial_t h=\Delta_L h$. 
Suppose we have
\begin{equation}\label{e: bound on |h(0)|}
    |h(x,y,\theta,0)|\le A\,e^{\alpha\sqrt{x^2+y^2}}
\end{equation}
for any $(x,y,\theta)\in\RR\times S^1$. Then
\begin{equation*}
    |h|(0,0,\theta,T)\le A\,e^{2\alpha T}\cdot e^{-\delta_0 T}.
\end{equation*}
\end{prop}

Note that at $(0,0,\theta)$, the upper bound $|h|\le A$ at time $0$ becomes $|h|\le A\,e^{2\alpha T}\cdot e^{-\delta_0T}$ at time $T$.
For $\alpha=0$, the bound at time $T$ is $|h|\le A\cdot e^{-\delta_0T}$, in which case the exponential decay is clear. 
For $\alpha\neq0$, there is an extra increasing factor $e^{2\alpha T}$ which seems to cancel out the effect of the decreasing factor $e^{-\delta_0T}$. In this case, the exponential decay rate is measured by the time-dependent distance to the `base point' in a suitable Ricci flow. So the increasing factor $e^{2\alpha T}$ will be compensated by the distance shrinking as going forward along the flow.

In Section \ref{s: Approximating $SO(2)$-symmetric metrics}, we will apply the non-linear version of this proposition on the 3D flying wing with $\lim_{s\rii}R(\Gamma_i(s))=4$, $i=1,2$.
We will consider $h$ satisfying the initial bound $|h|\le e^{\alpha\,d_g(\cdot,\Gamma)}$. Since the soliton converges to $\R\times\cigar$ along $\Gamma$, it follows the distance to $\Gamma$ shrinks at a speed arbitrarily close to $2$. This will outweigh the increasing caused by $e^{2\alpha T}$.
It is crucial that $\alpha$ can be slightly greater than $2$ since we will rely on this to find a $SO(2)$-symmetric metric sufficiently close to the soliton metric so that the error decays like $e^{-(2+\delta)\,d_g(\cdot,\Gamma)}$ for some small but positive $\delta$.
So the error can decay faster than the scalar curvature as a consequence of Theorem \ref{t': curvature estimate}.

\begin{proof}
Since $h$ is oscillatory, we can write it as 
\begin{equation*}\begin{split}
    h(x,y,\theta,t)=F_1&(x,y,\theta,t)dx^2+F_2(x,y,\theta,t)dxdy+F_3(x,y,\theta,t)dy^2\\
    &+F_4(x,y,\theta,t)dxd\theta+F_5(x,y,\theta,t)dyd\theta +F_6(x,y,\theta,t)d\theta^2,
\end{split}\end{equation*}
where $F_i$ are in the following form
\begin{equation*}\begin{split}\label{e: F_i}
    F_i(x,y,\theta,t)&=\sum_{j=1}^{\infty}A_{i,j}(x,y,t)\cos (j\theta) + B_{i,j}(x,y,t)\sin (j\theta),\\
    A_{i,j}(x,y,t)&=\frac{1}{\pi}\int_0^{2\pi}F_i(x,y,\theta',t)\cos(j\theta')\,d\theta'\\
    B_{i,j}(x,y,t)&=\frac{1}{\pi}\int_0^{2\pi}F_i(x,y,\theta',t)\sin(j\theta')\,d\theta'.
\end{split}\end{equation*}
So by the assumption \eqref{e: bound on |h(0)|} we have $|F_i(x,y,\theta,0)|\le A\,e^{\alpha\sqrt{x^2+y^2}}$,
and hence
\begin{equation}\label{e: A and B upper bound}
    |A_{i,j}|(x,y,0),|B_{i,j}|(x,y,0)\le 2A\,e^{\alpha\sqrt{x^2+y^2}}.
\end{equation}

Since $h$ satisfies $\partial_th=\Delta_L h$, which in the coordinate $(x,y,\theta)$ is equivalent to
\begin{equation}\label{e: solve}
     \partial_t F_i(x,y,\theta,t)=(\partial_{xx}+\partial_{yy}+\partial_{\theta\theta})F_i(x,y,\theta,t)=(\Delta_{\R^2}+\partial_{\theta\theta})F_i(x,y,\theta,t).
\end{equation}
Solving \eqref{e: solve} term by term we see 
\begin{equation*}
    \begin{split}
       \partial_t A_{i,j}= \Delta_{\R^2} A_{i,j}-j^2\,A_{i,j},\quad \partial_t B_{i,j}= \Delta_{\R^2} B_{i,j}-j^2\,B_{i,j}.
    \end{split}
\end{equation*}
So $A_{i,j}(x,y,t)\cdot e^{j^2(t-T)}$ and $B_{i,j}(x,y,t)\cdot e^{j^2(t-T)}$ satisfy the heat equation on $\R^2$.
In the following, we will estimate these terms from above at $(0,0,T)$.

For convenience, we will omit the indices for a moment and let
\begin{equation}\label{e: firstapril}
    u(x,y,t)=A_{i,j}(x,y,t)\cdot e^{j^2(t-T)}.
\end{equation}
Then $u$ satisfies the heat equation
\begin{equation*}
    \partial_t u=\Delta_{\R^2} u,
\end{equation*}
and by \eqref{e: A and B upper bound} we have
\begin{equation}\label{e: april}
    |u|(x,y,0)\le 2A\,e^{\alpha\sqrt{x^2+y^2}}\cdot e^{-j^2T}.
\end{equation}

Since $\cos 0.4\ge \frac{1}{1.1}$, it follows that for any $(x,y)\in\RR$, there is
\begin{equation*}
    \sqrt{x^2+y^2}\le 1.1\,(x\cos\theta+y\sin\theta),
\end{equation*}
for all $\theta\in[\theta_0-0.4,\theta_0+0.4]$, where $\theta_0$ satisfies $\cos\theta_0=\frac{x}{\sqrt{x^2+y^2}}$ and $\sin\theta_0=\frac{y}{\sqrt{x^2+y^2}}$. So for any $\alpha$ we have that
\begin{equation}\label{e: exp}
    e^{\alpha\sqrt{x^2+y^2}}\le \frac{1}{0.8} \int_0^{2\pi}e^{1.1\,\alpha(x\cos\theta+y\sin\theta)} \;d\theta.
\end{equation}
Let
\begin{equation*}
    v(x,y,t)=2A\,e^{-j^2T}\cdot e^{(1.1\alpha)^2t}\cdot\frac{1}{0.8} \int_0^{2\pi}e^{(1.1\alpha)(x\cos\theta+y\sin\theta)} \;d\theta.
\end{equation*}
For any fixed $\theta$, by a straightforward computation we see that the function $e^{(1.1\alpha)^2t}\cdot e^{(1.1\alpha)(x\cos\theta+y\sin\theta)}$ is a solution to the heat equation on $\R^2$. So it follows that $v$ also satisfies the heat equation, i.e.
\begin{equation*}
    \partial_t v=\Delta_{\R^2} v.
\end{equation*}
Note by \eqref{e: april} and \eqref{e: exp} we have $|u|(x,y,0)\le v(x,y,0)$.
Moreover, by Lemma \ref{l: heat equation} we have a linear exponential growth bound on $|u|(x,y,t)$ for all later times $t\in[0,T]$ which may depend on $T$. This allows us to 
use the maximum principle (see e.g. \cite{Evans2010PartialDE}) and deduce that 
\begin{equation}\label{e: h_i}
    |u|(0,0,T)\le v(0,0,T)=2.5 \,A\,e^{-j^2T} \cdot e^{(1.1\alpha)^2 T}.
\end{equation}

Since $\alpha\in[0,2.02]$, it is easy to check that the following holds
\begin{equation*}
    2\alpha-(1.1\alpha)^2+1\ge0.1.
\end{equation*}
Take $T_0=\frac{\ln 2.5}{0.05}$, then $e^{0.05\,T}>  2.5 $ for all $T\ge T_0$, and hence
\begin{equation*}
      2.5\,  e^{-T}\cdot e^{(1.1\alpha)^2 T}\le e^{2\alpha T}\cdot e^{-0.05\, T}.
\end{equation*}
Substituting this into \eqref{e: h_i} we obtain
\begin{equation*}
    |u|(0,0,T)\le A\,e^{-(j^2-1)T}\cdot e^{2\alpha T}\cdot e^{-0.05\, T}.
\end{equation*}
Restoring the indices in \eqref{e: firstapril}, we obtain
\begin{equation*}
    |A_{i,j}|(0,0,T)\le A\,e^{-(j^2-1)T}\cdot e^{2\alpha T}\cdot e^{-0.05\, T}.
\end{equation*}
Similarly, we can show that $|B_{i,j}|(0,0,T)$ satisfies the same inequality. 
Therefore, assuming $T_0\ge\frac{\ln 400}{0.04}$, we obtain
\begin{equation*}\begin{split}
    |F_i|(0,0,T)&\le \sum_{j=1}^{\infty}|A_{i,j}|(0,0,T)+\sum_{j=1}^{\infty}|B_{i,j}|(0,0,T)\\
    &\le 2A\,e^{2\alpha T}\cdot e^{-0.05\, T}\cdot\sum_{j=1}^{\infty}e^{-(j^2-1)T}\\
    &\le 4A\,e^{2\alpha T}\cdot e^{-0.05\, T}
    \le \frac{1}{100}A\,e^{2\alpha T}\cdot e^{-0.01\, T},
\end{split}\end{equation*}
which impiles $|h|(0,0,T)\le A\,e^{2\alpha T}\cdot e^{-0.01\, T}$, and hence proves the lemma.
\end{proof}

\subsection{A symmetry improvement theorem in the nonlinear case}

In Theorem \ref{t: symmetry improvement}, we prove the non-linear version of Proposition \ref{l: h on R2xS1}. 
In the theorem, $h$ is a symmetric 2-tensor satisfying the Ricci Deturck flow perturbation with background metric $g(t)$ being a $SO(2)$-symmetric complete Ricci flow which is sufficiently close to $\RR\times S^1$ at a base point. 
We will show that the oscillatory part of $h$ has a similar exponential decay in time as in Proposition \ref{l: h on R2xS1}.

We briefly recall some facts of Ricci Deturck flow perturbations from \cite[Appendix A]{bamler2022uniqueness}. Let $(M,g(t))$ be a complete Ricci flow and $h$ be a solution to the following Ricci Deturck flow perturbation equation with background metric $g(t)$. 
\begin{equation*}
    \nabla_{\partial_t}h=\Delta_{g(t)}h+2\,\Rm_{g(t)}(h)+\mathcal{Q}_{g(t)}[h],
\end{equation*}
where $\mathcal{Q}_{g(t)}[h]$ is quadratic in $h$ and its spacial derivatives, and the left-hand side contains the conventional Uhlenbeck trick:
\begin{equation*}
    (\nabla_{\partial_t}h)_{ij}=(\partial_th)_{ij}+g^{pq}(h_{pj}\Ric_{qi}+h_{ip}\Ric_{qj}).
\end{equation*}
Then $\widetilde{h}:=\alpha^{-1}h$ satifies the rescaled Ricci Deturck flow perturbation equation
\begin{equation*}
    \nabla_{\partial_t}\widetilde{h}=\Delta_{g(t)}\widetilde{h}+2\,\Rm_{g(t)}(\widetilde{h})+\mathcal{Q}^{(\alpha)}_{g(t)}[\widetilde{h}],
\end{equation*}
which converges to the linearized Ricci-Deturck equation as $\alpha\ri0$,
\begin{equation*}
    \nabla_{\partial_t}\widetilde{h}=\Delta_{g(t)}\widetilde{h}+2\,\Rm_{g(t)}(\widetilde{h}),
\end{equation*}
which can also be written as $\partial_t\widetilde{h}=\Delta_{L}\widetilde{h}$, where $\Delta_{L}$ is the lichnerowicz laplacian
\begin{equation*}
    \Delta_{L}h_{ij}=\Delta h_{ij}+2\,g^{kp}g^{\ell q}R_{kij\ell}h_{pq}-g^{pq}(h_{pj}\Ric_{qi}+h_{ip}\Ric_{qj}).
\end{equation*}

Theorem \ref{t: symmetry improvement} is proved using a limiting argument: We consider a sequence of blow-ups of solutions $h_i$ to the Ricci Deturck flow perturbation, and show that they converge to a solution to the linearized Ricci Deturck flow to which we can apply Proposition \ref{l: h on R2xS1}.
To take the limit, we need to derive uniform bounds for $h_i$ and the derivatives. 

To this end, we first observe that for a solution $h$ to the Ricci Deturck flow perturbation, $|h|^2$ satisfies the following evolution inequality, see \cite[Appendix A.1]{bamler2022uniqueness},
\begin{equation*}\begin{split}\label{e: original}
    \partial_t |h|^2\le (g+h)^{ij}\nabla^2_{ij}|h|^2-2(g+h)^{ij}g^{pq}g^{uv}\nabla_ih_{pu}\nabla_jh_{qv}\\
    +C(n)|\Rm_g|\cdot|h|^2+C(n)|h|\cdot|\nabla h|^2,
\end{split}\end{equation*}
where $C(n)>0$ is some dimensional constant.
Note that the elliptic operator $(g+h)^{ij}\nabla^2_{ij}$ is not exactly a laplacian of metrics.
So in order to use the standard heat kernel estimates, we compare this operator with the exact laplacian $\Delta_{g(t)+h(t)}$ in the following lemma and show that $\partial_t |h|^2\le \Delta_{g(t)+h(t)}|h|^2+|h|^2$.

\begin{lem}\label{l: laplacian}
For any $n\in\mathbb{N}$, there are constants $C_0(n),C_1(n)$ such that the following holds:
Let $(M^n,g(t))$, $t\in[0,T]$, be a Ricci flow (not necessarily complete) with $|\Rm|\le 1$ and $\textnormal{inj}\ge 1$. and $h(t)$ be a Ricci flow perturbation with background $g(t)$. Suppose $|\nabla^kh|\le\frac{1}{C_0(n)}<\frac{1}{100}$, then $|h|^2$ satisfies the following evolution inequality
\begin{equation}\label{e: garden}
    \partial_t |h|^2\le \Delta_{g(t)+h(t)}|h|^2+C_1(n)\,|h|^2.
\end{equation}
\end{lem}

\begin{proof}
In the following the covariant derivatives and curvature quantities are taken with respect to $g(t)$ and the time-index $t$ in  $g(t),h(t)$ is suppressed.
Let $C$ denote all dimensional constants whose values may change from line to line.
Let $(x^1,x^2,...,x^n)$ be local coordinates on an open subset $U\subset M$, such that $|\nabla^kg|\le C$, $k=0,1$, for example we may choose the distance coordinates, \cite[Theorem 74]{petersen}.
By the formula of the Hessian $\nabla^2_{ij} f=\partial^2_{ij}f-\Gamma^k_{ij}\partial_kf$ for any smooth function $f$ on $U$, it is easy to see
\begin{equation*}\begin{split}
    (g+h)^{ij}\nabla^2_{ij}|h|^2&=(g+h)^{ij}\partial^2_{ij}|h|^2-(g+h)^{ij}\Gamma^k_{ij}\partial_k|h|^2,\\
        \Delta_{g+h}|h|^2&=(g+h)^{ij}\partial^2_{ij}|h|^2-(g+h)^{ij}\widetilde{\Gamma}^k_{ij}\partial_k|h|^2,
\end{split}\end{equation*}
where $\widetilde{\Gamma}^k_{ij}$ is the Christoffel symbol of $g+h$.
So we have 
\begin{equation*}\begin{split}\label{e: lemon}
    (g+h)^{ij}\nabla^2_{ij}|h|^2-\Delta_{g+h}|h|^2&=(g+h)^{ij}(\widetilde{\Gamma}^k_{ij}-\Gamma^k_{ij})\partial_k|h|^2.
\end{split}\end{equation*}
Seeing that $|\widetilde{\Gamma}^k_{ij}-\Gamma^k_{ij}|\le C(|h|+|\nabla h|)$ and using the assumption $|\nabla^kh|\le\frac{1}{C_0(n)}$, $k=0,1$, for sufficiently large $C_0(n)$ we obtain
\begin{equation*}\begin{split}
    |(g+h)^{ij}\nabla^2_{ij}|h|^2-\Delta_{g+h}|h|^2|&\le C\,|h|^2\,|\nabla h|+C\,|h|\,|\nabla h|^2
    \le \frac{1}{2}\,|h|^2+\frac{1}{2}|\nabla h|^2.
\end{split}\end{equation*}
Combining this with \eqref{e: original}, and note that 
\begin{equation*}
    2(g+h)^{ij}g^{pq}g^{uv}\nabla_ih_{pu}\nabla_jh_{qv}\ge 1.8\,|\nabla h|^2,
\end{equation*}
we obtain \eqref{e: garden}.

\end{proof}

Now we prove the main result of this section.

\begin{theorem}[On almost cylindrical part, non-linear]\label{t: symmetry improvement}
There exist $\delta_0,T_0>0$ such that for any $T\ge T_0$, there exists $\overline{\epsilon}(T),\overline{\delta}(T),\underline{D}(T)>0$ such that for any
\begin{equation*}
  \alpha\in[0,2.02],\quad \epsilon<\overline{\epsilon},\quad\delta<\overline{\delta},\quad D_{\#}>\underline{D}.
\end{equation*}
the following holds:

Let $(M,g(t),x_0)$, $t\in[0,T]$, be a $SO(2)$-symmetric complete Ricci flow with $|\Rm|_{g(t)}\le 1$, and $(M,g,x_0)$ is $\delta$-close to $(\RR\times S^1,g_{stan})$ in the $C^{98}$-norm. Let $h(t)$ be a Ricci Deturck flow perturbation with background metric $g(t)$ on $B_0(x_0,D_{\#})\times[0,T]$ and $|\nabla^{k}h(t)|\le\frac{1}{1000}$, $k=0,1$.
Suppose also
\begin{equation}\label{e: double}
\begin{cases}
    |\nabla^kh|(x,0)\le \epsilon\cdot e^{100\,d_0(x,x_0)} \quad&\textit{for}\quad x\in B_0(x_0,D_{\#}),\quad k=0,1,2,\\
    |h|(x,t)\le \epsilon\cdot e^{10\,D_{\#}}\quad&\textit{for}\quad (x,t)\in \partial B_0(x_0,D_{\#})\times[0,T].
\end{cases}
\end{equation}
Suppose also
\begin{equation}\label{e: alpha}
    |h|(x,0)\le\epsilon\cdot e^{\alpha\,d_{g(0)+h(0)}(x,x_0)}\quad\textit{for}\quad x\in B_0(x_0,D_{\#}).
\end{equation}
Then we have
\begin{equation*}
    |\nabla^kh_-|(x_0,T)\le \epsilon\cdot e^{-\delta_0T}\cdot e^{2\alpha T},\quad k=0,1,...,100.
\end{equation*}
where $h_-$ is the oscillatory part of $h$.
\end{theorem}

\begin{proof}[Proof of Theorem \ref{t: symmetry improvement}]
Let $T_0>0$ be from Proposition \ref{l: h on R2xS1}, and the value of $\delta_0$ will be determined later.
Suppose the assertion does not hold for some $T\ge T_0$ and $C_0>0$.
Then there are sequences of numbers $\epsilon_i\ri0,\delta_i\ri0,D_i\ri\infty$, a sequence of $SO(2)$-symmetric complete Ricci flows $(M_i,g_i(t))$, $t\in[0,T]$, which is $\delta_i$-close to $\RR\times S^1$ at $(x_i,0)\in M_i\times[0,T]$, and a sequence of Ricci deturck flow perturbation $h_i(t)$ with background metric $g_i(t)$, defined on $B_0(x_0,D_i)\times[0,T]$, which satisfies \eqref{e: double} and
\begin{equation}\label{e: initial bd on h}
    |h_i|(x,0)\le\epsilon_i\cdot e^{\alpha\, d_{g_i(0)+h_i(0)}(x,x_i)},
\end{equation}
but there is some $k_i\in\{0,1,...,100\}$ such that
\begin{equation}\label{e: h-lower bound}
    |\nabla^{k_i}h_{i,-}|(x_i,T)\ge \epsilon_i\cdot e^{-\delta_0 T}\cdot e^{2\alpha T}.
\end{equation}

After passing to a subsequence we may assume that the pointed Ricci flows $(M_i,g_i(t),x_i)$ on $[0,T]$ converge to  $(\RR\times S^1,g_{stan},x_0)$ in the $C^{96}$-sense.
We will show that $\frac{h_i}{\epsilon_i}$ converge to a solution to the linearized Ricci Deturck flow on $\RR\times S^1$. To this end, since $|\nabla^{k}h|\le\frac{1}{1000}$, $k=0,1$, by Lemma \ref{l: laplacian}, there is $C_0>0$ such that
\begin{equation*}
    \partial_t|h_i|^2\le\Delta_{g_i+h_i}|h_i|^2+ C_0\cdot|h_i|^2 \quad\textit{on}\quad B_0(x_i,D_i)\times[0,T].
\end{equation*}
Let $u_i:=e^{-C_0t}|h_i|^2$, then this implies $\partial_t u_i\le\Delta_{g_i+h_i}u_i$. Moreover, \eqref{e: double} implies
\begin{equation*}\begin{cases}
    u_i(x,0)
    \le \epsilon_i^2\cdot e^{20\, d_{g_i(0)+h_i(0)}(x,x_i)} &\textit{for}\quad x\in B_0(x_i,D_{i}),\\
    u_i(x,t)\le \epsilon_i^2\cdot e^{20\,D_{i}} &\textit{for}\quad x\in\partial B_0(x_i,D_{i}), t\in[0,T].
\end{cases}
\end{equation*}
Applying the heat kernel estimate Lemma \ref{l: heat equation} and the weak maximum principle on $B_0(x_i,D_i)\times[0,T]$, we obtain bounds on $|u_i|$ which are independent of all $i$: For any $A>0$, there exists $C(A,T)>0$, which is uniform for all $i$, such that 
\begin{equation*}
    |u_i|(x,t)\le C(A,T)\cdot \epsilon_i^2\quad\textit{on}\quad B_0(x_i,A)\times[0,T],
\end{equation*}
and hence
\begin{equation}\label{e: C0 bound}
    |h_i|(x,t)\le C(A,T)\cdot \epsilon_i\quad\textit{on}\quad B_0(x_i,A)\times[0,T],
\end{equation}
with a possibly larger $C(A,T)$.
By the local derivative estimates for the Ricci deturck flow perturbations \cite[Lemma A.14]{bamler2022uniqueness}, this implies bounds for higher derivatives,
\begin{equation*}
    |\nabla^mh_i|\le C_m(A,T)\cdot\epsilon_i\cdot t^{-m/2}\quad\textit{on}\quad B_0(x_i,A/2)\times(0,T],
\end{equation*}
where $m\in\mathbb{N}$ and $C_m(A,T)>0$ are constants depending on $A,T,m$. Moreover, the first inequality in \eqref{e: double} implies
\begin{equation*}
    |\nabla^k h_i|\le C(A,T)\cdot\epsilon_i\quad\textit{on}\quad B_0(x_i,A/2)\times[0,T],\quad k=0,1,2.
\end{equation*}

Therefore, letting $H_i=\frac{h_i}{\epsilon_i}$, then
\begin{equation*}
    |\nabla^mH_i|\le C_m(A,T)\cdot t^{-m/2}\quad\textit{on}\quad B_0(x_i,A/2)\times(0,T],
\end{equation*}
and also
\begin{equation*}
    |\nabla^k H_i|\le C(A,T)\quad\textit{on}\quad B_0(x_i,A/2)\times[0,T],\quad k=0,1,2.
\end{equation*}
So after passing to a subsequence, $H_i$ converges to a symmetric 2-tensor $H_{\infty}$ on $(\RR\times S^1)\times[0,T]$ in the $C^0$-sense, and the convergence is smooth on $(0,T]$.

On the one hand, by the contradiction assumption \eqref{e: h-lower bound} there is some $k_0\in\{0,...,100\}$ such that
\begin{equation}\label{e: greater than}
    |\nabla^{k_0}H_{\infty,-}|(x_0,T)\ge e^{-\delta_0T}\cdot e^{2\alpha T}.
\end{equation}
On the other hand, since $\epsilon_i\ri0$, it follows that $H_{\infty}$ satisfies the linearized Ricci Deturck equation $\partial_t H_{\infty}=\Delta_L H_{\infty}$. The initial bound \eqref{e: initial bd on h} passes to the limit and implies
\begin{equation*}
    |H_{\infty}|(x,0)\le e^{\alpha \,d_{g_{stan}}(x,x_0)}.
\end{equation*}
So we can apply Proposition \ref{l: h on R2xS1} to the oscillatory part $H_{\infty,-}$ at every point in the backward parabolic neighborhood $U:=B_T(x_0,1)\times[T-1,T]\subset (\RR\times S^1)\times[0,T]$ centered at $(x_0,T)$, and obtain
\begin{equation*}
    |H_{\infty,-}|< e^{\alpha+0.01}\cdot e^{-0.01\,T}\cdot e^{2\alpha T}\quad \textit{on}\quad U.
\end{equation*}
Since the components of $H_{\infty,-}$ satisfy the heat equation \eqref{e: solve}, it follows by the standard derivative estimates of heat equations that for all $k=0,1,...,100$,
\begin{equation*}
    |\nabla^kH_{\infty,-}|(x_0,T)< C_k\cdot e^{-0.01\,T}\cdot e^{2\alpha T}\le e^{-\delta_0T}\cdot e^{2\alpha T},
\end{equation*}
for some $\delta_0<0.01$, which contradicts with \eqref{e: greater than}.

\end{proof}

\section{Construction of an approximating SO(2)-symmetric metric}\label{s: Approximating $SO(2)$-symmetric metrics}
The main goal in this section is to construct an approximating $SO(2)$-symmetric metric which approximates the soliton metric within error $e^{-2(1+\epsilon_0)\,d_g(\cdot,\Gamma)}$, for some positive constant $\epsilon_0>0$, and moreover the error goes to zero as we move towards the infinity of the soliton.  Here $\Gamma=\Gamma_1(-\infty,\infty)\cup\Gamma_2(-\infty,\infty)\cup\{p\}$, where $p$ is the critical point of $M$, and $\Gamma_1,\Gamma_2$ are two integral curves from Corollary \ref{l: new Gamma}.

The construction consists of two parts: First, in subsection \ref{ss: the desired exponential decay part 2}, we do an inductive construction to obtain a $SO(2)$-symmetric metric $\overline{g}$ that approximates the soliton metric within the error $e^{-2(1+\epsilon_0)d_g(\cdot,\Gamma)}$.
Next, in subsection \ref{ss: near the edge part 3}, we extend $\overline{g}$ to a neighborhood of $\Gamma$ to obtain the desired approximating metric.

In the first step, we repeating the following process in an induction scheme: We consider the harmonic map heat flows from the Ricci flow of the soliton to the Ricci flow of some approximating $SO(2)$-symmetric metric.
The error between the two flows is characterized by the Ricci deturck flow perturbation, whose oscillatory mode decays in time by our symmetric improvement theorem.
Therefore, the accuracy of the approximation will improve by the flow, after adding the rotationally symmetric mode in the Ricci deturck flow perturbation to the approximating metric.

Note that the perturbation could grow very fast in the compact regions because the soliton is not close to $\RR\times S^1$ and we do not have a symmetry improvement theorem there. In order to deal with this, we will do surgeries to the soliton metric $g$ and  approximating $SO(2)$-symmetric metrics, by cutting-off their compact regions and glue-back regions that are sufficiently close to $\RR\times S^1$. The resulting manifolds are diffeomorphic to $\RR\times S^1$, and close to $\RR\times S^1$ everywhere.
So the harmonic map heat flows between the flows of these manifolds exist up to a long enough time for us to apply Theorem \ref{t: symmetry improvement}.
In the surgeries we need to glue up $\epsilon$-cylindrical planes, and for the $SO(2)$-symmetric metrics we also need to preserve the $SO(2)$-symmetry in the resulting metrics. This needs some gluing-up lemmas in subsection \ref{ss: glueup}.
We conduct the surgeries in subsection \ref{ss: Surgery on the soliton metric part1} and \ref{ss: Surgery on the background SO(2)-symmetric metrics part 1}.

\subsection{Glue up SO(2)-symmetric metrics}\label{ss: glueup}
In order to do the surgeries, we need to know how to glue up $SO(2)$-symmetric metrics. This is done in this subsection.
Recall that for a 3D Riemannian manifold $(M,g)$,
we say it is $SO(2)$-symmetric if there is 
a one parameter group of isometries $\psi_{\theta}$, $\theta\in\R$, such that $\psi_{\theta}=i.d.$ if and only if $\theta=2k\pi$, $k\in\mathbb{Z}$.
In this subsection, we show how to glue up several $SO(2)$-symmetric metrics which are close to $(\RR\times S^1,g_{stan})$ and also close to each other on their intersections.
Since the metrics
throughout this subsection are $\epsilon$-close to $\RR\times S^1$ for some very small $\epsilon$, we will take derivatives and measure norms using the metric $g_{stan}$ on $\RR\times S^2$.

First, we show in following lemma that if a 3D Riemannian manifold $(M,g)$ is $\epsilon$-close to $\RR\times S^1$ at $x_0\in M$ under two $\epsilon$-isometries $\phi_{1}$ and $\phi_{2}$, then the two vector fields $\phi_{1*}\left(\partial_{\theta}\right)$ and $\phi_{2*}\left(\partial_{\theta}\right)$ are $C_0\epsilon$-close.
Therefore, the vector field $\partial_{\theta}$ is well-defined on an $\epsilon$-cylindrical plane up to sign and an error $\epsilon$.

\begin{lem}\label{l: two epsilon cylindrical planes are close}
Let $k\in\mathbb{N}$.
There exists $C_0,\overline{\epsilon}>0$ such that the following holds for all $\epsilon<\overline{\epsilon}$. 
Let $(M,g)$ be a 3-dimensional Riemannian manifold.
Suppose $(M,g,x_0)$ is $\epsilon$-close to $(\RR\times S^1,g_{stan})$ in the $C^k$-norm under two $\epsilon$-isometries $\phi_i:S^1\times(-\epsilon^{-1},\epsilon^{-1})\times(-\epsilon^{-1},\epsilon^{-1})\rightarrow U_i\subset\subset M$, $i=1,2$, where $V:=\phi_1(S^1\times(-100,100)\times(-100,100))\subset U_2$. Then after possibly replacing $\phi_2$ by $\phi_2\circ p$, where $p(\theta,x,y)=(-\theta,x,y)$ for $\theta\in[0,2\pi)$ and $x,y\in(-\epsilon^{-1},\epsilon^{-1})$, we have
\begin{equation*}
    \left|\phi_{1*}\left(\partial_{\theta}\right)-\phi_{2*}\left(\partial_{\theta}\right)\right|_{C^{k-1}(V)}\le C_0\epsilon.
\end{equation*}
\end{lem}

\begin{proof}
We shall use $\epsilon$ to denote all constants $C_0\epsilon$, where $C_0>0$ is a constant depending only on  $k$.
Let $g_i=(\phi_i^{-1})^{*}g_{stan}$ and $X_i=\phi_{i*}\left(\partial_{\theta}\right)$, $i=1,2$. 
Let $(x,y,\theta)$ be the coordinates on $U_1$ induced by $\phi_1$ such that $g_1$ can be written as 
$d\theta^2+dx^2+dy^2$, where $x,y\in(-\epsilon^{-1},\epsilon^{-1})$ and $\theta\in [0,2\pi)$. So $X_1=\partial_{\theta}$.
The coordinate function $\theta$ can be lifted to a function $z$ on the universal covering $\widetilde{U}_1\rightarrow U_1$, so that the metric on $\widetilde{U}_1$ can be written as $dz^2+dx^2+dy^2$ under the coordinates $(x,y,z)$. 

Assume $\epsilon$ is sufficiently small, then
\begin{equation}\label{e: conference}
    |g_1-g_2|_{C^k(V)}\le|g_1-g|_{C^k(V)}+|g_2-g|_{C^k(V)}\le\epsilon.
\end{equation}
Since $\LL_{X_2}g_2=0$, this implies
\begin{equation}\label{e: sloan}
    \left|\LL_{X_2}g_1\right|_{C^{k-1}(V)}\le\left|\LL_{X_2}(g_2+(g_1-g_2))\right|_{C^{k-1}(V)}=\left|\LL_{X_2}(g_1-g_2)\right|_{C^{k-1}(V)}\le\epsilon.
\end{equation}
Note that the following is the form of all killing fields on $\R^3$.
\begin{equation*}
    a_1\partial_{x}+a_2\partial_{y}+a_3\partial_{z}+b_1(x\partial_{z}-z\partial_{x})+b_2(x\partial_{y}-y\partial_{x})+b_3(z\partial_{y}-y\partial_{z}),
\end{equation*}
where $a_1,a_2,a_3,b_1,b_2,b_3\in\R$.
A direct computation using \eqref{e: sloan}, we see that $X_2$ is $\epsilon$-close in the $C^{k-1}$-norm to a vector field on $V$ of the following form,
\begin{equation}\label{e: rutgers}
    a_1\partial_{x}+a_2\partial_{y}+a_3\partial_{\theta}+b_1(x\partial_{\theta}-{\theta}\partial_{x})+b_2(x\partial_{y}-y\partial_{x})+b_3({\theta}\partial_{y}-y\partial_{\theta}).
\end{equation}

In the following we will estimate these coefficients and show that $|a_3\pm1|\le\epsilon$ and other coefficients are bounded in absolute values by $\epsilon$.
First, for a vector field in \eqref{e: rutgers} to be well-defined at $\theta=0$, we must have
\begin{equation*}
    |b_1|+|b_3|\le\epsilon.
\end{equation*}
Next, since $\nabla^{g_2} X_2=0$, it follows by \eqref{e: conference} that $|\nabla^{g_1} X_2|_{C^{k-1}(V)}\le\epsilon$, which implies $|b_2|\le\epsilon$.
So $X_2$ is $\epsilon$-close in the $C^{k-1}$-norm to the vector
\begin{equation*}
    Y:=a_1\partial_{x}+a_2\partial_{y}+a_3\partial_{\theta},
\end{equation*}
and hence the flow generated by $Y$, which is
\begin{equation*}
    \begin{cases}
    x(t;x_0,y_0,\theta_0)=x_0+a_1t,\\
    y(t;x_0,y_0,\theta_0)=y_0+a_2t,\\
    \theta(t;x_0,y_0,\theta_0)=\theta_0+a_3t,
    \end{cases}
\end{equation*}
is $\epsilon$-close in the $C^{k-1}$-norm to the flow generated by $X_2$ on $V$.
Since the flow of $X_2$ is $2\pi$-periodic, it follows that
\begin{equation}\label{e: a1a2a3}
    |a_1|+|a_2|+|2\pi\,a_3-2m\pi|\le \epsilon,
\end{equation}
for some $m\in\mathbb{N}$.
Note that $X_i$ is a unit speed velocity vector of the $g_i$-minimal geodesic loop, $i=1,2$, it is easy to see that for sufficiently small $\epsilon$ we must have either $|X_2-X_1|\le\frac{1}{1000}$ or $|X_2+X_1|\le\frac{1}{1000}$. 
This implies $|a_3\pm1|\le\frac{1}{100}$, which combined with \eqref{e: a1a2a3} implies 
\begin{equation*}
    |a_1|+|a_2|+|a_3\pm1|\le\epsilon,
\end{equation*}
which proves the lemma.
\end{proof}

The next lemma is a step further than Lemma \ref{l: two epsilon cylindrical planes are close}, which shows that if a $SO(2)$-symmetric metric is $\epsilon$-close to an $\epsilon$-cylindrical plane, then their killing fields are $C_0\epsilon$-close.

\begin{lem}\label{l: killing fields are close}
Let $k\in\mathbb{N}$.
There exists $C_0,\overline{\epsilon}>0$ such that the following holds for all $\epsilon<\overline{\epsilon}$. 
Let $(U,g)$ be a 3D Riemannian manifold, $x_0\in U$.
Suppose $g$ is $SO(2)$-symmetric metric and $X$ is the killing field of the $SO(2)$-isometry.
Suppose also $(U,g,x_0)$ is $\epsilon$-close to $(\RR\times S^1,g_{stan})$ in the $C^k$-norm, and
\begin{equation*}
    |X-\partial_{\theta}|\le\frac{1}{1000}\quad \textit{on}\quad B_g(x_0,1000)\subset U,
\end{equation*}
where $\partial_{\theta}$ denotes the killing field along the $S^1$-direction in an $\epsilon$-cylindrical plane.
Then
\begin{equation*}
    \left|X-\partial_{\theta}\right|_{C^{k-1}(B_g(x_0,1000))}\le C_0\epsilon.
\end{equation*}
\end{lem}

Note that
the assumption $|X-\partial_{\theta}|\le\frac{1}{1000}$ in this lemma is necessary even if we derive a better bound using it.
For example, on $\RR\times S^1$, a $SO(2)$-isometry could be either a rotation in the $xy$-plane around the origin, or a rotation in the $S^1$-factor, but their killing fields are not close to each other.

\begin{proof}
We shall use $\epsilon$ to denote all constants $C_0\epsilon$, where $C_0>0$ is a constant independent of $\epsilon$.
First, let $\phi_1$ be the $\epsilon$-isometry to $(\RR\times S^1,g_{stan})$, and let $(x,y,\theta)$, $x,y\in(-\epsilon^{-1},\epsilon^{-1})$ and $\theta\in [0,2\pi)$, be local coordinates on an open subset $V$ containing $B_g(x_0,1000)$, such that
$g_1:=(\phi_1^{-1})^{*}g_{stan}$ can be written as $d\theta^2+dx^2+dy^2$. 
Then $\theta$ can be lifted to a function $z$ on the universal covering $\widetilde{V}\ri V$, and the induced metric on $\widetilde{V}$ is $dz^2+dx^2+dy^2$. 

Assume $\epsilon$ is sufficiently small, then
\begin{equation*}
    |g-g_1|_{C^k(U)}\le\epsilon.
\end{equation*}
Seeing also $\LL_Xg=0$, this implies
\begin{equation*}
    \left|\LL_{X}g_1\right|_{C^{k-1}(U)}
    =\left|\LL_{X}(g_1-g)\right|_{C^{k-1}(U)}\le\epsilon.
\end{equation*}
Similarly as in Lemma \ref{l: two epsilon cylindrical planes are close}, this implies that $X$ is $\epsilon$-close in the $C^{k-1}$-norm to the following vector field on $V$,
\begin{equation}\label{e: Y vector}
    Y:=a_1\partial_{x}+a_2\partial_{y}+a_3\partial_{\theta}+b_2(x\partial_{y}-y\partial_{x}),
\end{equation}
where $a_1,a_2,a_3,b_2\in\R$.

In the following we will show that $|a_3-1|\le\epsilon$ and $|a_1|,|a_2|,|b_2|\le\epsilon$.
First, assume $b_2\neq0$.
Then the flow generated by $Y$ is
\begin{equation}\label{e: case}
    \begin{cases}
    y(t;x_0,y_0,\theta_0)=y_0\cos b_2t+\frac{a_1}{b_2}(1-\cos b_2t)+\frac{a_2}{b_2}\sin b_2t\\
    x(t;x_0,y_0,\theta_0)=x_0\cos b_2t+\frac{a_2}{b_2}(\cos b_2t-1)+\frac{a_1}{b_2}\sin b_2t\\
    \theta(t;x_0,y_0,\theta_0)=\theta_0+a_3t.
    \end{cases}
\end{equation}
Since the flow generated by $X$ is $2\pi$-periodic and $Y$ is $\epsilon$-close to $X$, it follows that 
\begin{equation}\label{e: y0x0}
    |x(2\pi;x_0,y_0,\theta_0)-x_0|+|y(2\pi;x_0,y_0,\theta_0)-y_0|\le\epsilon,
\end{equation}
for some $m\in\mathbb{N}$.
First, by $|X-\partial_{\theta}|\le\frac{1}{1000}$ we have
\begin{equation*}
    \left|Y-\partial_{\theta}\right|\le |X-Y|+\left|X-\partial_{\theta}\right|\le \epsilon+\left|X-\partial_{\theta}\right|\le\frac{1}{500}
\end{equation*}
which implies $|b_2|\le\frac{1}{100}$.
Moreover, by taking $y_0=100,-100$ in \eqref{e: y0x0} and using \eqref{e: case}, we get
\begin{equation*}
    200|\cos (2\pi b_2)-1|= |y(2\pi;0,100,0)-100-y(2\pi;0,-100,0)+100|\le\epsilon,
\end{equation*}
which combined with $|b_2|\le\frac{1}{100}$ implies $|b_2|\le\epsilon$. 

So we may assume $b_2=0$ in \eqref{e: Y vector}, so $X$ is $\epsilon$-close to the following vector field in the $C^{k-1}$-norm on $V$,
\begin{equation}\label{e: Z vector}
    Z:=a_1\partial_{x}+a_2\partial_{y}+a_3\partial_{\theta},
\end{equation}
which generates the flow
\begin{equation}\label{e: case again}
    \begin{cases}
    y(t;x_0,y_0,\theta_0)=y_0+a_2t;\\
    x(t;x_0,y_0,\theta_0)=x_0+a_1t;\\
    \theta(t;x_0,y_0,\theta_0)=\theta_0+a_3t.
    \end{cases}
\end{equation}
The $2\pi$-periodicity of the flow of $X$ implies immediately 
\begin{equation*}
    |a_1|+|a_2|+|2\pi a_3-2m\pi|\le \epsilon.
\end{equation*}
Using $|X-\partial_{\theta}|\le\frac{1}{1000}$ this implies $|a_3-1|\le\epsilon$, which proves the lemma.

\end{proof}

In the following lemma, we show that one can glue up one parameter groups of diffeomorphisms that are $\epsilon$-close to each other on their intersections, to obtain a global one parameter group of diffeomorphisms that are $C_0\epsilon$-close to them, where the constant $C_0$ does not depend on $\epsilon$.

\begin{lem}\label{l: glue up local diffeomorphisms from N x S1 to M}
Let $m,k\in\mathbb{N}$. There exists $C_0,\overline{\epsilon}>0$ such that the following holds for all $\epsilon<\overline{\epsilon}$. 
Let $(M,g)$ be a 3D Riemannian manifold. Suppose $(M,g)$ is $\epsilon$-close to $(\RR\times S^1,g_{stan})$ at all $x\in M$.
Suppose $\{U_i\}_{i=1}^{\infty}$ is an open covering of $M$ such that at most $m$ of them intersect at one point. Moreover, there is a  one parameter group of diffeomorphisms $\{\phi_{i,t}\}_{t\in\R}$ on each $U_i$, which satisfies:
\begin{enumerate}
    \item\label{it: 1} $\phi_{i,0}=\phi_{i,2\pi}=i.d.$.
    \item\label{it: 2} $|\phi_{i,t*}(\partial_{t})-\partial_{\theta}|\le\frac{1}{1000}$, where $\partial_{\theta}$ denotes the killing field along the $S^1$-direction in an $\epsilon$-cylindrical plane up to sign.
    \item\label{it: 3} $|\phi_{i}-\phi_{j}|_{C^{k}((U_i\cap U_j)\times S^1)}\le\epsilon$, where $\phi_i(x,\theta)=\phi_{i,\theta}(x)$ for any $(x,\theta)\in U_i\times S^1$.
\end{enumerate}
Then there exists a one parameter group of diffeomorphisms $\{\psi_{t}\}_{t\in\R}$ on $M$  satisfying
\begin{enumerate}
    \item\label{property 1} $\psi_{0}=\psi_{2\pi}=i.d.$.
    \item\label{property 2} $|\psi-\phi_{i}|_{C^{k}(U_i\times S^1)}\le C_0 \epsilon$ for all $i$.
    \item\label{property 3} $\psi_{t}=\phi_{i,t}$ on $\{x\in M: B_g(x,1000\,r(x))\subset U_i\}$.
\end{enumerate}
\end{lem}

\begin{proof}
In the following we denote $\phi_{i,t}$ as $\phi_{i,\theta}$, $\theta\in[0,2\pi]$, and $C_0$ denotes for all positive constants that depend on $k$.
Since $(M,g)$ is covered by $\epsilon$-cylindrical planes, 
by a standard gluing up argument, 
we can find a smooth complete surface $N$ embedded in $M$, such that the tangent space of $N$ is $\frac{1}{100}$-almost orthogonal to $\partial_{\theta}$ in each $\epsilon$-cylindrical plane. 
Equip the manifold $N\times S^1$ with a warped product metric $\overline{g}=g_N+d\theta^2$, $\theta\times[0,2\pi)$, where $g_N$ is the induced metric of $(M,g)$ on $N$.

First, we will use the local one parameter groups $\phi_{i,\theta}$ to construct local diffeomorphisms $F_i: N\times S^1\ri M$.
Let $V_i=U_i\cap N$, then $\{V_i\}_{i=1}^{\infty}$ is an open covering of $N$, and at most $m$ of them intersect at one point.
Let $F_i: V_i\times S^1\ri U_i$ be defined by 
\begin{equation*}
    F_i(x,\theta)=\phi_{i,\theta}(\varphi(x)).
\end{equation*}
Then $F_i$ is a diffeomorphism, and 
\begin{equation*}
    |F_i-F_j|_{C^k((V_i\cap V_j)\times S^1)}\le C_0\epsilon.
\end{equation*}

Next, we will first construct a global diffeomorphism $F:N\times S^1\ri M$ by gluing up the diffeomorphisms $F_i$, such that $F$ is $C_0\epsilon$-close to each $F_i$. 
Suppose $\overline{\epsilon}$ is sufficiently small such that for any $x\in M$, $B_g(x,\overline{\epsilon})$ is a convex neighborhood of $x$, i.e. the minimizing geodesics connecting any two points in $B_g(x,\overline{\epsilon})$ are unique and contained in $B_g(x,\overline{\epsilon})$.
Let $\Delta_2$ be the open neighborhood of $\{(x,x)\in M^2\}$ the diagonal of $M^2$, 
\begin{equation*}
    \Delta_2=\{(x,y)\in M^2: d_g(x,y)\le\overline{\epsilon}\}\subset M^2.
\end{equation*}
Define the smooth map
\begin{equation*}
    \Sigma_{2}: \{(s_1,s_2)\in[0,1]^2: s_1+s_2=1\}\times \Delta_2\ri M,
\end{equation*}
as $\Sigma_2(s_1,s_2,x_1,x_2)=\gamma_{x_2,x_1}(s_1)$ where $\gamma_{x_2,x_1}:[0,1]\ri M$ is a minimizing geodesic from $x_2$ to $x_1$.
Then $\Sigma_2$ satisfies the following properties for all $s_1,s_2\in[0,1]$, $s_1+s_2=1$, and $x,x_1,x_2\in M$:
\begin{enumerate}
    \item $\Sigma_2(1,0,x_1,x_2)=x_1$, $\Sigma_2(0,1,x_1,x_2)=x_2$.
    \item $\Sigma_2(s_1,s_2,x,x)=x$.
\end{enumerate}
Then for each $K\in\mathbb{N}$, we can inductively construct the open neighborhood of $\{(x,...,x)\in M^K\}$ the diagonal in $M^K$, see also \cite{Bamler2020CompactnessTO},
\begin{equation*}
    \Delta_K=\{(x_1,...,x_K)\in M^K: d_g(x_i,x_j)\le\overline{\epsilon}\}\subset M^K,
\end{equation*}
and the smooth map
\begin{equation*}
    \Sigma_K: \{(s_1,...,s_K)\in[0,1]^K: s_1+\cdots+s_K=1\}\times \Delta_K\ri M,
\end{equation*}
by defining
\begin{equation*}
    \Sigma_K(s_1,...,s_K,x_1,...,x_K):=\Sigma_2(1-s_K,s_K,\Sigma_{K-1}(s_1,...,s_{K-1},x_1,...,x_{K-1}),x_K),
\end{equation*}
with the following properties for all $s_1,...,s_K\in[0,1]$, $s_1+\cdots+s_K=1$, and $x,x_1,x_2,...,x_K\in M$:
\begin{enumerate}
    \item If for some $j\in\{1,...,K\}$ we have $s_j=1$ and $s_i=0$ for all $i\neq j$, then $\Sigma_K(s_1,...,s_K,x_1,...,x_K)=x_j$.
    \item $\Sigma_K(s_1,...,s_K,x,...,x)=x$.
    \item If $s_{K-i+1}=...=s_K=0$ for some $i\ge 1$, then $\Sigma_K(s_1,...,s_K,x_1,...,x_K)=\Sigma_{K-1}(s_1,...,s_{K-i},x_1,...,x_{K-i})$.
\end{enumerate}

Let $\{h_i\}_{i=1}^{\infty}$ be a partition of unity of $N\times S^1$ subordinated to $\{V_i\times S^1\}_{i=1}^{\infty}$, such that $h_i$ is constant on each $S^1$-factor, and $h_i\equiv1$ on $\widetilde{V}_i\times S^1$, where $\widetilde{V}_i=\{x\in N: B_{g_N}(x,500\,r(x))\subset V_i\}\times S^1$.
Then let $F: N\times S^1\ri M$ be such that for any $x\in N\times S^1$, 
\begin{equation*}
    F(x):=\Sigma_K(h_1(x),...,h_K(x),F_1(x),...,F_K(x)),
\end{equation*}
where $K\in\mathbb{N}$ is some integer such that $h_{K'}(x)=0$ for any $K'>K$.
By the properties of $\Delta_K$ and $\Sigma_K$, we see that $F$ is a well-defined smooth map, and it satisfies
\begin{equation}\label{e: diffeos close}
    |F-F_i|_{C^k(V_i\times S^1)}\le C_0\epsilon.
\end{equation}

Next, we will show that $F$ is a diffeomorphism. 
First, by \eqref{e: diffeos close} and the definition of $\overline{g}$ we see that for any $p\in N\times S^1$ and $v\in T_p(N\times S^1)$ we have
\begin{equation}\label{e: covering map}
    |F_{*p}(v,v)|_g\ge0.9|v|_{\overline{g}}.
\end{equation}
So $F_*$ is non-degenerate.
Next, we argue that $F$ is injective. To see this, observe that 
by \eqref{e: diffeos close} and assumption \eqref{it: 2} we have that $F$ is injective on $B_{g_N}(x,2)\times S^1$ for any $x\in N$, and $F(x_1\times S^1)\cap F(x_2\times S^1)=\emptyset$ for any $x_1,x_2\in N$ such that $d_{g_N}(x_1,x_2)\ge 1$.
Now suppose $F(x_1,\theta_1)=F(x_2,\theta_2)$, then we must have $d_{g_N}(x_1,x_2)<1$, and hence $x_1=x_2,\theta_1=\theta_2$ as desired.
Therefore, $F$ is a diffeomorphism.

Next, let $\theta\in[0,2\pi)$ be the parametrization of $S^1$, and let $X:=F_*(\partial_{\theta})$. Then 
$X$ generates a $2\pi$-periodic one-parameter group of diffeomorphisms $\psi_{\theta}$, $\theta\in[0,2\pi)$ on $M$.
In the following we will show that $\psi_{\theta}$ satisfies all the properties.
Denote $\psi(x,\theta)=\psi_{\theta}(x)$ for all $(x,\theta)\in M\times S^1$.
First, let $\widetilde{U}_i=\{x\in M: B_g(x,1000\,r(x))\subset U_i\}$. Then it is easy to see that $\widetilde{U}_i\subset F(\widetilde{V}_i\times S^1)$. Since $h_i\equiv1$ on $\widetilde{V}_i\times S^1$, it follows that $F=F_i$, and hence $\psi_{\theta}=\phi_{i,\theta}$ on $\widetilde{U}_i$, which verifies property \eqref{property 3}.

Lastly,
to verify property \eqref{property 2}, we note that for any $x\in U_i=F_i(V_i\times S^1)$, suppose $x=F_i(z,\theta_1)$, then for any $\theta\in[0,2\pi)$, we have
\begin{equation*}
    \phi_{i,\theta}(x)=\phi_{i,\theta+\theta_1}(x)=F_i(z,\theta_1+\theta),
\end{equation*}
and also
\begin{equation*}
    \psi(F\circ F_i^{-1}(x),\theta)
    =\psi(F(z,\theta_1),\theta)=F(z,\theta_1+\theta).
\end{equation*}
Therefore, using \eqref{e: diffeos close} the closeness of $F$ and $F_i$, as well as the closeness of $F\circ F_i^{-1}:U_i\ri M$ and $i.d.: U_i\ri M$, we can deduce that
\begin{equation*}
    |\psi-\phi_{i}|_{C^k(U_i\times S^1)}\le C_0\epsilon,
\end{equation*}
which verifies \eqref{property 2}.
\end{proof}

Now we prove the main result of this subsection, which shows that 
if there are some $SO(2)$-symmetric metrics which are $\epsilon$-close to each other, then we can glue them up to obtain a global $SO(2)$-symmetric metric which is $C_0\epsilon$-close to the original metrics.

\begin{lem}\label{l: glue cylindrical plane}
Let $k,m\in\mathbb{N}$. There exists $C_0,\overline{\epsilon}>0$ such that the following holds for all $\epsilon<\overline{\epsilon}$. 
Let $(M,g)$ be a 3D Riemannian manifold diffeomorphic to $\RR\times S^1$, which is $\epsilon$-close to $(\RR\times S^1,g_{stan})$ at all $x\in M$.
Suppose $\{U_i\}_{i=1}^{\infty}$ is a locally finite covering of $M$ such that at most $m$ of them intersects at one point, and there is a $SO(2)$-symmetric metric $g_i$ on $W_i:=\bigcup_{x\in U_i}B_g(x,1000r(x))$, with the $SO(2)$-isometry $\{\phi_{i,\theta}\}_{\theta\in[0,2\pi)}$ and killing field $X_i$, which satisfies
\begin{equation}\label{e: gi and g}
    |g_i-g|_{C^k(W_i)}\le\epsilon,\quad\textit{and}\quad |X_i-\partial_{\theta}|_{C^0(W_i)}\le\frac{1}{1000},
\end{equation}
where $\partial_{\theta}$ denotes the killing field along the $S^1$-direction in an $\epsilon$-cylindrical plane up to sign.
Let $\{h_i\}_{i=1}^{\infty}$ be a partition of unity that subordinates to $\{U_i\}_{i=1}^{\infty}$. Then we can find a $SO(2)$-symmetric metric $\overline{g}$ on $M$ with $SO(2)$-isometry $\{\psi_{\theta}\}_{\theta\in[0,2\pi)}$ such that
\begin{equation}\label{e: g3 and g}
    |\overline{g}-g|_{C^{k-1}(M)}\le C_0\epsilon\quad\textit{and}\quad|\psi_{\theta}-\phi_{i,\theta}|_{C^{k-1}(U_i)}\le C_0\epsilon.
\end{equation}
Moreover, $\overline{g}=g_i$ on the subset $\{x\in M: h_i(\phi_{i,\theta}(x))=1,\,\theta\in[0,2\pi)\}$.
\end{lem}

\begin{proof}
In the following $C_0$ denotes for all positive constants that depend on $k,m$.
Let 
\begin{equation*}
    V_i=\{x\in M: h_i(\phi_{i,\theta}(x))=1,\,\theta\in[0,2\pi)\}.
\end{equation*}
First, applying Lemma \ref{l: killing  fields are close} to $g_i$ we have
\begin{equation}\label{e: Xtheta}
    \left|X_i-\partial_{\theta}\right|_{C^{k-1}(U_i)}\le C_0\epsilon,
\end{equation}
where $\partial_{\theta}$ is the killing field along the $S^1$-direction in an $\epsilon$-cylindrical plane up to sign.
As in Lemma \ref{l: glue up local diffeomorphisms from N x S1 to M}, let 
$N$ be a 2D complete surface smoothly embedded in $M$, whose the tangent space is $\frac{1}{100}$-almost orthogonal to $\partial_{\theta}$. 

Next, we will construct a diffeomorphism $\sigma:N\times S^1\ri M$, such that $\sigma|_{N\times\{0\}}=id_N$ and  $\sigma_*(\partial_{\theta})$ is $\frac{1}{100}$-close to the $S^1$-factor of any $\epsilon$-cylindrical planes. 
To do this, first we can find a covering of $M$ by $\epsilon$-cylindrical planes such that the number of them intersecting at any point is bounded by a universal constant. Then we claim that after reversing the $\theta$-coordinate in certain $\epsilon$-cylindrical planes we can arrange that the vector fields $\partial_{\theta}$ are $C_0\epsilon$-close in the intersections.
Suppose the claim does not hold, then by Lemma \ref{l: two epsilon cylindrical planes are close} it is easy to find an embedded Klein bottle in $M$. Since $M$ is diffeomorphic to $\RR\times S^1$, which can be embedded into $\R^3$ as a tubular neighborhood of a circle, it follows that the Klein bottle can be embedded in $\R^3$, which is impossible by \cite[Corollary 3.25]{Hatcher:478079}.
Now the diffeomorphism $\sigma$ follows immediately from applying Lemma \ref{l: glue up local diffeomorphisms from N x S1 to M}.

Therefore, we can replacing $X_i$ by $-X_i$ for some $i$ so that they are all $\frac{1}{100}$-close to $\sigma_*(\partial_{\theta})$.
So \eqref{e: Xtheta} implies
\begin{equation}\label{e: vector field close}
    |X_i-X_j|_{C^{k-1}(U_i\cap U_j)}\le C_0\epsilon.
\end{equation}
Replacing $\phi_{i,\theta}$ by $\phi_{i,-\theta}$ for such $i$, then this implies
\begin{equation}\label{e: closeness of maps}
    |\phi_{i,\theta}-\phi_{j,\theta}|_{C^{k-1}(U_i\cap U_j)}\le C_0\epsilon.
\end{equation}
Then by Lemma \ref{l: glue up local diffeomorphisms from N x S1 to M} we can construct a one parameter group of diffeomorphisms $\{\psi_{\theta}\}$ on $M$, such that $\psi_0=\psi_{2\pi}=i.d.$ and
\begin{equation*}
    |\psi_{\theta}-\phi_{i,\theta}|_{C^{k-1}(U_i)}\le C_0\epsilon,
\end{equation*}
and $\psi_{\theta}=\phi_{i,\theta}$ on $V_i$.

Let $\widehat{g}=\sum_{i=1}^{\infty}h_i\cdot g_i$, then $\widehat{g}=g_i$ on $V_i$ and by \eqref{e: gi and g} we have
\begin{equation*}
    |\widehat{g}-g|_{C^{k-1}}\le C_0\epsilon\quad\textit{on}\quad M.
\end{equation*}
Let
\begin{equation*}
    \overline{g}=\frac{1}{2\pi}\int_0^{2\pi}\psi^*_{\theta}\widehat{g}\,d\theta,
\end{equation*}
then $(M,\overline{g})$ is $SO(2)$-symmetric under the isometries $\psi_{\theta}$, and
\begin{equation*}\begin{split}
    |\overline{g}-g_i|_{C^{k-2}(U_i)}&
    \le\frac{1}{2\pi}\int_0^{2\pi}|\psi_{\theta}^*g-\phi^*_{i,\theta}g_i|_{C^{k-2}(U_i)}\,d\theta\\
    &\le\frac{1}{2\pi}\int_0^{2\pi}|\psi_{\theta}^*(g-g_i)|_{C^{k-2}(U_i)}+|\psi_{\theta}^*g_i-\phi^*_{i,\theta}g_i|_{C^{k-2}(U_i)}\,d\theta\le C_0\epsilon,
\end{split}\end{equation*}
which combined with \eqref{e: gi and g} implies \eqref{e: g3 and g}.

Moreover, if $x\in V_i$, then
$\psi_{\theta}(x)=\phi_{i,\theta}(x)$ and $\widehat{g}(x)=g_i(x)$, so we have
\begin{equation*}\begin{split}
    \overline{g}(x)
    =\frac{1}{2\pi}\int_0^{2\pi}\psi^*_{\theta}(\widehat{g}(\psi_{\theta}(x)))\,d\theta
    =\frac{1}{2\pi}\int_0^{2\pi}\phi^*_{i,\theta}(g_i(\phi_{i,\theta}(x)))\,d\theta=g_i(x),
\end{split}\end{equation*}
which finishes the proof.
\end{proof}

\subsection{Surgery on the soliton metric}
\label{ss: Surgery on the soliton metric part1}
In this subsection we will conduct a surgery on the soliton by first removing a neighborhood of the edges $\Gamma$ and then grafting a region covered by $\epsilon$-cylindrical planes onto the soliton. After the surgery, we obtain a complete metric on $\RR\times S^1$, which is covered everywhere by $\epsilon$-cylindrical planes.

We fix some conventions and notations.
First, in the rest of the entire section we assume $(M,g)$ is a 3D steady gradient soliton with positive curvature that is not a Bryant soliton.
Then by Theorem \ref{l: wing-like} we may assume after a rescaling that
\begin{equation}\label{e: dinner}
    \lim_{s\rii}R(\Gamma_1(s))=\lim_{s\rii}R(\Gamma_2(s))= 4.
\end{equation}
Next, let $\rho=d_g(\cdot,\Gamma)$, for any $0<A<B$ we write $\Gamma_{[A,B]}=\rho^{-1}([A,B])$,  $\Gamma_{\le A}=\rho^{-1}([0,A])$,  $\Gamma_{\ge A}=\rho^{-1}([A,\infty))$, and
$\Gamma_A=\rho^{-1}(A)$.
In following lemma we extend the soliton metric on $\Gamma_{\ge A}$ to obtain a complete manifold $(M',g')$ diffeomorphic to $\RR\times S^1$, which is covered by $\epsilon$-cylindrical planes.

\begin{lem}[The $\epsilon$-grafted soliton metric $g'$]\label{l: lemma one}
Let $(M,g)$ be a 3D steady gradient soliton with positive curvature that is not a Bryant soliton satisfying \eqref{e: dinner}.
For any $\epsilon>0$ and $m\in\mathbb{N}$, there exist $A>0$ and a complete Riemannian manifold $(M',g')$ diffeomorphic to $\RR\times S^1$ such that the followings hold:
\begin{enumerate}
    \item There is an isometric embedding $\phi:(\Gamma_{> A},g)\ri (M',g')$.
    \item $(M',g')$ is an $\epsilon$-cylindrical plane at any point $x\in M'$ in the $C^m$-norm.
\end{enumerate} 
\end{lem}

\begin{proof}
Let $A>0$ be sufficiently large such that $\Gamma_{\ge A}$ is covered by $\epsilon$-cylindrical planes.
We may furthermore increase $A$ depending on $m,\epsilon$ and the soliton $(M,g)$.

By using the Ricci flow equation $\partial_tg(t)=-2\Ric(g(t))$, $g(t)=\phi_{-t}^*g$, the quadratic curvature upper bound in Theorem \ref{t: R upper bd}, and Shi's derivative estimates \cite{Shi1987derivative1}, we may assume when $A$ is sufficiently large that
\begin{equation*}
    |\nabla^{\ell}(g-\phi^*_{-(k+1)}g)|\le\epsilon\quad\textit{on}\quad\phi_{k}(\ga)\quad \ell=0,...,m,
\end{equation*}
where the covariant derivatives are taken with respect to $g$.
Therefore, for each $k\in\mathbb{N}$, by a standard gluing-up argument, we can construct a metric $g_k$ on $\ga$ which satisfies
\begin{equation*}
    g_k=\phi^*_{-k}g \quad\textit{on}\quad \phi_{k}(\ga),
\end{equation*}
and for all $i=0,1,2,...,k-1$,
\begin{equation}\label{costco}
    |\nabla^{\ell}(g_k-g)|\le C_0\,\epsilon\quad \textit{on}\quad M,\quad \ell=0,...,m,
\end{equation}
where here and below $C_0$ denotes all positive constants that only depend on the soliton and $m$.

Now fixed a point $p\in\ga\subset M$ and let $p_k:=\phi_{k}(p)$.
Then after passing to a subsequence we may assume that the pointed manifolds $(\ga,g_k,p_k)$ converge to a smooth manifold $(M',g',p')$.
At the same time, the isometric embeddings $\phi_{k}:(\ga,g,p)\ri(\phi_{k}(\ga),g_k,p_k)\subset(M,g_k,p_k)$ smoothly converge to an isometric embedding $\pi:(\ga,g,p)\rightarrow(M',g',p')$ in the $C^m$-sense.

Furthermore, by \eqref{costco}, we see that $(M',g',p')$ is $C_0\epsilon$-close to the smooth limit of $(M,g,p_k)$, which by Lemma \ref{l: DR to cigar} must be isometric to $\RR\times S^1$ after a suitable rescaling. In particular, this implies that $(M',g')$ is complete, diffeomorphic to $\RR\times S^1$, and covered by $C_0\epsilon$-cylindrical planes.

\end{proof}

We now use the $\epsilon$-grafted soliton metric $g'$ from Lemma \ref{l: lemma one} to generate a family of metrics $g'(t)$ on $M'$, which satisfies the Ricci flow equation in a staircase region of $M'\times[0,\infty)$. 
In future proofs in this section, the flow $(M',g'(t))$ will be used as domains of harmonic map heat flows.

First, let $X$ be a smooth vector field such that $X=\nabla f$ on $\Gamma_{\ge A+200}$, and $X=0$ on $M'\setminus\ga$, and $|\nabla^mX|\le C_0$, $m=0,1,...,100$.
Second, let $\psi_{t}$ be the family of diffeomorphisms generated by $X$ with $\psi_{0}=\textnormal{id}$. 
Let 
\begin{equation*}
    g'(t)=\psi_{t}^*g',\quad t\ge0,
\end{equation*}
be a smooth family of metrics on $M'$.
Then $(M',g'(t))$ is covered by  $\epsilon$-cylindrical planes everywhere for all $t\ge0$.

For a subset $U\subset M'$ and $t\ge0$, let
\begin{equation*}
    U^t=\{x\in M': x\in\psi_t(U)\}.
\end{equation*}
Replacing $A$ by $A+200$, then we have
$X=\nabla f$, $\psi_t=\phi_t$, and $g'=g$ on $\Gamma_{>A}$.
Note also that $\phi_t(\Gamma_{>A})\subset\Gamma_{>A}$, we have that for any $B\ge A$, 
\begin{equation*}\begin{split}
    \Gamma^t_{>B}&=\phi_t(\Gamma_{>B})=\{x\in \Gamma_{>A}: d_{g(t)}(x,\Gamma)>B\},
\end{split}\end{equation*}
where in the second equality we identified $\Gamma_{>A}\subset M'$ with $\Gamma_{>A}\subset M$, and used the fact that $g(t)=\phi_t^*g$.
Since $g(t)$ satisfies the Ricci flow equation, it follows that the Ricci flow equation holds on the open subset
\begin{equation*}
    \bigcup_{t\ge0}\,(\Gamma^t_{> A}\times\{t\})\subset M'\times[0,\infty).
\end{equation*}
We call $(M',g')$ the $\epsilon$-grafted soliton, and $(M',g'(t))$ the $\epsilon$-grafted soliton flow.

\subsection{Surgery on SO(2)-symmetric metrics}
\label{ss: Surgery on the background SO(2)-symmetric metrics part 1}
The next lemma allows us to do a surgery on a $SO(2)$-symmetric metric $\widehat{g}$ on an open subset containing $\Gamma_{\ge B}$ for some large $B>0$.
This surgery extends the incomplete $SO(2)$-symmetric metric $\widehat{g}$ to a complete $SO(2)$-symmetric metric.
Moreover, if $\widehat{g}$ is close to the soliton metric $g$, then the resulting complete metric is close to the grafted soliton metric $g'$.

In future proofs in this section, we will run harmonic map heat flows from the grafted soliton flow $(M',g'(t))$ to Ricci flows starting from some suitable $SO(2)$-symmetric metrics we obtained by the surgery.

\begin{lem}[A global $SO(2)$-symmetric metric]\label{l: lemma two}
There are constants $C_0>0$ such that the following holds:

For any $\epsilon>0$, let $A>0$ and $(M',g')$ be the $\epsilon$-grafted soliton from Lemma \ref{l: lemma one}. Then for any $B>A$, suppose $\widehat{g}$ is a $SO(2)$-symmetric metric on an open subset $U\supset\Gamma_{\ge B}$ with the $SO(2)$-isometry $\psi_{\theta}$ and the killing field $X$, such that
\begin{equation}\label{e: banana}
    |\widehat{g}-g'|_{C^{100}(U)}\le \epsilon,\quad |X-\partial_{\theta}|_{C^0(U)}<\frac{1}{1000},
\end{equation}
where $\partial_{\theta}$ is the killing field of an $\epsilon$-cylindrical plane. 
Then there is a $SO(2)$-symmetric metric $\widetilde{g}$ on $M'$ with the $SO(2)$-isometry $\widetilde{\psi}_{\theta}$ and the killing field $\widetilde{X}$, such that
\begin{equation*}
    |\widetilde{g}-g'|_{C^{98}(M')}\le C_0\,\epsilon\quad\textit{and}\quad|\widetilde{X}-\partial_{\theta}|_{C^{98}(M')}\le C_0\,\epsilon.
\end{equation*}
Moreover, we have $\widetilde{g}=\widehat{g}$ and $\widetilde{\psi}_{\theta}=\psi_{\theta}$ on $\Gamma_{\ge B+100}$.
\end{lem}

\begin{proof}[Proof of Lemma \ref{l: lemma two}]
Let $\psi_{\theta}$, $\theta\in[0,2\pi)$ be the $SO(2)$-isometry of $\widehat{g}$.
It is easy to find a covering of $M'$ by a sequence of $\epsilon$-cylindrical planes $\{U_i\}_{i=1}^{\infty}$ and $U_0=U$ such that the number of them intersecting at any point is bounded by a universal constant.
Then we can find a partition of unity $\{h_i\}_{i=0}^{\infty}$ subordinate to $\{U_i\}_{i=0}^{\infty}$ such that the function $h_0$ satisfies $h_0(\psi_{\theta}(x))=1$ for all $x\in\Gamma_{\ge B+100}$. 
Now the assertions follow immediately from applying Lemma \ref{l: glue cylindrical plane}.
\end{proof}

\subsection{An approximating metric away from the edge}\label{ss: the desired exponential decay part 2}
In this subsection, we construct a $SO(2)$-symmetric approximation metric away from the edge $\Gamma$ such that the error decays at the rate $O(e^{-2(1+\epsilon_1)d_g(\cdot,\Gamma)})$.

In the proof of Theorem \ref{t: precise approximation}, we will need to choose some constants that are sufficiently large or small such that certain requirements are satisfied. In order to show that the dependence between these constants is not circular, we introduce the following parameter order, 
\begin{equation*}
    \delta_0,C_0,T,\underline{D},\epsilon,A,D.
\end{equation*}
such that each parameter is chosen depending only on the preceding parameters.

\begin{theorem}[Approximation with a good exponential decay]\label{t: precise approximation}
Let $(M,g,f,p)$ be a 3D steady gradient soliton that is not a Bryant soliton satisfying \eqref{e: dinner}. Then there exist a constant $\epsilon_1,A_1>0$ and a $SO(2)$-symmetric metric $\widehat{g}$ on an open subset containing $\Gamma_{\ge A_1}$ such that $\Gamma_{\ge A_1}$ is covered by $\epsilon$-cylindrical planes, and
\begin{equation*}
    |\nabla^m(g-\widehat{g})|\le  e^{-2(1+\epsilon_1)d_g(\cdot,\Gamma)}\quad\textit{on}\quad\Gamma_{\ge A_1},\quad m=0,...,100.
\end{equation*}
Moreover, let $X$ be the killing field of the $SO(2)$-isometry of $\widehat{g}$ and $\partial_{\theta}$ be the $SO(2)$-killing field of an $\epsilon$-cylindrical plane, then we have $|X-\partial_{\theta}|\le\frac{1}{1000}$.
\end{theorem}

\begin{proof}
We choose the following constants which satisfies the above parameter order, and whose values may be further adjusted later: 
\begin{enumerate}
    \item Let $\delta_0>0$ be from Theorem \ref{t: symmetry improvement}. 
    \item Let $C_0>0$ be the maximum of $1$ and the constants $C_0>0$ from Lemma \ref{l: two epsilon cylindrical planes are close}, \ref{l: glue cylindrical plane}, and \ref{l: lemma two}. 
    \item Let $T>T_0$, where $T_0$ is from Theorem \ref{t: symmetry improvement}. 
    Assume also that $2C_0^2\cdot e^{400}<e^{\frac{\delta_0}{2}T}$ and $T^{1/2}>\frac{160}{\delta_0}$.
    Let $\underline{D}(T)>0$ be determined by Theorem \ref{t: symmetry improvement}.
    \item\label{i: 3} Let $0<\epsilon<\min\{\frac{1}{1000C_0}\frac{\overline{\delta}(T)}{C_0},\frac{\overline{\epsilon}(T)}{e^{400}C_0^2}\}$, where $\overline{\delta}(T),\overline{\epsilon}(T)>0$ are constants determined by Theorem \ref{t: symmetry improvement}.
    \item Let $A>0$ be sufficiently large so that
    \begin{enumerate}
        \item $g'=g$ on $\Gamma_{\ge A}\subset M'$, where $(M',g')$ is the $\frac{\epsilon}{C_0}$-grafted soliton from Lemma \ref{l: lemma one} and $\Gamma_{\ge A}$ is covered by $\frac{\epsilon}{C_0}$-cylindrical planes. Let $(M',g'(t))$ be the $\frac{\epsilon}{C_0}$-grafted soliton flow, which it satisfies the Ricci flow equation on
        \begin{equation*}
            \bigcup_{t\ge0}\Gamma^t_{> A}\subset M'\times[0,\infty).
        \end{equation*}
        \item By Theorem \ref{l: wing-like} and the assumption \eqref{e: dinner} we may assume that for any point $x\in\Gamma_{\ge A}\subset M$ and $t\ge0$, we have
    \begin{equation*}
        2-\frac{\delta_0}{16}\le \frac{d}{dt}d_g(\phi_t(x),\Gamma)\le2+\frac{\delta_0}{16}.
    \end{equation*}
    \end{enumerate}
    \item Let $D=\max\{A,\ln C_0,\ln\epsilon^{-1},10\underline{D},100T^{1/2},e^{400}\}$, and $\frac{\ln C_0}{D}<0.01$. 
\end{enumerate}

We will impose two inductive assumptions. The first one produces a finite sequence of metrics $\widehat{g}_n$ on $\Gamma_{\ge D_n}$ until $n$ is sufficiently large so that $\widehat{g}_n$ satisfies the assertion of the theorem.
For each fixed $n$, suppose the first inductive assumption holds for $n$, then the second inductive assumption produces an infinite sequence of metrics $\widehat{g}_{n,i}$, $i=0,1,2,...$, on $\Gamma_{\ge D_n}$, where $\widehat{g}_{n,0}=\widehat{g}_{n}$. We will then take a limit of these metrics $\widehat{g}_{n,i}$ as $i\rii$ and obtain a metric $\widehat{g}_{n+1}$ that satisfies the first inductive assumption for $n+1$.

We will see that all metrics in the proof are $\frac{1}{2000C_0}$-close to $(\RR\times S^1,g_{stan})$. So all the following derivatives and norms at a point $x$ are taken and measured with respect to $(\phi_x^{-1})^*g_{stan}$, where $\phi_x$ is a $\frac{1}{2000C_0}$-isometry at $x$. Note that for a different choice of $\phi_x$, the estimates differ at most by the factor $1.1$.

\textbf{Inductive assumption one:} 
For any $n\in\mathbb{N}$ such that $\frac{n\ln C_0}{D}\le2.02$,
there are a sequence of increasing constants $D_n>0$ and a $SO(2)$-symmetric metric $\widehat{g}_{n}$ on $\Gamma_{\ge D_n}$ such that for $\alpha_n:=\frac{n\ln C_0}{D}\le2.02$ we have
\begin{equation}\label{e: induction on n}
    |\nabla^m(\widehat{g}_{n}-g)|\le \epsilon\cdot e^{-\alpha_n(d_g(\cdot,\Gamma)-D_n)} \quad \textit{on}\quad \Gamma_{\ge D_n}\quad m=0,...,100.
\end{equation}
Moreover, let $X_n$ be the killing field of the $SO(2)$-isometry of $\widehat{g}_n$ and $\partial_{\theta}$ be the $SO(2)$-killing field of an $\epsilon$-cylindrical plane, then we have
\begin{equation}\label{e: indone-vector}
    |X_n-\partial_{\theta}|\le\frac{1}{1000}.
\end{equation}
Suppose inductive assumption one is true for a moment. Since $\frac{\ln C_0}{D}<0.01$, we can find an integer $n$ such that $2<\frac{n\ln C_0}{D}\le2.02$. Then the metric $\widehat{g}_n$ on $\Gamma_{\ge D_n}$ satisfies the assertion of the theorem, with $\epsilon_1=\frac{1}{2}(\frac{n\ln C_0}{D}-2)$ and $A_1=D_n$.
So the theorem follows immediately after establishing inductive assumption one.

First, for $n=0$,
since $(M',g')$ is covered everywhere by $\frac{\epsilon}{C_0}$-cylindrical planes and $g'=g$ on $\Gamma_{\ge A}$, by applying Lemma \ref{l: glue cylindrical plane} we obtain a $SO(2)$-symmetric metric on $M'$ which is covered by $C_0\epsilon$-cylindrical planes, and its restriction on $\Gamma_{\ge D_0}$ satisfies the inductive assumptions for some $D_0\ge A+D$.
Now suppose the inductive assumption holds for $n\ge 0$, in the rest of the proof we show that it also holds for $n+1$.
Without loss of generality, we may assume $\frac{(n+1)\delta_0T}{2D}<2.02$, because otherwise we are done.
Now we impose a second inductive assumption.

\textbf{Inductive assumption two:}
Let $n\ge0$ be fixed.
Then for any $k\in\mathbb{N}$ there exists a $SO(2)$-symmetric metric $\widehat{g}_{n,k}$ on an open subset in $M$ containing $\Gamma_{\ge D_n}$, which satisfies
\begin{equation}\begin{split}\label{e: induction on i}
|\nabla^m(\widehat{g}_{n,k}-g)|\le \epsilon\cdot C_0^{-i}\cdot e^{-\alpha_n( d_g(\cdot,\Gamma)-D_n)}\quad \textit{on}\quad\Gamma_{\ge D_n+iD}\quad\\
    \textit{for}\quad i=0,...,k,\quad m=0,...,100.
\end{split}\end{equation}
Moreover, let $X_{n,k}$ be the killing field of the $SO(2)$-isometry of $\widehat{g}_{n,k}$ and $\partial_{\theta}$ be the $SO(2)$-killing field of an $\epsilon$-cylindrical plane, then we have
\begin{equation}\label{e: vector n,k}
    |X_{n,k}-\partial_{\theta}|\le\frac{1}{1000}.
\end{equation}

For $k=0$, inductive assumption two clearly holds for $\widehat{g}_{n,0}=\widehat{g}_n$.
Now assume it is true for an integer $k\ge0$, we will show it also holds for $k+1$.
First, since
\begin{equation*}
   |\nabla^m(\widehat{g}_{n,k}-g)|\le \epsilon\quad
    \textit{on}\quad \Gamma_{\ge D_n},\quad m=0,...,100,
\end{equation*}
by applying Lemma \ref{l: lemma two} we obtain a $SO(2)$-symmetric metric $\widetilde{g}_{n,k}$ on $M'$ with the $SO(2)$-isometry $\psi_{n,k}$ and the killing field $\widetilde{X}_{n,k}$, such that $\widetilde{g}_{n,k}=\widehat{g}_{n,k}$ on $\Gamma_{\ge D_n+100}$ and 
\begin{equation}\label{e: global epsilon}
    |\nabla^m(\widetilde{g}_{n,k}-g')|\le C_0\epsilon,\quad\textit{and}\quad|\widetilde{X}_{n,k}-\partial_{\theta}|\le C_0\epsilon\quad\textit{on}\quad M',\quad m=0,...,98.
\end{equation}

Moreover, we claim that the following holds:
\begin{equation}\label{e: induction on i1}
    \begin{split}
|\nabla^m(\widetilde{g}_{n,k}-g)|\le C_0 \cdot e^{400}\cdot \epsilon\cdot C_0^{-i}\cdot e^{-\alpha_n( d_g(\cdot,\Gamma)-D_n)}\quad \textit{on}\quad\Gamma_{\ge D_n+iD}\quad\\
    \textit{for}\quad i=-1,0,...,k,\quad m=0,...,98.
\end{split}
\end{equation}
To show this, for all $i\ge 1$, since $\widetilde{g}_{n,k}=\widehat{g}_{n,k}$ on $\Gamma_{\ge D_n+100}$, and $D\ge 100$, the claim clearly holds by \eqref{e: induction on i}. 
For $i=-1$, the claim follows directly from \eqref{e: global epsilon}.
For $i=0$, to show the claim holds on $\Gamma_{\ge D_n}$, on the one hand we note that on $\Gamma_{\ge D_n+100}$, we have $\widetilde{g}_{n,k}=\widehat{g}_{n,k}$, and thus the claim holds by \eqref{e: induction on i}; On the other hand, on $\Gamma_{[D_n, D_n+100]}$, we have $d_g(\cdot,\Gamma)-D_n\le 100$, so the claim \eqref{e: induction on i1} follows from \eqref{e: global epsilon} and $\alpha_n<4$.

Let $(M',\widetilde{g}_{n,k}(t))$ be the Ricci flow that starts from $\widetilde{g}_{n,k}(0)=\widetilde{g}_{n,k}$. 
We may take $\epsilon$ sufficiently small so that $(M',\widetilde{g}_{n,k}(t))$ exists up to time $T$ and $|\Rm|_{\widetilde{g}(t)}\le\frac{1}{T}$ for all $t\in[0,T]$,
there is a smooth harmonic map heat flow $\{\chi_{n,k,t}\}:(M',g'(t))\ri(M',\widetilde{g}_{n,k}(t))$, $t\in[0,T]$, with $\chi_{n,k,0}=\textit{i.d.}$, and the perturbation $h_{n,k}(t):=(\chi_{n,k,t}^{-1})^*g'(t)-\widetilde{g}_{n,k}(t)$ satisfies $|h_{n,k}(t)|\le\frac{1}{1000}$.
See \cite[Lemma A.24]{bamler2022uniqueness} for the existence of harmonic map heat flows and estimates of perturbations.

For the fixed $n$ and $k$, we will omit the subscripts $n,k$ in $\chi_{n,k,t},\widetilde{g}_{n,k}(t),h_{n,k}(t)$ for a moment. 
For a fixed $i=0,1,...,k+1$, let 
\begin{equation*}
    x\in\Gamma^T_{\ge D_n+iD}\quad\textit{and}\quad x'=\chi_T(x).
\end{equation*}
We will apply Theorem \ref{t: symmetry improvement} (symmetry improvement) at $(x',0)$ with suitable constants that will be determined later. 
In the following we will
verify all assumptions of Theorem \ref{t: symmetry improvement} and determine the constants. We first prove the following claim.
\begin{claim}\label{claim: carro}
For any $L>10 T^{1/2}$, we have $B_{\widetilde{g}(0)}(x',L)\subset \chi_t(B_{g'(t)}(x,10\,L))$ for all $t\in[0,T]$.
\end{claim}

\begin{proof}
First, we observe that by $|h|\le\frac{1}{1000}$ we have
\begin{equation}\label{e: beef0}
    B_{\widetilde{g}(t)}(\chi_t(x),5L)\subset \chi_t(B_{g'(t)}(x,10\,L)).
\end{equation}
Now let $y\in B_0(x',L)$, by triangle inequality we have
\begin{equation}\label{e: beef}
    \begin{split}
        d_{\widetilde{g}(t)}(y,\chi_t(x))&\le d_{\widetilde{g}(t)}(x',\chi_t(x))+d_{\widetilde{g}(t)}(x',y).\\
    \end{split}
\end{equation}
On the one hand, by the local drift estimate of harmonic map heat flows \cite[Lemma A.18]{bamler2022uniqueness}, and the curvature bound $|\Rm|<\frac{1}{T}$, we have
\begin{equation}\label{e: beef1}
    d_{\widetilde{g}(t)}(x',\chi_t(x))=d_{\widetilde{g}(t)}(\chi_t(x),\chi_T(x))\le 10\,(T-t)^{1/2}\le 10\,T^{1/2}<L.
\end{equation}
On the other hand, by the distance distortion estimate on the Ricci flow $(M',\widetilde{g}_k(t))$ under the curvature bound $|\Rm|<\frac{1}{T}$, we have
\begin{equation}\label{e: beef2}
    d_{\widetilde{g}(t)}(x',y)\le 2\,e^{\sup|\Rm|\cdot T}\,d_{\widetilde{g}(0)}(x',y) \le4\,d_{\widetilde{g}(0)}(x',y)\le 4\,L.
\end{equation}
Combining \eqref{e: beef}\eqref{e: beef1}\eqref{e: beef2} we obtain $d_{\widetilde{g}(t)}(y,\chi_t(x))\le 5\,L$, which by \eqref{e: beef0} implies the claim.
\end{proof}

First, we verify assumption \eqref{e: alpha} in Theorem \ref{t: symmetry improvement}. 
Let 
\begin{equation*}
    \mathcal{H}(x')=C_0\cdot e^{400}\cdot\epsilon\cdot C_0^{-(i-1)}\cdot e^{-\alpha_n( d_g(x',\Gamma)-D_n)}.
\end{equation*}
Let $y\in\Gamma_{\ge D_n+(i-1)D}$, 
then by the triangle inequality we have
\begin{equation*}
    d_g(y,\Gamma)\ge d_g(x',\Gamma)-d_g(y,x'),
\end{equation*}
and thus by \eqref{e: induction on i1} we obtain
\begin{equation}\label{e: star}
    \begin{split}
        |\nabla^m h|(y,0)&\le
        \mathcal{H}(x')\cdot 
        e^{\alpha_n\, d_g(y,x')}\quad\textit{on}\quad \Gamma_{\ge D_n+(i-1)D},\quad m=0,...,98.
    \end{split}
\end{equation}
Since $x\in\Gamma^T_{\ge D_n+iD}$, by the definition of the flow $g'(t)$ we see that $x\in\Gamma^t_{\ge D_n+iD}$ for all $t\in[0,T]$. In particular, we have $x\in \Gamma^0_{\ge D_n+iD}=\Gamma_{\ge D_n+iD}$.
So taking $L=D/10$ in Claim \ref{claim: carro} we see that
\begin{equation*} 
    B_{\widetilde{g}(0)}(x',\underline{D})\subset B_{\widetilde{g}(0)}(x',D/10)\subset B_g(x,D)\subset\Gamma_{\ge D_n+(i-1)D}
\end{equation*}
By \eqref{e: star} this verifies the assumption \eqref{e: alpha} in Theorem \ref{t: symmetry improvement} at $(x',0)$, with the constants $\alpha$ there equal to $\alpha_n$, and $\epsilon$ there equal to $\mathcal{H}(x')$.

Second, we verify the assumption of Theorem \ref{t: symmetry improvement} that the perturbation $h$ restricted on $B_{\widetilde{g}(0)}(x',D_{\#})\times[0,T]$ is a Ricci-Deturck flow perturbation for
\begin{equation*}
    D_{\#}=\frac{1}{10}(d_{g(T)}(x,\Gamma)-D_n+D)>\underline{D}.
\end{equation*}
This follows from Claim \ref{claim: carro} because by taking $L=D_{\#}$ we have
\begin{equation*}\begin{split}
    B_{\widetilde{g}(0)}(x',D_{\#})&\subset\chi_t(B_{g'(t)}(x,d_{g(T)}(x,\Gamma)-D_n+D))\\
    &\subset\chi_t(B_{g'(t)}(x,d_{g(t)}(x,\Gamma)-D_n+D))\subset\chi_t(\Gamma^t_{\ge D_n-D})\subset\chi_t(\Gamma^t_{\ge A}),
    \end{split}
\end{equation*}
for all $t\in[0,T]$. 
Moreover, by taking $\epsilon$ sufficiently small we may assume $|\nabla^kh|\le\frac{1}{1000}$, $k=0,1$, on $B_{\widetilde{g}(0)}(x',D_{\#})\times[0,T]$. See \cite[Lemma A.14]{bamler2022uniqueness} for the local derivative estimate of Ricci flow perturbations.

Lastly, we will verify assumption \eqref{e: double} of Theorem \ref{t: symmetry improvement}. Recall that assumption \eqref{e: double} consists of two estimates of $|h|$ on the parabolic boundary of $B_{\widetilde{g}(0)}(x',D_{\#})\times[0,T]=(\partial B_{\widetilde{g}(0)}(x',D_{\#})\times[0,T])\cup (B_{\widetilde{g}(0)}(x',D_{\#})\times\{0\})$. 
We first verify the assumption on $\partial B_{\widetilde{g}(0)}(x',D_{\#})\times[0,T]$, which follows immediately from the following claim and $|h|\le\frac{1}{1000}$.
\begin{claim}
$e^{100 D_{\#}}\cdot \mathcal{H}(x')\ge\frac{1}{1000}$.
\end{claim}

\begin{proof}
Note by \eqref{e: beef} we have $d_g(x,x')\le 10T^{1/2}$, and $d_g(x,\Gamma)\le d_{g(T)}(x,\Gamma)+2.1T$, we have
\begin{equation*}
    d_g(x',\Gamma)\le d_g(x,\Gamma)+d_g(x,x')
    \le d_{g(T)}(x,\Gamma)+2.1T+10T^{1/2}
    =10D_{\#}+2.1T+10T^{1/2}+D_n.
\end{equation*}
Substituting this into $\mathcal{H}(x')$ and seeing that $e^{-4(2.1T+10T^{1/2})}\ge e^{-10T}\ge e^{-D}\ge e^{-10 D_2}$, and $\epsilon\ge e^{-D}\ge e^{-10D_2}$, the claim follows.
\end{proof}

Since $B_{\widetilde{g}(0)}(x',D_{\#})\subset\Gamma_{\ge D_n-D}$, the assumption \eqref{e: double} on $B_{\widetilde{g}(0)}(x',D_{\#})\times\{0\}$ follows immediately from the following claim

\begin{claim}\label{c: check condition 2}
$|\nabla^mh|(y,0)\le  \mathcal{H}(x')\cdot e^{4\,d_g(x',y)}$ for all $y\in \Gamma_{\ge D_n-D}$, $m=0,...,98$.
\end{claim}

\begin{proof}

If $y\in\Gamma_{\ge D_n+(i-1)D}$, then the claim holds by \eqref{e: star}.
Now assume $y\in \Gamma_{[D_n+(j-1)D, D_n+jD)}$ for some $0\le j\le i-1$. Then we have
\begin{equation}\label{e: dij}
    d_g(x',y)\ge (i-j)D,
\end{equation}
which combined with
\eqref{e: induction on i1} implies
\begin{equation}\begin{split}\label{e: hd}
    |\nabla^mh|(y,0)&
    \le \mathcal{H}(x')\cdot  C_0^{i-j}\cdot e^{\alpha_n\,(d_g(x',\Gamma)-d_g(y,\Gamma))}\\
\end{split}\end{equation}
Note that \eqref{e: dij} implies $C_0^{i-j}\le e^{\frac{\ln C_0}{D}\cdot d_g(x',y)}\le e^{0.01d_g(x',y)}$,
using also $d_g(x',y)\ge d_g(x',\Gamma)-d_g(y,\Gamma)$ and \eqref{e: hd} we obtain
the claim.
\end{proof}

Therefore, applying Theorem \ref{t: symmetry improvement} (symmetric improvement) at $(x',0)$ we obtain
\begin{equation}\begin{split}\label{e: first h}
    |\nabla^mh_-|(x',T)&\le\mathcal{H}(x')\cdot e^{2\alpha_n T}\cdot e^{-\delta_0T}\\
    &\le C_0\cdot e^{400}\cdot\epsilon\cdot C_0^{-(i-1)}\cdot e^{-\alpha_n( d_g(x',\Gamma)-D_n)}\cdot e^{2\alpha_n T}\cdot e^{-\delta_0T},
\end{split}\end{equation}
for $m=0,...,100$.
Since $T^{1/2}>\frac{160}{\delta_0}$, using the triangle inequality $d_g(x',\Gamma)-d_g(x,\Gamma)\ge -d_g(x,x')$ and 
\begin{equation*}
    d_g(x',x)=d_g(\chi_{T}(x),x)\le10\,T^{1/2},
\end{equation*}
it is easy to see that $e^{-\frac{\delta_0}{4}T}\cdot e^{-\alpha_n(d_g(x',\Gamma)-D_n)}\le e^{-\alpha_n(d_g(x,\Gamma)-D_n)}$,
which combined with \eqref{e: first h} implies
\begin{equation}\begin{split}\label{e: hen}
    |\nabla^mh_-|(x',T)
    &\le C_0\cdot e^{400}\cdot\epsilon\cdot C_0^{-(i-1)}\cdot e^{-\alpha_n( d_g(x,\Gamma)-D_n)}\cdot e^{2\alpha_n T}\cdot e^{-\frac{3\delta_0}{4}T},
\end{split}\end{equation}
for $m=0,...,100$.
Moreover, by the assumption of $A$ we have
\begin{equation*}
    d_{g(T)}(x,\Gamma)\le d_g(x,\Gamma)-\left(2-\frac{\delta_0}{16}\right)T\le d_g(x,\Gamma)-\left(2-\frac{\delta_0}{4\alpha_n}\right)T,
\end{equation*}
and hence $e^{-\frac{\delta_0}{4}T}\cdot e^{-\alpha_n(d_g(x,\Gamma)-D_n)}\cdot e^{2\alpha_nT}\le e^{-\alpha_n(d_{g(T)}(x,\Gamma)-D_n)}$,
which together with \eqref{e: hen} implies
\begin{equation}\begin{split}\label{e: ce}
    |\nabla^mh_-|(x',T)
    &\le C_0\cdot e^{400}\cdot\epsilon\cdot C_0^{-(i-1)}\cdot e^{-\alpha_n( d_{g(T)}(x,\Gamma)-D_n)}\cdot e^{-\frac{\delta_0}{2}T},
\end{split}\end{equation}
for $m=0,...,100$.
So by the assumption $2C_0^2\cdot e^{400}<e^{\frac{\delta_0}{2}T}$ and \eqref{e: ce} we obtain
\begin{equation*}
    |\nabla^m h_-|(x',T)\le\frac{1}{2}\epsilon\cdot C_0^{-i}\cdot e^{-\alpha_n( d_{g(T)}(x,\Gamma)-D_n)},
\end{equation*}
and hence
\begin{equation}\label{e: pullback h}
    |\nabla^m\chi_T^*h_-|(x,T)\le\epsilon\cdot C_0^{-i}\cdot e^{-\alpha_n( d_{g(T)}(x,\Gamma)-D_n)}.
\end{equation}

Now we restore the subscripts $n,k$.
Since $g'(T)=\chi_{n,k,T}^*(h_{n,k}(T)+\widetilde{g}_{n,k}(T))$ and $h_{n,k}(T)=h_{n,k,+}(T)+h_{n,k,-}(T)$,
by letting 
\begin{equation*}
    \widehat{g}_{n,k+1}=\chi_{n,k,T}^*(h_{n,k,+}(T)+\widetilde{g}_{n,k}(T)),
\end{equation*}
we see that $\widehat{g}_{n,k+1}$ is a $SO(2)$-symmetric metric on $M'$ and \eqref{e: pullback h} implies
\begin{equation}\begin{split}\label{e: like2}
    |\nabla^m(g'(T)-\widehat{g}_{n,k+1})|\le\epsilon\cdot C_0^{-i}\cdot e^{-\alpha_n( d_{g(T)}(\cdot,\Gamma)-D_n)},\quad\textit{on}\quad\Gamma^T_{\ge D_n+iD},
\end{split}\end{equation}
where $m=0,...,100$.
This is true for all $i=0,1,...,k+1$. 

Since $g'(T)=\phi_{-T}^*g$ and $\phi_{-T}:(\Gamma^T_{
\ge A},g'(T))\ri(\Gamma_{\ge A},g)$ is an isometry, replacing $\widehat{g}_{n,k+1}$ by $\phi_T^*(\widehat{g}_{n,k+1})$, then by \eqref{e: like2} we obtain a $SO(2)$-symmetric metric $\widehat{g}_{n,k+1}$ on $M'$ such that
\begin{equation}\label{e: firstind}
    |\nabla^m(g-\widehat{g}_{n,k+1})|\le\epsilon\cdot C_0^{-i}\cdot e^{-\alpha_n( d_{g}(\cdot,\Gamma)-D_n)},\quad\textit{on}\quad\Gamma_{\ge D_n+iD},\quad m=0,...,98.
\end{equation}
This verifies \eqref{e: induction on i} in inductive assumption two for $k+1$.

It remains to verify \eqref{e: vector n,k} in inductive assumption two for $k+1$.
To do this, let $\psi_x:(\RR\times S^1,g_{stan})\ri (\RR\times S^1,\widetilde{g},x)$ be an $C_0\epsilon$-isometry at $x\in \RR\times S^1$, we may assume that $\epsilon$ is sufficiently small so that $\psi_x:(\RR\times S^1,g_{stan})\ri (\RR\times S^1,\widetilde{g}(t),x)$ is an $\frac{1}{2000C_0}$-isometry for all $t\in[0,T]$. 
By \eqref{e: global epsilon} we see
\begin{equation}\label{e: so2killing field}
    |(\chi_{n,k,T}^{-1})_*\widetilde{X}_{n,k}-(\chi_{n,k,T}^{-1})_*\partial_{\theta}|\le 4C_0\epsilon.
\end{equation}
Note that by the uniqueness of Ricci flow, see e.g. \cite{ChenBL}, the $SO(2)$-isometry $\psi_{n,k,\theta}$ of $\widetilde{g}_{n,k}(0)$ is $SO(2)$-isometry of $\widetilde{g}_{n,k}(t)$ for all $t\in[0,T]$.
In particular, $\widetilde{X}_{n,k}$ is the killing field of  $\widetilde{g}_{n,k}(T)$ and hence $(\chi_{n,k,T}^{-1})_*\widetilde{X}_{n,k}$ is the killing field of the $\widehat{g}_{k+1}$.
So the vector field $(\chi_{n,k,T}^{-1})_*\partial_{\theta}$ is the $SO(2)$-killing field of an $\frac{1}{2000C_0}$-cylindrical plane. 
By Lemma \ref{l: two epsilon cylindrical planes are close}, for the $SO(2)$-killing field $\partial_{\theta'}$ of an $\epsilon$-cylindrical plane, we have $|\partial_{\theta'}-(\chi_{n,k,T}^{-1})_*\partial_{\theta}|\le\frac{1}{2000}$, which combined with 
\eqref{e: so2killing field} implies 
\begin{equation*}
    |(\chi_{n,k,T}^{-1})_*\widetilde{X}_{n,k}-\partial_{\theta'}|\le\frac{1}{1000},
\end{equation*}
which confirms \eqref{e: vector n,k} in inductive assumption two for $k+1$.
Combining this with \eqref{e: firstind} we verified inductive assumption two for $k+1$.

Therefore, inductive assumption two holds for all $k\ge0$.
Now let $p\in \Gamma_{> D_n}$ be some fixed point, then we may assume after passing to a subsequence and letting $k\rii$ that the pointed $SO(2)$-symmetric manifolds $(M',\widehat{g}_{n,k},p)$ converge in the $C^{98}$-norm to a $SO(2)$-symmetric manifold $(M',\widehat{g}_{n+1},p)$, which satisfies
\begin{equation}\label{e: readyh}
    |\nabla^m(g-\widehat{g}_{n+1})|\le \epsilon\cdot C_0^{-i}\cdot e^{-\alpha_n(d_g(\cdot,\Gamma)-D_n)} \quad \textit{on}\quad \Gamma_{\ge D_n+iD},
\end{equation}
for $m=0,...,98$.
Moreover, let $X_{n+1}$ be the killing field of the $SO(2)$-isometry of $\widehat{g}_{n+1}$ and $\partial_{\theta}$ be the $SO(2)$-killing field of an $\epsilon$-cylindrical plane, then  $|X_{n+1}-\partial_{\theta}|\le\frac{1}{1000}$, which verifies \eqref{e: indone-vector} for $n+1$.
For any $i\ge0$, and $y\in\Gamma_{[D_n+iD,D_n+(i+1)D)}$, we have $d_g(y,\Gamma)\le D_n+(i+1)D$, which together with \eqref{e: readyh} implies 
\begin{equation}\label{e: readyh2}
    |\nabla^m(g-\widehat{g}_{n+1})|(y)\le \epsilon\cdot e^{-\alpha_{n+1}(d_g(y,\Gamma)-D_{n+1})},
\end{equation}
where $D_{n+1}=\frac{(D+D_n)\ln C_0+\alpha_n D_nD}{\ln C_0+\alpha_nD}>D_n$
and $\alpha_{n+1}=\frac{(n+1)\ln C_0}{D}$. Therefore, we have
\begin{equation*}
    |\nabla^m(g-\widehat{g}_{n+1})|\le \epsilon\cdot e^{-\alpha_{n+1}(d_g(\cdot,\Gamma)-D_{n+1})} \quad \textit{on}\quad \Gamma_{\ge D_{n+1}},
\end{equation*}
for $m=0,...,100$, which verifies inductive assumption one for $n+1$ and thus proves the theorem.

\end{proof}

\subsection{Extend the approximating metric near the edges}
\label{ss: near the edge part 3}
The approximating $SO(2)$-symmetric metric in Theorem \ref{t: precise approximation} is constructed on an open subset away from $\Gamma$.
Next, we want to extend it to a $SO(2)$-symmetric metric which is also defined on a neighborhood of $\Gamma$. 

Seeing that the soliton dimension reduces to $\R\times\cigar$ along the $\Gamma$, we can find a sequence of $SO(2)$-symmetric metrics in balls centered at $\Gamma$ whose radius going to infinity, such that these metrics are $e^{-2(1+\epsilon_1)d_g(\cdot,\Gamma)}$-close to the soliton metric.
Then by using Lemma \ref{l: glue cigar} we can glue up these metrics near $\Gamma$ with the metric $\overline{g}$ at suitable distance to $\Gamma$, and obtain a $SO(2)$-symmetric metric defined everywhere outside of a compact subset of $M$.
This new metric will inherit their closeness to $g$ both near $\Gamma$ and away from $\Gamma$.
To achieve this goal, we need the following lemma which compares an almost killing field $Y$ in $(\R\times\cigar,g_{\Sigma})$ with the actual killing field.

\begin{lem}\label{l: glue cigar}
For any $C_1,\epsilon_1>0$, there exists $C(C_1,\epsilon_1)>0$ such that the following holds:
Write the metric of $(\R\times\cigar,g_{\Sigma})$ as $g_{\Sigma}=ds^2+dr^2+\varphi^2(r)d\theta^2$ under the coordinate $(s,r,\theta)$, where $r$ is the distance to the line $\R\times\{x_{tip}\}$. Suppose $R((0,x_{tip}))=4$. 
Suppose $Y$ is a vector field defined on $A\le r\le B$ for some $0<A<B$, such that
\begin{equation}\label{e: L_Yg}
    |\nabla^k(\LL_{Y}g_{\Sigma})|\le C_1\,e^{-2(1+\epsilon_1)d_g(\cdot,\Gamma)}\quad k=0,...,98.
\end{equation}
Suppose also that 
\begin{equation}\label{e: assump limit}
    \lim_{r\rii}|\nabla^k(Y-\partial_{\theta})|(\cdot,B,\cdot)\le C_1\,e^{-2(1+\epsilon_1)A},\quad k=0,...,98.
\end{equation}
Then we have
\begin{equation*}
    |\nabla^k(Y-\partial_{\theta})|(\cdot,A,\cdot)\le C\,e^{-2(1+\epsilon_1)A}\quad k=0,...,96.
\end{equation*}
\end{lem}

\begin{proof}
Write $Y=Y^s\partial_s+Y^r\partial_r+Y^{\theta}\partial_{\theta}$ under the coordinate $(s,r,\theta)$, then
by the formula of Lie derivatives for a symmetric 2-tensor, we have
\begin{equation}\label{e: lie}
    \begin{split}
        \LL_Y g_{\Sigma}(\partial_{r},\partial_{\theta})&=\partial_{\theta} Y^r+\pr Y^{\theta}\varphi^2;\\
        \LL_Y g_{\Sigma}(\pr,\ps)&=\pr Y^s+\ps Y^r;\\
        \LL_Y g_{\Sigma}(\pr,\pr)&=2\,\pr Y^r.\\
    \end{split}
\end{equation}
Moreover, by assumption \eqref{e: assump limit} we have
\begin{equation*}
   (|Y^s|_{C^k}+|Y^r|_{C^k}+|Y^{\theta}-1|)_{C^k}(\cdot,B,\cdot)\le C\,e^{-2(1+\epsilon_1)A}.
\end{equation*}
By the third equation in \eqref{e: lie} and \eqref{e: L_Yg} and integrating from $r=A$ to $r=B$, we see that $|Y^r|_{C^{k-1}}(\cdot,A,\cdot)\le C\,e^{-2(1+\epsilon_1)A}$. Substituting this into the first two equations in \eqref{e: lie} and integrating from $r=A$ to $r=B$ we obtain $(|Y^{\theta}-1|_{C^{k-2}}+|Y^s|)_{C^{k-2}}(\cdot,A,\cdot)\le C\,e^{-2(1+\epsilon_1)A}$, which proves the lemma.
\end{proof} 

We now prove the main result of this section.

\begin{theorem}\label{c: best approximation}
Let $(M,g,f,p)$ be a 3D steady gradient soliton that is not a Bryant soliton. Assume $\lim_{s\rii}R(\Gamma_1(s))=\lim_{s\rii}R(\Gamma_2(s))= 4$. Then there exist a constant $C,\epsilon_1>0$ and a $SO(2)$-symmetric metric $\overline{g}$ defined outside of a compact subset of $M$ such that for any $k=0,...,98$, the followings hold:
\begin{enumerate}
    \item $|\nabla^k(g-\overline{g})|\rightarrow 0$ as $x\rightarrow\infty$.
    \item $|\nabla^k(g-\overline{g})|\le C\, e^{-2(1+\epsilon_1)d_g(\cdot,\Gamma)}$.
\end{enumerate}
\end{theorem}

\begin{proof}
We will use $\epsilon(D)$ to denote all functions which goes to zero as $D\rii$.
Let $C>0$ denote all constants that depending only on the soliton.
On the one hand, since the manifold dimension reduces along $\Gamma$ to $\R\times\cigar$, for each $i=1,2$, by a standard gluing-up argument we can find a $SO(2)$-symmetric metric $\widetilde{g}_i$ on an open subset $U_i$ containing the balls $B_g(\Gamma_i(s),D_1(s))$ for all large $s$, such that $\lim_{s\rii}D_1(s)=\infty$, and the followings hold:
\begin{enumerate}
    \item Let $\epsilon_1>0$ be from Theorem \ref{t: precise approximation}. Then for each $i=1,2$, we have
    \begin{equation}\label{e: Li}
        |\nabla^k(\widetilde{g}_i-g)|\le\min\{\epsilon(d_g(\cdot,p)),\, C\,e^{-2(1+\epsilon_1)d_g(\cdot,\Gamma)}\},\quad\textit{on}\quad U_i
    \end{equation}
   \item $U_1\cap U_2=\emptyset$.
   \item Let $\psi_{i,\theta}$, $\theta\in[0,2\pi)$, be the $SO(2)$-isometries for $(U_i,\widetilde{g}_i)$, $i=1,2$,
we can also assume that there is an embedded surface $N_i$ in $U_i\cap\Gamma_{\ge 1000}$, which is diffeomorphic to $\RR$, and intersects each $S^1$-orbit of $\psi_{i,\theta}$ exactly once, and its tangent space $T_xN_i$ at $x\in N_i$ is $\frac{1}{100}$-close to the orthogonal space of the $S^1$-orbit passing through $x$.
\item For each large $s$, there is a smooth map $\psi_{i,s}$ from $B_g(\Gamma_i(s),D(s))$ into $(\R\times\cigar,g_{\Sigma})$ which is a diffeomorphism onto the image, such that 
\begin{equation}\label{e: liying}
    |\nabla^k(\widetilde{g}_i-\psi_{i,s}^*g_{\Sigma})|\le C\,e^{-2(1+\epsilon_1)d_g(\cdot,\Gamma)}.
\end{equation}
Let $X_i$ be the killing field of the $SO(2)$-isometry of $\widetilde{g}_i$, then
\begin{equation}\label{e: X}
    |\nabla^k(X_i-(\psi_{i,s})^{-1}_*(\partial_{\theta}))|\le C\,e^{-2(1+\epsilon_1)d_g(\cdot,\Gamma)}.
\end{equation}
\end{enumerate}

On the other hand, by Theorem \ref{t: precise approximation} we have a $SO(2)$-symmetric metric $\widehat{g}$ defined on an open subset $U\supset\Gamma_{\ge A}$ for some $A>0$ such that 
\begin{equation}\label{e: ying}
    |\nabla^k(g-\widehat{g})|\le C\, e^{-2(1+\epsilon_1)d_g(\cdot,\Gamma)}.
\end{equation}
Let $\psi_{\theta}$, $\theta\in[0,2\pi)$, be the $SO(2)$-isometries for $(U,\widehat{g})$, and
$Y$ be the killing field of $\psi_{\theta}$. 
Next, we will compare the two vector fields $X_i$ and $Y$ on $\Gamma_{> A}\cap U_i$.
First, by \eqref{e: Li},\eqref{e: liying},\eqref{e: ying} we have 
\begin{equation}\label{e: stem}
    |\nabla^k(\LL_Y\psi^*_{i,s}g_{\Sigma})|\le C\, e^{-2(1+\epsilon_1)d_g(\cdot,\Gamma)},
\end{equation}
Second, we also have
\begin{equation*}
    |\nabla^k(\widehat{g}-\psi^*_{i,s}g_{\Sigma})|\le|\nabla^k(\widehat{g}-g)|+|\nabla^k(g-\widetilde{g}_i)|+|\nabla^k(\widetilde{g}_i-\psi^*_{i,s}g_{\Sigma})|\le\epsilon(d_g(\cdot,\Gamma)).
\end{equation*}
Moreover, the manifold is an $\epsilon(d_g(\cdot,\Gamma))$-cylindrical plane at all points where the metrics $\widehat{g}$ and $\psi^*_{i,s}g_{\Sigma}$ are defined. 
So we can apply Lemma \ref{l: killing fields are close} and deduce
\begin{equation}\label{e: shanzhuxiao}
    |\nabla^k(Y-(\psi_{i,s})^{-1}_*(\partial_{\theta}))|\le\epsilon(d_g(\cdot,\Gamma)).
\end{equation}
In particular, for any $A>0$, there exists $B>A$ such that $\epsilon(B)<e^{-2(1+\epsilon_1)A}$.

Therefore, for each $i=1,2$,
by \eqref{e: stem} and \eqref{e: shanzhuxiao} we can apply Lemma \ref{l: glue cigar} and deduce that there is a function $D_2:[s_0,\infty)\ri\R_+$ for a sufficiently large $s_0$ such that $B_g(\Gamma_i(s),D_2(s))\subset U_i$ and $D_2(s)\rii$ as $s\rii$, and the following holds on $B_g(\Gamma_i(s),D_2(s))$,
\begin{equation*}
    |\nabla^k(Y-(\psi_{i,s})^{-1}_*(\partial_{\theta}))|\le C\, e^{-2(1+\epsilon_1)d_g(\cdot,\Gamma)}.
\end{equation*}
Combining this with \eqref{e: X} we see that the following holds on $V_i=\cup_{s> s_0}B_g(\Gamma(s),D_2(s))$,
\begin{equation}\label{e: paufu}
    |\nabla^k(Y-X_i)|\le C\, e^{-2(1+\epsilon_1)d_g(\cdot,\Gamma)}, 
\end{equation}
which implies
\begin{equation*}
    |\nabla^k(\psi-\psi_i)|\le C\, e^{-2(1+\epsilon_1)d_g(\cdot,\Gamma)}.
\end{equation*}

Let $h: V_1\cup V_2\cup U\ri[0,1]$ be a smooth cut-off function satisfying $h\equiv1$ on a $\psi_{i,\theta}$-invariant open subset $W_i\subset V_i$ which contains $B_g(\Gamma_i(s),D_3(s))$ with $D_3(s)\rii$ as $s\rii$ for each $i=1,2$, and $h\equiv0$ on a $\psi_{\theta}$-invariant open subset containing $U\setminus (V_1\cup V_2)$.
Let $V_{i,0}$ be a $\psi_{i,\theta}$-invariant open subset such that
\begin{equation*}
    W_i\cap\Gamma_{\ge 100}\subset V_{i,0}\subset V_i\cap\Gamma_{\ge 100}.
\end{equation*}

Next, we define a smooth map $F_i:N_i\times S^1\ri M$ by letting 
\begin{equation*}
    F_i(x,\theta)=\Sigma_2(1-h,h,\psi_{\theta}(x),\psi_{i,\theta}(x)),
\end{equation*}
where $x\in N_i$, $\theta\in[0,2\pi)$.
Then in the same way as in Lemma \ref{l: glue up local diffeomorphisms from N x S1 to M} we can show that $F_i$ is a diffeomorphism onto the image, and
\begin{equation*}
    |\nabla^k(F_{i*}(\partial_{\theta})-X_i)|\le C\,e^{-2(1+\epsilon_1)d_g(\cdot,\Gamma)},
\end{equation*}
Let $\Psi_{i,\theta}$ be the $2\pi$-periodic one parameter group of diffeomorphisms generated by  $F_{i*}(\partial_{\theta})$, $i=1,2$.
Since $V_1\cap V_2=\emptyset$, we have $\Psi_{1,\theta}=\Psi_{2,\theta}$ on $U\setminus(V_1\cup V_2)$. Therefore, we obtain a $2\pi$-periodic one parameter group of diffeomorphisms $\Psi_{\theta}$ on $V_1\cup V_2\cup U$, which satisfies
\begin{enumerate}
    \item $\Psi_{\theta}=\psi_{i,\theta}$ on $W_i$.
    \item $\Psi_{\theta}=\psi_{\theta}$ on $U\setminus(V_1\cup V_2)$.
    \item $|\Psi-\psi|_{C^k(U\times S^1)}+|\Psi-\psi_i|_{C^k(V_i\times S^1)}\le C\,e^{-2(1+\epsilon_1)d_g(\cdot,\Gamma)}$.
\end{enumerate}
So by letting
\begin{equation*}
    \overline{g}=\frac{1}{2\pi}\int_0^{2\pi}\Psi^*_{\theta}(h\cdot \widetilde{g}_i+(1-h)\cdot\widehat{g})\,d\theta.
\end{equation*}
we obtain a $SO(2)$-symmetric metric on $V_1\cup V_2\cup U$, which contains the complement of a compact subset. Moreover, $\overline{g}$ satisfies the following properties:
\begin{enumerate}
    \item $\overline{g}=\widetilde{g}_i$ on $W_i$, for each $i=1,2$.
    \item $\overline{g}=\widehat{g}$ on $U\setminus (V_1\cup V_2)$.
    \item For some $D_0>0$, we have
    \begin{equation}\label{e: happy}
    \begin{cases}
            |\overline{g}-g|\le C\,e^{-2(1+\epsilon_1)d_g(\cdot,\Gamma)}\quad\textit{on}\quad \Gamma_{\ge D_0}\\
            |\overline{g}-g|\le\epsilon(d_g(\cdot,p))\quad\textit{on}\quad W_1\cup W_2.
    \end{cases}
\end{equation}
\end{enumerate}
The first inequality implies that $\overline{g}$ satisfies the assertion of exponential decay away from $\Gamma$.
Moreover, since $d_g(x,p)\rii$ as $d_g(x,\Gamma)\rii$ for any $x\in M\setminus (W_1\cup W_2)$, the first inequality implies $|\overline{g}-g|\le\epsilon(d_g(\cdot,p))$ also holds on $M\setminus (W_1\cup W_2)$. So $\overline{g}$ satisfies the assertion of decaying to zero at infinity.

\end{proof}

\section{The evolution of the Lie derivative}\label{s: lie}
In this section, $(M,g)$ is a 3D steady gradient soliton that is not a Bryant soliton, and $(M,g(t))$ is the Ricci flow of the soliton. Let $h(t)$ be a linearized Ricci-DeTurck flow with background metric $g(t)$. The main result is Proposition \ref{p: h decays to zero}, which shows that $h(t)$ tends to zero as $t$ goes to infinity, if the initial value $h(0)$ satisfies the condition $\frac{h(x,0)}{R(x)}\ri0$ as $x\rii$. 
In particular, let $\overline{g}$ be the approximating $SO(2)$-symmetric metric obtained from Theorem \ref{c: best approximation}, and $\partial_{\theta}$, $\theta\in[0,2\pi)$, be the killing field of the $SO(2)$-symmetry. We show that the Lie derivative $\LL_{\partial_{\theta}}g$ satisfies this initial condition, hence decays to zero as $t\rii$ under the linearized Ricci-DeTurck equation.

\subsection{The vanishing of a heat kernel at infinity}
In this subsection we prove a vanishing theorem of the heat kernel to a certain heat type equation at time infinity.
We will see that for a linearized Ricci Deturck flow $h(t)$, the norm $|h|(\cdot,t)$ is controlled by the convoluted integral of this heat kernel and $|h|(\cdot,0)$.

Let $G$ be the heat kernel of the heat type equation 
\begin{equation}\label{e: special}
    \pt H=\Delta H+\frac{2|\Ric|^2}{R}H.
\end{equation}
That is,  for any $t>s$, $x,y\in M$, 
\begin{equation*}\begin{split}
    \pt G(x,t;y,s)&=\Delta_{x,t} G(x,t;y,s)+\frac{2|\Ric|^2(x,t)}{R(x,t)}G(x,t;y,s),\\
    \lim_{t\searrow s}G(\cdot,t;y,s)&=\delta_{y}.
\end{split}\end{equation*}

\begin{lem}[Vanishing of heat kernel at time infinity]\label{l: vanishing of heat kernel}
Let $p$ be the point where $R$ attains the maximum. For any fixed $D>0$, let 
\begin{equation*}
    u_D(x,t)=\int_{B_0(p,D)}G(x,t;y,0)\,d_0y,
\end{equation*}
then $\sup_{B_{t}(p,D)}u_D(\cdot,t)\ri0$ as $t\rii$.
\end{lem}

\begin{proof}
Note that $u_D$ satisfies the equation \eqref{e: special}.
First, we show that there exists $C_1>0$ such that $u_D(x,t)\le C_1$ for all $(x,t)\in M\times[0,\infty)$.
Since the scalar curvature satisfies the equation $\partial_t R=\Delta R+\frac{2|\Ric|^2}{R}\cdot R$, by using the reproduction formula we have
\begin{equation*}
    R(x,t)=\int_M G(x,t;y,0)R(y,0)\,d_0y.
\end{equation*}
By compactness we have for some $c>0$ that $R(y,0)\ge c$ for all $y\in B_0(p,D)$, so it follows that $u_D(x,t)\le c^{-1}R(x,t)\le c^{-1}R(p)$. So we may take $C_1=c^{-1}R(p)$.

Now suppose the claim does not hold, then there exist $\epsilon>0$ and a sequence of $t_i\rii$ and $x_i\in B_{t_i}(p,D)$ such that $u_D(x_i,t_i)\ge\epsilon>0$. Without loss of generality we may assume that $t_{i+1}\ge t_i+1$.
Since $u\le C_1$, by a standard parabolic estimate we see that $|\partial_t u_D|(x,t)+|\nabla u_D|(x,t)\le C_2$ for some $C_2>0$.
Therefore, there exists $\delta_1\in(0,1)$ such that
\begin{equation}\label{e: this}
    \int_{B_t(p,D)}\int_{B_0(p,D)}G(x,t;y,0)\,d_0y\,d_tx=\int_{B_t(p,D)}u(x,t)\,d_tx\ge\delta_1,
\end{equation}
for all $t\in [t_i,t_i+\delta_1]$.

Since $\Rm>0$, we can choose an orthonormal basis $\{e_1,e_2,e_3\}$ such that $\Rm(e_1\wedge e_2)=-\lambda_3\,e_1\wedge e_2$, $\Rm(e_1\wedge e_3)=-\lambda_2\,e_1\wedge e_3$, $\Rm(e_2\wedge e_3)=-\lambda_1\,e_2\wedge e_3$ for some $\lambda_1,\lambda_2,\lambda_3>0$. So it is easy to see  $2|\Ric|^2=2(\lambda_2+\lambda_3)^2+2(\lambda_1+\lambda_3)^2+2(\lambda_1+\lambda_2)^2$ and $R^2=4(\lambda_1+\lambda_2+\lambda_3)^2$, which implies
\begin{equation*}
  2|\Ric|^2-R^2<0.
\end{equation*}
So by compactness there is $\delta_2>0$ such that
\begin{equation}\label{e: onehand}
    \sup_{B_t(p,D)}\frac{2|\Ric|^2-R^2}{R}\le-\delta_2.
\end{equation}

Let $F_D(t)=\int_{B_0(p,D)}\int_M G(x,t;y,0)\,d_tx\,d_0y$, then since
\begin{equation}\label{e: integratey}
\begin{split}
    \pt&\int_M G(x,t;y,0)\,d_tx
    =\int_M\pt G(x,t;y,0)-R(x,t)G(x,t;y,0)\,d_tx\\
    =&\int_M\Delta_{x,t} G(x,t;y,0)+\frac{2|\Ric|^2(x,t)}{R(x,t)}G(x,t;y,0)-R(x,t)G(x,t;y,0)\,d_tx\\
    =&\int_M     \frac{G(x,t;y,0)}{R(x,t)}(2|\Ric|^2(x,t)-R^2(x,t))\,d_tx,
\end{split}
\end{equation}
where we used the fact that the heat kernel $G$ satisfies a Gaussian upper bound so that $\int_{M}\Delta_{x,t} G(x,t;y,0)\,d_tx$ vanishes by the divergence theorem. Hence we obtain
\begin{equation}\label{e: sum}
    \pt F_D(t)=\int_{B_0(p,D)}\int_M \frac{G(x,t;y,0)}{R(x,t)}(2|\Ric|^2-R^2)(x,t)\,d_tx\,d_0y<0,
\end{equation}
and by \eqref{e: this} and \eqref{e: onehand} we see that $\pt F_D(t)\le-\delta_1\delta_2$ for all $t\in[t_i,t_i+\delta_1]$.
Note $G$ is everywhere positive so that $F_D(t)>0$, this implies
\begin{equation*}
\begin{split}
F_D(0)&<\lim_{t\rii}F_D(t)-F_D(0)\le -\sum_{i=1}^{\infty}\int_{t_i}^{t_i+\delta_1}\delta_1\cdot\delta_2\,dt=-\sum_{i=1}^{\infty}\delta_1^2\cdot \delta_2,
\end{split}
\end{equation*}
which is a contradiction. So this proves the lemma.
\end{proof}

\subsection{The vanishing of the Lie derivative at infinity}
We prove the main result in this subsection by applying the heat kernel estimates.
First, we prove a lemma using the Anderson-Chow pinching estimate. 

\begin{lem}[c.f.\cite{AC}]\label{l: AC}
Let $(M,g(t))$, $t\in[0,T]$, be a 3D Ricci flow with positive sectional curvature. Consider a solution $h$ to the linearized Ricci-Deturck flow on $(M,g(t))$, and a positive solution $H$ to the following equation
\begin{equation}\label{e: H evolve}
    \partial_t H=\Delta H+\frac{2|\Ric|^2(x,t)}{R(x,t)}H.
\end{equation}
Then the following holds:
\begin{equation*}
    \partial_t\left(\frac{|h|^2}{H^2}\right)\le\Delta\left(\frac{|h|^2}{H^2}\right)+2\nabla H\cdot\nabla\left(\frac{|h|^2}{H^2}\right).
\end{equation*}
\end{lem}

\begin{proof}
By a direct computation using $\partial_th=\Delta_{L,g(t)}h$ and \eqref{e: H evolve} we have, see \cite{AC},
\begin{equation*}\begin{split}
    \partial_t\left(\frac{|h|^2}{H^2}\right)=\Delta\left(\frac{|h|^2}{H^2}\right)+2\nabla H\cdot\nabla\left(\frac{|h|^2}{H^2}\right)+\frac{4}{H^2}\left(R_{ijkl}h_{il}h_{jk}-\frac{|\Ric|^2}{R}|h|^2\right)\\
    -2\frac{|H\nabla_ih_{jk}-(\nabla_iH)h_{jk}|^2}{H^4}.
\end{split}\end{equation*}
Then the lemma follows immediately form the pinching estimate from \cite{AC} that for any non-zero symmetric 2-tensor $h$,
\begin{equation*}
    \frac{\Rm(h,h)}{|h|^2}\le\frac{|\Ric|^2}{R}.
\end{equation*}
\end{proof}

We now prove the main result of this section. Note that $|h|(\cdot,t)$ is controlled by the convoluted integral of the heat kernel of \eqref{e: special},
we split the integral into two parts, where in the compact region, the integral tends to zero as a consequence of our vanishing theorem. In the non-compact subset, we use the assumption $\frac{h(x,0)}{R(x)}\ri0$ as $x\rii$ and the reproduction formula of the scalar curvature to deduce that the integral is bounded above by arbitrarily small multiples of the scalar curvature. 

\begin{prop}\label{p: h decays to zero}
Let $(M,g(t))$ be the Ricci flow of a 3D steady gradient Ricci soliton that is not a Bryant soliton. Consider a solution $h(t)$, $t\in[0,\infty)$, to the linearized Ricci-Deturck flow on $M$, i.e.
\begin{equation*}
    \partial_t h(t)=\Delta_{L,g(t)}h(t).
\end{equation*}
Suppose $h$ satisfies the following initial condition
\begin{equation*}
    \frac{|h|(x,0)}{R(x)}\ri0\quad \textit{as}\quad x\rii.
\end{equation*}
Then $|h|(x,t)$ converges to $0$ smoothly and uniformly on any compact subset $B_t(p,D)$ as $t\rii$.
\end{prop}

\begin{proof}

Now let $H(x,t)=\int_M G(x,t;y,0)|h|(y,0)\,d_0y$, 
then $H$ solves the equation
\begin{equation*}
    \partial_t H=\Delta H+\frac{2|\Ric|^2(x,t)}{R(x,t)}H.
\end{equation*}
Therefore, by Lemma \ref{l: AC} and applying the weak maximum principle to $\frac{|h|^2}{H^2}$, we see that $|h|\le H$, that is, 
\begin{equation}\label{e: compare}
    |h(x,t)|\le\int_M G(x,t;y,0)|h|(y,0)\,d_0y.
\end{equation}

For any $\epsilon>0$, by the assumption on $|h|(\cdot,0)$, we can find some $D>0$ such that $|h(y,0)|\le\epsilon R(y,0)$ for all $y\in M\setminus B_0(p,D)$. We may assume $D>\frac{1}{\epsilon}$.
So by \eqref{e: compare} and using Lemma \ref{l: vanishing of heat kernel} (vanishing of heat kernel) we have for all sufficiently large $t$ and $x\in B_{t}(p,D)$ that    
\begin{equation*}
\begin{split}
    |h(x,t)|&\le\int_{B_0(p,D)}G(x,t;y,0)|h|(y,0)\,d_0y+\int_{M\setminus B_0(p,D)}G(x,t;y,0)|h|(y,0)\,d_0y\\
    &\le\epsilon+\epsilon\int_MG(x,t;y,0)R(y,0)\,d_0y\\
    &=\epsilon+\epsilon R(x,t)\le\epsilon(1+R(p)),
\end{split}
\end{equation*}
This implies $\sup_{B_t(p,D)}|h|(\cdot,t)\le\epsilon(1+R(p))$ for all large $t$, and the assertion follows by letting $\epsilon\ri0$. 
\end{proof}

As a direct application, we prove the following

\begin{cor}\label{c: Lie derivative tends to zero}
Let $\overline{g}$ be the $SO(2)$-symmetric metric from Theorem \ref{c: best approximation}, and $X$ be a vector field on $M$ which is a smooth extension of the killing field of the $SO(2)$-isometry of $\overline{g}$, and $X$ has a bounded norm. 
Let $h(t)$ be the solution to the following initial value problem of the linearized Ricci-Deturck flow
\begin{equation*}\begin{cases}
    \pt h(t)=\Delta_{L,g(t)}h(t),\\
    h(0)=\LL_Xg.
\end{cases}\end{equation*}
Then $|h|(x,t)$ converges to $0$ smoothly and uniformly on any compact subset $B_t(p,D)$ as $t\rii$.
\end{cor}

\begin{proof}
By Proposition \ref{p: h decays to zero} it suffices to show that $\frac{|h|(x,0)}{R(x)}\ri0$ as $x\rii$. We claim this is true.
On the one hand, by Theorem \ref{c: best approximation}, there exists $\epsilon_1,C_1>0$ such that the following holds,
\begin{equation}\label{e: sit}
\begin{cases}
        |h(0)|\le C_1\cdot e^{-2(1+\epsilon_1)d_g(\cdot,\Gamma)}\quad\textit{on}\quad M,\\
        |h(x,0)|\ri0\quad\textit{as}\quad x\rii.
\end{cases}\end{equation}
On the other hand, by Theorem \ref{l: R>e^{-2r}} (Scalar curvature exponential lower bound) we can find a constant $C_2>0$ depending only on $\epsilon_1$ such that 
\begin{equation}\label{e: straight}
    R\ge C_2^{-1}e^{-2(1+\frac{\epsilon_1}{2})d_g(\cdot,\Gamma)}.
\end{equation}

For any $\epsilon>0$, by the second condition in \eqref{e: sit} we can find $D(\epsilon)>0$ such that $|h|(x,0)\le\epsilon$ for all $x\in M\setminus B_g(p,D(\epsilon))$. First, if $x\in \Gamma_{\le L(\epsilon)}$, where $L(\epsilon)=\frac{\ln\frac{C}{\sqrt{\epsilon}}}{2(1+\epsilon_1)}$, we have 
\begin{equation*}
    |h|(x,0)\le\epsilon=\sqrt{\epsilon}\cdot C\cdot e^{-2(1+\epsilon_1)L(\epsilon)}\le\sqrt{\epsilon}\cdot R(x).
\end{equation*}
Second, if $x\in \Gamma_{\ge L(\epsilon)}$, then by the first condition in \eqref{e: sit} and \eqref{e: straight} we obtain 
\begin{equation*}
    |h|(x,0)\le C_1\cdot e^{-\epsilon_1\,d_g(x,\Gamma)}\cdot e^{-2(1+\frac{\epsilon_1}{2})\,d_g(x,\Gamma)}\le C_1C_2\cdot e^{-\epsilon_1\,L(\epsilon)}\cdot R(x).
\end{equation*}
Note $L(\epsilon)\rii$ as $\epsilon\ri0$, the claim follows immediately.
\end{proof}

\section{Construction of a killing field}\label{s: killing field}
Let $(M,g)$ be a 3D steady gradient soliton that is not a Bryant soliton, and $(M,g(t))$, $t\in(-\infty,\infty)$, be the Ricci flow of the soliton.
In this section, we study the evolution of a vector field $X(t)$ under the equation 
\begin{equation}\label{e: X(t)}
    \partial_t X(t)=\Delta_{g(t)} X(t)+\Ric_{g(t)}(X),
\end{equation}
In particular, we will choose $X(t)$ such that $X(0)$ to be the killing field of the $SO(2)$-isometry of the approximating metric obtained from Theorem \ref{c: best approximation}, and show that $X(t_i)$ converges to a non-zero killing field of the soliton $(M,g)$ for a sequence $t_i\rii$.

Throughout this section, we assume 
\begin{equation*}
    \lim_{s\rii}R(\Gamma_1(s))=\lim_{s\rii}R(\Gamma_2(s))= 4.
\end{equation*} and $\overline{g}$ is the $SO(2)$-symmetric metric from Theorem \ref{c: best approximation} which satisfies
\begin{equation}\label{e: quote}
  |\nabla^k(g-\overline{g})|\le C\, e^{-2(1+\epsilon_1)d_g(\cdot,\Gamma)},\quad k=0,...,98.
\end{equation}

Let $A>0$ be sufficiently large so that $\Gamma_{\ge A}$ is covered by $\epsilon$-cylindrical planes. In particular, the $SO(2)$-isometry $\psi_{\theta}$ of $\overline{g}$ acts freely on an open subset $U\supset\Gamma_{\ge A}$.
So we can find a 2D manifold $(N,g_N)$ and a Riemannian submersion $\pi:(U,g)\ri(N,g_N)$ which maps a $S^1$-orbit to a point in $N$.

Let $\rho:(-1,1)^2\ri B\subset N$ be a local coordinate at $p\in N$ such that $\rho(0,0)=p$, and $s:B\ri U$ be  a section of the Riemannian submersion $\pi$. Then the map $\Phi:(-1,1)^2\times[0,2\pi)\ri U$ defined by $\Phi(x,y,\theta)=\psi_{\theta}(s(\rho(x,y)))$, gives a local coordinate at $s(p)\in U$ under which $\overline{g}$ can be written as follows,
\begin{equation*}
    \overline{g}=\sum_{\alpha,\beta=x,y}g_{\alpha\beta}\,d{\alpha}\,d{\beta}+G(d\theta+A_xdx+A_ydy)^2,
\end{equation*}
where $G, A_x,A_y, g_{\alpha\beta}$ are functions that are independent of $\theta$.
Then $Y=\partial_{\theta}$.
Note that a change of section changes the connection form $A=A_xdx+A_ydy$ by an exact form, and leaves invariant the curvature form
\begin{equation*}
    dA=(\partial_{x}A_y-\partial_{y}A_x)dx\wedge dy=F_{xy}\,dx\wedge dy.
\end{equation*}

\begin{lem}\label{l: dA and nabla dA}
For $k=0,...,96$, there are $C_k>0$ such that for all $q\in U$ we have
\begin{equation}\label{e: nabla dA0}
    |\widetilde{\nabla}^{\ell} ( dA)|(\pi(q))\le\frac{C_k}{d_g^k(q,\Gamma)},\quad \ell=0,1,
\end{equation}
and 
\begin{equation}\label{e: nablavarphi}
    |\widetilde{\nabla}^{\ell}G^{1/2}|(\pi(q))\le\frac{C_k}{d_g^k(q,\Gamma)},\quad\ell=1,2,
\end{equation}
for $k=0,...,96$, where $\widetilde{\nabla}$ denotes the covariant derivative on $(N,g_N)$.
\end{lem}

\begin{proof}
We adopt the notion such that for a tensor $\tau=\tau^{i_1...i_r}_{j_1...j_s}\,dx^{j_1}\otimes\cdots \otimes dx^{j_s}\otimes \partial_{x_{i_1}}\otimes \cdots\otimes\partial_{x_{i_r}}$ on $B$, we have $\tau^{i_1...i_r}_{j_1...j_s,k}=\partial_{x_k}(\tau^{i_1...i_r}_{j_1...j_s})$, $\tau^{i_1...i_r}_{j_1...j_s,k\ell}=\partial^2_{x_k,x_{\ell}}(\tau^{i_1...i_r}_{j_1...j_s})$, and
\begin{equation*}\begin{split}
    \widetilde{\nabla}_{\partial_{x_k}}\tau&=\tau^{i_1...i_r}_{j_1...j_s\,;\,k}\;dx^k\otimes dx^{j_1}\otimes\cdots \otimes dx^{j_s}\otimes \partial_{x_{i_1}}\otimes \cdots\otimes\partial_{x_{i_r}}\\
     \widetilde{\nabla}^2_{\partial_{x_{k},x_{\ell}}}\tau&=\tau^{i_1...i_r}_{j_1...j_s\,;\,k\ell}\;dx^k\otimes dx^{\ell}\otimes dx^{j_1}\otimes\cdots \otimes dx^{j_s}\otimes \partial_{x_{i_1}}\otimes \cdots\otimes\partial_{x_{i_r}}.
\end{split}\end{equation*}

For a point $p$ in the base manifold parametrized by $x,y$, it is convenient to choose the section so that $A(p)=0$. Then the non-zero components of the curvature tensor $\overline{R}_{IJKL}$ of $\overline{g}$ are given in terms of the components of the curvature tensor $R_{\alpha\beta\gamma\delta}$ of $(B,g_N)$, the components $F_{xy}$, and the function $G$ by the following, see \cite[Section 4.2]{Lott2007DimensionalRA},
\begin{equation}\label{e: DR}
\begin{split}
    \overline{R}_{\theta\alpha\theta\beta}&=-\frac{1}{2}G_{;\alpha\beta}+\frac{1}{4}G^{-1}G_{,\beta}G_{,\alpha}+\frac{1}{4}g^{\alpha\beta}G^2F_{xy}^2,\\
    \overline{R}_{\theta\alpha\beta\alpha}&=-\frac{1}{2}GF_{xy;\alpha}-\frac{3}{4}G_{,\alpha}F_{xy},\\
    \overline{R}_{\alpha\beta\alpha\beta}&=R_{\alpha\beta\alpha\beta}-\frac{3}{4}GF_{xy}^2.
\end{split}
\end{equation}

Let $R_{IJKL}$ be the components of the curvature tensor of the soliton metric $g$, then by Theorem \ref{t: R upper bd} we have $|R_{IJKL}|\le\frac{C_k}{d^k_g(\cdot,\Gamma)}$ for any $k\in\mathbb{N}$ and $C_k>0$. So by \eqref{e: quote} we obtain $|\overline{R}_{IJKL}|\le\frac{C_k}{d^k_g(\cdot,\Gamma)}$ after replacing $C_k$ by a possibly larger number. 
In particular, by the second equation in \eqref{e: DR} we obtain
\begin{equation}\label{e: nabla dA}
    |\widetilde{\nabla} (G^{\frac{3}{2}} dA)|(\pi(q))\le\frac{C_k}{d_g^k(q,\Gamma)},
\end{equation}
for all $q\in U$.
By Kato's inequality this implies 
\begin{equation}\label{e: im}
    |\widetilde{\nabla} |G^{\frac{3}{2}} dA||(\pi(q))\le|\widetilde{\nabla} (G^{\frac{3}{2}} dA)|(\pi(q))\le\frac{C_k}{d_g^k(q,\Gamma)}.
\end{equation}
By Theorem \ref{l: wing-like} there exists $C>0$ such that 
\begin{equation*}
    d_g(\phi_t(q),\Gamma)\ge d_g(q,\Gamma)+C^{-1}\,t,\quad t\ge0,
\end{equation*}
for any point $q\in\Gamma_{\ge A}$ where $A>0$ is sufficiently large.
Note that $(M,g,\phi_t(q))$ converge to $\RR\times S^1$ as $t\rii$, we have
$\lim_{t\rii}|G^{\frac{3}{2}}dA|(\pi(\phi_t(q)))=0$.
Since by \eqref{e: im} we have
\begin{equation}
    \left|\frac{d}{dt}\,|G^{\frac{3}{2}}dA|(\pi(\phi_t(q)))\right|=\left|\langle\widetilde{\nabla}|G^{\frac{3}{2}}dA|,\pi_*(\nabla f(\phi_t(q)))\rangle\right|\le \frac{C_k}{d_g^k(\phi_t(q),\Gamma)},
\end{equation}
integrating which from zero to infinity, we see that there is some $C_{k-1}>0$ such that
\begin{equation*}
    |G^{\frac{3}{2}}dA|(\pi(q))\le\frac{C_{k-1}}{d_g^{k-1}(q,\Gamma)}.
\end{equation*}
This together with \eqref{e: nabla dA} implies \eqref{e: nabla dA0}.
It also implies $|\frac{1}{4}g^{\alpha\beta}G^2F_{xy}^2|(\pi(q))\le\frac{C_k}{d_g^k(q,\Gamma)}$
in the first equation in \eqref{e: DR}, and hence implies $|\widetilde{\nabla}^2G^{1/2}|(\pi(q))\le\frac{C_k}{d_g^k(q,\Gamma)}$.
Similarly as before, we obtain $|\widetilde{\nabla}G^{1/2}|(\pi(q))\le\frac{C_k}{d_g^k(q,\Gamma)}$ by an integration.

\end{proof}

Let $\partial_{\theta}$ be the killing field of the $SO(2)$-isometry of $\overline{g}$ outside of a compact subset of $M$. We can extend it to a smooth vector field $Y$ on $M$ such that $|Y|\le 10$. Let $Y(t)=\phi_{t*}Y$ for all $t\ge0$, and 
\begin{equation*}
    Q(t)=-\partial_t(Y(t))+\Delta_{g(t)}Y(t)+\Ric_{g(t)}(Y(t)).
\end{equation*}
We will often abbreviate it as $Q(t)=-\partial_tY+\Delta Y+\Ric(U)$ when there is no confusion. Next, we show that $Q(t)$ has a polynomial decay away from $\Gamma$.
\begin{lem}\label{l: Q(t)}
For $k=0,...,94$, there are $C_k>0$ such that 
\begin{equation*}
    |Q(x,t)|\le\frac{C_k}{d^k_t(x,\Gamma)},
\end{equation*}
for all $t\ge0$.
\end{lem}

\begin{proof}

Since $g(t)=\phi^*_{-t}g$, $\phi_t:(M,g)\ri(M,g(t))$ is an isometry, it follows
\begin{equation*}
    \partial_t Y(t)=\partial_t(\phi_{t*}Y)=\phi_{t*}(\LL_{\nabla f}Y),\quad 
    \Delta_{g(t)}Y(t)=\phi_{t*}(\Delta Y),\quad 
    \Ric_{g(t)}(Y(t))=\phi_{t*}(\Ric(Y)).
\end{equation*}
So the lemma reduces to show the following:
\begin{equation*}
    |-\LL_{\nabla f}Y+\Delta Y+\Ric(Y)|\le\frac{C_k}{d^k_g(\cdot,\Gamma)}.
\end{equation*}
Let $h=\LL_Yg$, then by a direct computation we have the identity, see e.g. \cite{brendlesteady3d},
\begin{equation}\label{e: oxtail}
    div(h)-\frac{1}{2}tr\nabla h=\Delta Y+\Ric(Y).
\end{equation}
Since by \eqref{e: quote} we have
\begin{equation*}
    |\nabla^{k-1}h|=|\nabla^{k-1}(\LL_Yg)|=|\nabla^{k-1}(\LL_Y(g-\overline{g}))|\le C\cdot e^{-2(1+\epsilon_1)d_g(\cdot,\Gamma)},
\end{equation*}
it then follows that 
\begin{equation*}
    |div(h)|+|tr\nabla h|\le C\cdot e^{-2(1+\epsilon_1)d_g(\cdot,\Gamma)}.
\end{equation*}
Therefore, by \eqref{e: oxtail} and Theorem \ref{t: R upper bd} (scalar curvature polynomial upper bounds) 
\begin{equation}\label{water}
    |\Delta Y+\Ric(Y)|\le C\cdot e^{-2(1+\epsilon_1)d_g(\cdot,\Gamma)}\le\frac{C_k}{d^k_g(\cdot,\Gamma)}.
\end{equation}
So the lemma reduces to estimate $|\LL_{\nabla f}Y|$.

To do this, we assume $\nabla f=F^{\theta}\partial_{\theta}+F^{\alpha}\partial_{\alpha}$, $\alpha=x,y$, where $F^{\theta}=\partial_{\theta} f\cdot G^{-1}$. 
Then we can compute that
\begin{equation}\label{e: Y derivative}
    \LL_{\nabla f}Y=\LL_{\nabla f}\partial_{\theta}=\partial_{\theta}F^{\theta}\cdot\partial_{\theta}+\partial_{\theta}F^{\alpha}\cdot\partial_{\alpha}.
\end{equation}
We will see in the following equations that the two components $\partial_{\theta}F^{\theta}$ and $\partial_{\theta}F^{\alpha}$ also appear in the components of $\LL_{\nabla f}\overline{g}$ which we will compare to. 
Replace the coordinate function $\theta$ by $\theta+\int_0^{y}\int_0^{x}\partial_{y}A_x(x',y')\,dx'\,dy'\mod 2\pi$, then by \eqref{e: nabla dA0} we have
\begin{equation*}
    \overline{g}=\sum_{\alpha,\beta=x,y}g_{\alpha\beta}\,d{\alpha}\,d{\beta}+Gd\theta^2+\overline{h},
\end{equation*}
where $\overline{h}$ is a 2-tensor satisfies $|\nabla^{\ell}\overline{h}|\le \frac{C_k}{d_g^k(\cdot,\Gamma)}$, $\ell=0,1$. In particular, this implies
\begin{equation*}
    |\overline{g}_{\alpha\theta}|+|\partial_{\beta}\overline{g}_{\alpha\theta}|\le\frac{C_k}{d_g^k(\cdot,\Gamma)},\quad \alpha,\beta=x,y.
\end{equation*}
So by a direct computation we obtain
\begin{equation}\label{e: g derivative}
    \begin{split}
        (\LL_{\nabla f}\overline{g})_{\theta \beta}&
        =2\,\partial_{\theta}F^{\alpha}\cdot\overline{g}_{\alpha\beta}-F^{\theta}\partial_{\beta}G+O(d_g^{-k}(\cdot,\Gamma)),\\
        (\LL_{\nabla f}\overline{g})_{\theta\theta}&=F^{\alpha}\partial_{\alpha}G+2(\partial_{\theta}F^{\theta})\cdot G+O(d_g^{-k}(\cdot,\Gamma)).
    \end{split}
\end{equation}
where $O(d_g^{-k}(\cdot,\Gamma))$ denotes functions that are bounded by $\frac{C_k}{d_g^{k}(\cdot,\Gamma)}$ in absolute values.

Since $\frac{1}{2}\LL_{\nabla f}g=\nabla^2 f=\Ric$, we have
\begin{equation*}\begin{split}
    |\LL_{\nabla f}\overline{g}|&\le|\LL_{\nabla f}(\overline{g}-g)|+|\LL_{\nabla f}g|\le C\cdot e^{-2(1+\epsilon_1)d_g(\cdot,\Gamma)}+2|\Ric_g|
    \le \frac{C_k}{d_g^{k}(\cdot,\Gamma)}.
\end{split}\end{equation*}
Therefore, by comparing \eqref{e: Y derivative} and \eqref{e: g derivative} we can deduce 
\begin{equation*}
    |\LL_{\nabla f}Y|\le \frac{C_k}{d_g^{k}(\cdot,\Gamma)}+C\cdot|\widetilde{\nabla} G|,
\end{equation*}
which combined with \eqref{e: nablavarphi} implies
\begin{equation*}
    |\LL_{\nabla f}Y|\le\frac{C_k}{d_g^k(\cdot,\Gamma)},
\end{equation*}
which proves the lemma.
\end{proof}

Let $Z(t)$ be a vector field which solves the following equation
\begin{equation}\label{e: Z equation}
\begin{cases}
    -\partial_t Z+\Delta Z+\Ric(Z)=Q(t),\\
    Z(0)=0.
\end{cases}
\end{equation}
In the next lemma, we show that $Z(t)$ has a polynomial decay away from $\Gamma$.

\begin{lem}\label{l: Z equation}
For $k=0,...,92$, there are $C_k>0$ such that $|Z(t)|\le\frac{C_k}{d^k_t(\cdot, \Gamma)}$.
\end{lem}

\begin{proof}
We can compute that
\begin{equation*}\begin{split}
    \Delta|Z|^2&=2\langle\Delta Z, Z\rangle+2\langle\nabla Z, \nabla Z\rangle\\
    \partial_t|Z|^2&=2\langle\partial_t Z, Z\rangle-2\,\Ric(Z,Z),
\end{split}\end{equation*}
combining which with \eqref{e: Z equation} we obtain 
\begin{equation*}
    \partial_t|Z(t)|^2=\Delta_{g(t)}|Z(t)|^2-2|\nabla_{g(t)}Z(t)|_{g(t)}^2-2\langle Q(t),Z(t)\rangle_{g(t)}, 
\end{equation*}\
which we will often abbreviate as $\partial_t|Z|^2=\Delta|Z|^2-2|\nabla Z|^2-2\langle Q,Z\rangle$.
Similarly, we can show $\partial_t|X|^2=\Delta|X|^2-2|\nabla X|^2\le\Delta|X|^2$.
So by the maximum principle we get $|X(t)|\le|X(0)|\le C$, and hence
\begin{equation*}
    |Z(t)|\le|Y(t)|+|X(t)|\le C.
\end{equation*}
So $\partial_t|Z|^2\le\Delta|Z|^2+C\cdot|Q|\le\Delta|Z|^2+\frac{C}{d^m_t(\cdot, \Gamma)}$. The lemma now follows immediately from the following lemma.
\end{proof}

\begin{lem}
Let $(M,g(t))$ be a 3D steady gradient soliton that is not a Bryant soliton. Let $u: M\times[0,T]\rii$ be a smooth non-negative function, which satisfies $u(\cdot,0)=0$, and
\begin{equation*}
    \partial_tu\le\Delta u+\frac{C_0}{d_t^k(\cdot,\Gamma)},
\end{equation*}
for some $k\ge 2$. Then there exists $C=C(C_0,k)>0$ such that $u(\cdot,t)\le\frac{C}{d_t^{k-1}(\cdot,\Gamma)}$.
\end{lem}

\begin{proof}
Let $C>0$ denote all constants depending on $C_0,k$. Denote $d_t(x,\Gamma)$ by $r(x,t)$, which satisfies the distance distortion estimates \eqref{e: distance change}. By the maximum principle and the integral formula for solutions of heat type equations, we obtain
\begin{equation*}
    u(x,t)\le\int_0^t\int_{M}G(x,t;y,s)\frac{C_0}{r^k(y,s)}\,d_sy\,ds,
\end{equation*}
where $G(x,t;y,s)$ is the heat kernel of the heat equation under $g(t)$, see \eqref{e: standard heat kernel}.

For a fixed $s\in[0,t]$, we split the integral $\int_{M}G(x,t;y,s)\frac{C_0}{r^k(y,s)}\,d_sy$ into two integrals on $B_s(x,\frac{r(x,s)}{1000})$ and $M\setminus B_s(x,\frac{r(x,s)}{1000})$, and denote them respectively by $I(s)$ and $II(s)$. 
We will estimate them similarly as in the proof of Theorem \ref{t: R upper bd}.
For $II(s)$, note $\frac{d_s^2(y,x)}{t-s}\ge\frac{r(x,s)}{C}$ for all $y\in M\setminus B_s(x,\frac{r(x,s)}{1000})$, by the heat kernel estimates Lemma \ref{l: L-geodesic} and Lemma \ref{l: heat kernel lower bound implies upper bound} we obtain
\begin{equation*}
    II(s)\le C\int_{M\setminus B_s(x,\frac{r(x,s)}{1000})}e^{-\frac{d_s^2(y,x)}{C(t-s)}}\,d_sy
    \le C\cdot e^{-\frac{r(x,s)}{C}}\le\frac{C}{r^k(x,s)}.
\end{equation*}
For $I(s)$, we have $r(y,s)\ge\frac{r(x,s)}{2}$ for all $y\in B_{s}(x,\frac{r(x,s)}{1000})$, and thus
\begin{equation*}
    I(s)\le C\sup_{y\in B_s(x,\frac{r(x,s)}{1000})}r^{-k}(y,s)\le\frac{C}{r^{k}(x,s)}.
\end{equation*}
Therefore, by \eqref{e: distance change}
and the estimates of $I(s)$ and $II(s)$, we obtain 
\begin{equation*}
    u(x,t)\le\int_0^t\frac{C}{r^k(x,s)}\,ds
    \le C\left(\frac{1}{r^{k-1}(x,t)}-\frac{1}{(r(x,t)+1.9\,t)^{k-1}}\right)\le\frac{C}{r^{k-1}(x,t)}.
\end{equation*}
\end{proof}

Now we prove the main result of this section, which finds a non-trivial killing field of $(M,g)$ as time goes to infinity.

\begin{prop}\label{l: non-vanishing}
There exists a vector field $X_{\infty}$ which does not vanish everywhere and has bounded norm such that $\LL_{X_{\infty}}g=0$.
\end{prop}

\begin{proof}
Let $X(t)=Y(t)-Z(t)$ and $\widetilde{X}(t)=\phi_{-t*}X(t)$.
We will show that there exists a sequence $t_i\rii$ such that the vector fields $\widetilde{X}(t_i)$ on $M$ smoothly converge to a non-zero killing field $X_{\infty}$.

By the definitions of $Z(t)$ and $X(t)$, it is easy to see
\begin{equation*}
\begin{cases}
    \partial_t X(t)=\Delta X(t)+\Ric(X(t)),\\
    X(0)=Y(0).
\end{cases}
\end{equation*}
Let $h(t)=\LL_{X(t)}g(t)$, then a direct computation shows that $h(t)$ satisfies the  linearized Ricci-DeTurck equation
\begin{equation*}
    \begin{cases}
    \partial_t h(t)=\Delta_L h(t),\\
    h(0)=\LL_Xg=\LL_Yg.
\end{cases}
\end{equation*}

Note we have the isometry 
\begin{equation*}
    (M,g(t),p,X(t),h(t))\xrightarrow[]{\phi_{-t}} (M,g,p,\widetilde{X}(t),\widetilde{h}(t)).
\end{equation*}
So $\widetilde{h}(t):=\phi_t^*h(t)=\LL_{\widetilde{X}(t)}g$. 

By Theorem \ref{l: wing-like}, for any $\epsilon>0$, the manifold is covered by $\epsilon$-cylindrical planes on scale $1$ on $\Gamma_{\ge A}$ for sufficiently large $A$. So we may pick a point $q\in M$ such that $||Y|(q,0)-1|\le\epsilon$ and $r(q,0)>2C_1+1$, where $C_1>0$ is the constant from Lemma \ref{l: Z equation}.
Then by the definition of $Y$ we have $|Y|(\phi_t(q),t)=|Y|(q,0)\ge1-\epsilon$, and by Lemma \ref{l: Z equation} we have
$|Z|(\phi_t(q),t)=|Z|(q,0)\le\frac{1}{2}$.
Therefore,
\begin{equation*}
    |\widetilde{X}|_g(q,t)=|X|_{g(t)}(\phi_t(q),t)\ge |Y|_{g(t)}(\phi_t(q),t)-|Z|_{g(t)}(\phi_t(q),t)\ge\frac{1}{2}-\epsilon.
\end{equation*}

Next, by the $C^0$-upper bound $|X|(t)\le C$, and the standard interior estimates for linear parabolic equations; see e.g. \cite[Theorem 7.22]{lieberman1996second}, we have uniform $C^k$-upper bounds on $|\widetilde{X}|(t)$.
Therefore, by the Arzella-Ascoli theorem, there exists $t_i\rii$ such that $\widetilde{X}(t_i)$ smoothly uniformly converges to a vector field $X_{\infty}$ and correspondingly $\widetilde{h}(t_i)$ converges to a smooth symmetric 2-tensor $\LL_{X_{\infty}}g$.

First, we have $X_{\infty}\neq0$, because $|X_{\infty}|(q)=|\widetilde{X}|(q,t_i)\ge\frac{1}{2}-\epsilon$. 
Moreover, by Corollary \ref{c: Lie derivative tends to zero} we see that $\widetilde{h}(t_i)$ converges to $0$ smoothly and uniformly on any compact subsets of $M$, which implies $\LL_{X_{\infty}}g=0$.
\end{proof}

\section{Proof of the O(2)-symmetry}\label{s: O(2)-symmetry}
In this section we prove the $O(2)$-symmetry for all 3D steady gradient solitons that are not the Bryant soliton. By Proposition \ref{l: non-vanishing} we find a non-zero smooth vector field $X$ such that $\LL_Xg=0$. We will show that $X$ induces an isometric $O(2)$-action $\{\chi_{\theta}\}_{\theta\in[0,2\pi)}$.
Throughout this section we assume $(M,g,f,p)$ is a 3D steady gradient solitons with positive curvature that is not the Bryant soliton, where $p$ is the critical point of $f$, and $f(p)=0$.

First, let $\{\chi_{\theta}\}_{\theta\in\mathbb{R}}$, be the one parameter group of isometries generated by $X$, we show that $X$ and $\nabla f$ are commutative, and hence the diffeomorphisms they generate are commutative.
\begin{lem}\label{l: comm}
$\left[X,\nabla f\right]=0$, and $\chi_{\theta}\circ\phi_t=\phi_t\circ\chi_{\theta}$, $t\in\R,\theta\in\R$.
\end{lem}

\begin{proof}
We first show that the potential function is invariant under $\chi_{\theta}$. Let $p$ be the critical point of $f$. Since $p$ is the unique maximum point of $R$, we have $\chi_{\theta}(p)=p$ for all $t$, and hence 
\begin{equation*}
    f\circ\chi_{\theta}(p)=f(p)=0,\quad \nabla (f\circ\chi_{\theta})(p)=\nabla f(p)=0,\quad \nabla^2(f\circ\chi_{\theta})=\Ric=\nabla^2 f.
\end{equation*}
For any $x\in M$, let $\sigma:[0,1]\ri M$ be a minimizing geodesic from $p$ to $x$, then
\begin{equation*}\begin{split}
    f(\chi_{\theta}(x))&=f(\chi_{\theta}(p))+\int_0^1\int_0^r\nabla^2(f\circ\chi_{\theta})(\sigma'(s),\sigma'(s))\,ds\,dr\\
    &=f(p)+\int_0^1\int_0^r\nabla^2f(\sigma'(s),\sigma'(s))\,ds\,dr
    =f(x).
\end{split}\end{equation*}
So $f\circ\chi_{\theta}\equiv f$.
Now since $\chi^*_{\theta}(f)=f$ and $\chi^*_{\theta}g=g$, it is easy to see 
$\chi^*_{\theta}\left(\nabla f\right)=\nabla f$.
So $\left[X,\nabla f\right]=0$ and hence $\chi_{\theta}\circ\phi_t=\phi_t\circ\chi_{\theta}$. 
\end{proof}

Second, we show that $\chi_{\theta}$ is a $SO(2)$-isometry.

\begin{lem}
There exists $\lambda>0$ such that after replacing $\{\chi_{\theta}\}$ by $\{\chi_{\lambda\theta}\}$, we have that $\{\chi_{\theta}\}$ is a $SO(2)$-isometry on $M$.
\end{lem}

\begin{proof}
Since $f$ is invariant under $\chi_{\theta}$, it follows that the level sets of $f$ is invariant under $\chi_{\theta}$. So $\chi_{\theta}$ induces an isometry on each level set of $f$.
Since the level sets $f^{-1}(a)$, $a>0$, are compact and diffeomorphic to $S^2$, it is easy to see that $X|_{f^{-1}(a)}$ vanishes at exactly two points, and $\chi_{\theta}|_{f^{-1}(a)}$ acts by rotations with two fixed points.

Therefore, after replacing $X$ by $\lambda X$ for some $\lambda>0$ we may assume that $\chi_{\theta}|_{f^{-1}(a)}=id$ if and only if $\theta=2k\pi$, for $k\in\mathbb{Z}$.
In particular, for a point $y\in f^{-1}(a)$, we have $\chi_{2\pi}(y)=y$, and $(\chi_{2\pi}|_{f^{-1}(a)})_{*y}$ is the identity transformation of the tangent space $T_yf^{-1}(a)$.
Since $\chi_{\theta}$ is a smooth family of diffeomorphisms, and $\chi_0=id$, it follows that $\chi_{2\pi}$ preserves the orientation. So $(\chi_{2\pi})_{*y}$ is the identity transformation of $T_yM$, and hence $\chi_{2\pi}=id$.
Therefore, $\chi_{\theta}$, $\theta\in[0,2\pi)$ is a $SO(2)$-isometry. 
\end{proof}

Next, we show that the fixed point set
Let $\Gamma'=\{x\in M: X(x)=0\}=\{x\in M: \chi_{\theta}(x)=x,\theta\in\R\}$ of the $SO(2)$-isometry $\chi_{\theta}$ coincides with $\Gamma=\Gamma_1(-\infty,\infty)\cup\Gamma_2(-\infty,\infty)\cup\{p\}$, where $\Gamma_1,\Gamma_2$ are two integral curves of $\nabla f$ from Corollary \ref{l: new Gamma}.

\begin{lem}
$\Gamma=\Gamma'$.
\end{lem}

\begin{proof}
Note that $\Gamma\setminus\{p\},\Gamma'\setminus\{p\}$ are both unions of two integral curves of $\nabla f$.
Let $\Gamma'_1,\Gamma'_2$ are the two components of $\Gamma'\setminus\{p\}$.
It suffices to show that for each $j=1,2$, the integral curves $\Gamma_j,\Gamma'_j$, intersect at some point, after possibly switching the order of $\Gamma_1$ and $\Gamma_{2}$.
To see this, note that on the one hand, by Corollary \ref{l: new Gamma} we have that the manifolds $(M,r^{-2}(x)g,x)$ converge to $(\R\times\cigar,r^{-2}(x_{tip})g_c,x_{tip})$ for any sequence $x\rii$ along $\Gamma$.
On the other hand, since the points on $\Gamma'$ are fixed points of the $SO(2)$-isometry, it is easy to see that the manifolds $(M,r^{-2}(x)g,x)$ converge to $(\R\times\cigar,r^{-2}(x_{tip})g_c,x_{tip})$ for any sequence $x\rii$ along $\Gamma'$.

Therefore, for any $i\in\mathbb{N}$, switching the order of $\Gamma_1$ and $\Gamma_{2}$ we may assume that there are two points $x_i\in\Gamma_1\cap (M\setminus B_g(p,2))$ and $y_i\in\Gamma'_1\cap(M\setminus B_g(p,2))$ such that $d_g(x_i,y_i)<i^{-1}$. Let $t_i>0$ be a constant such that $\phi_{-t_i}(x_i)\in B_g(p,2)\setminus B_g(p,1)$. Then
\begin{equation*}
    d_g(\phi_{-t_i}(x_i),\phi_{-t_i}(y_i))=d_{g(t_i)}(x_i,y_i)\le d_g(x_i,y_i)\le i^{-1}\ri0.
\end{equation*}
So after passing to a subsequence we may assume $\phi_{-t_i}(x_i), \phi_{-t_i}(y_i)\ri q\neq p$, and hence
$q\in\Gamma_1\cap\Gamma'_1$ and $\Gamma_1=\Gamma_2$. Similarly we can show $\Gamma_2=\Gamma'_2$. 
\end{proof}

Lastly, we prove the $O(2)$-symmetry, that is, there exist a totally geodesic surface $N\subset M$ and a diffeomorphism $\Phi: N\times S^1\ri M\setminus\Gamma$ such that the pull-back metric $\Phi^*g$ is a warped-product metric $\Phi^*g=g_N+\varphi^2d\theta^2$, $\theta\in[0,2\pi)$, where $g_N$ is the induced metric on $N$ and $\varphi$ is a function on $N$. 

\begin{lem}
The $SO(2)$-isometry $\chi_{\theta}$ is an $O(2)$-isometry.
\end{lem}

\begin{proof}
Let $\Sigma=f^{-1}(a)$ for some fixed $a>0$, and $\sigma:[0,1]\ri \Sigma$ be a minimizing geodesic in $\Sigma$ connecting the two fixed points $\{x_a,\overline{x}_a\}$.
Let $\Phi:(0,1)\times(-\infty,\infty)\times [0,2\pi)\ri M\setminus\Gamma$ be a diffeomorphism defined as
\begin{equation*}
    \Phi(r,t,\theta)=\phi_t(\chi_{\theta}(\sigma(r)))=\chi_{\theta}(\phi_t(\sigma(r)).
\end{equation*}
Then we can write the metric under this coordinate as
\begin{equation*}
    g=\sum_{\alpha,\beta=r,t}g_{\alpha\beta}dx_{\alpha}dx_{\beta}+G(d\theta+A)^2.
\end{equation*}
Since the vectors $\partial_{r},\partial_t=\nabla f,\partial_{\theta}=X$ are orthogonal at all points in $\Sigma\setminus\{x_a,\overline{x}_a\}=\Phi(\{(r,0,\theta): r\in(0,1),\theta\in[0,2\pi)\})$, 
the connection form $A=A_r\,dr+A_t\,dt$ vanishes at these points.
Moreover, we have $A_t=0$
everywhere because $\langle\partial_t,\partial_{\theta}\rangle=\langle\nabla f,X\rangle=0$. 

On the one hand, by the curvature formula \eqref{e: DR} for a $SO(2)$-symmetric metric we have
\begin{equation}\label{e: Ricextra}
    \Ric_{t\theta}=R_{\theta r  t r}=-\frac{1}{2}G\Phi_{rt;r}-\frac{3}{4}G_{,r}F_{rt}.
\end{equation}
On the other hand, by the soliton equation $\Ric=\nabla^2 f$ we have
\begin{equation}\label{e: Ricci by the soliton}
    \Ric_{t\theta}=\nabla^2_{t\theta} f=-\left\langle\nabla_{\nabla f}\nabla f,\partial_{\theta}\right\rangle=-\left\langle\nabla_{\partial_t}\partial_t,\partial_{\theta}\right\rangle
    =\frac{1}{2}\partial_{\theta}\left\langle\partial_t,\partial_t\right\rangle=0,
\end{equation}
where we used $\left\langle\partial_t,\partial_{\theta}\right\rangle=\langle X,-\nabla f\rangle=0$ and $\partial_{\theta}\left\langle\partial_t,\partial_t\right\rangle=X(|\nabla f|^2)=0$.
So by \eqref{e: Ricextra} and \eqref{e: Ricci by the soliton} we have
\begin{equation}\label{e: der of F is zero}
    \widetilde{\nabla}_{\partial_{r}}(G^{3/2}dA)=G^{\frac{1}{2}}(GF_{rt;r}+\frac{3}{2}G_{,r}F_{rt})\,dr\wedge dt=0,
\end{equation}
at points in $\Sigma\setminus\{x_a,\overline{x}_a\}$.

\begin{claim}
$dA=0$ holds on $M\setminus\Gamma$.
\end{claim}

\begin{proof}[Proof of the claim]

Consider the rescaled manifolds $(M,r_i^{-2}g,x_a)$ where $r_i>0$ is an arbitrary sequence going to zero.
Then it is easy to see that
$(M,r_i^{-2}g,q)$ smoothly converges to the Euclidean space $\R^3$, with $\Gamma$ converging to a straight line, which we may assume to be the $z$-axis after a change of coordinates.
So the $SO(2)$-isometries on $(M,r_i^{-2}g,q)$ converges to the rotation around the $z$-axis.
Note that $|G\,dA|$ is scaling invariant, this convergence implies $|G\,dA|(r_i,\theta,0)\ri0$ as $i\rii$, which proves $\lim_{r\ri0}|G\,dA|(r,\theta,0)=0$. So $\lim_{r\ri0}|G^{3/2}dA|(r,\theta,0)=0$.
Then by \eqref{e: der of F is zero}, we get $G^{3/2}dA(r,\theta,0)=0$ and $dA(r,\theta,0)=0$. So $dA=0$ on $\Sigma\setminus\{x_a,\overline{x}_a\}$.

Note that we may choose $\Sigma$ to be $f^{-1}(a)$ for any $a>0$, the same argument implies $dA=0$ everywhere on $M\setminus\Gamma$, which proves the claim.
\end{proof}

Now since $dA=0$ and $A_t=0$, we have $\frac{\partial A_r}{\partial t}=\frac{\partial A_t}{\partial r}=0$. Note $A_r(r,0,\theta)$=0, this implies $A_r(r,t,\theta)=0$ for all $t\in\R$.
So $A=0$ and hence the metric can be written as the following warped-product form under the coordinates $(r,t,\theta)$, 
\begin{equation*}
    g=\sum_{\alpha,\beta=r,t}g_{\alpha\beta}dx_{\alpha}dx_{\beta}+Gd\theta^2.
\end{equation*}

\end{proof}

Using the $O(2)$-symmetry and the $\mathbb{Z}_2$-symmetry at infinity, we can follow the same line as in \cite[Theorem 1.5]{Lai2020_flying_wing} to prove Theorem \ref{t': quantitative relation}.

\bibliography{bib}

\begin{thebibliography}{10}

\bibitem{AC}
G.~Anderson and B.~Chow.
\newblock A pinching estimate for solutions of the linearized ricci flow system
  on 3-manifolds.
\newblock {\em Calculus of Variations and Partial Differential Equations},
  23:1--12, 2002.

\bibitem{angenent2022unique}
S.~Angenent, S.~Brendle, P.~Daskalopoulos, and N.~{\v{S}}e{\v{s}}um.
\newblock Unique asymptotics of compact ancient solutions to three-dimensional
  ricci flow.
\newblock {\em Communications on Pure and Applied Mathematics},
  75(5):1032--1073, 2022.

\bibitem{Bakas2009AncientSO}
I.~Bakas, S.~Kong, and L.~Ni.
\newblock Ancient solutions of ricci flow on spheres and generalized hopf
  fibrations.
\newblock {\em Journal für die reine und angewandte Mathematik (Crelles
  Journal)}, 663, 2012.

\bibitem{Bamler_thesis}
R.~Bamler.
\newblock {Ricci flow with surgery}.
\newblock {\em diploma thesis}, 2007.

\bibitem{Bamler2020CompactnessTO}
R.~Bamler.
\newblock Compactness theory of the space of super ricci flows.
\newblock {\em arXiv: Differential Geometry}, 2020.

\bibitem{Bamler2021OnTR}
R.~Bamler and B.~Kleiner.
\newblock On the rotational symmetry of 3-dimensional $\kappa$-solutions.
\newblock {\em Journal f{\"u}r die reine und angewandte Mathematik (Crelles
  Journal)}, 2021:37 -- 55, 2021.

\bibitem{BamA}
R.~H. Bamler.
\newblock {Long-time behavior of 3-dimensional ricci flow A: Generalizations of
  perelman's long-time estimates}.
\newblock {\em Geometry and Topology}, 22(2):775--844, 2018.

\bibitem{BamD}
R.~H. Bamler.
\newblock {Long-time behavior of 3-dimensional ricci flow D: Proof of the main
  results}.
\newblock {\em Geometry and Topology}, 22(2):949--1068, 2018.

\bibitem{bamler2022uniqueness}
R.~H. Bamler and B.~Kleiner.
\newblock Uniqueness and stability of ricci flow through singularities.
\newblock {\em Acta Mathematica}, 228(1):1--215, 2022.

\bibitem{Bourni_convex}
T.~Bourni, M.~Langford, and G.~Tinaglia.
\newblock {Convex ancient solutions to mean curvature flow}.
\newblock {\em arXiv:1907.03932}, 2019.

\bibitem{Bourni_jdg}
T.~Bourni, M.~Langford, and G.~Tinaglia.
\newblock {Collapsing ancient solutions of mean curvature flow}.
\newblock {\em Journal of Differential Geometry}, 119(2):187 -- 219, 2021.

\bibitem{brendlesteady3d}
S.~Brendle.
\newblock {Rotational symmetry of self-similar solutions to the Ricci flow}.
\newblock {\em Inventiones Mathematicae}, 194(3):731--764, 2013.

\bibitem{Brendle_jdg_high}
S.~Brendle.
\newblock {Rotational symmetry of Ricci solitons in higher dimensions}.
\newblock {\em Journal of Differential Geometry}, 97(2):191 -- 214, 2014.

\bibitem{Brendle2019UniquenessOC}
S.~Brendle and K.~Choi.
\newblock Uniqueness of convex ancient solutions to mean curvature flow in
  $\mathbb{R}^3$.
\newblock {\em Inventiones mathematicae}, 217:35--76, 2019.

\bibitem{Brendle2021OC}
S.~Brendle and K.~Choi.
\newblock Uniqueness of convex ancient solutions to mean curvature flow in
  higher dimensions.
\newblock {\em Geometry and Topology}, 2021.

\bibitem{BrendleNaffDasSesum}
S.~Brendle, P.~Daskalopoulo, K.~Naff, and N.~Sesum.
\newblock Uniqueness of compact ancient solutions to the higher dimensional
  ricci flow.
\newblock {\em arXiv:2102.07180}, 2021.

\bibitem{Brendle2011AncientST}
S.~Brendle, G.~Huisken, and C.~Sinestrari.
\newblock Ancient solutions to the ricci flow with pinched curvature.
\newblock {\em Duke Mathematical Journal}, 158:537--551, 2011.

\bibitem{brendle2023rotational}
S.~Brendle and K.~Naff.
\newblock Rotational symmetry of ancient solutions to the ricci flow in higher
  dimensions.
\newblock {\em Geometry \& Topology}, 27(1):153--226, 2023.

\bibitem{bryant}
R.~Bryant.
\newblock {Ricci flow solitons in dimension three with SO (3)-symmetries}.
\newblock {\em preprint, Duke Univ}, pages 1--24, 2005.

\bibitem{BuragoBuragoIvanov}
D.~Burago, Y.~D. Burago, and S.~O. Ivanov.
\newblock {\em A Course in Metric Geometry}, volume~33.
\newblock Graduate Studies in Mathematics. Providence, RI: American
  Mathematical Society, 2001.

\bibitem{BGP}
Y.~Burago, M.~Gromov, and G.~Perelman.
\newblock {A.D. Alexandrov spaces with curvature bounded below}.
\newblock {\em Russian Mathematical Surveys}, 47(2):1--58, 1992.

\bibitem{CaoHD}
H.-D. Cao.
\newblock {Recent Progress on Ricci Solitons}.
\newblock {\em Advanced Lectures in Mathematics}, 11:1--38, 2010.

\bibitem{Cao-Kahler}
H.-D. Cao.
\newblock Existence of gradient kahler-ricci solitons.
\newblock {\em Elliptic and Parabolic Methods in Geometry}, 03 2012.

\bibitem{infinitesimal}
H.~D. Cao and C.~He.
\newblock {Infinitesimal rigidity of collapsed gradient steady Ricci solitons
  in dimension three}.
\newblock {\em Communications in Analysis and Geometry}, 26(3):505--529, 2018.

\bibitem{Catino}
G.~Catino, P.~Mastrolia, and D.~D. Monticelli.
\newblock {Classification of expanding and steady ricci solitons with integral
  curvature decay}.
\newblock {\em Geometry and Topology}, 20(5):2665--2685, 2016.

\bibitem{Chan2019CurvatureEF}
P.-Y. Chan.
\newblock Curvature estimates for steady ricci solitons.
\newblock {\em Transactions of the American Mathematical Society}, 2019.

\bibitem{chan2022dichotomy}
P.-Y. Chan and B.~Zhu.
\newblock On a dichotomy of the curvature decay of steady ricci solitons.
\newblock {\em Advances in Mathematics}, 404:108458, 2022.

\bibitem{ChenBL}
B.-L. Chen.
\newblock {Strong Uniqueness of the Ricci Flow}.
\newblock {\em J. Differential Geom.}, 82(2):363--382, 2007.

\bibitem{Chen-Zhu-Pinch}
B.-L. Chen and X.-P. Zhu.
\newblock Complete riemannian manifolds with pointwise pinched curvature.
\newblock {\em Inventiones Mathematicae}, 140:423--452, 05 2000.

\bibitem{Chen2011OnFA}
X.~Chen and Y.~Wang.
\newblock On four-dimensional anti-self-dual gradient ricci solitons.
\newblock {\em The Journal of Geometric Analysis}, 25:1335--1343, 2011.

\bibitem{Chow2007a}
B.~Chow, S.-C. Chu, and D.~Glickenstein.
\newblock {The Ricci flow: techniques and applications Volume 2– Part I :
  Geometric Aspects}.
\newblock {\em Part I: Geometric Aspects {\ldots}}, 2, 2007.

\bibitem{RFTandA3}
B.~Chow, S.-C. Chu, D.~Glickenstein, C.~Guenther, J.~Isenberg, T.~Ivey,
  D.~Knopf, P.~Lu, F.~Luo, and L.~Ni.
\newblock {The Ricci flow: techniques and applications. Part III:
  geometric-analytic aspects}.
\newblock {\em American Mathematical Society}, 163, 2010.

\bibitem{chow2022four}
B.~Chow, Y.~Deng, and Z.~Ma.
\newblock On four-dimensional steady gradient ricci solitons that dimension
  reduce.
\newblock {\em Advances in Mathematics}, 403:108367, 2022.

\bibitem{HaRF}
B.~Chow, P.~Lu, and L.~Ni.
\newblock {\em {Hamilton's Ricci Flow}}.
\newblock American Mathematical Society, 2006.

\bibitem{Chu}
S.-C. Chu.
\newblock {Type II ancient solutions to the Ricci flow on surface}.
\newblock {\em Comm. Anal. Geom.}, 15(1):195--215, 2007.

\bibitem{2dancientcompact}
P.~Daskalopoulos, R.~Hamilton, and N.~Sesum.
\newblock {Classification of ancient compact solutions to the ricci flow on
  surfaces}.
\newblock {\em Journal of Differential Geometry}, 91(2):171--214, 2012.

\bibitem{Sesum}
P.~Daskalopoulos and N.~Sesum.
\newblock {Eternal Solutions to the Ricci Flow on $\mathbb{R}^2$}.
\newblock {\em Int Math Res Notic}, 2006.

\bibitem{DZ}
Y.~Deng and X.~Zhu.
\newblock {3d steady Gradient Ricci Solitons with linear curvature decay}.
\newblock {\em arXiv:1612.05713}, 2016.

\bibitem{De15}
A.~Deruelle.
\newblock {Smoothing out positively curved metric cones by Ricci expanders}.
\newblock {\em Geometric and Functional Analysis}, 26(1):188--249, 2016.

\bibitem{Cao2011BachflatGS}
H.~dong Cao, G.~Catino, Q.~Chen, C.~Mantegazza, and L.~Mazzieri.
\newblock Bach-flat gradient steady ricci solitons.
\newblock {\em Calculus of Variations and Partial Differential Equations},
  49:125--138, 2011.

\bibitem{Cao2009OnLC}
H.~dong Cao and Q.~Chen.
\newblock On locally conformally flat gradient steady ricci solitons.
\newblock {\em Transactions of the American Mathematical Society},
  364:2377--2391, 2009.

\bibitem{du2021hearing}
W.~Du and R.~Haslhofer.
\newblock Hearing the shape of ancient noncollapsed flows in r4.
\newblock {\em Comm. Pure Appl. Math.(to appear)}, 2021.

\bibitem{Evans2010PartialDE}
L.~C. Evans.
\newblock Partial differential equations, second edition.
\newblock 2010.

\bibitem{Fateev}
V.~A. Fateev.
\newblock The sigma model (dual) representation for a two-parameter family of
  integrable quantum field theories.
\newblock {\em Nucl. Phys.}, 473(B), 1996.

\bibitem{FIK}
M.~Feldman, T.~Ilmanen, and D.~Knopf.
\newblock {Rotationally Symmetric Shrinking and Expanding Gradient
  Kähler-Ricci Solitons}.
\newblock {\em Journal of Differential Geometry}, 65(2):169 -- 209, 2003.

\bibitem{cigar}
R.~Hamilton.
\newblock {The Ricci flow on surfaces}.
\newblock {\em Contemporary Mathematics}, 71:237--261, 1988.

\bibitem{Hamilton_singularity_formation}
R.~Hamilton.
\newblock {The formations of singularities in the Ricci Flow}.
\newblock {\em Surveys in Differential Geometry}, 2(1):7--136, 1993.

\bibitem{Hamilton_ric}
R.~S. Hamilton.
\newblock {Three-manifolds with positive Ricci curvature}.
\newblock {\em Journal of Differential Geometry}, 17(2):255 -- 306, 1982.

\bibitem{supR}
R.~S. Hamilton.
\newblock {Eternal solutions to the Ricci flow}.
\newblock {\em Journal of Differential Geometry}, 38(1):1--11, 1993.

\bibitem{distanceH}
R.~S. Hamilton.
\newblock {The formation of singularities in the Ricci flow}.
\newblock {\em Surveys in differential geometry}, 2:7--136, 1995.

\bibitem{Hatcher:478079}
A.~Hatcher.
\newblock {\em {Algebraic topology}}.
\newblock Cambridge Univ. Press, Cambridge, 2000.

\bibitem{Hatcher}
A.~Hatcher.
\newblock Notes on basic 3-manifold topology.
\newblock 2001.

\bibitem{HN}
H.-J. Hein and A.~Naber.
\newblock {New logarithmic Sobolev inequalities and an $\epsilon$-regularity
  theorem for the Ricci flow}.
\newblock {\em Communications on Pure and Applied Mathematics},
  67(9):1543--1561, 2014.

\bibitem{white}
D.~Hoffman, T.~Ilmanen, F.~Mart{\'{i}}n, and B.~White.
\newblock {Graphical translators for mean curvature flow}.
\newblock {\em Calculus of Variations and Partial Differential Equations},
  58(4), 2019.

\bibitem{Huisken2015ConvexAS}
G.~Huisken and C.~Sinestrari.
\newblock Convex ancient solutions of the mean curvature flow.
\newblock {\em Journal of Differential Geometry}, 101:267--287, 2015.

\bibitem{Lu_Wang_Ricci_shrinker}
B.~Kotschwar and L.~Wang.
\newblock {Rigidity of asymptotically conical shrinking gradient Ricci
  solitons}.
\newblock {\em Journal of Differential Geometry}, 100(1):55 -- 108, 2015.

\bibitem{Lai2019}
Y.~Lai.
\newblock {Ricci flow under local almost non-negative curvature conditions}.
\newblock {\em Advances in Mathematics}, 343:353--392, 2019.

\bibitem{Lai2020_flying_wing}
Y.~Lai.
\newblock {A family of 3d Steady Gradient solitons that are flying wings}.
\newblock {\em arXiv:2010.07272}, 2020.

\bibitem{Lai2021Producing3R}
Y.~Lai.
\newblock Producing 3d ricci flows with non-negative ricci curvature via
  singular ricci flows.
\newblock {\em Geometry and Topology}, 25-7:3629--3690, 2021.

\bibitem{lieberman1996second}
G.~M. Lieberman.
\newblock {\em Second order parabolic differential equations}.
\newblock World scientific, 1996.

\bibitem{Lott2007DimensionalRA}
J.~Lott.
\newblock Dimensional reduction and the long-time behavior of ricci flow.
\newblock {\em Commentarii Mathematici Helvetici}, 85:485--534, 2007.

\bibitem{MT}
J.~Morgan and G.~Tian.
\newblock {Ricci flow and the Poincare conjecture}.
\newblock {\em arXiv:math/0607607}, 2009.

\bibitem{MT2}
J.~W. Morgan and G.~Tian.
\newblock Completion of the proof of the geometrization conjecture.
\newblock {\em arXiv:0809.4040}, 2008.

\bibitem{Munteanu2019PoissonEO}
O.~Munteanu, C.-J.~A. Sung, and J.~Wang.
\newblock Poisson equation on complete manifolds.
\newblock {\em Advances in Mathematics}, 2019.

\bibitem{Pel1}
G.~Perelman.
\newblock {The entropy formula for the Ricci flow and its geometric
  applications.}
\newblock {\em http://arxiv.org/abs/math/0211159}, 2002.

\bibitem{petersen}
P.~Petersen.
\newblock {\em {Riemannian geometry}}.
\newblock Springer, second edition, 2006.

\bibitem{Shi1987derivative1}
W.-X. Shi.
\newblock {Deforming the metric on complete Riemannian manifolds}.
\newblock {\em Journal of differential geometry}, 30(1):223--301, 1987.

\bibitem{simon2021local}
M.~Simon and P.~M. Topping.
\newblock Local mollification of riemannian metrics using ricci flow, and ricci
  limit spaces.
\newblock {\em Geometry \& Topology}, 25(2):913--948, 2021.

\bibitem{Spruck2020CompleteTS}
J.~Spruck and L.~Xiao.
\newblock Complete translating solitons to the mean curvature flow in r3 with
  nonnegative mean curvature.
\newblock {\em American Journal of Mathematics}, 142:993--1015, 2020.

\bibitem{Wangxujia}
X.~J. Wang.
\newblock {Convex solutions to the mean curvature flow}.
\newblock {\em Annals of Mathematics}, 173(3):1185--1239, 2011.

\bibitem{ZhaoZhu}
Z.~Zhao and X.~Zhu.
\newblock Rigidity of the bryant ricci soliton.
\newblock {\em arXiv:2212.02889}, 2022.

\bibitem{Zhu2021RotationalSO}
J.~Zhu.
\newblock Rotational symmetry of uniformly 3-convex translating solitons of
  mean curvature flow in higher dimensions.
\newblock {\em arXiv:2103.16382}, 2021.

\bibitem{zhu2022so}
J.~Zhu.
\newblock {$SO(2)$ Symmetry of the Translating Solitons of the Mean Curvature
  Flow in $\mathbb{R}^4$}.
\newblock {\em Annals of PDE}, 8(1):6, 2022.

\end{thebibliography}
\bibliographystyle{abbrv}

\end{document}